\documentclass[12pt]{article}
\usepackage[a4paper,vmargin={3cm,3cm},hmargin={2.5cm,2.5cm},bottom=30mm]{geometry}
\pdfoutput=1
\usepackage{dsfont}
\usepackage{mathrsfs}
\usepackage{multirow}
\usepackage{slashbox,multirow}
\usepackage{threeparttable}
\usepackage{lscape}
\usepackage[sectionbib]{natbib}
\usepackage{fancyhdr,latexsym,amsmath,amsfonts,amssymb,amsbsy,amsthm,url}
\usepackage{amsmath}
\usepackage[utf8]{inputenc}
\usepackage{tikz}
\date{}

 \usepackage{color}
\usepackage[pdfpagemode=UseNone,bookmarksopen=false,colorlinks=blue]{hyperref}
\hypersetup{ colorlinks=true, linkcolor=blue, citecolor=blue,
  filecolor=blue, urlcolor=red}

\allowdisplaybreaks
\bibfont{\footnotesize}

\newtheorem{assumption}{Assumption}

\newtheorem{proposition}{Proposition}[section]
\newtheorem{theorem}{Theorem}[section]
\theoremstyle{remark}

\begin{document}



\title{\textbf{Quantile processes and their applications in finite populations}}
\author{Anurag Dey and Probal Chaudhuri\\
\textit{Indian Statistical Institute, Kolkata}}
\maketitle


\begin{abstract}
The weak convergence of the quantile processes, which are constructed based on different estimators of the finite population quantiles, is shown under various well-known sampling designs based on a superpopulation model. The results related to the weak convergence of these quantile processes are applied to find asymptotic distributions of the smooth $L$-estimators and the estimators of smooth functions of finite population quantiles. Based on these asymptotic distributions, confidence intervals are constructed for several finite population parameters like the median, the $\alpha$-trimmed means, the interquartile range and the quantile based measure of skewness. Comparisons of various estimators are carried out based on their asymptotic distributions. We show that the use of the auxiliary information in the construction of the estimators sometimes has an adverse effect on the performances of the smooth $L$-estimators and the estimators of smooth functions of finite population quantiles under several sampling designs. Further, the performance of each of the above-mentioned estimators sometimes becomes worse under sampling designs, which use the auxiliary information, than their performances under simple random sampling without replacement (SRSWOR). \\ 
\end{abstract}
\newpage


\section{Introduction}
 The estimation of the finite population median instead of the population mean is meaningful, when the population observations are generated from skewed and heavy-tailed distributions. The estimation of the population trimmed means, which are constructed based on the population quantile function, can also be considered for a similar reason. Moreover, the median and the trimmed means are more robust and resistant to outliers than the mean.  \cite{MR2748939}, \cite{MR1091862}, \cite{kuk1988estimation}, \cite{MR1007458}, \cite{rao1990estimating}, \cite{wang2011asymptotic}, etc. considered the estimation of the population median, whereas \cite{MR1292550} considered the estimation of the population trimmed means. The estimation of some specific population quantiles (eg., population quartiles) are also of interest because estimators of population parameters like interquartile range, quantile based measure of skewness (Bowley's measure of skewness), etc. can be constructed based on the estimators of the population quantiles. \cite{MR1091862} considered the estimation of the interquartile range, whereas \cite{MR1292550} considered the estimation of the Bowley's measure of skewness and several other functions of the population quantiles.
\par

The weak convergence of quantiles and qauntile processes were studied in classical set up, when sample observations are independently and identically distributed (i.i.d.) random variables from a probability distribution (see \cite{serfling2009approximation}, \cite{shorack2009empirical}, etc.). It becomes challenging, when we deal with samples drawn from a finite population using a without replacement sampling design. In this case, we face difficulty as sample observations are neither independent nor identical. It becomes more challenging, when we consider the quantile processes constructed based on estimators other than the sample quantile.  Furthermore, different quantile processes are considered under different sampling designs unlike in the case of i.i.d. sample observations. 
\par

 The weak convergence of several empirical processes were shown in the earlier literature (see \cite{MR3670194}, \cite{han2021complex} and references therein) under some conditions on sampling designs. These conditions seem to hold under only simple random sampling without replacement (SRSWOR), Poisson sampling design and rejective sampling design. There is no result available in the literature related to the weak convergence of empirical processes under Lahiri-Midzuno-Sen (LMS, see \cite{Lahiri1951sampling}, \cite{Midzuno1952sampling} and \cite{Sen1953sampling}), $\pi$PS (see \cite{MR1624693} and \cite{bondesson2006pareto}), Rao-Hartley-Cochran (RHC, see \cite{rao1962simple}) and stratified multistage cluster sampling designs (see \cite{MR1091862} and \cite{MR1292550}). These sampling designs, especially stratified multistage cluster sampling designs, are of practical importance in sample surveys. In this article,  we show the weak convergence of an empirical process similar to the H\'ajek empirical process considered in \cite{MR3670194} and \cite{han2021complex} under high entropy sampling designs, which include SRSWOR, LMS and high entropy $\pi$PS (HE$\pi$PS) sampling designs. We also show the weak convergence of the above empirical process under RHC and stratified multistage cluster sampling designs. 
\par 
 
Asymptotic results related to the weak convergence of empirical processes were applied to study the asymptotic behaviour of poverty rate (see \cite{MR3670194}) and to deal with different regression and classification problems (see \cite{han2021complex}). However, neither \cite{MR3670194} nor \cite{han2021complex} considered quantiles and quantile processes in the context of sample survey.  \cite{shi1990weak} proved strong and weak versions of Bahadur type representations for the sample quantile process under simple random sampling in the presence of superpopulation model. \cite{conti2015inference} constructed a quantile process based on the sample quantile, which is obtained by inverting the H\'ajek estimator of finite population distribution function under high entropy sampling designs (see \cite{MR1624693},  \cite{MR3670194}, \cite{cardot2014variance},\cite{conti2014estimation},  etc.). There is no available result related to the weak convergence of quantile processes based on well-known quantile estimators like the ratio (see \cite{rao1990estimating}), the difference (see \cite{rao1990estimating}), and the regression (see \cite{del2003difference} and \cite{rueda2004some}) estimators, which are constructed using an auxiliary information. There is also no result available in the literature related to the weak convergence of a quantile process under RHC and stratified multistage cluster sampling designs. In this paper, we establish the weak convergence of the quantile processes, which are constructed based on the sample quantile as well as the ratio, the difference and the regression estimators of the finite population quantile, under the aforementioned sampling designs using the weak convergence of empirical process, Hadamard differentiability of the quantile map and the functional delta method. The weak convergence of the empirical and the quantile processes are shown under a probability distribution, which is generated by a sampling design and a superpopulation model jointly. 
\par

In this article, we apply asymptotic results for quantile processes to derive asymptotic distributions of the smooth $L$-estimators (see \cite{MR1292550}) and the estimators of smooth functions of population quantiles. We estimate asymptotic variances of these estimators consistently. Confidence intervals for finite population parameters like the median, the $\alpha$-trimmed means, the interquartile range and the quantile based measure of skewness are constructed based on asymptotic distributions of these estimators.
\par

  We also compare several estimators based on their asymptotic distributions. It is shown that the use of the auxiliary information in the estimation stage may have an adverse effect on the performances of the smooth $L$-estimators and the estimators of smooth functions of population quantiles based on the ratio, the difference and the regression estimators under each of SRSWOR, LMS, HE$\pi$PS and RHC sampling designs. Moreover, each of the aforementioned estimators may have worse performance under HE$\pi$PS and RHC sampling designs, which use the auxiliary information, than under SRSWOR. In practice, SRSWOR is easier to implement than the sampling designs that use the auxiliary information. Thus the above result is significant in view of selecting the appropriate sampling design. 
\par

In this paper, it is further shown that the sample median is more efficient than the sample mean under SRSWOR, whenever the finite population observations are generated from some symmetric and heavy-tailed superpopulation distributions with the same superpopulation mean and median. A similar result is known to hold in the classical set up with i.i.d. sample observations. However, for the cases of symmetric superpopulation distributions with the same superpopulation mean and median, it is shown that the generalized regression (GREG) estimator of the finite population mean is more efficient than the sample median under SRSWOR, whenever there is substantial correlation present between the study and the auxiliary variables. This stands in contrast to what happens in the case of i.i.d. observations. 
\par

 In Section \ref{sec 3}, we give the expressions of  several estimators of the population quantile and construct quantile processes based on these estimators.  We present asymptotic results related to the weak convergence of  empirical and quantiles  processes in Sections \ref{sec 4} and \ref{subsec 3.3} for single stage and stratified multistage cluster sampling designs. Asymptotic results related to the smooth $L$-estimators and the estimators of smooth functions of population quantiles are presented in Section \ref{sec 6.2}. In Section \ref{sec 7}, we compare different estimators.  Some numerical results based on real data are presented in Section \ref{sec 5}. Proofs of several results are given in the \hyperref[appn 1]{Appendix}.
\section{Quantile processes based on different estimators}\label{sec 3}
Suppose that $\mathcal{P}$=$\{1,\ldots,N\}$ is a finite population of size $N$, and a sample $s$ of size $n$ $(< N)$ is a subset of $\mathcal{P}$ having size $n$. Further, suppose that $\mathcal{S}$ is the collection of all subsets of $\mathcal{P}$ having size $n$. Then, a sampling design $P(s)$ is a probability distribution on $\mathcal{S}$ such that $0 \leq P(s)\leq 1$ for all $s\in \mathcal{S}$ and $\sum_{s\in\mathcal{S}}P(s)$=$1$. In this article, we consider high entropy sampling designs, RHC sampling design and stratified multistage cluster sampling design with SRSWOR. Let $(Y_i,X_i)$ be the value of $(y,x)$ for the $i^{th}$ population unit, $i$=$1,\ldots,N$, where $y$ is a real valued study variable and $x$ is a positive real valued size/auxiliary variable. In sample survey, the population values $\{X_i\}_{i=1}^N$ on $x$ are assumed to be known and utilized to implement sampling designs as well as to construct estimators. 
\par

Let $F_{y,N}(t)$=$\sum_{i=1}^N \mathds{1}_{[Y_{i}\leq t]}/N$ be the finite population distribution function of $y$  for  $t\in \mathds{R}$. Then, the finite population $p^{th}$ quantile of $y$ is defined as $Q_{y,N}(p)$=$\inf\{t\in\mathds{R}:F_{y,N}(t)\geq p\}$  for  $0<p< 1$. The Horvitz-Thompson (HT) estimator $\sum_{i\in s}(N\pi_i)^{-1} \mathds{1}_{[Y_{i}\leq t]}$ (cf. \cite{horvitz1952generalization}) and the RHC estimator  $\sum_{i\in s}(A_i/NX_i) \mathds{1}_{[Y_{i}\leq t]}$ (cf. \cite{rao1962simple}) are well-known design unbiased estimators of $F_{y,N}(t)$. Here, $\pi_i$ is the inclusion probability of the $i^{th}$ population unit under any sampling design $P(s)$, and $A_i$ is the $x$ total of that group of population units formed in the first step of the RHC sampling design from which the $i^{th}$ population unit is selected in the sample (see the $2^{nd}$ paragraph following Theorem \ref{Cor 1} in Section \ref{sec 4}). A unified way of writing these estimators is $\sum_{i \in s}d(i,s)\mathds{1}_{[Y_{i}\leq t]}$. An estimator of $Q_{y,N}(p)$ can be constructed as $\inf\{t\in\mathds{R}:\sum_{i \in s}d(i,s)\mathds{1}_{[Y_{i}\leq t]}\geq p\}$ (see \cite{kuk1988estimation}). However, $\inf\{t\in\mathds{R}:\sum_{i \in s}d(i,s)\mathds{1}_{[Y_{i}\leq t]}\geq p\}$ is not well defined, when $\max_{t\in\mathds{R}}\sum_{i \in s}d(i,s)\mathds{1}_{[Y_{i}\leq t]}$=$\sum_{i\in s}d(i,s)< p$ for some $0<p<1$ and $s\in \mathcal{S}$.  On the other hand, $\sum_{i \in s}d(i,s)\mathds{1}_{[Y_{i}\leq t]}$ violates the properties of the distribution functions, when  $\max_{t\in\mathds{R}}\sum_{i \in s}d(i,s)\mathds{1}_{[Y_{i}\leq t]}>1$ for some $s\in \mathcal{S}$. To eliminate these problems,  we consider $\hat{F}_{y}(t)$=$\sum_{ i \in s }d(i,s)\mathds{1}_{[Y_{i}\leq t]}/\sum_{i \in s}d(i,s)$ (see \cite{conti2015inference}, \cite{kuk1988estimation} and \cite{wang2011asymptotic}) as an estimator of $F_{y,N}(t)$. Based on $\hat{F}_{y}(t)$, the sample $p^{th}$ quantile of $y$ is defined as 
\begin{equation}\label{eq 59}
\hat{Q}_{y}(p)=\inf\{t\in\mathds{R}:\hat{F}_{y}(t)\geq p\}.
\end{equation}
 Note that $\hat{F}_{y}(t)$ satisfies all the properties of a distribution function, and  $\max_{t\in\mathds{R}}\hat{F}_{y}(t)$=$1>p$ for any $0<p<1$ and $s\in \mathcal{S}$. Thus $\hat{Q}_{y}(p)$ is a well defined estimator of $Q_{y,N}(p)$. The estimator $\hat{Q}_{y}(p)$ was considered for $d(i,s)$=$(N\pi_i)^{-1}$ in \cite{conti2015inference}, \cite{wang2011asymptotic}, etc. We also consider $\hat{Q}_{y}(p)$ for $d(i,s)$=$A_i/NX_i$ under RHC sampling design. Further, we consider some estimators of $Q_{y,N}(p)$, which are constructed using the auxiliary variable $x$ in the estimation stage. Suppose that $Q_{x,N}(p)$ and $\hat{Q}_{x}(p)$ are the population and the sample $p^{th}$ quantiles of $x$, respectively. Then, the ratio, the difference and the regression estimators of $Q_{y,N}(p)$ are defined as 
\begin{align}
\begin{split}
&\hat{Q}_{y,RA}(p)=\big(\hat{Q}_{y}(p)/\hat{Q}_{x}(p)\big)Q_{x,N}(p),\\
&\hat{Q}_{y,DI}(p)=\hat{Q}_{y}(p)+\bigg(\sum_{i \in s}d(i,s)Y_i\bigg/\sum_{i \in s} d(i,s)X_i\bigg)(Q_{x,N}(p)-\hat{Q}_{x}(p))\text{ and }\\
&\hat{Q}_{y,REG}(p)=\hat{Q}_{y}(p)+\hat{\beta}(Q_{x,N}(p)-\hat{Q}_{x}(p)),
\end{split}
\end{align}
 respectively, where $\hat{\beta}$=$\sum_{i\in s}d(i,s)X_iY_i/\sum_{i \in s}d(i,s)X_i^2$. The estimators $\hat{Q}_{y,RA}(p)$ and $\hat{Q}_{y,DI}(p)$ were considered in \cite{del2003difference} and  \cite{rao1990estimating} for $d(i,s)$=$(N\pi_i)^{-1}$, and $\hat{Q}_{y,REG}(p)$ was considered in \cite{del2003difference} and \cite{rueda2004some} for $d(i,s)$=$(N\pi_i)^{-1}$. We also consider these estimators for $d(i,s)$=$A_i/NX_i$ under RHC sampling design.
\par

Now, suppose that for any $0<\alpha<\beta< 1$, $D[\alpha,\beta]$ is the space of all left continuous functions on $[\alpha,\beta]$ having right hand limits at each point,  and $\mathcal{D}$ is the $\sigma$-field on $D[\alpha,\beta]$ generated by the open balls (ball $\sigma$-field) with respect to the sup norm metric. Note that $\mathcal{D}$ coincides with the Borel $\sigma$-field on $D[\alpha,\beta]$ with respect to the Skorohod metric (cf. \cite{billingsley2013convergence} and \cite{shorack2009empirical}). Thus  the quantile processes  $\{\sqrt{n}(G(p)-Q_{y,N}(p)): p\in [\alpha,\beta
]\}$ for $G(p)$=$\hat{Q}_{y}(p)$, $\hat{Q}_{y,DI}(p)$, $\hat{Q}_{y,RA}(p)$ and $\hat{Q}_{y,REG}(p)$  are random functions in  $(D[\alpha,\beta],\mathcal{D})$. Following the notion of weak convergence in \cite{billingsley2013convergence} and \cite{shorack2009empirical}, we shall show that the above quantile processes converge weakly in $(D[\alpha,\beta],\mathcal{D})$ with respect to the sup norm metric (see Sections \ref{sec 4} and \ref{subsec 3.3}). The weak convergence in $(D[\alpha,\beta],\mathcal{D})$ with respect to the sup norm metric implies and is implied by the weak convergence in $(D[\alpha,\beta],\mathcal{D})$ with respect to the Skorohod metric given that the limiting process has almost sure continuous paths. 
\section{Weak convergence of quantile processes under single stage sampling designs}\label{sec 4}
We consider a superpopulation model such that $\{(Y_i,X_i):1\leq i\leq N\}$ are i.i.d. random vectors on some probability space $(\Omega,\mathcal{F},\textbf{P})$ (see \cite{MR3670194}, \cite{conti2014estimation}, \cite{conti2015inference} and references therein). Recall that the population values, $X_1,\ldots,X_N$, are sometimes used to implement sampling designs (eg., RHC sampling design, $\pi$PS sampling designs, etc.). In such cases, a sampling design can be expressed as a function $P(s,\omega)$ defined on $\mathcal{S}\times\Omega$ so that for each $s\in \mathcal{S}$, $P(s,\omega)$ is a random variable on $\Omega$, and for each $\omega\in \Omega$, $P(s,\omega)$ is a probability distribution on $\mathcal{S}$. If a sampling design is independent of $X_1,\ldots,X_N$ (eg., SRSWOR), then, for each $s\in \mathcal{S}$, the corresponding $P(s,\omega)$ is a degenerate random variable on $\Omega$. Suppose that $\mathcal{A}$  is the power set of $\mathcal{S}$.  Then, we consider the probability measure $\textbf{P}^*(B\times E)$=$\int_E \sum_{s\in B}P(s,\omega)\text{d}\textbf{P}(\omega)$ (see \cite{MR3670194} and \cite{han2021complex}) defined on the product space ($\mathcal{S}\times\Omega$, $\mathcal{A}\times\mathcal{F}$), where $B\in \mathcal{A}$, $E\in \mathcal{F}$ and $B\times E$ is a cylinder subset of $\mathcal{S}\times\Omega$. We denote expectations of random quantities with respect to $P(s,\omega)$, $\textbf{P}$ and $\textbf{P}^*$ by $E$, $E_{\textbf{P}}$ and $E_{\textbf{P}^{*}}$, respectively. We define our asymptotic framework as follows. Let $\{\mathcal{P}_{\nu}\}$ be a sequence of populations with $N_\nu,n_\nu\rightarrow \infty$ as $\nu \rightarrow \infty$ (see \cite{MR3670194}, \cite{conti2015inference}, \cite{wang2011asymptotic}, etc.), where $N_\nu$ and $n_\nu$ are, respectively, the population and sample sizes corresponding to the $\nu^{th}$ population. We suppress the limiting index $\nu$ for the sake of notational simplicity. 
\par

We shall first show the weak convergence of the quantile processes introduced in Section \ref{sec 3} under high entropy sampling designs. The sampling design $P(s,\omega)$ is called the high entropy sampling design (cf. \cite{MR1624693},  \cite{MR3670194}, \cite{cardot2014variance}, \cite{conti2014estimation}, etc.) if 
\begin{equation}
D(P||R)=\sum_{s\in\mathcal{S}}P(s,\omega)\log (P(s,\omega)/R(s,\omega))\rightarrow 0\text{ as }\nu\rightarrow \infty \textit{ a.s.}[\textbf{P}] 
\end{equation} 
for some rejective sampling design $R(s,\omega)$ (see \cite{MR0178555}).  Some examples of high entropy sampling designs are SRSWOR, Rao-Sampford (RS) sampling design (see \cite{MR1624693}), LMS sampling design (see \cite{deychaudhuri2023a} and \cite{deychaudhuri2023b}), etc.  
\par

Suppose that $F_y$ and $F_x$ are superpopulation distribution functions of $y$ and $x$, respectively. Further, suppose that $Q_y(p)$=$\inf\{t\in\mathds{R}:F_y(t)\geq p\}$ and $Q_x(p)$=$\inf\{t\in\mathds{R}:F_x(t)\geq p\}$ are superpopulation $p^{th}$ quantiles of $y$ and $x$, respectively, and $\textbf{V}_i$=$\textbf{G}_i-\sum_{i=1}^N \textbf{G}_i/N$ for $i$=$1,\ldots,N$, where 
$$
\textbf{G}_i=\big(\mathds{1}_{[Y_i\leq Q_y(p_1)]},\ldots,\mathds{1}_{[Y_i\leq Q_y(p_k)]},\mathds{1}_{[X_i\leq Q_x(p_1)]}\ldots,\mathds{1}_{[X_i\leq Q_x(p_k)]}\big)
$$ 
 for $p_1,\ldots,p_k\in (0,1)$ and $k\geq 1$. Moreover, let $\textbf{T}$=$\sum_{i=1}^N \textbf{V}_i(1-\pi_i)/\sum_{i=1}^N \pi_i(1-\pi_i)$, where $\{\pi_i\}_{i=1}^N$ denote inclusion probabilities. All vectors in Euclidean spaces will be taken as row vectors and superscript $T$ will be used to denote their transpose. Before, we state the main results, let us consider the following assumptions. 
\begin{assumption}\label{ass 0} 
$n/N\rightarrow \lambda$ as $\nu\rightarrow\infty$, where $0<\lambda<1$.
\end{assumption}
\begin{assumption}\label{ass 2}
The inclusion probabilities $\{\pi_i\}_{i=1}^N$ are such that the following hold.
\begin{itemize}
\item[\textsl{(i)}] Given any $k\geq 1$ and $p_1,\ldots,p_k\in (0,1)$, $(n/N^2)\sum_{i=1}^{N}(\textbf{V}_i-\textbf{T}\pi_i)^T(\textbf{V}_i-\textbf{T}\pi_i)(\pi_i^{-1}-1)\rightarrow \Gamma$ as $\nu\rightarrow\infty$ \textit{a.s.} $[\textbf{P}]$ for some positive definite (p.d.) matrix $\Gamma$.
\item[\textsl{(ii)}] There exist constants $K_1,K_2>0$ such that for all $i$=$1,\ldots,N$ and $\nu\geq 1$, $K_1\leq N \pi_{i}/n\leq K_2$ \textit{a.s.} $[\textbf{P}]$.
\end{itemize}
\end{assumption}
Suppose that  $supp(F)$=$(a_1,a_2)$  is the support (see \cite{conti2015inference}) of any distribution function $F$, where $a_1$=$\sup\{t\in \mathds{R}:F(t)=0\}$ and $a_2$=$\inf\{t\in\mathds{R}:F(t)=1\}$. Note that $-\infty\leq a_1<a_2\leq\infty$. Then, we consider the following assumption on superpopulation distributions of $y$ and $x$.
\begin{assumption}\label{ass 3}
Superpopulation distribution functions $F_y$ of $y$ and $F_x$ of $x$ are continuous and are differentiable with positive continuous derivatives $f_y$ and $f_x$ on $supp(F_y)\subseteq (-\infty,\infty)$ and $supp(F_x)\subseteq (0,\infty)$, respectively.
\end{assumption}
 Assumptions \ref{ass 0}--\ref{ass 3} were considered earlier in the literature (see \cite{MR3670194}, \cite{conti2014estimation}, \cite{conti2015inference}, \cite{deychaudhuri2023a}, \cite{wang2011asymptotic}, etc.)   It can be shown that Assumption \ref{ass 2} holds under SRSWOR, LMS and any $\pi$PS sampling designs (see Lemma $1.1$ in the supplement (\cite{supp2024})).  Assumptions \ref{ass 0} and \ref{ass 2} are required to show the finite dimensional convergence of the empirical process $\{\sqrt{n}(\hat{F}_{z}(t)-t):t\in [0,1]\}$ for $d(i,s)$=$(N\pi_i)^{-1}$ under high entropy sampling designs, where 
\begin{equation}\label{eq 36}
\hat{F}_{z}(t)=\sum_{i\in s}d(i,s)\mathds{1}_{[Z_i\leq t]}\bigg/\sum_{i\in s}d(i,s)\text{ and }Z_i=F_{y}(Y_i)
\end{equation}
for $i$=$1,\ldots,N$ and $0\leq t\leq 1$. Here, $F_y$ is as in the paragraph preceding Assumption \ref{ass 0}. On the other hand, Assumptions \ref{ass 0}, \ref{ass 2}-\textsl{(ii)} and \ref{ass 3} are used to establish the tightness of this empirical process under the same sampling designs. Based on the weak convergence of this empirical process, we shall prove the weak convergence of the aforementioned quantile processes under high entropy sampling designs.  Suppose that $\tilde{D}[0,1]$ is the class of all right continuous functions defined on $[0,1]$ with finite left limits, and $\tilde{\mathcal{D}}$ is the $\sigma$-field on $\tilde{D}[0,1]$ generated by the open balls (ball $\sigma$-field) with respect to the sup norm metric. Then, we state the following proposition.
\begin{proposition}\label{lem s1}
Suppose that Assumptions \ref{ass 0} and \ref{ass 3} hold. Then, under $\textbf{P}^*$,
$$
\big\{\sqrt{n}(\hat{F}_{z}(t)-t):t\in [0,1]\big\}\xrightarrow{\mathcal{L}} \mathds{H} \text{ as }\nu\rightarrow\infty
$$ 
in $(\tilde{D}[0,1],\tilde{\mathcal{D}})$ with respect to the sup norm metric, for $d(i,s)$=$(N\pi_i)^{-1}$ and any high entropy sampling design satisfying Assumption \ref{ass 2}, where $\mathds{H}$ is a mean $0$ Gaussian process in $\tilde{D}[0,1]$ with almost sure continuous paths.
\end{proposition}
\cite{MR3670194} and \cite{han2021complex} showed the weak convergence of a similar version of the above empirical process under some conditions on sampling designs (e.g., (HT$2$) in \cite{MR3670194}, and (F$2$) and (F$3$) in \cite{han2021complex}). These conditions hold under very few sampling designs (with fixed sample size) like SRSWOR and rejective sampling designs. We are able to dispense with those conditions and show the weak convergence of $\big\{\sqrt{n}(\hat{F}_{z}(t)-t):t\in [0,1]\big\}$ for $d(i,s)$=$(N\pi_i)^{-1}$ under any high entropy sampling design satisfying Assumption \ref{ass 2}. Examples of such a sampling design are SRSWOR, LMS and HE$\pi$PS sampling designs. A sampling design is called HE$\pi$PS sampling design if it is a high entropy as well as a $\pi$PS sampling design (e.g., RS sampling design, rejective sampling design, etc.).  In particular, we are able to show the weak convergence of the aforementioned empirical process under LMS and HE$\pi$PS sampling designs, which are not covered in the earlier literature.  Now, we state the following theorem.
\begin{theorem}\label{Cor 1}
Fix any $0<\alpha<\beta<1$. Suppose that Assumptions \ref{ass 0} and \ref{ass 3} hold, and $E_{\textbf{P}}||\textbf{W}_i||^2<\infty$ for $\textbf{W}_i$=$(X_i,Y_i,X_iY_i,X_i^2)$. Then, under the probability distribution $\textbf{P}^*$, 
$$
\{\sqrt{n}(G(p)-Q_{y,N}(p)):p\in [\alpha,\beta]\}\xrightarrow{\mathcal{L}}\mathds{Q}\text{ as }\nu\rightarrow\infty 
$$
in  $(D[\alpha,\beta], \mathcal{D})$ with respect to the sup norm metric,  for any high entropy sampling design satisfying Assumption \ref{ass 2}, where $G(p)$ denotes one of $\hat{Q}_{y}(p)$, $\hat{Q}_{y,RA}(p)$, $\hat{Q}_{y,DI}(p)$ and $\hat{Q}_{y,REG}(p)$ with $d(i,s)$=$(N\pi_i)^{-1}$, and $\mathds{Q}$ is a mean $0$ Gaussian process in $D[\alpha,\beta]$ with  almost sure continuous path and p.d. covariance kernel 
\begin{align}\label{eq 52}
\begin{split}
&K(p_1,p_2)=\lim_{\nu\rightarrow\infty}(n/N^2)E_{\textbf{P}}\big(\sum_{i=1}^N(\zeta_i(p_1)-\overline{\zeta}(p_1)-S(p_1)\pi_i)\times\\
&(\zeta_i(p_2)-\overline{\zeta}(p_2)-S(p_2)\pi_i)(\pi_i^{-1}-1)\big)\text{ for }p_1,p_2\in [\alpha,\beta].
\end{split}
\end{align}
Here, $\overline{\zeta}(p)$=$\sum_{i=1}^N\zeta_i(p)/N$, $S(p)$=$\sum_{i=1}^N(\zeta_i(p)-\overline{\zeta}(p))(1-\pi_i)\big/\sum_{i=1}^N\pi_i(1-\pi_i)$, and $\zeta_i(p)$'s are as in Table \ref{table 4} below. 
\end{theorem}
The weak convergence of the quantile process constructed based on the sample quantile for $d(i,s)$=$(N\pi_i)^{-1}$ was considered earlier in \cite{conti2015inference} under a high entropy sampling design. However, in \cite{conti2015inference}, the author did not provide much details of the derivation of the main weak convergence result.  Using dominated convergence theorem (DCT) and Lemma $1.1$ in the supplement (\cite{supp2024}), $K(p_1,p_2)$ in \eqref{eq 52} can be expressed in terms of superpopulation moments under SRSWOR, LMS and any HE$\pi$PS sampling designs as in Table \ref{table 5} below. 

\begin{table}[h]

\caption{Expressions of $\zeta_i(p)$'s appearing in \eqref{eq 52} and \eqref{eq 51} for different $G(p)$'s in the cases of high entropy and RHC sampling designs}
\label{table 4}
\begin{center}
\begin{tabular}{cc}
\hline 
$G(p)$& $\zeta_i(p)$\\
\hline
$\hat{Q}_{y}(p)$& $\mathds{1}_{[Y_i\leq Q_y(p)]}/f_y(Q_y(p))$\\
\hline
$\hat{Q}_{y,RA}(p)$ & $\mathds{1}_{[Y_i\leq Q_y(p)]}/f_y(Q_y(p))-(Q_y(p)/Q_x(p))\mathds{1}_{[X_i\leq Q_x(p)]}/f_x(Q_x(p))$\\
\hline
$\hat{Q}_{y,DI}(p)$& $\mathds{1}_{[Y_i\leq Q_y(p)]}/f_y(Q_y(p))-(E_{\textbf{P}}(Y_i)/E_{\textbf{P}}(X_i))\mathds{1}_{[X_i\leq Q_x(p)]}/f_x(Q_x(p))$\\
\hline
$\hat{Q}_{y,REG}(p)$ &$\mathds{1}_{[Y_i\leq Q_y(p)]}/f_y(Q_y(p))-(E_{\textbf{P}}(X_iY_i)/E_{\textbf{P}}(X_i)^2)\mathds{1}_{[X_i\leq Q_x(p)]}/f_x(Q_x(p))$\\
\hline
\end{tabular}
\end{center}
\end{table}

\begin{table}[h] 

\caption{$K(p_1,p_2)$ in \eqref{eq 52} under different high entropy sampling designs}
\label{table 5}
\begin{threeparttable}[b]
\renewcommand{\arraystretch}{1.2}
\centering
\begin{tabular}{cc} 
\hline
Sampling design & $K(p_1,p_2)$\\ 
\hline
SRSWOR and &\multirow{2}{*}{$(1-\lambda)E_{\textbf{P}}\big[\zeta_i(p_1)-E_{\textbf{P}}(\zeta_i(p_1))\big]\big[\zeta_i(p_2)-E_{\textbf{P}}(\zeta_i(p_2))\big]$}\\
LMS&\\
\hline
\multirow{3}{*}{HE$\pi$PS} & \tnote{1} \ $E_{\textbf{P}}\big[\zeta_i(p_1)-E_{\textbf{P}}(\zeta_i(p_1))+\lambda \chi^{-1}\mu_x^{-1} X_i E_{\textbf{P}}\big((\zeta_i(p_1)-E_{\textbf{P}}(\zeta_i(p_1)))X_i \big)\big]\times$\\
&$\big[\zeta_i(p_2)-E_{\textbf{P}}(\zeta_i(p_2))+\lambda \chi^{-1}\mu_x^{-1} X_i E_{\textbf{P}}\big((\zeta_i(p_2)-E_{\textbf{P}}(\zeta_i(p_2)))X_i \big)\big]\times$\\
& $\big[\mu_x X_i^{-1}-\lambda]$\\
\hline
\end{tabular}
\begin{tablenotes}
\item[1] $\mu_x$=$E_{\textbf{P}}(X_i)$ and $\chi$=$\mu_x-(\lambda E_{\textbf{P}}(X_i)^2/\mu_x)$.
\end{tablenotes}
\end{threeparttable}
\end{table}
\par

 Next, we shall show the weak convergence of the quantile processes considered in Section \ref{sec 3} under RHC sampling design. In RHC sampling design (see \cite{chaudhuri2006feasibility} and \cite{rao1962simple}), $\mathcal{P}$ is first divided randomly into $n$ disjoint groups of sizes $N_1\cdots,N_n$, respectively, by taking a sample of $N_1$ units from $N$ units with SRSWOR, a sample of $N_2$ units from $N-N_1$ units with SRSWOR and so on. Then, one unit is selected in the sample from each of these groups independently with probability proportional to the size variable $x$. Before, we state the next results, let us consider some assumptions on the superpopulation distribution $\textbf{P}$.
\begin{assumption}\label{ass 4}
 There exists a constant $K$ such that $\max_{1\leq i\leq N} X_i/\min_{1\leq i\leq N}X_i\leq K$ \textit{a.s.} $[\textbf{P}]$.
\end{assumption}
\begin{assumption}\label{ass 5}
 The support of the joint distribution of $(Y_i,X_i)$ is not a subset of a straight line in $\mathds{R}^2$.
\end{assumption}
Let us also assume that $\{N_r\}_{r=1}^n$ are as follows (see \cite{rao1962simple}). 

\begin{equation}\label{eq 150}
N_r=
\begin{cases}
N/n,\text{ for } r=1,\cdots,n, \text{ when } N/n \text{ is an integer},\\
\lfloor N/n\rfloor,\text{ for } r=1,\cdots,k, \text{ and }\\
\lfloor N/n\rfloor+1,\text{ for } r=k+1,\cdots,n, \text{ when }N/n \text{ is not an integer},\\
\end{cases}
\end{equation}
where $k$ is such that $\sum_{r=1}^n N_r$=$N$. Here, $\lfloor N/n\rfloor$ is the integer part of $N/n$. Assumption \ref{ass 4} is equivalent to Assumption \ref{ass 2}–\textsl{(ii)} under any $\pi$PS sampling design.  Assumptions \ref{ass 4} and \ref{ass 5}, and \eqref{eq 150} were used earlier in \cite{deychaudhuri2023a}. These assumptions are required to show the finite dimensional convergence of the empirical process $\{\sqrt{n}(\hat{F}_z(t)-t)):t\in [0,1]\}$ for $d(i,s)$=$A_i/NX_i$ under RHC sampling design, where $A_i$'s are as in the $2^{nd}$ paragraph of Section \ref{sec 3}. Assumption \ref{ass 4}, and \eqref{eq 150} are also required to establish the tightness of this empirical process. As in the case of high entropy sampling designs, here also we shall show the weak convergence of several quantile processes based on the weak convergence of the above-mentioned empirical process. 
\begin{proposition}\label{lem 10}
Suppose that $E_{\textbf{P}}(X_i)^{-1}<\infty$, Assumptions \ref{ass 0} and \ref{ass 3}-\ref{ass 5} hold, and \eqref{eq 150}  holds. Then, the conclusion of Proposition \ref{lem s1} holds for $d(i,s)$=$A_i/NX_i$ and RHC sampling design.
\end{proposition} 
\begin{theorem}\label{Cor 3}
Fix any $0<\alpha<\beta<1$. Suppose that  $E_{\textbf{P}}(X_i)^{-1}<\infty$, $E_{\textbf{P}}||\textbf{W}_i||^2<\infty$ for $\textbf{W}_i$=$(X_i,Y_i,X_iY_i,X_i^2)$, Assumptions \ref{ass 0} and \ref{ass 3}--\ref{ass 5} hold,  and \eqref{eq 150}  holds.  Then, the conclusion of Theorem \ref{Cor 1} holds for $d(i,s)$=$A_i/NX_i$ and RHC sampling design with p.d. covariance kernel 
\begin{align}\label{eq 51}
\begin{split}
&K(p_1,p_2)=\lim_{\nu\rightarrow\infty}n\gamma E_{\textbf{P}}\big[(\overline{X}/N)\sum_{i=1}^N(\zeta_i(p_1)-\overline{\zeta}(p_1)) (\zeta_i(p_2)-\overline{\zeta}(p_2))X_i^{-1}\big]\\
&=c E_{\textbf{P}}(X_i)E_{\textbf{P}}\big[\big(\zeta_i(p_1)-E_{\textbf{P}}(\zeta_i(p_1))\big) \big(\zeta_i(p_2)-E_{\textbf{P}}(\zeta_i(p_2))\big)X_i^{-1}\big]\text{ for }p_1,p_2\in [\alpha,\beta].
\end{split}
\end{align}
Here, $\gamma$=$\sum_{r=1}^n N_r(N_r-1)/N(N-1)$, $c$=$\lim_{\nu\rightarrow\infty}n\gamma$, and $\zeta_i(p)$'s are as in Table \ref{table 4} above. 
\end{theorem}
 It follows from Lemma S$1$ in \cite{deychaudhuri2023} that $c$=$1-\lambda$ for $\lambda^{-1}$ an integer, and $c$=$\lambda \lfloor \lambda^{-1} \rfloor (2-\lambda \lfloor \lambda^{-1} \rfloor-\lambda)$ when $\lambda^{-1}$ is a non-integer. If we replace $Q_{y,N}$ by $Q_y$ in the quantile processes considered in this section, then the weak convergence of these quantile processes can be shown under high entropy and RHC sampling designs using the key ideas of the proofs of Theorems \ref{Cor 1} and \ref{Cor 3}.
\section{Weak convergence of quantile processes under stratified multistage cluster sampling design}\label{subsec 3.3}
Stratified multistage cluster sampling design with SRSWOR is used instead of single stage sampling designs mentioned in the preceding section, when heterogeneity is present in the population values of $(y,x)$. Suppose that the finite population $\mathcal{P}$ is divided into $H$ strata or subpopulations, where stratum $h$ consists of $M_h$ clusters for $h$=$1,\ldots,H$. Further, the $j^{th}$ cluster in stratum $h$ consists of $N_{hj}$ units for $j$=$1,\ldots,M_{h}$. For any given $h$=$1,\ldots,H$, $j$=$1,\ldots,M_h$ and $l$=$1,\ldots,N_{hj}$, we assume that the $l^{th}$ unit from cluster $j$ in stratum $h$ is the $i^{th}$ unit in the population $\mathcal{P}$, where $i$=$\sum_{h^{\prime}=1}^h \sum_{j^{\prime}=1}^{M_{h^{\prime}}}N_{h^{\prime}j^{\prime}}-\sum_{j^{\prime}=j}^{M_h} N_{h j^{\prime}}+l$. In stratified multistage cluster sampling design with SRSWOR, first a sample $s_h$ of $m_h$ ($<M_h$) clusters is selected from stratum $h$ under SRSWOR for each $h$. Then, a sample $s_{hj}$ of $r_h$ ($<N_{hj}$) units is selected from $j^{th}$ cluster in stratum $h$ if it is selected in the sample of clusters $s_h$ in the first stage for $h$=$1,\ldots,H$. Thus the resulting sample is $s$=$\cup_{1\leq h\leq H, j\in s_h}s_{hj}$. The samplings in two stages are done independently across the strata and the clusters. Under the above sampling design, the inclusion probability of the $i^{th}$ population unit is $\pi_i$=$m_h r_h/M_h N_{hj}$ if it belongs to the $j^{th}$ cluster of  stratum $h$. Note that stratified multistage cluster sampling design with SRSWOR becomes stratified sampling design with SRSWOR, when $N_{hj}$=$1$ for any $h$=$1,\ldots,H$ and $j$=$1,\ldots,M_h$. Also, note that stratified multistage cluster sampling design with SRSWOR becomes multistage cluster sampling design with SRSWOR, when $H$=$1$.
\par

Suppose that $(Y_{hjl}^{\prime},X_{hjl}^{\prime})$ denotes the value of $(y,x)$ corresponding to the $l^{th}$ unit from cluster $j$ in stratum $h$. Note that given any $h,j$ and $l$, $(Y_{hjl}^{\prime},X_{hjl}^{\prime})$=$(Y_i,X_i)$, where $i$=$\sum_{h^{\prime}=1}^h \sum_{j^{\prime}=1}^{M_{h^{\prime}}}N_{h^{\prime}j^{\prime}}-\sum_{j^{\prime}=j}^{M_h} N_{h j^{\prime}}+l$ and $(Y_i,X_i)$ is the value of $(y,x)$ corresponding to the $i^{th}$ population unit. We assume that for any given $h$, $\{(Y_{hjl}^{\prime},X_{hjl}^{\prime}):l=1,\ldots,N_{hj}, j=1,\ldots,M_h\}$ are i.i.d. random vectors defined on $(\Omega,\mathcal{F},\textbf{P})$ with marginal distribution functions $F_{y,h}$ and $F_{x,h}$, where $F_{y,h}$'s and $F_{x,h}$'s are not necessarily identical for varying $h$. We also assume that the population observations on $(y,x)$ in any stratum are independent of the observations in other strata. \cite{MR1091862} used a similar superpopulation model set up for studying the asymptotic behaviour of sample quantiles. However, they considered all $F_{y,h}$'s to be the same. Note that $H$, $M_h$, $N_{hj}$, $m_h$ and $r_h$ depend on $\nu$, when we consider the sequence of populations $\{P_{\nu}\}$. However, for simplicity, we omit $\nu$.  As in the cases of single stage sampling designs, here also we shall show the weak convergence of various quantile processes based on the weak convergence of the empirical process $\big\{\sqrt{n}(\hat{F}_{z}(t)-t):t\in [0,1]\big\}$ for $d(i,s)$=$(N\pi_i)^{-1}$.  
\par

First, we consider the case, when $H$ is fixed as $\nu\rightarrow\infty$ (cf. \cite{breslow2007weighted}). In this case, we need the following assumptions to show that conclusions of  Proposition  \ref{lem s1} and Theorem \ref{Cor 1} hold for stratified multistage cluster sampling design with SRSWOR. Let $N_h$=$\sum_{j=1}^{M_h} N_{hj}$ and $n_h$=$m_h r_h$ be the number of population units in stratum $h$ and the number of population units sampled from stratum $h$, respectively, for any $h$=$1,\ldots, H$. 
\begin{assumption}\label{ass 8}
$\sum_{\nu=1}^\infty exp(-K  M_h)<\infty$, $0<\underline{\lim}_{\nu\rightarrow\infty}m_h/M_h\leq\overline{\lim}_{\nu\rightarrow\infty}m_h/M_h<1$, $\lim_{\nu\rightarrow\infty}$ $n_h/n$=$\lambda_h>0$, $\lim_{\nu\rightarrow\infty}N_h/N$=$\Lambda_h>0$ and $0<\underline{\lim}_{\nu\rightarrow\infty}\min_{1\leq j\leq M_h}r_h/N_{hj}\leq \overline{\lim}_{\nu\rightarrow\infty}$ $\max_{1\leq j\leq M_h}r_h/N_{hj}<1$ for any $h$=$1,\ldots,H$ and $K>0$, and  $\max_{1\leq h\leq H}\sum_{j=1}^{M_h}N_{hj}^4/M_h$ =$O(1)$  and $\max_{1\leq h\leq H}\sum_{j=1}^{M_h}(N_{hj}-N_h/M_h)^2/ M_h\rightarrow 0$ as $\nu\rightarrow\infty$.
\end{assumption}
\begin{assumption}\label{ass 12}
For any $h$=$1,\ldots,H$, the support of the joint distribution of $(Y_{hjl}^{\prime},X_{hjl}^{\prime})$ is not a subset of a straight line in $\mathds{R}^2$, and $E_{\textbf{P}}||\textbf{W}_{hjl}^{\prime}||^2<\infty$, where  $\textbf{W}_{hjl}^{\prime}$=$(X_{hjl}^{\prime},Y_{hjl}^{\prime},X_{hjl}^{\prime}Y_{hjl}^{\prime}$, $(X_{hjl}^{\prime})^2)$.
\end{assumption}
\begin{assumption}\label{ass 7}
$supp(F_{y,h})$=$\mathcal{C}_y$ and $supp(F_{x,h})$=$\mathcal{C}_x$ for any $h$=$1,\ldots,H$ and some open intervals  $\mathcal{C}_y\subseteq (-\infty,\infty)$ and $\mathcal{C}_x\subseteq (0,\infty)$. Moreover, $F_{y,h}$ and $F_{x,h}$ are continuous on $\mathds{R}$ and are differentiable with positive continuous derivatives $f_{y,h}$ and $f_{x,h}$ on $\mathcal{C}_y$ and $\mathcal{C}_x$, respectively, for any $h$=$1,\ldots,H$ and $\nu\geq 1$. 
\end{assumption}
The condition $\sum_{\nu=1}^\infty exp(-K  M_h)<\infty$ for any $h$=$1,\ldots,H$ and $K>0$ in Assumption \ref{ass 8} holds, when the number of clusters in each stratum is a strictly increasing function of $\nu$. This assumption implies that $M_1,\ldots,M_H$ grow infinitely as $\nu\rightarrow\infty$. The condition  $\max_{1\leq h\leq H}\sum_{j=1}^{M_h}N_{hj}^4/M_h$=$O(1)$ as $\nu\rightarrow\infty$ in Assumption \ref{ass 8} holds, when cluster sizes in any stratum are not arbitrarily large.  The condition $\max_{1\leq h\leq H}\sum_{j=1}^{M_h}(N_{hj}-N_h/M_h)^2/M_h\rightarrow 0$ as $\nu\rightarrow\infty$ in Assumption \ref{ass 8} implies that the variation among cluster sizes in each stratum is negligible. The rest of the conditions in Assumption \ref{ass 8} are often used in the earlier literature (see \cite{MR983089}, \cite{MR1292550} and references therein).  Assumptions \ref{ass 8}-\ref{ass 7} are required to establish the finite dimensional convergence of the empirical process $\big\{\sqrt{n}(\hat{F}_{z}(t)-t):t\in [0,1]\big\}$ for $d(i,s)$=$(N\pi_i)^{-1}$ under stratified multistage cluster sampling design with SRSWOR, whereas Assumptions \ref{ass 8} and \ref{ass 7} are required to show the tightness of this empirical process under the same sampling design. 
\par

Next, we consider the case, when $H\rightarrow\infty$ as $\nu\rightarrow\infty$ (cf. \cite{MR1091862} and \cite{MR1292550}). In this case, we replace Assumption \ref{ass 8} by Assumption \ref{ass 9} and Assumption \ref{ass 12} by Assumption \ref{ass 10} given below, and consider some further assumptions to show that conclusions of  Proposition  \ref{lem s1} and Theorem \ref{Cor 1} hold for stratified multistage cluster sampling design with SRSWOR.
\begin{assumption}\label{ass 9} 
$\sum_{\nu=1}^\infty exp(-K H)<\infty$ for any $K>0$, and $\max_{1\leq h\leq H,1\leq j\leq M_h}n M_h N_{hj}/$ $m_h N$=$O(1)$,  $\sum_{h=1}^{H}M_{h}^4/H$=$O(1)$ and  $\max_{1\leq h\leq H}\sum_{j=1}^{M_h}(N_{hj}-N_h/M_h)^2/M_h\rightarrow 0$ as $\nu\rightarrow\infty$.
\end{assumption}
Next, suppose that $F_{y,H}(t)$=$\sum_{h=1}^H(N_h/N)F_{y,h}(t)$ and $F_{x,H}(t)$=$\sum_{h=1}^H(N_h/N)F_{x,h}(t)$, and $Q_{y,H}$ and $Q_{x,H}$ are quantile functions corresponding to $F_{y,H}$ and $F_{x,H}$, respectively. Let 
$$
\textbf{G}_{hjl}^{\prime}=\big(\mathds{1}_{[Y_{hjl}^{\prime}\leq Q_{y,H}(p_1)]},\ldots,\mathds{1}_{[Y_{hjl}^{\prime}\leq Q_{y,H}(p_k)]},\mathds{1}_{[X_{hjl}^{\prime}\leq Q_{x,H}(p_1)]},\ldots,\mathds{1}_{[X_{hjl}^{\prime}\leq Q_{x,H}(p_k)]}\big)
$$ 
for any $k\geq 1$ and $p_1,\ldots,p_k\in(0,1)$, and $\Gamma_h$=$E_{\textbf{P}}(\textbf{G}_{hjl}^{\prime}-E_{\textbf{P}}(\textbf{G}_{hjl}^{\prime}))^T$ $(\textbf{G}_{hjl}^{\prime}-E_{\textbf{P}}(\textbf{G}_{hjl}^{\prime}))$. Then, we consider the following assumptions.
\begin{assumption}\label{ass 10}
Given any $k\geq 1$ and $p_1,\ldots,p_k\in (0,1)$, $\sum_{h=1}^H N_h(N_h-n_h)\Gamma_h/n_h N\rightarrow\Gamma_1$ and $\sum_{h=1}^H N_h\Gamma_h/N\rightarrow\Gamma_2$ as $\nu\rightarrow\infty$ for some positive definite matrices $\Gamma_1$ and $\Gamma_2$. Moreover, $\sum_{h=1}^H\sum_{j=1}^{M_h}\sum_{l=1}^{N_{hj}}\textbf{W}_{hjl}^{\prime}/N\rightarrow \Theta$=$(\Theta_1,\ldots,\Theta_4)$ and $\sum_{h=1}^H\sum_{j=1}^{M_h}\sum_{l=1}^{N_{hj}}||\textbf{W}_{hjl}^{\prime}||^2/N$ =$O(1)$ \textit{a.s.} $[\textbf{P}]$ as $\nu\rightarrow\infty$, where $\Theta_1>0$.
\end{assumption}
Further, suppose that $f_{y,H}(t)$=$dF_{y,H}/dt$ and $f_{x,H}(t)$=$dF_{x,H}/dt$, and consider the following assumptions. 
\begin{assumption}\label{ass 11}
$supp(F_{y,h})$=$\mathcal{C}_y$ and $supp(F_{x,h})$=$\mathcal{C}_x$ for any $h$=$1,\ldots,H$ and some open intervals  $\mathcal{C}_y\subseteq (-\infty,\infty)$ and $\mathcal{C}_x\subseteq (0,\infty)$.  Further, there exists a distribution function $\tilde{F}_y$ with $supp(\tilde{F}_y)$=$\mathcal{C}_y$ and positive continuous derivative $\tilde{f}_y$ such that $F_{y,H}(t)\rightarrow\tilde{F}_y(t)$ for any $t\in\mathds{R}$ and $\sup_{\mathcal{C}_y}|f_{y,H}(t)-\tilde{f}_y(t)|\rightarrow 0$ as $\nu\rightarrow\infty$. There also exists a distribution function $\tilde{F}_x$ with $supp(\tilde{F}_x)$=$\mathcal{C}_x$ and positive continuous derivative $\tilde{f}_x$ such that $F_{x,H}(t)\rightarrow\tilde{F}_x(t)$ for any $t\in\mathds{R}$ and $\sup_{\mathcal{C}_x}|f_{x,H}(t)-\tilde{f}_x(t)|\rightarrow 0$ as $\nu\rightarrow\infty$. 
\end{assumption}
The assumption $\max_{1\leq h\leq H,1\leq j\leq M_h}n M_h N_{hj}/m_h N$=$O(1)$ as $\nu\rightarrow\infty$ in Assumption \ref{ass 9} was considered earlier in the literature (cf. \cite{MR1292550}). This assumption and Assumption \ref{ass 0} imply that cluster sizes in any stratum cannot be arbitrarily large. The assumption  $\sum_{h=1}^{H}M_{h}^4/H$=$O(1)$ as $\nu\rightarrow\infty$ in Assumption \ref{ass 9} holds, when the number of clusters in any stratum is not arbitrarily large.  Assumption \ref{ass 11} implies that $F_{y,H}$ and $F_{x,H}$ can be approximated by the distribution functions $\tilde{F}_{y}$ and $\tilde{F}_{x}$, respectively, when $H\rightarrow\infty$ as $\nu\rightarrow\infty$. This assumption also implies that $f_{y,H}$ and $f_{x,H}$ can be approximated uniformly by the density functions of $\tilde{F}_{y}$ and $\tilde{F}_{x}$, respectively, when $H\rightarrow\infty$ as $\nu\rightarrow\infty$.  Assumptions \ref{ass 9} and \ref{ass 10} are required to show the finite dimensional convergence of the empirical process $\big\{\sqrt{n}(\hat{F}_{z}(t)-t):t\in [0,1]\big\}$ for $d(i,s)$=$(N\pi_i)^{-1}$ under stratified multistage cluster sampling design with SRSWOR, and Assumptions \ref{ass 9} and \ref{ass 11} are required to establish the tightness of this empirical process under the same sampling design.  Now, we state the following results. 
\begin{proposition}\label{lem s2}
\textsl{(i)} Suppose that $H$ is fixed as $\nu\rightarrow\infty$, and Assumptions \ref{ass 0} and \ref{ass 8}--\ref{ass 7} hold. Then, the conclusion of  Proposition \ref{lem s1} holds for stratified multistage cluster sampling design with SRSWOR.\\
\textsl{(ii)} Further, if $H\rightarrow\infty$ as $\nu\rightarrow\infty$, and Assumptions \ref{ass 0} and \ref{ass 7}--\ref{ass 11} hold, then the same result holds.
\end{proposition}
\begin{theorem}\label{Thm 2} 
\textsl{(i)} Suppose that $H$ is fixed as $\nu\rightarrow\infty$, and Assumptions \ref{ass 0} and \ref{ass 8}--\ref{ass 7} hold. Then, the conclusion of Theorem \ref{Cor 1} holds for stratified multistage cluster sampling design with SRSWOR  with p.d. covariance kernel  
\begin{align}\label{eq 15}
\begin{split}
K(p_1,p_2)=&\lim_{\nu\rightarrow\infty}(n/N^2)\sum_{h=1}^H N_h(N_h-n_h)E_{\textbf{P}}\big(\zeta_{hjl}^{\prime}(p_1)-
E_{\textbf{P}}(\zeta_{hjl}^{\prime}(p_1))\big)\times\\
&\big(\zeta_{hjl}^{\prime}(p_2)-E_{\textbf{P}}(\zeta_{hjl}^{\prime}(p_2))\big)/n_h\text{ for }p_1,p_2\in[\alpha,\beta].
\end{split}
\end{align} 
Here, $\zeta_{hjl}^{\prime}(p)$'s are as in Table \ref{table 6} below.\\ 
\textsl{(ii)} Further, if $H\rightarrow\infty$ as $\nu\rightarrow\infty$, and Assumptions \ref{ass 0} and \ref{ass 7}--\ref{ass 11} hold, then the same result holds. 
\end{theorem}
\begin{table}[h]

\caption{Expressions of $\zeta_{hjl}^{\prime}(p)$'s appearing in \eqref{eq 15} for different $G(p)$'s in the case of stratified multistage cluster sampling design with SRSWOR}
\label{table 6}
\begin{threeparttable}[b]
\centering
\begin{tabular}{ccc}
\hline 
&$G(p)$& $\zeta_{hjl}^{\prime}(p)$\\
\hline
&$\hat{Q}_{y}(p)$& $\mathds{1}_{[Y_{hjl}^{\prime}\leq Q_{y,H}(p)]}/f_{y,H}(Q_{y,H}(p))$\\
\cline{2-3}
&\multirow{2}{*}{$\hat{Q}_{y,RA}(p)$} & $\mathds{1}_{[Y_{hjl}^{\prime}\leq Q_{y,H}(p)]}/f_{y,H}(Q_{y,H}(p))-(Q_{y,H}(p)/Q_{x,H}(p))\times$\\
$H$ is fixed&& $\mathds{1}_{[X_{hjl}^{\prime}\leq Q_{x,H}(p)]}/f_{x,H}(Q_{x,H}(p))$\\
\cline{2-3}
as $\nu\rightarrow\infty$&\multirow{2}{*}{$\hat{Q}_{y,DI}(p)$}& $\mathds{1}_{[Y_{hjl}^{\prime}\leq Q_{y,H}(p)]}/f_{y,H}(Q_{y,H}(p))-(\sum_{h=1}^H (N_h/N) E_{\textbf{P}} (Y_{hjl}^{\prime})/$\\
&&$\sum_{h=1}^H (N_h/N) E_{\textbf{P}}(X_{hjl}^{\prime}))\mathds{1}_{[X_{hjl}^{\prime}\leq Q_{x,H}(p)]}/f_{x,H}(Q_{x,H}(p))$\\
\cline{2-3}
&\multirow{2}{*}{$\hat{Q}_{y,REG}(p)$} &$\mathds{1}_{[Y_{hjl}^{\prime}\leq Q_{y,H}(p)]}/f_{y,H}(Q_{y,H}(p))-(\sum_{h=1}^H (N_h/N) E_{\textbf{P}}(X_{hjl}^{\prime}Y_{hjl}^{\prime})/$\\\
&& $\sum_{h=1}^H (N_h/N) E_{\textbf{P}}(X_{hjl}^{\prime})^2)\mathds{1}_{[X_{hjl}^{\prime}\leq Q_{x,H}(p)]}/f_{x,H}(Q_{x,H}(p))$\\
\hline
\hline
&$\hat{Q}_{y}(p)$& $\mathds{1}_{[Y_{hjl}^{\prime}\leq Q_{y,H}(p)]}/f_{y,H}(Q_{y,H}(p))$\\
\cline{2-3}
&\multirow{2}{*}{$\hat{Q}_{y,RA}(p)$} & $\mathds{1}_{[Y_{hjl}^{\prime}\leq Q_{y,H}(p)]}/f_{y,H}(Q_{y,H}(p))-(Q_{y,H}(p)/Q_{x,H}(p))\times$\\
$H\rightarrow\infty$&& $\mathds{1}_{[X_{hjl}^{\prime}\leq Q_{x,H}(p)]}/f_{x,H}(Q_{x,H}(p))$\\
\cline{2-3}
as $\nu\rightarrow\infty$&\multirow{2}{*}{$\hat{Q}_{y,DI}(p)$}& $\mathds{1}_{[Y_{hjl}^{\prime}\leq Q_{y,H}(p)]}/f_{y,H}(Q_{y,H}(p))-(\text{\tnote{2} \ }\Theta_2/\text{\tnote{2} \ }\Theta_1)\times$\\
&&$\mathds{1}_{[X_{hjl}^{\prime}\leq Q_{x,H}(p)]}/f_{x,H}(Q_{x,H}(p))$\\
\cline{2-3}
&\multirow{2}{*}{$\hat{Q}_{y,REG}(p)$} &$\mathds{1}_{[Y_{hjl}^{\prime}\leq Q_{y,H}(p)]}/f_{y,H}(Q_{y,H}(p))-(\text{\tnote{2} \ }\Theta_3/\text{\tnote{2} \ }\Theta_4)\times$\\\
&& $\mathds{1}_{[X_{hjl}^{\prime}\leq Q_{x,H}(p)]}/f_{x,H}(Q_{x,H}(p))$\\
\hline
\end{tabular}
\begin{tablenotes}
\item[2] The $\Theta_1$, $\Theta_2$, $\Theta_3$ and $\Theta_4$ are as in Assumption \ref{ass 10} in Section \ref{subsec 3.3}. 
\end{tablenotes}
\end{threeparttable}
\end{table}
 The weak convergence of the above quantile processes with $Q_{y,N}$ replaced by $Q_y$ can be shown using the key ideas of the proof of Theorem \ref{Thm 2}. When $H$ is fixed as $\nu\rightarrow\infty$, the expression of $K(p_1,p_2)$ in \eqref{eq 15} can be further simplified (cf. ($46$) in Section $5$ of the supplement (\cite{supp2024})) because $N_h/N\rightarrow\Lambda_h$ and $(n/N^2)N_h(N_h-n_h)/n_h\rightarrow \lambda \Lambda_h(\Lambda_h/\lambda\lambda_h-1)$ as $\nu\rightarrow\infty$ for any $h$=$1,\ldots,H$ by Assumptions \ref{ass 0} and \ref{ass 8}.
\section{Functions of quantile processes}\label{sec 6.2}
In this section, we derive the asymptotic normality of the smooth $L$-estimators as well as the estimators of smooth functions of finite population quantiles, which are based on the sample quantile, and the ratio, the difference and the regression estimators of the population quantiles, under sampling designs considered in the preceding two sections. The smooth $L$-estimators include the estimators of the population $\alpha$-trimmed means, whereas the estimators of smooth functions of population quantiles include the estimators of any specific quantile (e.g., median), the interquartile range and the quantile based measure of skewness in the population. Note that non smooth $L$-estimators are also special cases of these latter estimators.  \cite{MR2748939} and \cite{MR1007458} derived the asymptotic normality of the sample median and the ratio estimator of the population median, respectively, under SRSWOR. \cite{MR1091862} showed the asymptotitc normality of the sample quantile under stratified cluster sampling design with SRSWOR, and \cite{wang2011asymptotic} showed the asymptotitc normality of the sample quantile under some conditions on sampling deigns. \cite{MR1292550} derived the asymptotic normality of smooth and non smooth $L$-estimators, which are constructed based on the sample quantile under stratified multistage cluster sampling design with SRSWOR.  However, there is no result present in the existing literature related to the asymptotic normality of the smooth $L$-estimators and the estimators of smooth functions of population quantiles, which are based on the ratio, the difference and the regression estimators of the population quantile. There is also no result present in the available literature related to the asymptotic normality of the above estimators under high entropy and RHC sampling designs. 
\par

 Let us fix $0<\alpha<\beta<1$ and consider the finite population parameter $\int_{[\alpha,\beta]} Q_{y,N}(p)J(p)dp$ for some known smooth function $J$ on $[0,1]$. It follows from the definition of $Q_{y,N}(p)$ that the above parameter coincides with the population $\alpha$-trimmed mean 
 $$
\tau_{\alpha,N}=\bigg (\sum_{i=\lfloor N\alpha\rfloor+2}^{N-\lfloor N\alpha\rfloor-1}Y_{(i)}+(1-\{N\alpha\})(Y_{(\lfloor N\alpha\rfloor+1)}+Y_{(N-\lfloor N\alpha\rfloor)})\bigg)\bigg/N(1-2\alpha)
 $$ 
  when $0<\alpha<1/2$, $\beta$=$1-\alpha$ and $J(p)$=$1/(1-2\alpha)$ for $p\in[0,1]$. Here, $Y_{(1)},\ldots,Y_{(N)}$ are the ordered population values of $y$, and $\lfloor N\alpha\rfloor$ and $\{N\alpha\}$ are, respectively, the integer and the fractional parts of $N\alpha$. Several estimators of $\int_{[\alpha,\beta]} Q_{y,N}(p)J(p)dp$ can be constructed by plugging $\hat{Q}_{y}(p), \hat{Q}_{y,RA}(p), \hat{Q}_{y,DI}(p)$ and $\hat{Q}_{y,DI}(p)$ into $\int_{[\alpha,\beta]} Q_{y,N}(p)J(p)dp$. Note that these estimators can be expressed as weighted linear combinations of the ordered sample observations on $y$, where the weights are mainly generated by the smooth function $J$. That is why these estimators are called smooth $L$-estimators (cf. \cite{MR1292550}). 
\par

Next, suppose that $k\geq 1$, $p_1,\ldots,p_k\in (0,1)$, $f:\mathds{R}^k\rightarrow \mathds{R}$ is a smooth function, and $f( Q_{y,N}(p_1),\ldots,Q_{y,N}(p_k))$ is a finite population parameter.  Some examples of such a parameter are given in Table \ref{table 2} below. Several estimators of $f( Q_{y,N}(p_1),\ldots,Q_{y,N}(p_k))$ can be constructed by plugging $\hat{Q}_{y}(p)$, $\hat{Q}_{y,RA}(p)$, $\hat{Q}_{y,DI}(p)$ and $\hat{Q}_{y,DI}(p)$ in $f( Q_{y,N}(p_1),\ldots,$ $Q_{y,N}(p_k))$. Now, we state the asymptotic normality results for the above estimators.
\begin{theorem}\label{Theorem 4} 
\textsl{(i)} Fix $0<\alpha<\beta<1$. Suppose that the conclusion of Theorem \ref{Cor 1} holds and $K(p_1,p_2)$ in  \eqref{eq 52}  is continuous on $[\alpha,\beta]\times[\alpha,\beta]$. Then, under $\textbf{P}^*$, 
\begin{align}\label{eq 14}
\begin{split}
&\sqrt{n}\left(\int_{[\alpha,\beta]} G(p)J(p)dp-\int_{[\alpha,\beta]} Q_{y,N}(p)J(p)dp\right)\xrightarrow{\mathcal{L}}N(0,\sigma_1^2)\text{ and }\\
&\sqrt{n}\left(f( G(p_1),\ldots,G(p_k))-f( Q_{y,N}(p_1),\ldots,Q_{y,N}(p_k))\right)\xrightarrow{\mathcal{L}}N(0,\sigma_2^2)\text{ as }\nu\rightarrow\infty
\end{split}
\end{align}
for any high entropy sampling design, where $k\geq 1$, $p_1,\ldots,p_k\in[\alpha,\beta]$, and $G(p)$ is one of $\hat{Q}_{y}(p)$, $\hat{Q}_{y,RA}(p)$, $\hat{Q}_{y,DI}(p)$ and $\hat{Q}_{y,REG}(p)$ with $d(i,s)$=$(N\pi_i)^{-1}$. Here, 
\begin{equation}\label{ceq 71}
\sigma^2_1=\int_{\alpha}^{\beta}\int_{\alpha}^{\beta}K(p_1,p_2) J(p_1)J(p_2)dp_1dp_2,\text{ }\sigma^2_2=a\Delta a^T,
\end{equation}
 $\Delta$ is a $k\times k$ matrix such that 
\begin{equation}
 ((\Delta))_{uv}=K(p_u,p_v)\text{ for }1\leq u,v\leq k,\text{ and }a=\lim_{\nu\rightarrow\infty}\nabla f(Q_{y,N} (p_1),\ldots,Q_{y,N}(p_k))
 \end{equation}
 \textit{a.s.} $[\textbf{P}]$.\\
\textsl{(ii)} Further, if the assumptions of Theorem \ref{Cor 3} hold, then the results in \eqref{eq 14} hold for $d(i,s)$=$A_i/NX_i$ in the case of RHC sampling design. 
\end{theorem}
\begin{table}[h] 
\centering  
\caption{Examples of $f( Q_{y,N}(p_1),\ldots,Q_{y,N}(p_k))$}

\label{table 2}
\begin{tabular}{cccc}
\hline
Parameter& $k$& $p_1,\ldots,p_k$&$f$\\
\hline
Median& $1$& $p_1$=$0.5$& $f(t)$=$t$\\
\hline
Interquartile&\multirow{2}{*}{$2$}& \multirow{2}{*}{$p_1$=$0.25$, $p_2$=$0.75$}&\multirow{2}{*}{ $f(t_1,t_2)$=$t_2-t_1$}\\
range&&&\\
\hline
Bowley's measure&\multirow{2}{*}{ $3$}&\multirow{2}{*}{$p_1$=$0.25$, $p_2$=$0.5$, $p_3$=$0.75$}&\multirow{2}{*}{$f(t_1,t_2,t_3)$=$(t_1+t_3-2 t_2)/ (t_3-t_1)$}\\
of skewness&& &\\
\hline
\end{tabular}
\end{table}
 It can be shown using the expressions in Table \ref{table 5}, and the Assumption \ref{ass 3} that  $K(p_1,p_2)$ in  \eqref{eq 52}  is continuous on $[\alpha,\beta]\times[\alpha,\beta]$ under SRSWOR, LMS and HE$\pi$PS sampling designs.  Now, we state the following theorem.
\begin{theorem}\label{Theorem 6} 
\textsl{(i)} Fix $0<\alpha<\beta<1$. Suppose that $H$ is fixed as $\nu\rightarrow\infty$, and Assumptions \ref{ass 0} and \ref{ass 8}--\ref{ass 7} hold, then the results in \eqref{eq 14} of Theorem \ref{Theorem 4} hold for $d(i,s)$=$(N\pi_i)^{-1}$= $M_{h} N_{hj}/N m_h r_h$ under stratified multistage cluster sampling design with SRSWOR.\\
\textsl{(ii)} On the other hand, if $H\rightarrow\infty$ as $\nu\rightarrow\infty$, Assumptions \ref{ass 0} and \ref{ass 7}--\ref{ass 11} hold, and $K(p_1,p_2)$ in \eqref{eq 15} is continuous on $[\alpha,\beta]\times[\alpha,\beta]$, then the same results hold.

\end{theorem} 
When $H\rightarrow\infty$ as $\nu\rightarrow\infty$ in the case of stratified multistage cluster sampling design with SRSWOR, it can be shown that $K(p_1,p_2)$ in \eqref{eq 15} is continuous on $[\alpha,\beta]\times[\alpha,\beta]$ if the limit in the expression of $K(p_1,p_2)$ in \eqref{eq 15} exists uniformly over $[\alpha,\beta]\times[\alpha,\beta]$.
\subsection{Estimation of asymptotic covariance kernels and confidence intervals}\label{sec 6}
Suppose that 
\begin{align*} 
\begin{split}
&  \theta_1=\int_{[\alpha,\beta]} Q_{y,N}(p)J(p)dp\text{,  }\theta_2=f( Q_{y,N}(p_1),\ldots,Q_{y,N}(p_k)),\\
&   \hat{\theta}_1=\int_{[\alpha,\beta]}G(p)J(p)dp\text{ and }\hat{\theta}_2=f( G(p_1),\ldots,G(p_k)), 
\end{split}
\end{align*}
 where $G(p)$ is one of $\hat{Q}_{y}(p)$, $\hat{Q}_{y,RA}(p)$, $\hat{Q}_{y,DI}(p)$ and $\hat{Q}_{y,REG}(p)$. Then, $\sqrt{n}(\hat{\theta}_i-\theta_i)\xrightarrow{\mathcal{L}}N(0,\sigma^2_i)$ for $i$=$1,2$, where $\sigma^2_i$'s are as in Theorem \ref{Theorem 4}. Further, suppose that $\hat{\sigma}_i\xrightarrow{p}\sigma_i$ for some estimator $\hat{\sigma}_i$ of $\sigma_i$. Then, a $100(1-\eta)\%$ confidence interval for $\theta_i$ is constructed as 
$$
[\hat{\theta}_i-Z_{\eta/2}\hat{\sigma}_i/\sqrt{n}, \hat{\theta}_i+Z_{\eta/2}\hat{\sigma}_i/\sqrt{n}]\text{ for } i=1,2,
$$ 
 where $Z_{\eta/2}$ is the $(1-\eta/2)^{th}$ quantile of the standard normal distribution. We now discuss the estimation of the asymptotic covariance kernels mentioned in  \eqref{eq 52}, \eqref{eq 51} and \eqref{eq 15}  based on which consistent estimators of $\sigma^2_i$'s will be constructed. 
\par

Following the approach of \cite{cardot2014variance}, $K(p_1,p_2)$, for $d(i,s)$=$(N\pi_i)^{-1}$, under any high entropy sampling design (see  \eqref{eq 52})   can be estimated by
\begin{align}\label{eq 3}
\begin{split}
\hat{K}(p_1,p_2)=&(n/N^2)\sum_{i\in s}(\hat{\zeta}_i(p_1)-\hat{\overline{\zeta}}(p_1)-\hat{S}(p_1)\pi_i)\times\\
&(\hat{\zeta}_i(p_2)-\hat{\overline{\zeta}}(p_2)-\hat{S}(p_2)\pi_i)(\pi_i^{-1}-1)\pi_i^{-1},
\end{split}
\end{align}
where $\hat{\overline{\zeta}}(p)$=$\sum_{i\in s}(N\pi_i)^{-1}\hat{\zeta}_i(p)$ and $\hat{S}(p)$=$\sum_{i\in s}(\hat{\zeta}_i(p)-\hat{\overline{\zeta}}(p))(\pi_i^{-1}-1)/\sum_{i\in s}(1-\pi_i)$. Here, $\hat{\zeta}_i(p)$ is obtained by replacing the superpopulation parameters involved in the expression of $\zeta_i(p)$ in Table \ref{table 4} by their estimators under high entropy sampling designs (see  Table $5$ in Section $2$ of the supplement (\cite{supp2024})). 
\par

 Next, $K(p_1,p_2)$, for $d(i,s)$=$A_i/NX_i$ under RHC sampling design (see  \eqref{eq 51}),   can be estimated as
\begin{equation}\label{eq 6}
\hat{K}(p_1,p_2)=n\gamma(\overline{X}/N)\sum_{i\in s}(\hat{\zeta}_i(p_1)-\hat{\overline{\zeta}}(p_1))(\hat{\zeta}_i(p_2)-\hat{\overline{\zeta}}(p_2))(A_i/X_i^2),
\end{equation}
where $\hat{\overline{\zeta}}(p)$=$\sum_{i\in s}\hat{\zeta}_i(p)(A_i/NX_i)$. Here, $\hat{\zeta}_i(p)$ is obtained by replacing the superpopulation parameters involved in the expression of $\zeta_i(p)$ in Table \ref{table 4} by their estimators under RHC sampling design (see  Table $5$ in Section $2$ of the supplement (\cite{supp2024})).
\par

 Given an estimator $\hat{K}(p_1,p_2)$ of $K(p_1,p_2)$, an estimator of $\sigma^2_1$ can be constructed as $\hat{\sigma}^2_1$=$\int_{\alpha}^{\beta}\int_{\alpha}^{\beta}\hat{K}(p_1,p_2)J(p_1)J(p_2)$ $dp_1dp_2$, whereas an estimator of $\sigma^2_2$ can be constructed as $\hat{\sigma}^2_2$=$\hat{a}\hat{\Delta}\hat{a}^T$. Here, $\hat{a}$=$\nabla f (\hat{Q}_{y}(p_1),\ldots, \hat{Q}_{y}(p_k))$, $k\geq 1$, $p_1,\ldots,p_k\in[\alpha,\beta]$, and $\hat{\Delta}$ is a $k\times k$ matrix such that $((\hat{\Delta}))_{uv}$=$\hat{K}(p_u,p_v)$ for $1\leq u, v\leq k$. In the following theorem,  we assert that the above estimators of $\sigma^2_1$ and $\sigma^2_2$ are consistent. 
\begin{theorem}\label{Theorem 5}
\textsl{(i)} Fix $0<\alpha<\beta<1$. Suppose that the assumptions of Theorem \ref{Cor 1} hold, $K(p_1,p_2)$ is as in  \eqref{eq 52},  and $\hat{K}(p_1,p_2)$ is as in \eqref{eq 3}. Then, under $\textbf{P}^*$,
\begin{equation}\label{eq 16}
\hat{\sigma}^2_i \xrightarrow{p}\sigma^2_i\text{ as }\nu\rightarrow\infty\text{ for }i=1,2
\end{equation}
and any high entropy sampling design satisfying Assumption \ref{ass 2}.\\ 
\textsl{(ii)} Further, if the assumptions of Theorem \ref{Cor 3} hold, $K(p_1,p_2)$ is as in  \eqref{eq 51},  and $\hat{K}(p_1,p_2)$ is as in \eqref{eq 6}. Then, the result in \eqref{eq 16} hold under RHC sampling design. 
\end{theorem}
Next, for the case of stratified multistage cluster sampling design with SRSWOR, $E_{\textbf{P}}(\zeta_{hjl}^{\prime}(p_1)-E_{\textbf{P}}(\zeta_{hjl}^{\prime}(p_1)))(\zeta_{hjl}^{\prime}(p_2)-E_{\textbf{P}}(\zeta_{hjl}^{\prime}(p_2)))$ in the expression of $K(p_1,p_2)$ in \eqref{eq 15} can be estimated as 
\begin{equation}
\sum_{j\in s_h}\sum_{l\in s_{hj}}M_h N_{hj} (\hat{\zeta}_{hjl}(p_1)-\hat{\overline{\zeta}}_{h}(p_1))(\hat{\zeta}_{hjl}(p_2)-\hat{\overline{\zeta}}_{h}(p_2))/m_h r_h N_h,
\end{equation} 
where $\hat{\overline{\zeta}}_{h}(p)$=$\sum_{j\in s_h}\sum_{l\in s_{hj}}M_h N_{hj}\hat{\zeta}_{hjl}(p)/m_h r_h N_h$, and $h$=$1,\ldots,H$. Here, $\hat{\zeta}_{hjl}(p)$ is obtained by replacing the superpopulation parameters involved in the expression of $\zeta_{hjl}^{\prime}(p)$ in Table \ref{table 6} by their estimators as mentioned in  Table $6$ in Section $2$ of the supplement (\cite{supp2024}).  Thus an estimator of $K(p_1,p_2)$ in \eqref{eq 15} under stratified multistage cluster sampling design with SRSWOR is obtained as 
\begin{align}\label{eq 7}
\begin{split}
\hat{K}(p_1,p_2)=&(n/N^2)\sum_{h=1}^H (N_h^2/n_h-N_h)\sum_{j\in s_h}\sum_{l\in s_{hj}}M_h N_{hj} (\hat{\zeta}_{hjl}(p_1)-\hat{\overline{\zeta}}_{h}(p_1))\times\\
&(\hat{\zeta}_{hjl}(p_2)-\hat{\overline{\zeta}}_{h}(p_2))/m_h r_h N_h.
\end{split}
\end{align}
\par

Given $\hat{K}(p_1,p_2)$, estimators of $\sigma^2_1$ and $\sigma^2_2$ can be constructed under stratified multistage cluster sampling design with SRSWOR in the same way as in the case of single stage sampling designs discussed in the paragraph preceding Theorem \ref{Theorem 5}. Now, we  state the following theorem.
\begin{theorem}\label{Theorem 7}
Fix $0<\alpha<\beta<1$. Suppose that the assumptions of Theorem \ref{Thm 2} hold, $K(p_1,p_2)$ is as in \eqref{eq 15}, and $\hat{K}(p_1,p_2)$ is as in \eqref{eq 7}. Then, the result in \eqref{eq 16} of Theorem \ref{Theorem 5} hold under stratified multistage cluster sampling design with SRSWOR.
\end{theorem}
\section{ Comparison of different estimators}\label{sec 7}
\subsection{Comparison of the estimators of functions of quantiles} 
In this section, we shall first compare different estimators of the finite population parameter $\int_{[\alpha,\beta]} Q_{y,N}(p)J(p)dp$ as well as $f(Q_{y,N}(p_1),\cdots,$ $Q_{y,N}(p_k))$ (see Section \ref{sec 6.2}) under each of SRSWOR, RHC and any HE$\pi$PS sampling designs in terms of their asymptotic variances given in Theorem \ref{Theorem 4}. Here, $0<\alpha<\beta<1$, $k\geq 1$ and $p_1,\ldots,p_k\in (0,1)$. Recall from Section \ref{sec 6.2} that these parameters include the median, the $\alpha$-trimmed mean, the interquartile range and the quantile based measure of skewness. Let us assume that $P(s,\omega)$ is one of SRSWOR, RHC and any HE$\pi$PS sampling designs. Let us also assume that $K_1(p_1,p_2)$, $K_2(p_1,p_2)$, $K_3(p_1,p_2)$ and $K_4(p_1,p_2)$ are the asymptotic covariance kernels of  the quantile processes constructed based on $\hat{Q}_y(p)$, $\hat{Q}_{y,RA}(p)$, $\hat{Q}_{y,DI}(p)$ and $\hat{Q}_{y,REG}(p)$  under $P(s,\omega)$, respectively (see Table \ref{table 5} and \eqref{eq 51}), and $\{\Delta_u:1\leq u\leq 4\}$ is a $k\times k$ matrix such that
\begin{equation}
((\Delta_u))_{wv}=K_u(p_w,p_v)\text{ for }1\leq w,v\leq k\text{ and }1\leq u\leq 4.
\end{equation}
 Then, we have the following theorem. 
\begin{theorem}\label{res 1}
Suppose that $X_i\leq b$ \textit{a.s.} $[\textbf{P}]$ for some $b>0$, $E_{\textbf{P}}(X_i)^{-1}<\infty$, Assumption \ref{ass 0} holds with $0<\lambda<E_{\textbf{P}}(X_i)/b$, Assumptions \ref{ass 4} and \ref{ass 5} hold,  and \eqref{eq 150}  holds. Then, we have the following results.\\
\textsl{(i)} Under $P(s,\omega)$,  the asymptotic variance of the estimator of $\int_{[\alpha,\beta]} Q_{y,N}(p)J(p)dp$ based on the sample quantile is smaller than the asymptotic variances of its estimators based on the ratio, the difference and the regression estimators of the finite population quantile if and only if 
\begin{equation}\label{eq 64}
\max_{2\leq u\leq 4}\bigg\{\int_{\alpha}^{\beta} \int_{\alpha}^{\beta} \big(K_1(p_1,p_2)-K_u(p_1,p_2)\big)J(p_1)J(p_2)dp_1 dp_2\bigg\}<0
\end{equation}
\textsl{(ii)} Under $P(s,\omega)$, the asymptotic variance of the estimator of $f(Q_{y,N}(p_1),\cdots,Q_{y,N}(p_k))$ based on the sample quantile is smaller than the asymptotic variances of its estimators based on the ratio, the difference and the regression estimators of the finite population quantile if and only if
\begin{equation}\label{eq 65}
\max_{2\leq u\leq 4}a(\Delta_1-\Delta_u)a^T<0,
\end{equation} 
where $a$=$\nabla f (Q_{y}(p_1),\cdots,Q_{y}(p_k))$ is the gradient of $f$ at $\big(Q_{y}(p_1),\cdots,Q_{y}(p_k)\big)$.   
\end{theorem} 
\par

Next, we shall compare the performances of each of the estimators of $\int_{[\alpha,\beta]} Q_{y,N}(p)J(p)dp$ as well as $f(Q_{y,N}(p_1),\cdots,$ $Q_{y,N}(p_k))$ considered in Section \ref{sec 6.2} under SRSWOR, RHC and any HE$\pi$PS sampling designs in terms of their asymptotic variances (see Theorem \ref{Theorem 4}). Let us assume that $G(p)$ is one of $\hat{Q}_y(p)$, $\hat{Q}_{y,RA}(p)$, $\hat{Q}_{y,DI}(p)$ and $\hat{Q}_{y,REG}(p)$, $K_1^*(p_1,p_2)$, $K_2^*(p_1,p_2)$ and $K_3^*(p_1,p_2)$ denote asymptotic covariance kernels of  the quantile process based on $G(p)$  under SRSWOR,  RHC and any HE$\pi$PS sampling designs (see Table \ref{table 5} and \eqref{eq 51}), respectively, and $\{\Delta^*_u: 1\leq u\leq 3\}$ are $k\times k$ matrices such that 
\begin{equation}
((\Delta_u^*))_{wv}=K_u^*(p_w,p_v)\text{ for }1\leq w,v\leq k\text{ and }1\leq u\leq 3.
\end{equation}
Then, we have the following theorem.
\begin{theorem}\label{res 2}
Suppose that $X_i\leq b$ \textit{a.s.} $[\textbf{P}]$ for some $b>0$, $E_{\textbf{P}}(X_i)^{-1}<\infty$, Assumption \ref{ass 0} holds with $0<\lambda<E_{\textbf{P}}(X_i)/b$, Assumptions \ref{ass 4} and \ref{ass 5} hold, and \eqref{eq 150}  holds. Then, we have the following results.\\
\textsl{(i)} The asymptotic variance of the estimator of $\int_{[\alpha,\beta]} Q_{y,N}(p)J(p)dp$ based on $G(p)$ under SRSWOR is smaller than its asymptotic variance under RHC as well as any HE$\pi$PS sampling design, which uses auxiliary information, if and only if 
\begin{equation}\label{eq 66}
\max_{2\leq u\leq 3}\bigg\{\int_{\alpha}^{\beta} \int_{\alpha}^{\beta} \big(K_1^*(p_1,p_2)-K_u^*(p_1,p_2)\big) J(p_1)J(p_2)dp_1 dp_2\bigg\}<0.
\end{equation}
\textsl{(ii)} The asymptotic variance of the estimator of  $f(Q_{y,N}(p_1),\cdots,Q_{y,N}(p_k))$ based on $G(p)$ under SRSWOR is smaller than its asymptotic variance under RHC as well as any HE$\pi$PS  sampling design if and only if 
\begin{equation}\label{eq 67}
\max_{2\leq u\leq 3} a (\Delta_1^*-\Delta_u^*) a^T<0,
\end{equation}
where $a$=$\nabla f (Q_{y}(p_1),\cdots,Q_{y}(p_k))$ is the gradient of $f$ at $\big(Q_{y}(p_1),\cdots,Q_{y}(p_k)\big)$.     
\end{theorem}
The condition, $X_i\leq b$ \textit{a.s.} $[\textbf{P}]$ for some $b>0$, and Assumption \ref{ass 0} along with $0< \lambda<E_{\textbf{P}}(X_i)/b$ ensure that $n\max_{1\leq i\leq N} X_i/\sum_{i=1}^N X_i<1$ for all sufficiently large $\nu$ \textit{a.s.} $[\textbf{P}]$, which is required for implementing $\pi$PS sampling designs. This condition was considered earlier in \cite{deychaudhuri2023a} and \cite{deychaudhuri2023b}. In Section $1$ of the supplement (\cite{supp2024}),  we consider some  examples where the conditions \eqref{eq  64}, \eqref{eq 65}, \eqref{eq  66} and \eqref{eq 67} hold, and some examples where these conditions fail to hold in the cases of various finite population parameters, sampling designs and estimators (see Tables $1$ and $2$ in the supplement (\cite{supp2024})). The above conditions depend on superpopulation quantiles, moments and densities. In practice, one can check these conditions by estimating the above-mentioned superpopulation parameters (see Table $5$ in the supplement (\cite{supp2024})) based on a pilot survey. 
\par

Theorem \ref{res 1} shows that in the case of the estimation of $\int_{[\alpha,\beta]} Q_{y,N}(p)J(p)dp$ and $f(Q_{y,N}(p_1),\cdots,$ $Q_{y,N}(p_k))$, the use of the auxiliary information in the estimation stage may have an adverse effect on the performances of their estimators based on the ratio, the difference and the regression estimators under each of SRSWOR, RHC and any HE$\pi$PS sampling designs. This is in striking contrast to the case of the estimation of the finite population mean, where the use of the auxiliary information in the estimation stage improves the performance of the GREG estimator under these sampling designs (see \cite{deychaudhuri2023a} and \cite{deychaudhuri2023b}). On the other hand, Theorem \ref{res 2} implies that the performance of each of the estimators of the above parameters considered in this paper may become worse under RHC and any HE$\pi$PS sampling deigns, which use the auxiliary information, than their performances under SRSWOR.
\subsubsection{Comparison of the estimators of quantile based location, spread and skewness} It follows from Theorem \ref{res 1} that in the cases of the median, the interquartile range and the Bowley's measure of skewness, the estimator based on the sample quantile becomes more efficient than the estimators based on the ratio, the difference and the regression estimators of the finite population quantile under $P(s,\omega)$ if and only if \eqref{eq 65} holds with $k$, $p_1,\ldots,p_k$ and $a$ as in  Table \ref{table 3} below. Here, $P(s,\omega)$ is one of SRSWOR, RHC and any HE$\pi$PS sampling designs. On the other hand, it follows from Theorem \ref{res 2} that in the cases of the above parameters the performance of the estimator based on $G(p)$ becomes worse under RHC and any HE$\pi$PS sampling deigns, which use the auxiliary information, than its performance under SRSWOR if and only if \eqref{eq 67} holds with $k$, $p_1,\ldots,p_k$ and $a$ as in Table \ref{table 3} below. Here, $G(p)$ is one of $\hat{Q}_y(p)$, $\hat{Q}_{y,RA}(p)$, $\hat{Q}_{y,DI}(p)$ and $\hat{Q}_{y,REG}(P)$.
\begin{table}[h] 
\centering  
\caption{ $k$, $p_1,\ldots,p_k$ and $a$ in \eqref{eq 65} and \eqref{eq 67} for different parameters}

\label{table 3}
\begin{tabular}{cccc}
\hline
Parameter& $k$& $p_1,\ldots,p_k$&$a$\\
\hline
Median& $1$& $p_1$=$0.5$& $1$\\
\hline
Interquartile&\multirow{2}{*}{$2$}& \multirow{2}{*}{$p_1$=$0.25$, $p_2$=$0.75$}&\multirow{2}{*}{$(-1, 1)$}\\
range&&&\\
\hline
Bowley's measure&\multirow{2}{*}{ $3$}&\multirow{2}{*}{$p_1$=$0.25$, $p_2$=$0.5$, $p_3$=$0.75$}&$2\big(Q_y(p_3)-Q_y(p_2),Q_y(p_1)-Q_y(p_3),$\\
of skewness&& &$Q_y(p_2)-Q_y(p_1)\big)\big/\big(Q_y(p_3)-Q_y(p_1)\big)^2$\\
\hline
\end{tabular}
\end{table}
\subsection{Comparison of the sample mean, the sample median and the GREG  estimator} Here, we compare the GREG estimator, say $\hat{\overline{Y}}_{GREG}$ (see \cite{MR0474575}, \cite{deychaudhuri2023a} and references therein),  of the finite population mean $\overline{Y}$=$\sum_{i=1}^N Y_i/N$, the sample mean $\overline{y}$=$\sum_{i\in s} Y_i/n$,  and the sample median $\hat{Q}_{y}(0.5)$ under SRSWOR in terms of asymptotic variances of $\sqrt{n}(\hat{\overline{Y}}_{GREG}-E_{\textbf{P}}(Y_i))$, $\sqrt{n}(\overline{y}-E_{\textbf{P}}(Y_i))$ and $\sqrt{n}(\hat{Q}_{y}(0.5)-Q_y(0.5))$, when the superpopulation median $Q_y(0.5)$ and the superpopulation mean $E_{\textbf{P}}(Y_i)$ are same. 
\begin{theorem}\label{res 3}
Suppose that $Q_y(0.5)$=$E_{\textbf{P}}(Y_i)$, and Assumptions \ref{ass 0} and \ref{ass 3} hold. Then, under SRSWOR, the asymptotic variance of the sample median is smaller than that of the sample mean and the asymptotic variance of the GREG estimator of the mean is smaller than that of the sample median if and only if
\begin{equation}\label{eq 62}
\sigma^2_{y}> 1/4 f_y^{2}( Q_{y}(0.5)),\text{ and }
\end{equation}
\begin{equation}\label{eq 63}
\rho_{xy}^2> (1-\lambda)^{-1}(1-1/4\sigma^2_y f_y^{2}( Q_{y}(0.5)) )
\end{equation}
respectively. Here, $\sigma^2_{y}$ and $f_y$ are the superpopulation variance and density function of $y$,  respectively, and $\rho_{xy}$ is the superpopulation correlation coefficient between $x$ and $y$.
\end{theorem}
In Section $1$ of the supplement (\cite{supp2024}), we construct some examples where the conditions \eqref{eq 62} and \eqref{eq 63} hold, and some examples where these conditions fail to hold (see Tables $3$ and $4$ in the supplement (\cite{supp2024})). As mentioned in the cases of \eqref{eq 64}, \eqref{eq 65}, \eqref{eq 66} and \eqref{eq 67} in the preceding section, in practice, the conditions \eqref{eq 62} and \eqref{eq 63} can also be checked using a pilot survey. 
\par

Theorem \ref{res 3} implies that under SRSWOR, the performance of the sample mean is worse than that of the sample median and the performance of the sample median is worse than that of the GREG estimator if and only if \eqref{eq 62} and \eqref{eq 63} hold, respectively. In the case of a finite population, if the population observations on $y$ are generated from heavy-tailed distributions (e.g., exponential power, student's $t$, etc.) and SRSWOR is used, the sample median becomes more efficient than the sample mean. It is well-known that a similar result holds in the classical set up involving i.i.d. sample observations. However, the GREG estimator of the mean becomes more efficient than the sample median under SRSWOR, whenever $y$ and $x$ are highly correlated. This is in striking contrast to what happens in the case of i.i.d. observations.

\section{Demonstration using real data}\label{sec 5}
In this section, we use the data on irrigated land area for the state of West Bengal in India from the District Census Handbook ($2011$) available in Office of the Registrar General and Census Commissioner, India (\textit{https://censusindia.gov.in}). In West Bengal, lands are irrigated by different sources like canals, wells, waterfalls, lakes, etc., and irrigated land area (in Hectares) in different villages are reported in the above dataset. We consider the population of $14224$ villages having lands irrigated by canals in this state. We use this dataset to demonstrate the accuracy of the asymptotic normal approximations for the distributions of several estimators of several parameters under single stage sampling designs like SRSWOR, LMS and RHC sampling designs. 
\par

We use the same dataset to demonstrate the accuracy of the asymptotic approximations for the distributions of different estimators of different parameters under stratified multistage cluster sampling design with SRSWOR. Note that the above-mentioned population can be divided into $18$ districts, and every district can further be divided into sub districts consisting of villages. We consider districts as strata, sub districts as clusters and villages as population units.
\par

We choose the land area irrigated by canals as the study variable $y$, and the total irrigated land area as the auxiliary variable $x$. We are interested in the estimation of the median and the $\alpha$-trimmed means of $y$, where $\alpha$=$0.1$ and $0.3$. We are also interested in the estimation of the interquartile range and the Bowley's measure of skewness of $y$. For each of the aforementioned parameters, we compute relative biases of the estimators, which are based on the sample quantile, and the ratio, the difference and the regression estimators of the population quantile. We consider $I$=$1000$ samples each of size $n$=$200$ and $n$=$500$ selected using single stage sampling designs mentioned in the $1^{st}$ paragraph of this section. Further, we consider $I$=$1000$ samples each of size $n$=$108$ (a sample of $6$ clusters from each stratum and a sample of $1$ village from each selected cluster) and $n$=$216$ (a sample of $6$ clusters from each stratum and a sample of $2$ villages from each selected cluster) selected using stratified multistage cluster sampling design with SRSWOR. Suppose that $\hat{\theta}$ is an estimator of the parameter $\theta$, and $\hat{\theta}_k$ is the estimate of $\theta$ computed based on the $k^{th}$ sample using a sampling design $P(s)$ for $k$=$1,\ldots,I$. The relative bias of $\hat{\theta}$ under $P(s)$ is computed as $RB(\hat{\theta},P)$=$\sum_{k=1}^I(\hat{\theta}_{k}-\theta_0)/I\theta_0$ (cf. \cite{MR3670194}), where $\theta_0$ is the true value of $\theta$ in the population.  Note that $\theta_0$ is known because we have all the population values available for $y$ and $x$ in the above-mentioned dataset.   We use the $R$ software for drawing samples as well as computing estimators. For sample quantiles, we use weighted.quantile function in $R$. The plots of relative biases for different parameters, estimators, sampling designs and sample sizes are presented in Figures $1$--$4$ and  $13$--$16$  in the supplement (\cite{supp2024}). Also, boxplots of relative biases for different parameters and estimators in the cases of single stage  sampling designs and stratified multistage cluster sampling design with SRSWOR are given in Figure  $25$  in the supplement (\cite{supp2024}).
\par

 Next, we compute coverage probabilities of nominal $90\%$ and $95\%$ confidence intervals (see Section \ref{sec 6}) of the parameters discussed in this section. While computing coverage probabilities, we consider the estimators $\hat{\sigma}^2_1$=$\int_{\alpha}^{\beta}\int_{\alpha}^{\beta}\hat{K}(p_1,p_2)J(p_1)J(p_2)$ $dp_1dp_2$ and $\hat{\sigma}^2_2$=$\hat{a}\hat{\Delta}\hat{a}^T$ discussed in the paragraph preceding Theorem \ref{Theorem 5}. We compute coverage probabilities of nominal $90\%$ and $95\%$ confidence intervals of a parameter by taking the proportion of times confidence intervals constructed based on $I$=$1000$ samples include the true value of the parameter.  We also compute the magnitude of the Monte Carlo standard errors of these coverage probabilities.  The plots of observed coverage probabilities of nominal $90\%$ and $95\%$ confidence intervals for different parameters, estimators, sampling designs and sample sizes are presented in Figures  $5$--$12$ and  $17$--$24$ in the supplement (\cite{supp2024}). Also, boxplots of observed coverage probabilities of nominal $90\%$ and $95\%$ confidence intervals for different parameters and estimators in the cases of single stage sampling designs and stratified multistage cluster sampling design with SRSWOR are given in Figures  $26$ and $27$  in the supplement (\cite{supp2024}).  We also carry out some additional data analysis in the supplement (see Section $6$ of the supplement (\cite{supp2024})). 
\par

 The results obtained from the above data analysis are summarised as follows. 
\begin{itemize}
\item[(i)]  It follows from Figures $1$--$4$ and $13$--$16$ in the supplement (see also the boxplot in Figure $25$ in the supplement (\cite{supp2024})) that for different parameters, estimators, sampling designs and sample sizes considered in this section, relative biases are quite close to $0$ except for the following cases.  Figures  $13$ and $14$ in the supplement (\cite{supp2024}) shows  that for $n$=$200$, the estimator of the interquartile range based on difference estimator under SRSWOR and the estimators of measure of skewness based on ratio, difference and regression estimators under LMS sampling design have somewhat large negative biases compared to the other estimators.  Also,  Figures $4$  and  $16$ in the supplement (\cite{supp2024}) shows  that the estimators of measure of skewness based on ratio, difference and regression estimators under stratified multistage cluster sampling design with SRSWOR have relatively large negative biases compared to the other estimators for both $n$=$108$ and $n$=$216$. 
\item[(ii)]  Figures $5$--$12$ and $17$--$24$ in the supplement (see also the boxplots in Figures $26$ and $27$ in the supplement (\cite{supp2024})) show that for different parameters, estimators, sampling designs and sample sizes, observed coverage probabilities of nominal $90\%$ and $95\%$ confidence intervals are quite close to $90\%$ and $95\%$, respectively, except for the following case. Observed coverage probability of  nominal $95\%$ confidence interval of $\alpha$-trimmed mean with $\alpha$=$0.1$ based on the sample quantile under SRSWOR and sample size $n$=$200$  is $97.2\%$. 
\item[(iii)] Overall, the asymptotic approximations of the distributions of different estimators of different parameters considered in this article seem to work well in finite sample situations. Also, the accuracy of the asymptotic approximations increases as the sample size increases.
\end{itemize}


\begin{appendix}
\section*{Appendix}\label{appn 1} 
\begin{proof}[Proof of Theorem \ref{Cor 1}]
 Recall $\mathds{H}_n$ and $\mathds{W}_n$ from ($31$) in the proof of Proposition \ref{lem s1} in the supplement (\cite{supp2024}), and suppose that $0\leq t_1,\ldots,t_k\leq 1$ for some $k\geq 1$. Then, for $d(i,s)$=$(N\pi_i)^{-1}$, we have 
\begin{align}\label{eq 48}
\begin{split}
&\textbf{c}\big(\mathds{H}_n(t_1),\ldots, \mathds{H}_n(t_k)\big)^T+\sqrt{n/N} \textbf{d}\big(\mathds{W}_N(t_1),\ldots, \mathds{W}_N (t_k)\big)^T\\
&=\textbf{c}\big(\mathds{H}_n(t_1)-\sqrt{n/N}\mathds{W}_N(t_1),\ldots,\mathds{H}_n(t_k)-\sqrt{n/N}\mathds{W}_N(t_k)\big)^T+\\
&\sqrt{n/N}(\textbf{d}+\textbf{c})\big(\mathds{W}_N(t_1),\ldots, \mathds{W}_N(t_k)\big)^T=\textbf{c}\big(\mathds{U}_n(t_1),\ldots, \mathds{U}_n(t_1)\big)^T+\\
&\sqrt{n/N}(\textbf{d}+\textbf{c})\big(\mathds{W}_N(t_1),\ldots, \mathds{W}_N(t_k)\big)^T
\end{split}
\end{align}
given any $\textbf{c}, \textbf{d}\in \mathds{R}^k$ and  $\textbf{c}, \textbf{d}\neq 0$, where $\mathds{U}_n$ is as in ($31$) of the supplement (\cite{supp2024}). Further, suppose that $P(s,\omega)$ denotes a high entropy sampling design satisfying Assumption \ref{ass 2}. Then, it can be shown in the same way as the derivation of the result in ($41$) of the supplement (\cite{supp2024}) that under $\textbf{P}^*$,
\begin{align}\label{eq 46}
\begin{split}
&\textbf{c}\big(\mathds{U}_n(t_1),\ldots, \mathds{U}_n(t_1)\big)^T+\sqrt{n/N}(\textbf{d}+\textbf{c})\big(\mathds{W}_N(t_1),\ldots, \mathds{W}_N(t_k)\big)^T\xrightarrow{\mathcal{L}}\\
&N\big(0,\textbf{c}\Gamma_3\textbf{c}^T+\lambda(\textbf{d}+\textbf{c})\Gamma_4 (\textbf{d}+\textbf{c})^T\big)
\end{split}
\end{align}
for $P(s,\omega)$. Here, $\Gamma_3$ is as in ($38$) of the supplement (\cite{supp2024}), and $\Gamma_4$  as in ($42$) of the supplement (\cite{supp2024}). Thus in view of \eqref{eq 48} and \eqref{eq 46}, we have
\begin{align}\label{eq 49}
\begin{split}
&\big(\mathds{H}_n(t_1),\ldots, \mathds{H}_n(t_k),\sqrt{n/N}\mathds{W}_N (t_1),\ldots, \sqrt{n/N}\mathds{W}_N (t_k)\big)^T\\
&\xrightarrow{\mathcal{L}} N_{2k}(0,\Gamma_5),\text{ for }d(i,s)=(N\pi_i)^{-1}\text{ and } P(s,\omega)\text{ under }\textbf{P}^*,\text{  where }
\end{split}
\end{align}
$$\
\Gamma_5=
\begin{bmatrix}
\Gamma_3+\lambda\Gamma_4 & \lambda\Gamma_4\\
\lambda\Gamma_4 & \lambda\Gamma_4
\end{bmatrix}.
$$
The result stated in \eqref{eq 49} implies weak convergence of finite dimensional distributions of the process $(\mathds{H}_n,\sqrt{n/N}\mathds{W}_N)$ for $d(i,s)$=$(N\pi_i)^{-1}$. Recall from the $3^{rd}$ paragraph in the proof of Proposition \ref{lem s1} that under $\textbf{P}$, 
\begin{equation}\label{eq 43}
\mathds{W}_N=\big\{\sqrt{N}(F_{z,N}(t)-t):t\in [0,1]\big\}\xrightarrow{\mathcal{L}} \mathds{W}
\end{equation}
as $\nu\rightarrow\infty$ in $(\tilde{D}[0,1],\tilde{\mathcal{D}})$ with respect to the Skorohod metric, where $\mathds{W}$ is the standard Brownian bridge in $\tilde{D}[0,1]$ and has almost sure continuous paths. Then, $(\mathds{H}_n, \sqrt{n/N}\mathds{W}_N)$ is tight in $(\tilde{D}[0,1]\times \tilde{D}[0,1], \tilde{\mathcal{D}}\times \tilde{\mathcal{D}})$ with respect to the  Skorohod metric, for $d(i,s)$=$(N\pi_i)^{-1}$ and $P(s,\omega)$ because both $\mathds{H}_n$ and $\sqrt{n/N}\mathds{W}_N$ are tight in $(\tilde{D}[0,1],\tilde{\mathcal{D}})$ with respect to the Skorohod metric, for $d(i,s)$=$(N\pi_i)^{-1}$ and $P(s,\omega)$ in view of \eqref{eq 43} and Proposition \ref{lem s1}. Therefore, under $\textbf{P}^*$ 
\begin{equation}\label{eq 50}
(\mathds{H}_n, \sqrt{n/N}\mathds{W}_N)\xrightarrow{\mathcal{L}}\mathds{V}=(\mathds{V}_1,\mathds{V}_2)
\end{equation}
as $\nu\rightarrow\infty$ in $(\tilde{D}[0,1]\times \tilde{D}[0,1], \tilde{\mathcal{D}}\times \tilde{\mathcal{D}})$ with respect to the  Skorohod metric, for $d(i,s)$=$(N\pi_i)^{-1}$ and $P(s,\omega)$, where $\mathds{V} $ is a mean $0$ Gaussian process in $\tilde{D}[0,1]\times \tilde{D}[0,1]$ with almost sure continuous paths. The covariance kernel of $\mathds{V}$ is obtained from $\Gamma_5$ above. Next, recall from the paragraph preceding Assumption \ref{ass 0} that $F_y$ denotes the superpopulation distribution function of $y$. Then, by \eqref{eq 50}, continuous mapping theorem and Skorohod representation theorem, we have
\begin{equation}\label{eq 44}
(\mathds{H}_n\circ F_y ,\sqrt{n/N}\mathds{W}_N\circ F_y)\xrightarrow{\mathcal{L}} (\mathds{V}_1\circ F_y, \mathds{V}_2\circ F_y)\text{ as }\nu\rightarrow\infty 
\end{equation}
in $(\tilde{D}(\mathds{R})\times \tilde{D}(\mathds{R}),\tilde{\mathcal{D}}_\mathds{R}\times \tilde{\mathcal{D}}_\mathds{R})$ with respect to the sup norm metric, for $d(i,s)$=$(N\pi_i)^{-1}$ and $P(s,\omega)$. Here, $\tilde{D}(\mathds{R})$ denotes the class of all bounded right continuous functions defined on $\mathds{R}$ with finite left limits, and $\tilde{\mathcal{D}}_{\mathds{R}}$ denotes the $\sigma$-field on $\tilde{D}(\mathds{R})$ generated by the open balls (ball $\sigma$-field) with respect to the sup norm metric. Note that $(\mathds{V}_1\circ F_y, \mathds{V}_2\circ F_y)$ has almost sure continuous paths because $F_y$ is continuous by Assumption \ref{ass 3}. Let us now consider the quantile map 
\begin{equation}\label{eq 58}
\phi(F)=F^{-1}=Q\text{ for any distribution function F},
\end{equation}  
where $F^{-1}(p)$=$Q(p)$=$\inf\{t\in\mathds{R}: F(t)\geq p\}$ for any $0<p<1$. Now, suppose that $\tilde{D}$ denotes the set of distribution functions on $\mathds{R}$ restricted to $[Q_{y}(\alpha)-\epsilon, Q_y(\beta)+\epsilon]$ for some $0<\alpha<\beta<1$ and $\epsilon>0$, where $Q_y$ is the superpopulation quantile function of $y$. Then, it can be shown in the same way as the proof of Lemma $3.9.23$--\textsl{(i)} in \cite{van1996weak} that $\phi$: $\tilde{D} \subset\tilde{D}[Q_{y}(\alpha)-\epsilon, Q_y(\beta)+\epsilon]\rightarrow D[\alpha,\beta]$ is Hadamard differentiable at $F_y$ tangentially to $C[Q_y(\alpha)-\epsilon,Q_y(\beta)+\epsilon]$. Note that 
\begin{align}\label{eq 57}
\begin{split}
&\mathds{H}_n\circ F_y=\{\sqrt{n}(\hat{F}_y(t)-F_{y}(t)):t\in \mathds{R}\}\text{ and }\\
&\sqrt{n/N}\mathds{W}_N\circ F_y=\{\sqrt{n}(F_{y,N}(t)-F_{y}(t)):t\in \mathds{R}\},
\end{split}
\end{align}
where $\hat{F}_y(t)$=$\sum_{i\in s} d(i,s)\mathds{1}_{[Y_i\leq t]}/\sum_{i\in s}d(i,s)$ and $F_{y,N}(t)$=$\sum_{i=1}^N \mathds{1}_{[Y_i\leq t]}/N$. This is because $F_{y}$ is continuous by Assumption \ref{ass 3}. Then by \eqref{eq 44}, \eqref{eq 57}, functional delta method (see Theorem $3.9.4$ in \cite{van1996weak}) and Hadamard differentiability of $\phi$, we have
\begin{align}\label{eq 54}
\begin{split}
&\big(\{\sqrt{n}(\hat{Q}_y(p)-Q_{y}(p)):p\in[\alpha,\beta]\},\{\sqrt{n}(Q_{y,N}(p)-Q_{y}(p)):p\in[\alpha,\beta]\}\big)\xrightarrow{\mathcal{L}}\\
& (-\tilde{\mathds{V}}_1,-\tilde{\mathds{V}}_2)/f_y\circ Q_y
\end{split}
\end{align} 
as $\nu\rightarrow\infty$ in $(D[\alpha,\beta]\times D[\alpha,\beta], \mathcal{D}\times \mathcal{D})$ with respect to the sup norm metric, for  $d(i,s)$=$(N\pi_i)^{-1}$ and $P(s,\omega)$. Here, $f_y$ is the superpopulation density function of $y$, $(\tilde{\mathds{V}}_1,\tilde{\mathds{V}}_2)$ is a mean $0$ Gaussian process in $D[\alpha,\beta]\times D[\alpha,\beta]$, and $(\tilde{\mathds{V}}_1,\tilde{\mathds{V}}_2)\overset{\mathcal{L}}{=}(\mathds{V}_1,\mathds{V}_2)$. Then, by continuous mapping theorem, we have
\begin{equation}\label{eq 55}
\{\sqrt{n}(\hat{Q}_y(p)-Q_{y,N}(p)):p\in[\alpha,\beta]\}\xrightarrow{\mathcal{L}}-(\tilde{\mathds{V}}_1-\tilde{\mathds{V}}_2)/f_y\circ Q_y=\mathds{Q}\text{ (say) }
\end{equation}
as $\nu\rightarrow\infty$ in $(D[\alpha,\beta], \mathcal{D})$ with respect to the sup norm metric, for $d(i,s)$=$(N\pi_i)^{-1}$ and $P(s,\omega)$. The covariance kernel of $\mathds{Q}$ is obtained from the matrix
$$
\big[I_k\text{ }-I_k\big]\Gamma_5\big[I_k\text{ }-I_k\big]^T
$$
Here, $I_k$ is the $k\times k$ identity matrix. 
\par

We shall next show the weak convergence of the quantile processes constructed based on $\hat{Q}_{y,RA}(p)$, $\hat{Q}_{y,DI}(p)$ and $\hat{Q}_{y,REG}(p)$ in $(D[\alpha,\beta],\mathcal{D})$ with respect to the sup norm metric, for $d(i,s)$=$(N\pi_i)^{-1}$ and $P(s,\omega)$. Recall $\hat{Q}_x$ and $Q_{x,N}$ from Section \ref{sec 3}, and $Q_x$ from the paragraph preceding Assumption \ref{ass 0}. Note that 
\begin{align}\label{eq 21}
\begin{split}
&\sqrt{n}\big(\hat{Q}_{y,RA}(p)-Q_{y,N}(p)\big)=\sqrt{n}\big(\hat{Q}_{y}(p)-Q_{y}(p)\big)- \sqrt{n}\big(Q_{y,N}(p)-Q_{y}(p)\big)+\\
&\big(\hat{Q}_{y}(p)/\hat{Q}_{x}(p)\big)\big\{ \sqrt{n}\big(Q_{x,N}(p)-Q_{x}(p)\big)-\sqrt{n}\big(\hat{Q}_{x}(p)-Q_{x}(p)\big)\big\}.
\end{split}
\end{align} 
First, it can be shown in the same way as the derivation of the results in \eqref{eq 44} and \eqref{eq 54} that under $\textbf{P}^*$, $\big(\{\sqrt{n}(\hat{F}_y(t)-F_{y}(t)):t\in\mathds{R}\},\{\sqrt{n}(F_{y,N}(t)-F_{y}(t)):t\in\mathds{R}\}, \{\sqrt{n}(\hat{F}_x(t)-F_{x}(t)):t\in\mathds{R}\},\{\sqrt{n}(F_{x,N}(t)-F_{x}(t)):t\in\mathds{R}\}\big)$ converges weakly to some mean $0$ Gaussian process with almost sure continuous paths as $\nu\rightarrow\infty$, and hence $\big(\{\sqrt{n}(\hat{Q}_y(p)-Q_{y}(p)):p\in[\alpha,\beta]\},\{\sqrt{n}(Q_{y,N}(p)-Q_{y}(p)):p\in[\alpha,\beta]\},\{\sqrt{n}(\hat{Q}_x(p)-Q_{x}(p)):p\in[\alpha,\beta]\},\{\sqrt{n}(Q_{x,N}(p)-Q_{x}(p)):p\in[\alpha,\beta]\}\big)$ converges weakly to some mean $0$ Gaussian process with almost sure continuous paths as $\nu\rightarrow\infty$. Then, we have 
\begin{equation}\label{eq 22}
\sup_{p\in [\alpha,\beta]}|\hat{Q}_{y}(p)/\hat{Q}_{x}(p)-Q_y(p)/Q_x(p)|\xrightarrow{p} 0
\end{equation}
as $\nu\rightarrow\infty$ under $\textbf{P}^*$. Further, it can be shown in the same way as the derivation of the result in \eqref{eq 55} that under $\textbf{P}^*$, 
\begin{align}\label{eq 56}
\begin{split}
&\big\{\sqrt{n}(\hat{Q}_{y}(p)-Q_{y,N}(p))+\big(Q_{y}(p)/Q_{x}(p)\big)\times\\
&\sqrt{n}\big(Q_{x,N}(p)-\hat{Q}_{x}(p)\big):p\in [\alpha,\beta]\big\}\xrightarrow{\mathcal{L}}\mathds{Q}\text{ as }\nu\rightarrow\infty
\end{split}
\end{align}
in $(D[\alpha,\beta],\mathcal{D})$ with respect to the sup norm metric, for $d(i,s)$=$(N\pi_i)^{-1}$ and $P(s,\omega)$. Here, $\mathds{Q}$ is a mean $0$ Gaussian process in $\tilde{D}[\alpha,\beta]$ with almost sure continuous paths. Therefore, in view of \eqref{eq 21}--\eqref{eq 56}, 
\begin{equation}
\big\{\sqrt{n}(\hat{Q}_{y,RA}(p)-Q_{y,N}(p)): p\in [\alpha,\beta]\big\}\xrightarrow{\mathcal{L}}\mathds{Q}\text{ as }\nu\rightarrow\infty
\end{equation}
in $(D[\alpha,\beta],\mathcal{D})$ with respect to the sup norm metric, for $d(i,s)$=$(N\pi_i)^{-1}$ and $P(s,\omega)$ under $\textbf{P}^*$. The covariance kernel of $\mathds{Q}$ is obtained from the asymptotic covariance kernel of $\big(\{\sqrt{n}(\hat{F}_y(t)-F_{y}(t)):t\in\mathds{R}\},\{\sqrt{n}(F_{y,N}(t)-F_{y}(t)):t\in\mathds{R}\}, \{\sqrt{n}(\hat{F}_x(t)-F_{x}(t)):t\in\mathds{R}\},\{\sqrt{n}(F_{x,N}(t)-F_{x}(t)):t\in\mathds{R}\}\big)$. Next, note that 
\begin{align}\label{eq 23}
\begin{split}
&\sqrt{n}\big(\hat{Q}_{y,DI}(p)-Q_{y,N}(p)\big)=\sqrt{n}\big(\hat{Q}_{y}(p)-Q_{y}(p)\big)- \sqrt{n}\big(Q_{y,N}(p)-Q_{y}(p)\big)+\\
&\big(\sum_{i\in s}\pi_i^{-1}Y_i\big/\sum_{i\in s}\pi_i^{-1}X_i\big)\big\{ \sqrt{n}\big(Q_{x,N}(p)-Q_{x}(p)\big)-\sqrt{n}\big(\hat{Q}_{x}(p)-Q_{x}(p)\big)\big\}\text{ and }
\end{split}
\end{align}
\begin{align}\label{eq 24}
\begin{split}
&\sqrt{n}\big(\hat{Q}_{y,REG}(p)-Q_{y,N}(p)\big)=\sqrt{n}\big(\hat{Q}_{y}(p)-Q_{y}(p)\big)- \sqrt{n}\big(Q_{y,N}(p)-Q_{y}(p)\big)+\\
&\big(\sum_{i\in s}\pi_i^{-1}X_i Y_i\big/\sum_{i\in s}\pi_i^{-1}X_i^2\big)\big\{ \sqrt{n}\big(Q_{x,N}(p)-Q_{x}(p)\big)-\sqrt{n}\big(\hat{Q}_{x}(p)-Q_{x}(p)\big)\big\}.
\end{split}
\end{align}
It can be shown using Theorem $6.1$ in \cite{MR0178555} and similar arguments in the last paragraph of the proof of Proposition \ref{lem s1} in the supplement (\cite{supp2024}) that under $P(s,\omega)$, 
\begin{equation}
\sum_{i\in s}(N\pi_i)^{-1}\textbf{W}_i-\sum_{i=1}^N \textbf{W}_i/N\xrightarrow{p} 0\text{ as }\nu\rightarrow\infty \textit{ a.s. } [\textbf{P}]
\end{equation}
because $E_{\textbf{P}}||\textbf{W}_i||^2<\infty$. Here, $\textbf{W}_i$=$(X_i,Y_i,X_iY_i,X_i^2)$. Since $\sum_{i=1}^N \textbf{W}_i/N\rightarrow E_{\textbf{P}}(\textbf{W}_i)$ as $\nu\rightarrow\infty$ \textit{a.s.} $[\textbf{P}]$ by SLLN, we have 
\begin{align}
\begin{split} 
& \sum_{i\in s}\pi_i^{-1}Y_i\bigg/\sum_{i\in s}\pi_i^{-1} X_i\xrightarrow{p}E_{\textbf{P}}(Y_i)\big/E_{\textbf{P}}(X_i)\text{ and }\\
&\sum_{i\in s}\pi_i^{-1} X_i Y_i\bigg/\sum_{i\in s}\pi_i^{-1} X^2_i\xrightarrow{p}E_{\textbf{P}}(X_iY_i)/E_{\textbf{P}}(X_i)^2
 \end{split}
 \end{align}
as $\nu\rightarrow\infty$ for $P(s,\omega)$ under $\textbf{P}^*$. Therefore, using \eqref{eq 23}, \eqref{eq 24} and similar arguments as in the case of $\big\{\sqrt{n}(\hat{Q}_{y,RA}(p)-Q_{y,N}(p)):p\in[\alpha,\beta]\big\}$, we can say that under $\textbf{P}^*$, $\big\{\sqrt{n}(\hat{Q}_{y,DI}(p)-Q_{y,N}(p)):p\in[\alpha,\beta]\big\}$ and $\big\{\sqrt{n}(\hat{Q}_{y,REG}(p)-Q_{y,N}(p)):p\in[\alpha,\beta]\big\}$ converge weakly to a mean $0$ Gaussian process with almost sure continuous paths in $(D[\alpha,\beta],\mathcal{D})$ with respect to the sup norm metric, for $d(i,s)$=$(N\pi_i)^{-1}$ and $P(s,\omega)$.
\end{proof} 
\end{appendix}

\section*{Acknowledgements}
The authors would like to thank three anonymous reviewers and an associate editor for their critical comments and constructive suggestions.


\section*{Supplementary material}
In this supplement, we first discuss the Assumption \ref{ass 2}, and \eqref{eq 64}, \eqref{eq 65} and \eqref{eq 66}--\eqref{eq 63}.  Then, we provide estimators of some superpopulation parameters related to asymptotic covariance kernels that appear in Sections \ref{sec 4} and \ref{subsec 3.3}. We state and prove some technical lemmas which are required to prove Theorems in Sections \ref{sec 4}--\ref{sec 7}.  Further, we give the proofs of Propositions \ref{lem s1}, \ref{lem 10} and \ref{lem s2}, and Theorems \ref{Cor 3},  \ref{Thm 2}, \ref{Theorem 4}--\ref{Theorem 7}  and \ref{res 1}--\ref{res 3}.  This supplement also contains some additional data analysis, and results of the data analysis described in Section \ref{sec 5}.




\newpage
\vskip .55cm
\noindent
Anurag Dey\\
{\it Indian Statistical Institute, Kolkata}
\vskip 1.3pt
\noindent
E-mail: deyanuragsaltlake64@gmail.com
\vskip 1.3pt

\noindent
Probal Chaudhuri\\
{\it Indian Statistical Institute, Kolkata}
\vskip 1.3pt
\noindent
E-mail:probalchaudhuri@gmail.com
\end{document}




\title{\textbf{Supplementary material for ``Quantile processes and their applications in finite populations"}}
\author{Anurag Dey and Probal Chaudhuri\\
\textit{Indian Statistical Institute, Kolkata}}
\maketitle


\begin{abstract}
 In this supplement, we first discuss the Assumption $2$, and ($18$), ($19$) and ($21$)--($24$) of the main text.  Then, we provide estimators of some superpopulation parameters related to asymptotic covariance kernels that appear in Sections $3$ and $4$ of the main article. We state and prove some technical lemmas which are required to prove Theorems in Sections $3$--$6$ of the main paper.  Further, we give the proofs of Propositions  $3.1$,  $3.2$ and $4.1$, and Theorems $3.2$, $4.1$, $5.1$--$5.4$ and  $6.1$--$6.3$  in the main article.  This supplement also contains some additional data analysis, and results of the data analysis described in Section $7$ of the main paper..\\ 
\end{abstract}

\textbf{Keywords and phrases:} Auxiliary information, Difference estimator, Hadamard differentiability, High entropy sampling design, Ratio estimator, Regression estimator, RHC sampling design, Stratified multistage cluster sampling design, Skorohod metric, Sup norm metric.

\section{ Discussion of some conditions and related results}\label{sec s1} 
 In this section, we first demonstrate some situations, when the Assumption $2$ in the main article holds. Recall from the paragraph preceding Assumption $1$ in the main text that $Q_y(p)$=$\inf\{t\in\mathds{R}:F_y(t)\geq p\}$ and $Q_x(p)$=$\inf\{t\in\mathds{R}:F_x(t)\geq p\}$ are superpopulation $p^{th}$ quantiles of $y$ and $x$, respectively, and $\textbf{V}_i$=$\textbf{G}_i-\sum_{i=1}^N \textbf{G}_i/N$ for $i$=$1,\ldots,N$, where 
$$
\textbf{G}_i=\big(\mathds{1}_{[Y_i\leq Q_y(p_1)]},\ldots,\mathds{1}_{[Y_i\leq Q_y(p_k)]},\mathds{1}_{[X_i\leq Q_x(p_1)]}\ldots,\mathds{1}_{[X_i\leq Q_x(p_k)]}\big)
$$
 for $p_1,\ldots,p_k\in (0,1)$ and $k\geq 1$. Then, we state the following lemma.
\begin{lemma}\label{lem s1}
Suppose that Assumptions $1$, $4$ and $5$ in the main text hold. Then, Assumption $2$ in the main text holds under SRSWOR and LMS sampling design. Moreover, if $X_i\leq b$ \textit{a.s.} $[\textbf{P}]$ for some $b>0$, $E_{\textbf{P}}(X_i)^{-1}<\infty$, Assumption $1$ holds with $0<\lambda<E_{\textbf{P}}(X_i)/b$, and Assumption $5$ holds, then Assumption $2$ holds under any $\pi$PS sampling design.
\end{lemma}
\begin{proof}
 Given any $k\geq 1$ and $p_1,\ldots,p_k\in (0,1)$ let us denote $(1/N^2)\sum_{i=1}^{N}(\textbf{V}_i-\textbf{T}\pi_i)^T(\textbf{V}_i-\textbf{T}\pi_i)(\pi_i^{-1}-1)$ by $\Sigma_N$. Here, $\textbf{T}$=$\sum_{i=1}^N \textbf{V}_i(1-\pi_i)/\sum_{i=1}^N \pi_i(1-\pi_i)$, and the $\pi_i$'s are inclusion probabilities.  Note that  
 $$
 n \Sigma_N=(1-n/N)\bigg(\sum_{i=1}^N\textbf{V}_i^T\textbf{V}_i/N-\overline{\bf{V}}^T\overline{\bf{V}}\bigg)
 $$ 
 under SRSWOR. Then, 
\begin{equation}\label{eq 19}
 n\Sigma_N\rightarrow (1-\lambda)E_{\textbf{P}}(G_i-E_{\textbf{P}}(G_i))^T(G_i-E_{\textbf{P}}(G_i))\text{ as }\nu\rightarrow\infty\textit{ a.s. } [\textbf{P}]
\end{equation}
by Assumption $1$ and SLLN. Note that $E_{\textbf{P}}(G_i-E_{\textbf{P}}(G_i))^T(G_i-E_{\textbf{P}}(G_i))$ is p.d. by Assumption $5$. Thus Assumption $2$--\textsl{(i)} holds under SRSWOR. Assumption $2$--\textsl{(ii)} holds under SRSWOR trivially. 
\par
\vspace{.1cm}

Next, suppose that $\Sigma_N^{(1)}$ and $\Sigma_N^{(2)}$ denote $(1/N^2)\sum_{i=1}^{N}(\textbf{V}_i-\textbf{T}\pi_i)^T(\textbf{V}_i-\textbf{T}\pi_i)(\pi_i^{-1}-1)$ under LMS sampling design and SRSWOR, respectively, and $\{\pi_{i}^{(1)}\}_{i=1}^N$ denote inclusion probabilities of LMS sampling design. Then, we have 
\begin{align}\label{eq 20}
\begin{split}
&\pi_{i}^{(1)}=(n-1)/(N-1)+(X_i/\sum_{i=1}^N X_i)((N-n)/(N-1))\text{ and }\\
&\pi_{i}^{(1)}-n/N=-(N-n)(N(N-1))^{-1}(X_i/\overline{X}-1).
\end{split}
\end{align}
 Further,
$$
\frac{|\pi_{i}^{(1)}-n/N|}{n/N}=\frac{N-n}{n(N-1)}\left|\frac{X_i}{\overline{X}}-1\right|\leq\frac{N-n}{n(N-1)}\left(\frac{\max_{1\leq i \leq N} X_{i}}{\min_{1\leq i \leq N} X_{ i}}+1\right).
$$
Therefore, $\max_{1\leq i \leq N}|N\pi_{i}^{(1)}/n-1|\rightarrow0$ as $\nu\rightarrow\infty$ \textit{a.s.} $[\textbf{P}]$ by Assumption $4$. It can be shown using this latter result that $n(\Sigma_N^{(1)}-\Sigma_N^{(2)})\rightarrow 0$ as $\nu\rightarrow\infty$ \textit{a.s.} $[\textbf{P}]$. Therefore, Assumption $2$--\textsl{(i)} holds under LMS sampling design in view of \eqref{eq 19}. Assumption $2$--\textsl{(ii)} holds under LMS sampling design by Assumption $1$ and \eqref{eq 20}. 
\par
\vspace{.1cm}

Next, under any $\pi$PS sampling design (i.e., a sampling design with $\pi_i$=$nX_i/\sum_{i=1}^N X_i$), we have 
\begin{align}\label{eq 29}
\begin{split}
&\lim_{\nu\rightarrow\infty}n\Sigma_N=E_{\textbf{P}}\big[ \big\{\textbf{G}_i-E_{\textbf{P}}(\textbf{G}_i)+\lambda\chi^{-1}\mu_x^{-1} C_{xg}X_i \big\}^T \times\\
& \big\{\textbf{G}_i-E_{\textbf{P}}(\textbf{G}_i)+\lambda\chi^{-1}\mu_x^{-1} C_{xg} X_i\big\}\big\{ \mu_x/X_i-\lambda\big\}\big]  \textit{ a.s. } [\textbf{P}]
\end{split}
\end{align}
 by SLLN because $E_{\textbf{P}}(X_i)^{-1}<\infty$ and  Assumption $1$ holds. Here, $\mu_x$=$E_{\textbf{P}}(X_i)$, $\chi$=$\mu_x-\lambda (E_{\textbf{P}}(X_i)^2/\mu_x)$ and $C_{xg}$=$E_{\textbf{P}}[(\textbf{G}_i-E_{\textbf{P}}(\textbf{G}_i))X_i]$. The matrix on the right hand side of \eqref{eq 29} is p.d. because  $X_i\leq b$ \textit{a.s.} $[\textbf{P}]$ for some $b>0$, Assumption $5$ holds and Assumption $1$ holds with $0<\lambda<E_{\textbf{P}}(X_i)/b$. Thus Assumption $2$--\textsl{(i)} holds under any $\pi$PS sampling design. Assumption $2$--\textsl{(ii)} holds under any $\pi$PS sampling design by Assumption $4$. This completes the proof of the lemma. 
\end{proof}
Next, we consider some examples where the conditions ($18$) and ($19$) in the main article hold, and some examples where these conditions fail to hold. Suppose that $Y_i$'s have a truncated normal distribution with distribution function $\big(\Phi((t-\mu)/\sigma)-\Phi((\log(0.05)-\mu)/\sigma)\big)\big/\big(\Phi((\log(500)-\mu)/\sigma)-\Phi((\log(0.05)-\mu)/\sigma)\big)$ for $\mu\in\{u\}_{u=0}^{10}$ and $\sigma$=$1$, $X_i$=$exp(Y_i)$ for $i$=$1,\ldots,N$, and $\lambda$=$(20+\lfloor500/E_{\textbf{P}}(X_i)\rfloor)^{-1}$. Then, the conditions ($18$) and ($19$) are discussed in Table \ref{table 8} below in the cases of various finite population parameters and sampling designs. Here, $\Phi(\cdot)$ is the distribution function of the standard normal distribution.
\clearpage

\begin{table}[h!] 
\centering  
\caption{Discussion of the conditions ($18$) and ($19$) in the main text}

\label{table 8}
\renewcommand{\arraystretch}{0.75}
\begin{tabular}{ccccc}
\hline
&&\multicolumn{3}{c}{ Sampling design}\\
\hline
Parameter& The condition  & SRSWOR& RHC& HE$\pi$PS\\
 \hline
\multirow{2}{*}{ Median }& ($19$) holds for&$\mu\geq 5$ &$\mu\geq 5$& $\mu\geq 5$\\  
&($19$) does not hold for& $\mu\leq 4$&$\mu\leq 4$& $\mu\leq 4$\\
\hline
$\alpha$-trimmed mean &($18$) holds for&$\mu\leq 1$ &$\mu\leq 1$& $\mu\leq 1$\\  
with $\alpha$=$0.1$&($18$) does not hold for& $\mu\geq 2$&$\mu\geq 2$& $\mu\geq 2$\\
\hline
$\alpha$-trimmed mean & ($18$) holds for&$\mu\geq 5$ &$\mu\geq 5$& $\mu\geq 5$\\ 
with $\alpha$=$0.3$& ($18$) does not hold for&$\mu\leq 4$&$\mu\leq 4$& $\mu\leq 4$\\
\hline
Interquartile & ($19$) holds for&$\mu\geq 4$ &$\mu\geq 4$& $\mu\geq 4$\\  
range& ($19$) does not hold for& $\mu\leq 3$&$\mu\leq 3$& $\mu\leq 3$\\
\hline
Bowley's measure & ($19$) holds for&$\mu\geq 4$&$\mu\geq 5$&$\mu\geq 5$\\  
of skewness& ($19$) does not hold for&  $\mu\leq 3$& $\mu\leq 4$&  $\mu\leq 4$\\
\hline
\end{tabular}
\end{table}
Next, we consider some examples where the conditions ($21$) and ($22$) in the main article hold, and some examples where these conditions fail to hold. Suppose that $Y_i$'s have a truncated normal distribution with distribution function $\big(\Phi((t-\mu)/\sigma)-\Phi((\log(200)-\mu)/\sigma)\big)\big/\big(\Phi((\log(300)-\mu)/\sigma)-\Phi((\log(200)-\mu)/\sigma)\big)$ for $\mu$=$10$ and $\sigma\in \{u\}_{u=1}^{10}$, $X_i$=$exp(Y_i)$ for $i$=$1,\ldots,N$, and $\lambda$=$(20+\lfloor 300/E_{\textbf{P}}(X_i)\rfloor)^{-1}$. Then, the conditions ($21$) and ($22$) are discussed in Table \ref{table 9} below in the cases of various finite population parameters and their estimators. 
\begin{table}[h] 
\centering  
\caption{Discussion of the conditions ($21$) and ($22$) in the main text}

\label{table 9}
\renewcommand{\arraystretch}{0.75}
\begin{tabular}{cccccc}
\hline
&&\multicolumn{4}{c}{ Estimator based on}\\
\hline
Parameter& The condition  & $\hat{Q}_y(p)$& $\hat{Q}_{y,RA}(p)$& $\hat{Q}_{y,DI}(p)$& $\hat{Q}_{y,REG}(p)$\\
 \hline
\multirow{2}{*}{ Median }& ($22$) holds for&$\sigma\leq 8$&$\sigma\leq 8$& $\sigma\leq 8$&$\sigma\leq 8$\\  
& ($22$) does not hold for& $\sigma\geq 9$ &$\sigma\geq 9$ &$\sigma\geq 9$ &$\sigma\geq 9$ \\
\hline
$\alpha$-trimmed mean & ($21$) holds for&$\sigma\geq 9$ &$\sigma\geq 9$& $\sigma\geq 9$& $\sigma\geq 9$\\  
with $\alpha$=$0.1$& ($21$) does not hold for& $\sigma\leq 8$&$\sigma\leq 8$& $\sigma\leq 8$& $\sigma\leq 8$\\
\hline
$\alpha$-trimmed mean & ($21$) holds for&$\sigma$=$1$&$\sigma$=$1$& $\sigma$=$1$& $\sigma\leq 2$\\ 
with $\alpha$=$0.3$& ($21$) does not hold for& $\sigma\geq 2$& $\sigma\geq 2$&  $\sigma\geq 2$&$\sigma\geq 3$\\
\hline
Interquartile & ($22$) holds for&$\sigma\leq 8$&$\sigma\leq 8$& $\sigma\leq 2$& $\sigma\leq 2$\\  
range& ($22$) does not hold for& $\sigma\geq 9$&$\sigma\geq 9$& $\sigma\geq 3$& $\sigma\geq 3$\\
\hline
Bowley's&\multirow{2}{*}{ ($22$) holds for } &$\sigma\geq 2$ \& &$\sigma\geq 2$ \& & $\sigma\geq 3$ \& & $\sigma\geq 3$ \& \\ 
 measure &&$\sigma\leq 8$&  $\sigma\leq 8$& $\sigma\leq 8$& $\sigma\leq 8$\\
of&\multirow{2}{*}{($22$) does not hold for}& $\sigma$=$1$ \& &$\sigma$=$1$ \& & $\sigma\leq 2$ \& & $\sigma\leq 2$ \& \\
 skewness&&$\sigma\geq 9$&$\sigma\geq 9$& $\sigma\geq 9$& $\sigma\geq 9$\\ 
\hline
\end{tabular}
\end{table}
\par

Finally, we discuss the conditions ($23$) and ($24$) of the main text in Tables \ref{table 3} and \ref{table 11} below for different superpopulation distributions of $Y_i$'s and $X_i$'s, and different values of $\lambda$.
\begin{table}[h] 
\centering  
\caption{Discussion of the condition ($23$) in the main article}
\label{table 3}
\begin{threeparttable}[b]
\begin{tabular}{cc}
\hline
Superpopulation distribution of $Y_i$'s& The condition ($23$) holds iff\\
\hline
Exponential power distribution with&  \multirow{2}{*}{\tnote{1} \ \ $\alpha^2\Gamma(3/\alpha)>\Gamma^3 (1/\alpha)$} \\
location $\mu\in \mathds{R}$, scale $\sigma>0$  and shape $\alpha>0$&\\
\hline
Student's $t$-distribution with& \multirow{2}{*}{\tnote{1} \ \  $4 \Gamma^2((m+1)/2)>(m-2)\pi\Gamma^2 (m/2)$}\\
degrees of freedom $m>2$&\\
\hline
\end{tabular}
\begin{tablenotes}
\item[1] Here, $\Gamma(\cdot)$ denotes the gamma function.
\end{tablenotes}
\end{threeparttable}
\end{table}
\begin{table}[h] 
\centering  
\caption{Discussion of the condition ($24$) in the main article}

\label{table 11}
\begin{tabular}{ccc}
\hline
 Superpopulation & Superpopulation & \multirow{2}{*}{$\lambda$}\\
 distribution of $Y_i$'s& distribution of $X_i$'s& \\
\hline
 Normal distribution & Any distribution& The condition ($24$)\\
 with mean $\mu\in \mathds{R}$& supported & holds for any \\
  and variance $\sigma^2>0$&on $(0,\infty)$&$\lambda\in (0,1)$\\
 \hline
 Standard  &$X_i$=$\max\{Y_i,0\}$& The condition ($24$)\\
 Laplace&for &holds iff \\
distribution&$i$=$1,\ldots,N$& $\lambda\in(0,0.25)$\\
\hline
\end{tabular}
\end{table}
\section{Estimators of some superpopulation parameters related to asymptotic covariance kernels}
In this section, we provide estimators of various superpopulation parameters involved in the expression of $\zeta_i(p)$ in  Table $1$  in the main article for high entropy and RHC sampling designs. We also provide estimators of various superpopulation parameters involved in the expression of $\zeta_{hjl}^{\prime}(p)$ in Table $3$ in the main text for stratified multistage cluster sampling design with SRSWOR. Note that $\sqrt{n}(\hat{Q}_{y}(p+1/\sqrt{n})-\hat{Q}_{y}(p-1/\sqrt{n}))/2$ with $d(i,s)$=$(N\pi_i)^{-1}$ in Table \ref{table 7} below was considered as an estimator of $1/f_y(Q_y(p))$ earlier in \cite{MR1292550}.

\begin{table}[h!]
\caption{Estimators of various superpopulation parameters involved in the expression of $\zeta_i(p)$ in  Table $1$  in the main article for high entropy and RHC sampling designs}
\label{table 7}
\renewcommand{\arraystretch}{0.85}
\begin{center}
\begin{tabular}{ccc} 
\hline
Parameters& \multicolumn{2}{c}{Estimators}\\
\hline
&High entropy sampling designs& RHC sampling design\\
\hline
$Q_{y}(p)$& $\hat{Q}_{y}(p)$ with $d(i,s)$=$(N\pi_i)^{-1}$& $\hat{Q}_{y}(p)$ with $d(i,s)$=$A_i/NX_i$\\
\hline
$Q_{x}(p)$ & $\hat{Q}_{x}(p)$ with $d(i,s)$=$(N\pi_i)^{-1}$& $\hat{Q}_{x}(p)$ with $d(i,s)$=$A_i/NX_i$\\
\hline
\multirow{2}{*}{$1/f_y(Q_y(p))$}& $\sqrt{n}(\hat{Q}_{y}(p+1/\sqrt{n})-$& $\sqrt{n}(\hat{Q}_{y}(p+1/\sqrt{n})-$\\
& $\hat{Q}_{y}(p-1/\sqrt{n}))/2$& $\hat{Q}_{y}(p-1/\sqrt{n}))/2$\\ 
\hline
\multirow{2}{*}{$1/f_x(Q_y(p))$}& $\sqrt{n}(\hat{Q}_{x}(p+1/\sqrt{n})-$& $\sqrt{n}(\hat{Q}_{x}(p+1/\sqrt{n})-$\\
& $\hat{Q}_{x}(p-1/\sqrt{n}))/2$& $\hat{Q}_{x}(p-1/\sqrt{n}))/2$\\
\hline
$E_{\textbf{P}}(Y_i)$&$\sum_{i\in s}Y_i/N\pi_i$&$\sum_{i\in s} Y_iA_i/NX_i$\\
\hline
$E_{\textbf{P}}(X_i)$&$\sum_{i\in s} X_i/N\pi_i$& $\sum_{i\in s} A_i/N$\\
\hline
$E_{\textbf{P}}(X_iY_i)$ &$\sum_{i\in s}X_iY_i/N\pi_i$&$\sum_{i\in s} Y_i A_i/N$\\
\hline
$E_{\textbf{P}}(X_i)^2$ & $\sum_{i\in s}X_i^2/N\pi_i$ & $\sum_{i\in s} X_i A_i/N$\\
\hline
\end{tabular}
\end{center}
\end{table}

\begin{table}[h!]
\caption{Estimators of various superpopulation parameters involved in the expression of $\zeta_{hjl}^{\prime}(p)$ in Table $3$ in the main article for stratified multistage cluster sampling design with SRSWOR}
\label{table 10}
\renewcommand{\arraystretch}{1}
\begin{threeparttable}[b]
\centering
\begin{tabular}{cc} 
\hline
Parameters& Estimators\\
\hline
$Q_{y,H}(p)$& $\hat{Q}_{y}(p)$ with $d(i,s)$=$(N\pi_i)^{-1}$=$M_h N_{hj}/N m_h r_h$\\
\hline
$Q_{x,H}(p)$ & $\hat{Q}_{x}(p)$ with $d(i,s)$=$(N\pi_i)^{-1}$=$M_h N_{hj}/N m_h r_h$\\
\hline
$1/f_{y,H}(Q_{y,H}(p))$& $\sqrt{n}(\hat{Q}_{y}(p+1/\sqrt{n})-\hat{Q}_{y}(p-1/\sqrt{n}))/2$\\
\hline
$1/f_{x,H}(Q_{x,H}(p))$& $\sqrt{n}(\hat{Q}_{x}(p+1/\sqrt{n})-\hat{Q}_{x}(p-1/\sqrt{n}))/2$\\
\hline
$\sum_{h=1}^H (N_h/N) E_{\textbf{P}}(X_{hjl}^{\prime})$ as well as $\text{\tnote{2} \ }\Theta_1$&$\sum_{h=1}^H \sum_{j\in s_h, l\in s_{hj}}M_h N_{hj}X_{hjl}^{\prime}/m_h r_h N $ \\
\hline
$\sum_{h=1}^H (N_h/N) E_{\textbf{P}}(Y_{hjl}^{\prime})$ as well as $\text{\tnote{2} \ }\Theta_2$& $\sum_{h=1}^H$ $\sum_{j\in s_h,l\in s_{hj}} M_h N_{hj} Y_{hjl}^{\prime}/m_h r_h N$\\
\hline
$\sum_{h=1}^H (N_h/N) E_{\textbf{P}}(X_{hjl}^{\prime}Y_{hjl}^{\prime})$ as well as $\text{\tnote{2} \ }\Theta_3$ &$\sum_{h=1}^H \sum_{j\in s_h, l\in s_{hj}} M_h N_{hj}X_{hjl}^{\prime}Y_{hjl}^{\prime}/m_h r_h N$ \\
\hline
$\sum_{h=1}^H (N_h/N) E_{\textbf{P}}(X_{hjl}^{\prime})^2$ as well as $\text{\tnote{2} \ }\Theta_4$ & $\sum_{h=1}^H\sum_{j\in s_h, l\in s_{hj}} M_h N_{hj}(X_{hjl}^{\prime})^2/m_h r_h N$\\
\hline
\end{tabular}
\begin{tablenotes}
\item[2] The $\Theta_1$, $\Theta_2$, $\Theta_3$ and $\Theta_4$ are as in Assumption $10$ in Section $4$ in the main text. 
\end{tablenotes}
\end{threeparttable}
\end{table}
\clearpage

\section{Additional results required for proving the main results for single stage sampling designs}\label{sec 1}
Let us fix $k\geq 1$ and $p_1,\ldots,p_k\in (0,1)$, and recall $\textbf{V}_1,\ldots,\textbf{V}_N$ from the begining of Section \ref{sec s1} in this supplement. Define  $\hat{\overline{\textbf{V}}}_1$=$\sum_{i\in s}(N\pi_i)^{-1} \textbf{V}_i$. Suppose that $P(s,\omega)$ denotes a high entropy sampling design satisfying Assumption $2$, and $Q(s,\omega)$ denotes a rejective sampling design having inclusion probabilities equal to those of $P(s,\omega)$. Such a rejective sampling design always exists (see \cite{MR1624693}). Now, we state the following lemma. 
\begin{lemma}\label{lem 1}
Fix $\textbf{b}\in \mathds{R}^{2k}$ such that $\textbf{b}\neq 0$. Suppose that Assumption $1$ in the main article holds. Then, under $Q(s,\omega)$ as well as $P(s,\omega)$, we have 
$$
\sqrt{n}\textbf{b}\hat{\overline{\textbf{V}}}_1^T\xrightarrow{\mathcal{L}}N(0,\textbf{b}\Gamma\textbf{b}^T)\text{ as }\nu\rightarrow\infty \textit{ a.s. }[\textbf{P}],
$$
 where $\Gamma$ is as mentioned in Assumption $2$-\textsl{(ii)}.
\end{lemma}
\begin{proof}
 The proof follows exactly the same way as the derivation of the result, which appears in the proof of Lemma S$4$ in \cite{deychaudhuri2023}, that $\sqrt{n}\textbf{m}(\hat{\overline{\textbf{V}}}_1-\overline{\textbf{V}})^T\xrightarrow{\mathcal{L}}N(0,\textbf{m} \Gamma_1 \textbf{m}^T)$ as $\nu\rightarrow\infty$ under each of SRSWOR, LMS and any HE$\pi$PS sampling designs for any $\textbf{m}\in\mathbb{R}^{p}$, $\textbf{m}\neq 0$ and $\Gamma_1$=$\lim_{\nu\rightarrow\infty}\Sigma_1$.  
\end{proof}
 Next, recall $\{Z_i\}_{i=1}^N$  from ($4$) in the main text and define
\begin{equation}\label{eq 37}
F_{z,N}(t)=\sum_{i=1}^N \mathds{1}_{[Z_i\leq t]}/N\text{ and }\mathds{Z}_n(t)=\sqrt{n}\sum_{i\in s}(N\pi_i)^{-1}(\mathds{1}_{[Z_i\leq t]}-F_{z,N}(t)) 
\end{equation}
for $0\leq t\leq 1$. Let us also define 
\begin{equation}\label{eq 39}
B_{z,N}(u,t)=F_{z,N}(t)-F_{z,N}(u)\text{ and }\mathds{B}_n(u,t)=\mathds{Z}_n(t)-\mathds{Z}_n(u)
\end{equation}
for $0\leq u<t\leq 1$. Now, we state the following lemma. 
\begin{lemma}\label{lem 2}
Suppose that Assumption $1$ in the main text holds. Then, there exist constants $L_1,L_2>0$ such that under $Q(s,\omega)$, 
$$
E\big[\big(\mathds{B}_n(t_1,t_2)\big)^2\big(\mathds{B}_n(t_2,t_3)\big)^2\big]\leq L_1\big(B_{z,N}(t_1,t_3)\big)^2 \textit{ a.s. }[\textbf{P}]
$$
for any $0\leq t_1<t_2<t_3\leq 1$ and $\nu\geq 1$, and
$$
\overline{\lim}_{\nu\rightarrow\infty} E\big(\mathds{B}_n(u,t)\big)^4\leq L_2 \big(t-u\big)^2\textit{ a.s. } [\textbf{P}]
$$
for any $0\leq u<t\leq 1$.  
\end{lemma}
\begin{proof}
Suppose that for $i$=$1,\ldots,N$, $\xi_i$=$1$, when the $i^{th}$ population unit is included in the sample, and $\xi_i$=$0$ otherwise. Further, suppose that $S_{k,N}$=$\{(i_1,\ldots,i_k): i_1,\ldots,i_k\in\{1,2,\ldots,N\}\text{ and }i_1,\ldots,i_k\text{ are all distinct}\}$ for $k$=$2,3,4$. Recall from the proof of the preceding Lemma that under $Q(s,\omega)$, $\sum_{i=1}^N\pi_i(1-\pi_i)/n$ is bounded away from $0$ as $\nu\rightarrow\infty$ \textit{a.s.} $[\textbf{P}]$. Then, it follows from the proof of Corollary $5.1$ in \cite{MR3670194} that there exists a constant $K_1>0$ such that for all $\nu\geq 1$ 
\begin{align}\label{eq A5} 
\begin{split}
&\max_{(i_1,i_2)\in S_{2,N}} \big|E\big((\xi_{i_1}-\pi_{i_1})(\xi_{i_2}-\pi_{i_2})\big)\big|<K_1 n/N^2,\max_{(i_1,i_2,i_3)\in S_{3,N}}\big|E\big((\xi_{i_1}-\pi_{i_1})\times\\
&(\xi_{i_2}-\pi_{i_2})(\xi_{i_3}-\pi_{i_3})\big)\big|<K_1 n^2/N^3,\text{ and }\max_{(i_1,i_2,i_3,i_4)\in S_{4,N}}\big|E\big((\xi_{i_1}-\pi_{i_1})(\xi_{i_2}-\pi_{i_2})\times\\
&(\xi_{i_3}-\pi_{i_3})(\xi_{i_4}-\pi_{i_4})\big)\big|<K_1 n^2/N^4
\end{split}
\end{align}
under $Q(s,\omega)$ \textit{a.s.} $[\textbf{P}]$. Now, let 
\begin{align*}
\begin{split}
&B_i=\mathds{1}_{[t_1<Z_i\leq t_2]}-B_{z,N}(t_1,t_2), \text{  }C_i=\mathds{1}_{[t_2<Z_i\leq t_3]}-B_{z,N}(t_2,t_3),\\
&\alpha_i=B_i(\xi_i/\pi_i-1)\text{ and }\beta_i=C_i(\xi_i/\pi_i-1)
\end{split}
\end{align*}
 for given any $i$=$1,\ldots,N$ and $0\leq t_1<t_2<t_3\leq 1$. Then, we have
\begin{align*}
\begin{split}
&E\big[\big(\mathds{B}_n(t_1,t_2)\big)^2\big(\mathds{B}_n(t_2,t_3)\big)^2\big]=(n^2/N^4)E\bigg[\sum_{i=1}^N \alpha_i^2\beta_i^2+\\
&\sum_{(i_1,i_2)\in S_{2,N}} \alpha_{i_1}\alpha_{i_2}\beta_{i_1} \beta_{i_2}+\sum_{(i_1,i_2)\in S_{2,N}} \alpha_{i_1}^2\beta_{i_2}^2+\sum_{(i_1,i_2)\in S_{2,N}} \alpha_{i_1}^2\beta_{i_1} \beta_{i_2}+\\
\end{split}
\end{align*}
\begin{align*}
\begin{split}
&\sum_{(i_1,i_2,i_3)\in S_{3,N}} \alpha_{i_1}^2\beta_{i_2} \beta_{i_3}+ \sum_{(i_1,i_2)\in S_{2,N}} \alpha_{i_1}\alpha_{i_2} \beta^2_{i_2}+\sum_{(i_1,i_2,i_3)\in S_{3,N}} \alpha_{i_1}\alpha_{i_2} \beta_{i_3}^2+\\
&\sum_{(i_1,i_2,i_3,i_4)\in S_{4,N}} \alpha_{i_1}\alpha_{i_2} \beta_{i_3}\beta_{i_4}\bigg].
\end{split}
\end{align*}
Note that $Q(s,\omega)$ satisfies Assumption $2$-\textsl{(ii)} because $P(s,\omega)$ satisfies Assumption $2$-\textsl{(ii)}, and $P(s,\omega)$ and $Q(s,\omega)$ have the same inclusion probabilities. Then, we have 
\begin{align}\label{eq A21}
\begin{split}
&(n^2/N^4)E\bigg[\sum_{i=1}^N \alpha_i^2\beta_i^2\bigg]=(n^2/N^4)\sum_{i=1}^N E(\xi_i-\pi_i)^4 B_i^2 C_i^2/\pi_i^4\leq (K_2/N)\sum_{i=1}^NB_i^2 C_i^2\\
&\leq K_3\big(B_{z,N}(t_1,t_3)\big)^2
\end{split}
\end{align}
\textit{a.s.} $[\textbf{P}]$ for all $\nu\geq 1$ and some constants $K_2, K_3>0$ since Assumption $1$ holds, and $\mathds{1}_{[t_1<Z_i\leq t_2]} \mathds{1}_{[t_2<Z_i\leq t_3]}$ =$0$ for any $0\leq t_1<t_2<t_3\leq 1$. Next, suppose that $\{\pi_{i_1 i_2}: 1\leq i_1<i_2\leq N\}$ are second order inclusion probabilities of $Q(s,\omega)$. Then, we note that
\begin{align}\label{eq A18}
\begin{split}
&(n^2/N^4)E\bigg[\sum_{(i_1,i_2)\in S_{2,N}} \alpha_{i_1}\alpha_{i_2}\beta_{i_1}\beta_{i_2}\bigg]=(n^2/N^4)\sum_{(i_1,i_2)\in S_{2,N}} E\bigg((\xi_{i_1}-\pi_{i_1})^2\times\\
&(\xi_{i_2}-\pi_{i_2})^2\bigg)B_{i_1}B_{i_2} C_{i_1} C_{i_2}/\pi_{i_1}^2\pi_{i_2}^2\leq (K_4/n^2)\sum_{(i_1,i_2)\in S_{2,N}} (|\pi_{i_1 i_2}-\pi_{i_1}\pi_{i_2}|+\pi_{i_1}\pi_{i_2})\times\\
& |B_{i_1}C_{i_1}||B_{i_2} C_{i_2}| \leq (K_5/N^2)\sum_{(i_1,i_2)\in S_{2,N}} |B_{i_1}C_{i_1}||B_{i_2} C_{i_2}| \leq K_6(B_{z,N}\big(t_1,t_3)\big)^2
\end{split}
\end{align}
\textit{a.s.} $[\textbf{P}]$ for all $\nu\geq 1$ and some constants $K_4, K_5, K_6>0$ since Assumption $2$-\textsl{(ii)} holds, $E((\xi_{i_1}-\pi_{i_1})^2(\xi_{i_2}-\pi_{i_2})^2)$=$(\pi_{i_1 i_2}-\pi_{i_1}\pi_{i_2})(1-2\pi_{i_1})(1-2\pi_{i_2})+\pi_{i_1}\pi_{i_2}(1-\pi_{i_1})(1-\pi_{i_2})$ for $(i_1,i_2)\in S_{2,N}$, and $\max_{(i_1,i_2)\in S_{2,N}} \big|E\big((\xi_{i_1}-\pi_{i_1})(\xi_{i_2}-\pi_{i_2})\big)\big|$=$\max_{(i_1,i_2)\in S_{2,N}}|\pi_{i_1 i_2}-\pi_{i_1}\pi_{i_2}|<K_1 n/N^2$ \textit{a.s.} $[\textbf{P}]$ by \eqref{eq A5}. An inequality similar to \eqref{eq A18} holds for $(n^2/N^4)E[\sum_{(i_1,i_2)\in S_{2,N}} \alpha_{i_1}^2\beta_{i_2}^2]$. Since, $\big|E\big((\xi_{i_1}-\pi_{i_1})^3(\xi_{i_2}-\pi_{i_2})\big)\big|\leq 7|\pi_{i_1 i_2}-\pi_{i_1}\pi_{i_2}|$, inequalities similar to \eqref{eq A18} also hold for $(n^2/N^4)E$ $[\sum_{(i_1,i_2)\in S_{2,N}} \alpha_{i_1}^2\beta_{i_1} \beta_{i_2}]$ and $(n^2/N^4)E[\sum_{(i_1,i_2)\in S_{2,N}} \alpha_{i_1}\alpha_{i_2} \beta^2_{i_2}]$. Note that 
\begin{align*}
\begin{split}
&E\big((\xi_{i_1}-\pi_{i_1})^2(\xi_{i_2}-\pi_{i_2})(\xi_{i_3}-\pi_{i_3})\big)=(1-2\pi_{i_1})E\big((\xi_{i_1}-\pi_{i_1})(\xi_{i_2}-\pi_{i_2})\times\\
&(\xi_{i_3}-\pi_{i_3})\big)+\pi_{i_1}(1-\pi_{i_1})E\big((\xi_{i_2}-\pi_{i_2})(\xi_{i_3}-\pi_{i_3})\big)\text{ for }(i_1,i_2,i_3)\in S_{3,N}. 
\end{split}
\end{align*}
 Also, note that 
\begin{align*}
\begin{split}
&\max_{(i_1,i_2,i_3)\in S_{3,N}} \big|E\big((\xi_{i_1}-\pi_{i_1})(\xi_{i_2}-\pi_{i_2})(\xi_{i_3}-\pi_{i_3})\big)\big|<K_1 n^2/N^3\text{ and }\\
&\max_{(i_1,i_2,i_3,i_4)\in S_{4,N}}\big |E\big((\xi_{i_1}-\pi_{i_1})(\xi_{i_2}-\pi_{i_2})(\xi_{i_3}-\pi_{i_3})(\xi_{i_4}-\pi_{i_4})\big)\big|<K_1 n^2/N^4 \textit{ a.s. }[\textbf{P}]
\end{split}
\end{align*}
 by \eqref{eq A5}. Therefore, it can be shown in the same way as in \eqref{eq A21} and \eqref{eq A18} that under $Q(s,\omega)$, 
\begin{align*}
\begin{split}
&(n^2/N^4)E\bigg[\sum_{(i_1,i_2,i_3)\in S_{3,n}} \alpha_{i_1}^2\beta_{i_2} \beta_{i_3}\bigg]\leq K_7\big(B_{z,N}(t_1,t_3)\big)^2,\\
& (n^2/N^4)E \bigg[\sum_{(i_1,i_2,i_3)\in S_{3,N}} \alpha_{i_1}\alpha_{i_2} \beta_{i_3}^2\bigg]\leq K_{7}\big(B_{z,N}(t_1,t_3)\big)^2\text{ and }\\
&(n^2/N^4)E\bigg[ \sum_{(i_1,i_2,i_3,i_4)\in S_{4,N}} \alpha_{i_1}\alpha_{i_2}\beta_{i_3} \beta_{i_4}\bigg] \leq K_7\big(B_{z,N} (t_1,t_3)\big)^2\textit{ a.s. } [\textbf{P}] 
\end{split}
\end{align*}
 for all $\nu\geq 1$ and some constant $K_7>0$. Hence, there exists a constant $K_8>0$ such that under $Q(s,\omega)$, $E\big[\big(\mathds{B}_n(t_1,t_2)\big)^2$ $\big(\mathds{B}_n(t_2,t_3)\big)^2\big]$ $\leq K_8 \big(B_{z,N}(t_1,t_3)\big)^2$ \textit{a.s.} $[\textbf{P}]$ for any $\nu\geq 1$ and $0\leq t_1<t_2<t_3\leq 1$. 
\par

Next, one can shown that 
\begin{align*}
\begin{split}
&E \big(\mathds{B}_n(t_1,t_2)\big)^4=(n^2/N^4)E\bigg[\sum_{i=1}^N \alpha_i^4+2\sum_{(i_1,i_2)\in S_{2,N}} \alpha_{i_1}^2\alpha_{i_2}^2+\\
&2\sum_{(i_1,i_2)\in S_{2,N}} \alpha_{i_1}^3\alpha_{i_2}+2 \sum_{(i_1,i_2,i_3)\in S_{3,N}} \alpha_{i_1}^2\alpha_{i_2} \alpha_{i_3}+
\sum_{(i_1,i_2,i_3,i_4)\in S_{4,N}} \alpha_{i_1}\alpha_{i_2} \alpha_{i_3}\alpha_{i_4}\bigg].
\end{split}
\end{align*}
  It can also be shown in the same way as in \eqref{eq A21} and \eqref{eq A18} that under $Q(s,\omega)$, 
\begin{align*}
\begin{split} 
& (n^2/N^4)E\bigg[\sum_{i=1}^N \alpha_i^4\bigg]=O(1/n)\text{ as }\nu\rightarrow\infty \textit{ a.s. } [\textbf{P}],\text{ and }\\
& (n^2/N^4)E\bigg[2\sum_{(i_1,i_2)\in S_{2,N}} \alpha_{i_1}^2\alpha_{i_2}^2+2\sum_{(i_1,i_2)\in S_{2,N}}\alpha_{i_1}^3\alpha_{i_2}+2\sum_{(i_1,i_2,i_3)\in S_{3,N}} \alpha_{i_1}^2\alpha_{i_2} \alpha_{i_3}+\\
&\sum_{(i_1,i_2,i_3,i_4)\in S_{4,N}} \alpha_{i_1}\alpha_{i_2} \alpha_{i_3}\alpha_{i_4}\bigg]\leq K_9\big(B_{z,N}(t_1,t_2)\big)^2\text{ given any }\nu\geq 1 \textit{ a.s. }[\textbf{P}] 
\end{split}
\end{align*} 
 for some constant $K_9>0$. Therefore, under $Q(s,\omega)$, $\overline{\lim}_{\nu\rightarrow\infty}$ $E\big(\mathds{B}_n(t_1,t_2)\big)^4$ $\leq K_9\big(t_2-t_1)\big)^2$ \textit{a.s.} $[\textbf{P}]$ because $B_{z,N}(t_1,t_2)\rightarrow (t_2-t_1)$ \textit{a.s.} $[\textbf{P}]$ by SLLN. Hence, the result follows.
\end{proof}
Next, fix $k\geq 1$ and $p_1,\ldots,p_k\in (0,1)$ and define $\hat{\overline{\textbf{V}}}_2$=$\sum_{i\in s} A_i \textbf{V}_i/NX_i$, where $\textbf{V}_i$'s are as in the beginning of Section \ref{sec s1} in this supplement, and $A_i$'s are as in the $2^{nd}$ paragraph of Section $2$ in the main text. Also, recall from the statement of Theorem $3.2$ in Section $3$ of the main text that $\gamma$=$\sum_{r=1}^n N_r(N_r-1)(N(N-1))^{-1}$ with $N_r$ being the size of the $r^{th}$ group formed randomly in the first step of the RHC sampling design for $r$=$1,\ldots,n$
\begin{lemma}\label{lem 4}
Fix $\textbf{b}\in \mathds{R}^{2k}$ such that $\textbf{b}\neq 0$. Suppose that $\{N_r\}_{r=1}^n$ are as in ($6$) in the main text. Further, suppose that  $E_{\textbf{P}}(X_i)^{-1}<\infty$,  and Assumptions $1$, $4$ and $5$ in the main text hold. Then, under RHC sampling design, we have 
$$
\sqrt{n}\textbf{b}\hat{\overline{\textbf{V}}}_2^T\xrightarrow{\mathcal{L}}N(0,\textbf{b} \Gamma_6\textbf{b}^T)\text{ as }\nu\rightarrow\infty\textit{ a.s. }[\textbf{P}], 
$$
where $\Gamma_6$=$c E_{\textbf{P}}(X_i)E_{\textbf{P}}[(\textbf{G}_i-E_{\textbf{P}}(\textbf{G}_i))^T(\textbf{G}_i-E_{\textbf{P}}(\textbf{G}_i))/X_i]$, and $c$=$\lim_{\nu\rightarrow\infty}n \gamma$.
\end{lemma}
 Note that $\Gamma_6$ is p.d. by Assumption $5$. Also, note that the $\lim_{\nu\rightarrow\infty}n \gamma$ exists by Lemma S$1$ in \cite{deychaudhuri2023}. 
\begin{proof}
The proof follows exactly the same way as the derivation of the result, which appears in the proof of Lemma S$4$ in \cite{deychaudhuri2023}, that $\sqrt{n}\textbf{m}(\hat{\overline{\textbf{V}}}_2-\overline{\textbf{V}})^T \xrightarrow{\mathcal{L}}N(0,\textbf{m} \Gamma_2 \textbf{m}^T)$ as $\nu\rightarrow\infty$ under RHC sampling design for any $\textbf{m}\in\mathbb{R}^{p}$, $\textbf{m}\neq 0$ and $\Gamma_2$=$\lim_{\nu\rightarrow\infty}\Sigma_2$.
\end{proof}
 Before we state the next lemma, recall  $\{Z_i\}_{i=1}^N$ from ($4$) in the main text, and $F_{z,N}(t)$ and $B_{z,N}(u,t)$ from the paragraph preceding Lemma \ref{lem 2}.  Define $\tilde{\mathds{Z}}_n(t)$=$\sqrt{n}\sum_{i\in s}(A_i/NX_i)\times$ $(\mathds{1}_{[Z_i\leq t]}- F_{z,N}(t))$ for $0\leq t\leq 1$ and $\tilde{\mathds{B}}_n(u,t)$=$\tilde{\mathds{Z}}_n(t)-\tilde{\mathds{Z}}_n(u)$ for $0\leq u<t\leq 1$. 
\begin{lemma}\label{lem 5}
Suppose that $\{N_r\}_{r=1}^n$ are as in ($6$) in the main text, and Assumption $4$ in the main article holds. Then, there exist constants $L_1,L_2>0$ such that under RHC sampling design, 
$$
E\big[\big(\tilde{\mathds{B}}_n(t_1,t_2)\big)^2 \big(\tilde{\mathds{B}}_n(t_2,t_3)\big)^2\big]\leq L_1\big(B_{z,N}(t_1,t_3)\big)^2\textit{ a.s. } [\textbf{P}]
$$
for any $0\leq t_1<t_2<t_3\leq 1$ and $\nu\geq 1$, and 
$$
\overline{\lim}_{\nu\rightarrow\infty} E\big(\tilde{\mathds{B}}_n(u,t)\big)^4\leq L_2 \big(t-u\big)^2 \textit{ a.s. }[\textbf{P}]
$$
for any $0\leq u<t\leq 1$. 
\end{lemma}
\begin{proof}
Recall from Section $3$ in the main text that RHC sampling design is implemented in two steps. In the first step, the entire population is randomly divided into $n$ groups, say $\mathcal{P}_1,\ldots,\mathcal{P}_n$ of sizes $N_1\cdots,N_n$ respectively. Then, in the second step, a unit is selected from each group independently. For each $r$=$1,\ldots,n$, the $q^{th}$ unit from $\mathcal{P}_r$ is selected with probability $X_{qr}^{\prime}/Q_r$, where $X_{qr}^{\prime}$ is the $x$ value of the $q^{th}$ unit in $\mathcal{P}_r$ and $Q_r$=$\sum_{q=1}^{N_r}X_{qr}^{\prime}$. Let $E_1$ and $E_2$ denote design expectations with respect to the $1^{st}$ and the $2^{nd}$ steps, respectively. Suppose that $(y_r,x_r)$ is the value of $(y,x)$ corresponding to the $r^{th}$ unit in the sample for $r$=$1,\ldots,n$. Further, suppose that $z_r$=$F_y(y_r)$ for $r$=$1,\ldots,n$, where $F_y$ is the superpopulation distribution function of $y$. Define 
$$
\alpha_r=Q_r(\mathds{1}_{[t_1<z_r\leq t_2]}-B_{z,N}(t_1,t_2))/x_r\text{ and }\beta_r=Q_r(\mathds{1}_{[t_2<z_r\leq t_3]}-B_{z,N}(t_2,t_3))/x_r
$$
 for $0\leq  t_1<t_2<t_3\leq 1$ and $r$=$1,\ldots,n$. Note that $\tilde{\mathds{Z}}_n(t)$=$\sqrt{n}\sum_{i\in s}(A_i/NX_i)(\mathds{1}_{[Z_i\leq t]}-F_{z,N}(t))$=$\sqrt{n}\sum_{r=1}^n Q_r(\mathds{1}_{[z_r\leq t]}-F_{z,N}(t))/Nx_r$. Then, we have 
\begin{align*}
\begin{split}
&E\big[\big(\tilde{\mathds{B}}_n(t_1,t_2)\big)^2\big(\tilde{\mathds{B}}_n(t_2,t_3)\big)^2\big]=(n^2/N^4)E_1E_2\bigg[\sum_{r=1}^n \alpha_r^2\beta_r^2+\\
&\sum_{(r_1,r_2)\in S_{2,n}} \alpha_{r_1}\alpha_{r_2}\beta_{r_1}\beta_{r_2}+\sum_{(r_1,r_2)\in S_{2,n}} \alpha_{r_1}^2\beta_{r_2}^2+\sum_{(r_1,r_2)\in S_{2,n}} \alpha_{r_1}^2\beta_{r_1} \beta_{r_2}+\\
&\sum_{(r_1,r_2)\in S_{2,n}} \alpha_{r_1}\alpha_{r_2} \beta^2_{r_2}+\sum_{(r_1,r_2,r_3)\in S_{3,n}} \alpha_{r_1}^2\beta_{r_2} \beta_{r_3}+ \sum_{(r_1,r_2,r_3)\in S_{3,n}} \alpha_{r_1}\alpha_{r_2} \beta_{r_3}^2+\\
&\sum_{(r_1,r_2,r_3,r_4)\in S_{4,n}} \alpha_{r_1}\alpha_{r_2} \beta_{r_3}\beta_{r_4}\bigg],
\end{split}
\end{align*}
where $S_{k,n}$=$\{(r_1,\ldots,r_k):r_1,\ldots,r_k\in\{1,2,\ldots,n\}\text{ and }r_1,\ldots,r_k\text{ are all distinct}\}$ for $k$= $2,3,4$. Suppose that for $i$=$1,\ldots,N$, 
\begin{align*}
\begin{split}
 \xi_{ir}=
\begin{cases}
1,\text{ when the }i^{th}\text{ population unit is selected in the }r^{th}\text{ group }\mathcal{P}_r,\text{ and }\\
0,\text{ otherwise.}\\
\end{cases}
\end{split}
\end{align*}
 Note that $n\max_{1\leq r\leq n}N_r/N\leq 2$ for all $\nu\geq 1$ because $\{N_r\}_{r=1}^n$ are as in ($6$) in the main text. Recall $B_i$ and $C_i$ from the proof of Lemma \ref{lem 2}. Then, we have 
\begin{align}\label{eq A19}
\begin{split}
&(n^2/N^4)E_1\bigg[\sum_{r=1}^n E_2(\alpha_r^2\beta_r^2)\bigg]=(n^2/N^4)E_1\bigg[\sum_{r=1}^n \bigg(\sum_{i=1}^N B_i^2 C_i^2 \xi_{ir}/X_i^3\bigg) Q_r^3\bigg]\leq\\
&  (K_1)^3 (n^2/N^4)E_1\bigg[\sum_{r=1}^n \bigg(\sum_{i=1}^N B_i^2 C_i^2 \xi_{ir}\bigg) N_r^3\bigg]\leq (K_2/N)\bigg[\sum_{i=1}^N B_i^2 C_i^2E_1\bigg(\sum_{r=1}^n \xi_{ir}\bigg)\bigg]\\
&=(K_2/N)\bigg[\sum_{i=1}^N B_i^2 C_i^2\bigg]\leq K_3\big(B_{z,N}(t_1,t_3)\big)^2
\end{split}
\end{align}
\textit{a.s.} $[\textbf{P}]$ for all $\nu\geq 1$ and some constants $K_2,K_3>0$ because Assumption $4$ holds, and $\sum_{r=1}^n \xi_{ir}$=$1$ for any $1\leq i\leq N$. Next, recall $S_{2,N}$ from the proof of Lemma \ref{lem 2} and note that
\begin{align}\label{eq A20}
\begin{split}
&(n^2/N^4)E_1\bigg[\sum_{(r_1,r_2)\in S_{2,n}} E_2(\alpha_{r_1}\alpha_{r_2}\beta_{r_1}\beta_{r_2})\bigg]=(n^2/N^4)\times\\
&E_1\bigg[\sum_{(r_1,r_2)\in S_{2,n}} E_2(\alpha_{r_1}\beta_{r_1})E_2(\alpha_{r_2}\beta_{r_2})\bigg]=(n^2/N^4)E_1\bigg[\sum_{(r_1,r_2)\in S_{2,n}} \\
\end{split}
\end{align}
\begin{align*}
\begin{split}
& \bigg(\sum_{(i_1,i_2)\in S_{2,N}} B_{i_1} C_{i_1}B_{i_2} C_{i_2} \xi_{i_1 {r_1}}\xi_{i_2 {r_2}}/X_{i_1}X_{i_2}\bigg) Q_{r_1} Q_{r_2}\bigg]\leq (K_1)^2(n^2/N^4)\times \\
&E_1\bigg[\sum_{(r_1,r_2)\in S_{2,n}} \bigg(\sum_{(i_1,i_2)\in S_{2,N}} |B_{i_1} C_{i_1}||B_{i_2} C_{i_2}|\xi_{i_1 {r_1}}\xi_{i_2 {r_2}}\bigg) N_{r_1} N_{r_2}\bigg]\leq \\
&K_4\sum_{i_1=1}^N |B_{i_1} C_{i_1}|\sum_{i_2=1}^N|B_{i_2} C_{i_2}|/N(N-1)\leq K_5\big(B_{z,N}(t_1,t_3)\big)^2
\end{split}
\end{align*}
\textit{a.s.} $[\textbf{P}]$ for all $\nu\geq 1$ and some constants $K_4, K_5>0$ since units are selected from $\mathcal{P}_{r_1}$ and $\mathcal{P}_{r_2}$ independently, $\{N_r\}_{r=1}^n$ are as in ($6$) in the main text, and $E_1(\xi_{i_1 r_1}\xi_{i_2 r_2})$=$N_{r_1}N_{r_2}/N(N-1)$ for any $(r_1,r_2)\in S_{2,n}$ and $(i_1,i_2)\in S_{2,N}$. It can be shown that an inequality similar to \eqref{eq A20} above holds for each of $(n^2/N^4)E_1E_2[\sum_{(r_1,r_2)\in S_{2,n}}\alpha_{r_1}^2\beta_{r_2}^2]$, $(n^2/N^4)E_1E_2[\sum_{(r_1,r_2)\in S_{2,n}}$ $ \alpha_{r_1}^2\beta_{r_1} \beta_{r_2}]$ and $(n^2/N^4)E_1E_2$ $[\sum_{(r_1,r_2)\in S_{2,n}} \alpha_{r_1}\alpha_{r_2} \beta^2_{r_2}]$. Note that 
$$
E_1( \xi_{i_1 {r_1}}\xi_{i_2 {r_2}}\xi_{i_3 {r_3}})=N_{r_1}N_{r_2}N_{r_3}/(N(N-1)(N-2))
$$
 for $(r_1,r_2,r_3)\in S_{3,n}$ and $(i_1,i_2,i_3)\in S_{3,N}$, and $\sum_{(r_1,r_2,r_3)\in S_{3,n}}$ $N_{r_1}N_{r_2}N_{r_3}/N(N-1)(N-2)$ is bounded. Also, note that 
$$
E_1(\xi_{i_1 r_1}\xi_{i_2 r_2}\xi_{i_3 r_3}\xi_{i_4 r_4})=(N_{r_1}N_{r_2} N_{r_3} N_{r_4}) /N(N-1)(N-2)(N-3)
$$
 for $(r_1,r_2,r_3,r_4)\in S_{4,n}$ and $(i_1,i_2,i_3,$ $i_4)\in S_{4,N}$, and $\sum_{(r_1,r_2,r_3,r_4)\in S_{4,n}} N_{r_1}N_{r_2}N_{r_3}N_{r_4} /N(N-1)(N-2)(N-3)$ is bounded. Then, it can be shown in the same way as in \eqref{eq A19} and \eqref{eq A20} above that 
\begin{align*}
\begin{split}
&(n^2/N^4)E_1E_2 \bigg [\sum_{(r_1,r_2,r_3)\in S_{3,n}}\alpha_{r_1}^2\beta_{r_2} \beta_{r_3}\bigg]\leq K_6 \big(B_{z,N}(t_1,t_3)\big)^2, \\
&(n^2/N^4)E_1E_2\bigg [\sum_{(r_1,r_2,r_3)\in S_{3,n}} \alpha_{r_1}\alpha_{r_2} \beta_{r_3}^2 \bigg] \leq K_6 \big(B_{z,N}(t_1,t_3)\big)^2\text{ and }\\
&(n^2/N^4)E_1E_2 \bigg[\sum_{(r_1,r_2,r_3,r_4)\in S_{4,n}} \alpha_{r_1}\alpha_{r_2} \beta_{r_3}\beta_{r_4}\bigg]\leq K_6 \big(B_{z,N}(t_1, t_3)\big)^2  \textit{ a.s. } [\textbf{P}]
\end{split}
\end{align*}
 for all $\nu\geq 1$ and some constant $K_6>0$. Thus  
$$
E\big[\big(\tilde{\mathds{B}}_n(t_1,t_2)\big)^2\big(\tilde{\mathds{B}}_n(t_2,t_3)\big)^2\big]\leq K_7 \big(B_{z,N}(t_1,t_3)\big)^2 \textit{ a.s. } [\textbf{P}]
$$
 for all $\nu\geq 1$ and some constant $K_7>0$. 
\par

Next, note that 
\begin{align*}
\begin{split}
&E\big(\tilde{\mathds{B}}_n(t_1,t_2)\big)^4=(n^2/N^4)E_1E_2\bigg[\sum_{r=1}^n \alpha_r^4+2\sum_{(r_1,r_2)\in S_{2,n}} \alpha_{r_1}^2\alpha_{r_2}^2+\\
\end{split}
\end{align*}
\begin{align*}
\begin{split}
&2 \sum_{(r_1,r_2)\in S_{2,n}} \alpha_{r_1}^3\alpha_{r_2}+2\sum_{(r_1,r_2,r_3)\in S_{3,n}}\alpha_{r_1}^2\alpha_{r_2} \alpha_{r_3}+\sum_{(r_1,r_2,r_3,r_4)\in S_{4,n}} \alpha_{r_1}\alpha_{r_2} \alpha_{r_3}\alpha_{r_4}\bigg].
\end{split}
\end{align*}
 It can be shown in the same way as in \eqref{eq A19} and \eqref{eq A20} above that 
\begin{align*}
\begin{split}
&(n^2/N^4)E_1E_2\bigg[\sum_{r=1}^n \alpha_r^4\bigg]=O(1/n)\text{ as }\nu\rightarrow\infty \textit{ a.s. } [\textbf{P}],\text{ and }(n^2/N^4)\times\\
&E_1E_2 \bigg[2\sum_{(r_1,r_2)\in S_{2,n}} \alpha_{r_1}^2\alpha_{r_2}^2+2\sum_{(r_1,r_2)\in S_{2,n}} \alpha_{r_1}^3\alpha_{r_2}+2\sum_{(r_1,r_2,r_3)\in S_{3,n}} \alpha_{r_1}^2\alpha_{r_2}\alpha_{r_3}+\\
& \sum_{(r_1,r_2,r_3,r_4)\in S_{4,n}}\alpha_{r_1}\alpha_{r_2} \alpha_{r_3}\alpha_{r_4}\bigg]\leq K_8\big(B_{z,N}(t_1,t_2)\big)^2\text{ given any }\nu\geq 1\textit{ a.s. } [\textbf{P}] 
\end{split}
\end{align*}
for some constant $K_8>0$. Therefore, $\overline{\lim}_{\nu\rightarrow\infty} E\big(\tilde{\mathds{B}}_n(t_1,t_2)\big)^4\leq K_8 \big(t_2-t_1\big)^2$ \textit{a.s.} $[\textbf{P}]$ since $B_{z,N}(t_1,t_2)$ $\rightarrow (t_2-t_1)$ as $\nu\rightarrow\infty$ \textit{a.s.} $[\textbf{P}]$ by SLLN.
\end{proof}
Next, we state the following lemma, which will be required to prove Theorem $5.3$. 
\begin{lemma}\label{lem 8}
\textsl{(i)} Fix $0<\alpha<\beta<1$. Suppose that the assumptions of Theorem $3.1$ hold, $K(p_1,p_2)$ is as in  ($5$) in Section $3$ of the main article, and $\hat{K}(p_1,p_2)$ is as in ($12$) in Section $5.1$ of the main text.  Then, under $\textbf{P}^*$,
\begin{equation}\label{eq A13}
\sup_{p_1,p_2\in[\alpha,\beta]}|\hat{K}(p_1,p_2)-K(p_1,p_2)|=O_p(1)\text{ and }\hat{K}(p_1,p_2)\xrightarrow{p}K(p_1,p_2)\text{ as }\nu\rightarrow\infty
\end{equation}
for any $p_1,p_2\in [\alpha,\beta]$ and high entropy sampling design satisfying Assumption $2$. \\
\textsl{(ii)} Further, if the assumptions of Theorem $3.2$ hold, $K(p_1,p_2)$ is as in  ($7$) in Section $3$ of the main article, and $\hat{K}(p_1,p_2)$ is as in ($13$) in Section $5.1$ of the main text.  Then, the above results hold under RHC sampling design.\\
\end{lemma}
\begin{proof} \textsl{(i)} Let us first consider a high entropy sampling design $P(s,\omega)$ satisfying Assumption $2$, and a rejective sampling design $Q(s,\omega)$ having inclusion probabilities equal to those of $P(s,\omega)$. Since, $K(p_1,p_2)$ in  ($5$) in Section $3$ of the main text and $\hat{K}(p_1,p_2)$ in ($12$) in Section $5.1$ of the main text depend on $P(s,\omega)$  only through its inclusion probabilities, and $P(s,\omega)$ and $Q(s,\omega)$ have equal inclusion probabilities, it is enough to show that the results in \eqref{eq A13} in the statement of the above lemma hold for $Q(s,\omega)$. The results in \eqref{eq A13} hold for $P(s,\omega)$ in the same way as the conclusion of Proposition $3.1$ holds for $P(s,\omega)$ (see the proof of Proposition $3.1$ in Section \ref{sec 5} below). We shall first show that under $\textbf{P}^*$, 
$$
\sup_{p_1,p_2\in[\alpha,\beta]}|\hat{K}(p_1,p_2)-K(p_1,p_2)|=O_p(1)\text{ as }\nu\rightarrow\infty\text{ for }Q(s,\omega).
$$
 It can be shown in the same way as the derivation of the result in  ($33$) in the main text  that under $\textbf{P}^*$, $\{\sqrt{n}(\hat{Q}_y(p)-Q_{y}(p)): p\in[\alpha/2,(1+\beta)/2]\}$ converges weakly to a mean $0$ Gaussian process as $\nu\rightarrow\infty$  in $(D[\alpha/2,(1+\beta)/2],\mathcal{D})$ with respect to the sup norm metric,  for $Q(s,\omega)$. Consequently, 
\begin{equation}\label{eq A22}
\sup_{p\in[\alpha/2,(1+\beta)/2]}|\sqrt{n}(\hat{Q}_y(p)-Q_{y}(p))|=O_p(1)
\end{equation} 
as $\nu\rightarrow\infty$ under $\textbf{P}^*$ by continuous mapping theorem. Then, under $\textbf{P}^*$, we have 
$$
\sup_{p\in [\alpha,\beta]} |\sqrt{n}(\hat{Q}_y(p+1/\sqrt{n})-\hat{Q}_y(p-1/\sqrt{n}))/2|=O_p(1)\text{ as }\nu\rightarrow\infty\text{ for }Q(s,\omega)
$$
 since $\alpha-1/\sqrt{n}\geq\alpha/2$ and $\beta+1/\sqrt{n}\leq (1+\beta)/2$ for all sufficiently large $\nu$, and $f_y\circ Q_y$ is bounded away from $0$ on $[\alpha/2,(1+\beta)/2]$ by Assumption $3$. Here, we recall from Table \ref{table 7} in this supplement that $\sqrt{n}(\hat{Q}_{y}(p+1/\sqrt{n})-\hat{Q}_{y}(p-1/\sqrt{n}))/2$ is the estimator of $1/f_y(Q_y(p))$. Similarly, under $\textbf{P}^*$, 
$$
\sup_{p\in [\alpha,\beta]}|\sqrt{n}(\hat{Q}_{x}(p+1/\sqrt{n})-\hat{Q}_{x}(p-1/\sqrt{n}))/2|=O_p(1)\text{ as }\nu\rightarrow\infty\text{ for }Q(s,\omega).
$$
 It further follows from  ($36$) and ($42$) in the proof of Theorem $3.1$ in the main text  that under $\textbf{P}^*$, 
\begin{align*}
\begin{split}
&\sup_{p\in [\alpha,\beta]}|\hat{Q}_{y}(p)/\hat{Q}_{x}(p)-Q_y(p)/Q_x(p)|\xrightarrow{p} 0,\sum_{i\in s}\pi_i^{-1}Y_i\bigg/\sum_{i\in s}\pi_i^{-1}X_i\xrightarrow{p} E_{\textbf{P}}(Y_i)\bigg/E_{\textbf{P}}(X_i)\\
&\text{ and }\sum_{i\in s}\pi_i^{-1}X_iY_i\bigg/\sum_{i\in s}\pi_i^{-1}X_i^2\xrightarrow{p} E_{\textbf{P}}(X_iY_i)\bigg/E_{\textbf{P}}(X^2_i)\text{ as }\nu\rightarrow\infty\text{ for }Q(s,\omega).
\end{split}
\end{align*}
 Similarly, it can be shown that under $\textbf{P}^*$, 
$$
\sum_{i\in s}(1-\pi_i)\bigg/\sum_{i=1}^N\pi_i(1-\pi_i)\xrightarrow{p} 1\text{ as }\nu\rightarrow\infty\text{ for }Q(s,\omega).
$$
 Consequently, under $\textbf{P}^*$, 
$$
\sup_{p_1,p_2\in[\alpha,\beta]}|\hat{K}(p_1,p_2)-K(p_1,p_2)|=O_p(1)\text{ as }\nu\rightarrow\infty\text{ for }Q(s,\omega).
$$
 This completes the proof of the first  result in \eqref{eq A13} for $Q(s,\omega)$.
\par
\vspace{.2cm}

 Next, if we establish that under $\textbf{P}^*$, 
$$
\hat{K}(p_1,p_2)-\tilde{K}(p_1,p_2)\xrightarrow{p}0\text{ and }\tilde{K}(p_1,p_2)\xrightarrow{p} K(p_1,p_2)
$$
 as $\nu\rightarrow\infty$ for $Q(s,\omega)$ and any $p_1,p_2\in [\alpha,\beta]$, then the result 
$$
\hat{K}(p_1,p_2)\xrightarrow{p}K(p_1,p_2)\text{ as }\nu\rightarrow\infty\text{ for }Q(s,\omega)\text{ and any }p_1,p_2\in [\alpha,\beta]\text{ under }\textbf{P}^*
$$
 will follow. Here, 
$$
\tilde{K}(p_1,p_2)=(n/N^2)\sum_{i\in s}(\zeta_i(p_1)-\overline{\zeta}(p_1)-S(p_1)\pi_i)(\zeta_i(p_2)-\overline{\zeta}(p_2)-S(p_2)\pi_i)(\pi_i^{-1}-1)\pi_i^{-1}.
$$ 
 Note that 
\begin{align*}
\begin{split}
&\tilde{K}(p_1,p_2)-(n/N^2)\sum_{i=1}^N (\zeta_i(p_1)-\overline{\zeta}(p_1)-S(p_1)\pi_i)(\zeta_i(p_2)-\overline{\zeta}(p_2)-S(p_2)\pi_i)(\pi_i^{-1}-1)\\
&\xrightarrow{p} 0\text{ as }\nu\rightarrow\infty\text{ for any }p_1,p_2\in [\alpha,\beta]\text{ under }\textbf{P}^*
\end{split}
\end{align*} 
 in the same way as the derivation of the result $\sum_{i\in s} (N\pi_i)^{-1}\xrightarrow{p} 1$ for $Q(s,\omega)$ under $\textbf{P}^*$ in the proof of Proposition $3.1$ (see the last few lines in $2^{nd}$ paragraph  of the proof of Proposition $3.1$). Also, note that $(n/N^2)\sum_{i=1}^N(\zeta_i(p_1)-\overline{\zeta}(p_1)-S(p_1)\pi_i)(\zeta_i(p_2)-\overline{\zeta}(p_2)-S(p_2)\pi_i)(\pi_i^{-1}-1)$ has a deterministic limit \textit{a.s.} $[\textbf{P}]$ for any $p_1,p_2\in [\alpha,\beta]$ in view of Assumption $2$-\textsl{(i)}. Further, 
\begin{align*}
\begin{split}
&E_{\textbf{P}}\bigg(\lim_{\nu\rightarrow\infty}(n/N^2)\sum_{i=1}^N(\zeta_i(p_1)-\overline{\zeta}(p_1)-S(p_1)\pi_i)(\zeta_i(p_2)-\overline{\zeta}(p_2)-S(p_2)\pi_i)(\pi_i^{-1}-1)\bigg)\\
&=K(p_1,p_2)\text{ for any }p_1,p_2\in [\alpha,\beta] 
\end{split}
\end{align*}
 in view of Assumption $2$-\textsl{(ii)} and DCT. Therefore, as $\nu\rightarrow\infty$, 
$$
(n/N^2)\sum_{i=1}^N(\zeta_i(p_1)-\overline{\zeta}(p_1)-S(p_1)\pi_i)(\zeta_i(p_2)-\overline{\zeta}(p_2)-S(p_2)\pi_i)(\pi_i^{-1}-1)\rightarrow K(p_1,p_2)
$$
 \textit{ a.s. } $[\textbf{P}]$, and hence $\tilde{K}(p_1,p_2)\xrightarrow{p} K(p_1,p_2)$ under $\textbf{P}^*$ for any $p_1,p_2\in [\alpha,\beta]$. 
\par
\vspace{.2cm}

Next, let us fix $\nu\geq 1$, $t>0$, $\delta>0$ and $p\in [\alpha,\beta]$. Then, we have 
\begin{align}\label{eq A14}
\begin{split}
&\bigg\{\sqrt{n}|\hat{Q}_{y}(p)-Q_y(p)|\leq t\text{ and }\sum_{i\in s}(\mathds{1}_{[Y_i\leq Q_y(p)+t/\sqrt{n}]}-\mathds{1}_{[Y_i\leq Q_y(p)-t/\sqrt{n}]})/N\pi_i\leq \delta\bigg\} \\
&\subseteq \bigg\{|\sum_{i\in s}\mathds{1}_{[Y_i\leq \hat{Q}_{y}(p)]}/N\pi_i-\sum_{i\in s} \mathds{1}_{[Y_i\leq Q_y(p)]}/N\pi_i|\leq \delta\bigg\}.
\end{split}
\end{align} 
Further, one can show that under $\textbf{P}^*$, 
\begin{align*}
\begin{split}
&\sum_{i\in s}(\mathds{1}_{[Y_i\leq Q_y(p)+t/\sqrt{n}]}-\mathds{1}_{[Y_i\leq Q_y(p)-t/\sqrt{n}]})\bigg/N\pi_i-\\
&F_{y,N} (Q_y(p)+t/\sqrt{n})+F_{y,N} (Q_y(p)-t/\sqrt{n})\xrightarrow{p} 0\text{ as }\nu\rightarrow\infty
\end{split}
\end{align*}
 in the same way as the derivation of the result $\sum_{i\in s} (N\pi_i)^{-1}\xrightarrow{p} 1$ for $Q(s,\omega)$ under $\textbf{P}^*$ in the proof of Proposition $3.1$. Moreover, under $\textbf{P}$, $F_{y,N}(Q_y(p)+t/\sqrt{n})-F_{y,N}(Q_y(p)-t/\sqrt{n})\xrightarrow{p} 0$ as $\nu\rightarrow\infty$ by Chebyshev's inequality and Assumption $3$. Thus as $\nu\rightarrow\infty$
\begin{equation}\label{eq A15}
\sum_{i\in s}(\mathds{1}_{[Y_i\leq Q_y(p)+t/\sqrt{n}]}-\mathds{1}_{[Y_i\leq Q_y(p)-t/\sqrt{n}]})/N\pi_i\xrightarrow{p} 0 \text{ under }\textbf{P}^{*}.
\end{equation}
Moreover, it follows from \eqref{eq A22} above that as $\nu\rightarrow\infty$, 
\begin{equation}\label{eq A16}
\sqrt{n}|\hat{Q}_{y}(p)-Q_y(p)|=O_p(1)\text{ under }\textbf{P}^{*}.
\end{equation}
Therefore, using \eqref{eq A14}, \eqref{eq A15} and \eqref{eq A16} above, one can show that 
$$
\sum_{i\in s}\mathds{1}_{[Y_i\leq \hat{Q}_{y}(p)]}\bigg/N\pi_i-\sum_{i\in s}\mathds{1}_{[Y_i\leq Q_y(p)]}\bigg/N\pi_i\xrightarrow{p} 0\text{ as }\nu\rightarrow\infty\text{ under }\textbf{P}^*.
$$
 Now, suppose that $p_n$=$p+c/\sqrt{n}$ for $c\in\mathds{R}$. Then, we have 
$$
Q_y(p_n)=Q_y(p)+(c/\sqrt{n})(1/f_y(Q_y(\epsilon_n)))
$$
 by Taylor expansion, where $\epsilon_n\rightarrow p$ as $\nu\rightarrow\infty$. Thus one can show that as $\nu\rightarrow\infty$, 
$$
\sqrt{n}(\hat{F}_{y}(Q_y(p_n))-\hat{F}_{y}(Q_y(p))-F_y(Q_y(p_n))+p)\xrightarrow{p} 0\text{ under }\textbf{P}^*
$$
 in the same way as the derivation of the result $\sum_{i\in s} (N\pi_i)^{-1}\xrightarrow{p} 1$ for $Q(s,\omega)$ under $\textbf{P}^*$ in the proof of Proposition $3.1$. Further, it can be shown that  
$$
\hat{Q}_{y}(p)-Q_y(p)=(p-\hat{F}_{y}(Q_y(p)))/f_y(Q_y(p))+o_p(1/\sqrt{n})\text{ as }\nu\rightarrow\infty\text{ under }\textbf{P}^*.
$$
 Similarly, we have 
$$
\hat{Q}_{y}(p_n)-Q_y(p_n)=(p_n-\hat{F}_{y}(Q_y(p_n)))/f_y(Q_y(p_n))+o_p(1/\sqrt{n})\text{ as }\nu\rightarrow\infty\text{ under }\textbf{P}^*.
$$
 Therefore, 
 $$
 \sqrt{n}(\hat{Q}_{y}(p+1/\sqrt{n})-\hat{Q}_{y}(p-1/\sqrt{n}))/2\xrightarrow{p} 1/f_y(Q_y(p))\text{ as }\nu\rightarrow\infty\text{ under }\textbf{P}^*.
 $$ 
  Similarly, 
\begin{align*}
\begin{split}
 &\sum_{i\in s}\mathds{1}_{[X_i\leq \hat{Q}_{x}(p)]}/N\pi_i-\sum_{i\in s}\mathds{1}_{[X_i\leq Q_x(p)]}/N\pi_i\xrightarrow{p} 0\text{ and }\\
 &\sqrt{n}(\hat{Q}_{x}(p+1/\sqrt{n})-\hat{Q}_{x}(p-1/\sqrt{n}))/2\xrightarrow{p} 1/f_x(Q_x(p))\text{ as }\nu\rightarrow\infty\text{ under }\textbf{P}^*. 
 \end{split}
\end{align*} 
 Hence, under $\textbf{P}^*$, $\hat{K}(p_1,p_2)-\tilde{K}(p_1,p_2)\xrightarrow{p}0$ as $\nu\rightarrow\infty$ for $Q(s,\omega)$ and any $p_1,p_2\in [\alpha,\beta]$. This completes the proof of \textsl{(i)}. The proof of \textsl{(ii)} follows exactly the same way as the proof of \textsl{(i)}.
\end{proof} 
\section{Additional results required for proving the main results for stratified multistage cluster sampling design} Suppose that $P(s,\omega)$ denotes the stratified multistage cluster sampling design with SRSWOR mentioned in Section $4$ of the main text. Fix $k\geq 1$ and $p_1,\ldots,p_k\in (0,1)$. Recall $\textbf{G}_{hjl}^{\prime}$ from the paragraph preceding Assumption $10$ in the main text. Define 
$$
\textbf{V}_{hjl}^{\prime}=\textbf{G}_{hjl}^{\prime}-\overline{\textbf{G}}^{\prime}\text{ and }\hat{\overline{\textbf{V}}}_3=\sum_{h=1}^H \sum_{j\in s_h}\sum_{l\in s_{hj}} M_h N_{hj}\textbf{V}_{hjl}^{\prime}/m_h r_h N
$$ 
 for $h$=$1,\ldots,H$, $j$=$1,\ldots,M_h$ and $l$=$1,\ldots,N_{hj}$, where $\overline{\textbf{G}}^{\prime}$=$\sum_{h=1}^H\sum_{j=1}^{M_h}\sum_{l=1}^{N_{hj}} \textbf{G}_{hjl}^{\prime}/N$. Now, we state the following lemma.
\begin{lemma}\label{lem s8}
\textsl{(i)} Fix $\textbf{b}\in \mathds{R}^{2k}$ such that $\textbf{b}\neq 0$. Suppose that $H$ is fixed as $\nu\rightarrow\infty$, and Assumptions $1$ and $6$-$8$ in the main text hold. Then, under $P(s,\omega)$, 
$$
\sqrt{n}\textbf{b}\hat{\overline{\textbf{V}}}^T_3\xrightarrow{\mathcal{L}}N(0,\lambda\textbf{b}\Gamma_7 \textbf{b}^T)\text{ as }\nu\rightarrow\infty \textit{ a.s. }[\textbf{P}] 
$$
 for some p.d. matrix $\Gamma_7$, where $\lambda$ is as in Assumption $1$. \\
\textsl{(ii)} Further, if $H\rightarrow\infty$ as $\nu\rightarrow\infty$, and Assumptions $1$ and $8$--$10$ in the main text hold, then the same result holds. 
\end{lemma}
\begin{proof}
Note that 
$$
\sqrt{n}\textbf{b}\hat{\overline{\textbf{V}}}^T_3=\sqrt{n}\sum_{h=1}^H \sum_{j\in s_h}\sum_{l\in s_{hj}}M_h N_{hj}\textbf{V}_{hjl}^{\prime}\textbf{b}^T/m_h r_h N=\sum_{h=1}^H\mathcal{T}_h\text{ (say)}.
$$
\vspace{.2cm}

 \textsl{(i)} We shall first show that $\mathcal{T}_h$=$\sqrt{n} \sum_{j\in s_h}\sum_{l\in s_{hj}}$ $M_h N_{hj}\textbf{V}_{hjl}^{\prime}\textbf{b}^T/m_h r_h N$ is asymptotically normal under two stage cluster sampling design with SRSWOR for each $h$=$1,\ldots, H$. Then, the asymptotic normality of $\sum_{h=1}^H \mathcal{T}_h$ follows from the independence of $\{\mathcal{T}_h\}_{h=1}^H$. For establishing the asymptotic normality of $\mathcal{T}_h$, we will use Theorem $2.1$ in \cite{MR983089}. 
\par
\vspace{.1cm}

Let $\Theta_h$=$\sum_{j\in s_h}\sum_{l=1}^{N_{hj}}\textbf{V}_{hjl}^{\prime}\textbf{b}^T/\sqrt{m_h}$ for $h$=$1,\ldots,H$. Note that $\Theta_h/\sqrt{m_h}$ is the HT estimator of $\sum_{j=1}^{M_h}\sum_{l=1}^{N_{hj}}\textbf{V}_{hjl}^{\prime}\textbf{b}^T/M_h$ under SRSWOR. Also, note that Assumption $2$-\textsl{(ii)} holds trivially under SRSWOR. It follows from Assumptions $1$ and $6$ that $\sum_{j=1}^{M_h}|\sum_{l=1}^{N_{hj}}\textbf{V}_{hjl}^{\prime}\textbf{b}^T|^{2+\delta} /M_h$ =$O(1)$ as $\nu\rightarrow\infty$ for any $0<\delta\leq 2$ and $\omega\in\Omega$. 
\par
\vspace{.1cm}

Now, it can be shown that $var(\Theta_h)$=$\sigma^2_{h,1}-\sigma^2_{h,2}+\sigma^2_{h,3}$. Here, 
\begin{align*}
\begin{split}
&\sigma^2_{h,1}=(1-f_h)\sum_{j=1}^{M_h}N_{hj}^2\big((\overline{\textbf{G}}_{hj}^{\prime}-\overline{\textbf{G}}^{\prime})\textbf{b}^T\big)^2/(M_h-1),\\
&\sigma^2_{h,2}=2(1-f_h)N_h\big((\overline{\textbf{G}}_{h}^{\prime}-\overline{\textbf{G}}^{\prime})\textbf{b}^T\big)\sum_{j=1}^{M_h}N_{hj}\big((\overline{\textbf{G}}_{hj}^{\prime}-\overline{\textbf{G}}^{\prime})\textbf{b}^T\big)/M_h(M_h-1)\\
&\text{and }\sigma^2_{h,3}=(1-f_h)N^2_h\big((\overline{\textbf{G}}_{h}^{\prime}-\overline{\textbf{G}}^{\prime})\textbf{b}^T\big)^2/M_h(M_h-1)
\end{split}
\end{align*}
 with $f_h$=$m_h/M_h$, $\overline{\textbf{G}}_{hj}^{\prime}$=$\sum_{l=1}^{N_{hj}}\textbf{G}_{hjl}^{\prime}/N_{hj}$ and $\overline{\textbf{G}}_h^{\prime}$=$\sum_{j=1}^{M_h}$ $\sum_{l=1}^{N_{hj}}\textbf{G}_{hjl}^{\prime}$ $/N_h$. Next, we note that 
\begin{align}\label{eq 1}
\begin{split}
&\sigma^2_{h,1}=(1-f_h)\bigg(\sum_{j=1}^{M_h}N_{hj}^2\big(\overline{\textbf{G}}_{hj}^{\prime}\textbf{b}^T\big)^2-2\big(\overline{\textbf{G}}^{\prime}\textbf{b}^T\big)\sum_{j=1}^{M_h}N_{hj}^2\big(\overline{\textbf{G}}_{hj}^{\prime}\textbf{b}^T\big)+\\
&\tilde{N}_h(\overline{\textbf{G}}^{\prime}\textbf{b}^T\big)^2\bigg)\bigg/(M_h-1),
\end{split}
\end{align} 
where $\tilde{N}_h$=$\sum_{j=1}^{M_h}N_{hj}^2$. Let us consider the first term on the right hand side of \eqref{eq 1}. Using Assumptions $1$ and $6$, and Hoeffding's inequality, it can be shown that 
$$
(1-f_h)\sum_{j=1}^{M_h}N_{hj}^2 \big((\overline{\textbf{G}}_{hj}^{\prime}\textbf{b}^T)^2-E_{\textbf{P}}(\overline{\textbf{G}}_{hj}^{\prime}\textbf{b}^T)^2\big)/(M_h-1)\rightarrow 0\text{ as }\nu\rightarrow\infty \textit{ a.s. } [\textbf{P}]. 
$$
 Further, we have
$$
(1-f_h)\sum_{j=1}^{M_h}N_{hj}^2 E_{\textbf{P}} (\overline{\textbf{G}}_{hj}^{\prime}\textbf{b}^T)^2/(M_h-1)=(1-f_h)(N_h\tilde{\sigma}^2_h+\tilde{N}_h\mu_h^2)/(M_h-1),
$$ 
where $\tilde{\sigma}^2_h$=$E_{\textbf{P}}\big[\big(\textbf{G}_{hjl}^{\prime}-E_{\textbf{P}}(\textbf{G}_{hjl}^{\prime})\big)\textbf{b}^T]^2$=$\textbf{b}\Gamma_h\textbf{b}^T$ (recall $\Gamma_h$ from the paragraph preceding Assumption $10$) and $\mu_h$=$E_{\textbf{P}}$ $(\textbf{G}_{hjl}^{\prime}\textbf{b}^T)$. Thus 
\begin{equation}\label{eq 2}
(1-f_h)\sum_{j=1}^{M_h}N_{hj}^2 (\overline{\textbf{G}}_{hj}^{\prime}\textbf{b}^T)^2/(M_h-1)=(1-f_h)(N_h\tilde{\sigma}^2_h+\tilde{N}_h\mu_h^2)/(M_h-1)+o(1)
\end{equation}
as $\nu\rightarrow\infty$ \textit{a.s.} $[\textbf{P}]$. Using similar arguments, we can say that 
\begin{align*}
\begin{split}
&\sigma^2_{h,1}=(1-f_h)(N_h\tilde{\sigma}^2_h+\tilde{N}_h(\mu_h-\tilde{\mu})^2)/(M_h-1)+o(1),\\
&\sigma^2_{h,2}=2(1-f_h)N_h^2(\mu_h-\tilde{\mu})^2/M_h(M_h-1)+o(1)\text{ and }\\
&\sigma^2_{h,3}=(1-f_h)N_h^2(\mu_h-\tilde{\mu})^2/M_h(M_h-1)+o(1)\text{ as }\nu\rightarrow\infty\textit{ a.s. } [\textbf{P}],
\end{split}
\end{align*}
 where $\tilde{\mu}$=$\sum_{h=1}^H\Lambda_h\mu_h$ (recall $\Lambda_h$'s from Assumption $6$). Then, we have 
\begin{equation}\label{eq 8}
var(\Theta_h)=(1-f_h)N_h\tilde{\sigma}^2_h/(M_h-1)+o(1)
\end{equation} 
as $\nu\rightarrow\infty$ \textit{a.s.} $[\textbf{P}]$ by Assumption $6$.  
\par
\vspace{.2cm}

Next, recall $F_{y,H}(t)$ and $F_{x,H}(t)$ from the paragraph preceding Assumption $10$. It can be shown that 
\begin{equation}\label{eq 26}
\sup_{t\in \mathds{R}}|F_{y,H}(t)-\tilde{F}_{y,H}(t)|\rightarrow 0\text{ and }\sup_{t\in \mathds{R}}|F_{x,H}(t)-\tilde{F}_{x,H}(t)|\rightarrow 0\text{ as }\nu\rightarrow \infty
\end{equation}
 by Assumption $6$, where $\tilde{F}_{y,H}(t)$=$\sum_{h=1}^H\Lambda_h F_{y,h}(t)$ and $\tilde{F}_{x,H}(t)$=$\sum_{h=1}^H\Lambda_h F_{x,h}(t)$. Then, it follows from Lemma \ref{lem s10} that 
\begin{equation}\label{eq 27}
Q_{y,H}(p_r)\rightarrow\tilde{Q}_{y,H}(p_r)\text{ as }\nu\rightarrow\infty\text{ for any }r=1,\ldots,k,
\end{equation}
 where $\tilde{Q}_{y,H}(p)$=$\inf\{t\in\mathds{R}:\tilde{F}_{y,H}(t)\geq p\}$. Similarly, 
\begin{equation}\label{eq 28}
Q_{x,H}(p_r)\rightarrow\tilde{Q}_{x,H}(p_r)\text{ as }\nu\rightarrow\infty\text{ for any }r=1,\ldots,k, 
\end{equation} 
 where $\tilde{Q}_{x,H}(p)$=$\inf\{t\in\mathds{R}:\tilde{F}_{x,H}(t)\geq p\}$. Let 
$$
\tilde{\textbf{G}}_{hjl}=\bigg(\mathds{1}_{[Y_{hjl}^{\prime}\leq \tilde{Q}_{y,H}(p_1)]},\ldots, \mathds{1}_{[Y_{hjl}^{\prime}\leq \tilde{Q}_{y,H}(p_k)]},\mathds{1}_{[X_{hjl}^{\prime}\leq \tilde{Q}_{x,H}(p_1)]}, \ldots,\mathds{1}_{[X_{hjl}^{\prime}\leq \tilde{Q}_{x,H}(p_k)]}\bigg),
$$
 where $(Y^{\prime}_{hjl},X^{\prime}_{hjl})$ is as in the $2^{nd}$ paragraph of Section $4$ in the main text. Then, 
$$
\tilde{\sigma}^2_h=\textbf{b}\Gamma_h \textbf{b}^T \rightarrow \textbf{b} E_{\textbf{P}}\big(\tilde{\textbf{G}}_{hjl}-E_{\textbf{P}}(\tilde{\textbf{G}}_{hjl}))^T(\tilde{\textbf{G}}_{hjl}-E_{\textbf{P}}(\tilde{\textbf{G}}_{hjl})\big)\textbf{b}^T
$$
 as $\nu\rightarrow\infty$ for any $h$=$1,\ldots, H$ in view of Assumption $8$. Moreover, $E_{\textbf{P}}\big(\tilde{\textbf{G}}_{hjl}-E_{\textbf{P}}(\tilde{\textbf{G}}_{hjl}))^T(\tilde{\textbf{G}}_{hjl}-E_{\textbf{P}}(\tilde{\textbf{G}}_{hjl})\big)$ is a p.d. matrix because Assumption $7$ holds. Therefore, 
$$
\underline{\lim}_{\nu\rightarrow\infty}((M_h-1)/M_h)var(\Theta_h)>0\textit{ a.s. } [\textbf{P}]
$$
 by \eqref{eq 8} above and Assumption $6$ in the main text. Hence, one can show that 
$$
(\Theta_h-E(\Theta_h))/\sqrt{var(\Theta_h)}\xrightarrow{\mathcal{L}}N(0,1)\text{ as }\nu\rightarrow\infty\text{ under SRSWOR }\textit{ a.s. } [\textbf{P}]
$$
in the same way as the derivation of the result, which appears in the proof of Lemma S$4$ in \cite{deychaudhuri2023}, that $\sqrt{n}\textbf{m}(\hat{\overline{\textbf{V}}}_1-\overline{\textbf{V}})^T\xrightarrow{\mathcal{L}}N(0,\textbf{m} \Gamma_1 \textbf{m}^T)$ as $\nu\rightarrow\infty$ under SRSWOR for any $\textbf{m}\in\mathbb{R}^{p}$, $\textbf{m}\neq 0$ and $\Gamma_1$=$\lim_{\nu\rightarrow\infty}\Sigma_1$.  Thus the condition C$1$ of Theorem $2.1$ in \cite{MR983089} holds \textit{a.s.} $[\textbf{P}]$.
\par
\vspace{.1cm}

Next, suppose that $\overline{\textbf{V}}_{hj}^{\prime}$=$\sum_{l=1}^{N_{hj}}\textbf{V}_{hjl}^{\prime}/N_{hj}$. Note that for any $h$=$1,\ldots, H$, $\sum_{l=1}^{N_{hj}}\big((\textbf{V}_{hjl}^{\prime}-\overline{\textbf{V}}_{hj}^{\prime})\textbf{b}^T\big)^2/N_{hj}$ are independent bounded random variables for $1\leq j\leq M_h$. Then, by Assumptions $1$ and $6$, and Hoeffding's inequality, we have 
$$
\sum_{j=1}^{M_h}(N_{hj}^2/r_h)(1/m_h)\bigg[\sum_{l=1}^{N_{hj}}\big((\textbf{V}_{hjl}^{\prime}-\overline{\textbf{V}}_{hj}^{\prime})\textbf{b}^T\big)^2/N_{hj}\bigg]=(1/r_h m_h)\sum_{j=1}^{M_h} N_{hj}(N_{hj}-1)\tilde{\sigma}^2_h+o(1)
$$
 as $\nu\rightarrow\infty$ \textit{a.s.} $[\textbf{P}]$. Thus 

$$
\underline{\lim}_{\nu\rightarrow\infty}\big[\sum_{j=1}^{M_h} (N_{hj}^2/r_h)(1/m_h)\big\{\sum_{l=1}^{N_{hj}}\big((\textbf{V}_{hjl}^{\prime}-\overline{\textbf{V}}_{hj}^{\prime})\textbf{b}^T\big)^2/N_{hj}\big\}\big]^2>0
$$
 \textit{a.s.} $[\textbf{P}]$. Further, in view of Assumption $6$, we have 
$$ 
\bigg[\sum_{j=1}^{M_h}(N_{hj}^4/r_h^2) (1/M_h)\bigg\{\sum_{l=1}^{N_{hj}}\big((\textbf{V}_{hjl}^{\prime}-\overline{\textbf{V}}_{hj}^{\prime}\big)\textbf{b}^T\big)^2/N_{hj}\bigg\}^2\bigg]\leq K
$$ 
 for all sufficiently large $\nu$ and some constant $K>0$ \textit{a.s.} $[\textbf{P}]$. Therefore, 
\begin{align*}
&\lim_{\nu\rightarrow\infty}\bigg[\sum_{j=1}^{M_h}(N_{hj}^4/r_h^2)(M_h/m_h)^3\bigg\{\sum_{l=1}^{N_{hj}}\big((\textbf{V}_{hjl}^{\prime}-\overline{\textbf{V}}_{hj}^{\prime})\textbf{b}^T\big)^2\big/N_{hj}\bigg\}^2\bigg]\bigg/\\
&\bigg[\sum_{j=1}^{M_h}(N_{hj}^2/r_h)(M_h/m_h)\bigg\{\sum_{l=1}^{N_{hj}}\bigg((\textbf{V}_{hjl}^{\prime}-\overline{\textbf{V}}_{hj}^{\prime})\textbf{b}^T\big)^2 \big/N_{hj}\bigg\}\bigg]^2=0 
\end{align*}
\textit{a.s.} $[\textbf{P}]$ by Assumption $6$. Thus the condition C$2$ of Theorem $2.1$ in \cite{MR983089} holds \textit{a.s.} $[\textbf{P}]$ by Assumption $6$ in the main text and Proposition $4.1$ in \cite{MR983089}. 
\par
\vspace{.1cm}

The condition C$3$ of Theorem $2.1$ in \cite{MR983089} holds for any $\omega\in\Omega$ by $(b)$ of Proposition $2.3$ in \cite{MR983089} since SRSWOR is used to select samples from clusters in the $1^{st}$ stage and from population units of the selected clusters in the $2^{nd}$ stage. Therefore, the conditions C$1$, C$2$ and C$3$ of Theorem $2.1$ in \cite{MR983089} hold \textit{a.s.} $[\textbf{P}]$. Hence, by Theorem $2.1$ in \cite{MR983089}, we have  
\begin{equation}\label{eq 11}
(\sqrt{n_h}/N_h)(N/\sqrt{n})(\mathcal{T}_h-E(\mathcal{T}_h))/(var((\sqrt{n_h}/N_h)(N/\sqrt{n})\mathcal{T}_h))^{1/2}\xrightarrow{\mathcal{L}}N(0,1)
\end{equation} 
as $\nu\rightarrow\infty$ under two stage cluster sampling design with SRSWOR \textit{a.s.} $[\textbf{P}]$ for any $h$=$1,\ldots,H$. Now, 
\begin{align*}
\begin{split}
&var((\sqrt{n_h}/N_h)(N/\sqrt{n})\mathcal{T}_h)=\sum_{j=1}^{M_h}\tilde{c}_{hj}\big((\overline{\textbf{G}}_{hj}^{\prime}-\overline{\textbf{G}}^{\prime})\textbf{b}^T\big)^2- \tilde{c}_h \big(N_h(\overline{\textbf{G}}_h^{\prime}-\overline{\textbf{G}}^{\prime})\textbf{b}^T/M_h\big)^2+\\
&\sum_{j=1}^{M_h}\tilde{d}_{hj}\sum_{l=1}^{N_{hj}}\big((\textbf{G}_{hjl}^{\prime}-\overline{\textbf{G}}_{hj}^{\prime})\textbf{b}^T\big)^2,
\end{split}
\end{align*}
 where 
$$
\tilde{c}_{hj}=(N/N_h)^2 (n_h/n)c_{hj},  \tilde{d}_{hj}=(N/N_h)^2(n_h/n)d_{hj},\text{ and }\tilde{c}_h=(N/N_h)^2(n_h/n)c_h.
$$ 
 Here,
\begin{align*}
\begin{split}
&c_{hj}=c_h N_{hj}^2/M_h, d_{hj}=nM_h(1-f_{hj})N_{hj}^2/m_h r_h(N_{hj}-1)N^2,\\
&c_h=nM_h^3(1-f_h)/m_h(M_h-1)N^2,  f_h=m_h/M_h,\text{ and }f_{hj}=r_h/ N_{hj}. 
\end{split}
\end{align*} 
 It can be shown using Hoeffding's inequality that 
$$
var((\sqrt{n_h}/ N_h)(N/\sqrt{n})\mathcal{T}_h)=(1-n_h/N_h)\tilde{\sigma}^2_h+o(1)\text{ as }\nu\rightarrow\infty \textit{ a.s. } [\textbf{P}].
$$
 Therefore, using \eqref{eq 11} above and Assumption $6$ in the main text, it can be shown that 
$$
\sum_{h=1}^H \mathcal{T}_h=\sum_{h=1}^H(\mathcal{T}_h-E(\mathcal{T}_h))\xrightarrow{\mathcal{L}}N(0,\Delta^2)\text{ as }\nu\rightarrow\infty \textit{ a.s. } [\textbf{P}].
$$
 Here, 

\begin{align*}
\begin{split}
&\Delta^2=\lim_{\nu\rightarrow\infty}\sum_{h=1}^H n N_h(N_h-n_h)\tilde{\sigma}^2_h/n_h N^2=\lim_{\nu\rightarrow\infty}\sum_{h=1}^H  n N_h(N_h-n_h)\textbf{b}\Gamma_h\textbf{b}^T/n_h N^2\\
&= \lambda\sum_{h=1}^H\Lambda_h(\Lambda_h/\lambda\lambda_h -1)E_{\textbf{P}}\big( \tilde{\textbf{G}}_{hjl} \textbf{b}^T-E_{\textbf{P}}(\tilde{\textbf{G}}_{hjl}\textbf{b}^T)\big)^2=\lambda \textbf{b}\Gamma_7 \textbf{b}^T>0
\end{split}
\end{align*}
 with $\Gamma_7$=$\sum_{h=1}^H\Lambda_h(\Lambda_h/$ $\lambda\lambda_h-1)E_{\textbf{P}}\big(\tilde{\textbf{G}}_{hjl}-E_{\textbf{P}}(\tilde{\textbf{G}}_{hjl}))^T$ $(\tilde{\textbf{G}}_{hjl}-E_{\textbf{P}}(\tilde{\textbf{G}}_{hjl})\big)$. This completes the proof of \textsl{(i)}.
\vspace{.3cm}

\textsl{(ii)} Since, population units are sampled independently across the strata in $P(s,\omega)$, asymptotic normality of $\sum_{h=1}^H\mathcal{T}_h$ under $P(s,\omega)$ follows by applying Lyapunov's CLT to independent random variables $\{\mathcal{T}_h\}_{h=1}^H$. Note that  for any $\delta>0$, we have
$$
|\mathcal{T}_h|^{2+\delta}\leq \mathcal{\epsilon}(\nu) (m_h/\sqrt{n})^{2+\delta}
$$
 by Assumption $9$, where $\mathcal{\epsilon}(\nu)$ does not depend on $s$ and $\omega$, and $\mathcal{\epsilon}(\nu)$=$O(1)$ as $\nu\rightarrow\infty$. Therefore, under $P(s,\omega)$, 
$$
\sum_{h=1}^H E|\mathcal{T}_h|^{2+\delta}\leq\epsilon(\nu) (H/n^{1+\delta/2}) \sum_{h=1}^H M_h^{2+\delta}/H = O(n^{-\delta/2})
$$
as $\nu\rightarrow\infty$ for any $0<\delta\leq 2$ and $\omega\in\Omega$.  Hence, under $P(s,\omega)$, $\sum_{h=1}^H E|\mathcal{T}_h-E(\mathcal{T}_h)|^{2+\delta}\rightarrow 0$ as $\nu\rightarrow\infty$ for any $0<\delta\leq 2$ and $\omega\in\Omega$. 
\par
\vspace{.1cm}

Next, we have 
\begin{align*}
\begin{split}
&\sum_{h=1}^H var(\mathcal{T}_h)=\sum_{h=1}^H\sum_{j=1}^{M_h}c_{hj}\big((\overline{\textbf{G}}_{hj}^{\prime}-\overline{\textbf{G}}^{\prime})\textbf{b}^T\big)^2-\sum_{h=1}^H c_h \big(N_h(\overline{\textbf{G}}_h^{\prime}-\overline{\textbf{G}}^{\prime})\textbf{b}^T/M_h\big)^2\\
&+\sum_{h=1}^H\sum_{j=1}^{M_h}d_{hj}\sum_{l=1}^{N_{hj}}
\big((\textbf{G}_{hjl}^{\prime}-\overline{\textbf{G}}_{hj}^{\prime})\textbf{b}^T\big)^2=\Delta^2_1-\Delta^2_2+\Delta^2_3\text{ (say).}
\end{split}
\end{align*}
 Now, it can be shown using Assumptions $1$ and $9$, and Hoeffding's inequality that 
\begin{equation}\label{eq 3}
\begin{split}
&\Delta^2_1-\Delta^2_2+\Delta^2_3=\sum_{h=1}^H c_h(\tilde{N}_h-N_h^2/M_h)(\mu_h-\mu^{*})^2/M_h+\sum_{h=1}^H nM_h\tilde{N}_h\tilde{\sigma}^2_h/m_h r_h N^2\\
&-\sum_{h=1}^H nN_h\tilde{\sigma}^2_h/N^2+o(1)\text{ as }\nu\rightarrow\infty\text{ \textit{a.s.} }[\textbf{P}], 
\end{split}
\end{equation}
in the same way as the derivation of the result in \eqref{eq 2}. Here, $\mu^{*}$=$\sum_{h=1}^H$ $N_h\mu_h/N$. The first term on the right hand side of \eqref{eq 3} converges to $0$ as $\nu\rightarrow\infty$ by Assumptions $1$ and $9$. Moreover, we have
\begin{equation}\label{eq 6}
\begin{split}
&\sum_{h=1}^H nM_h\tilde{N}_h\tilde{\sigma}^2_h/m_h r_h N^2-\sum_{h=1}^H nN_h\tilde{\sigma}^2_h/N^2\\
&=(n/N^2)\sum_{h=1}^H M_h(\tilde{N}_h-N_h^2/M_h)\tilde{\sigma}^2_h/m_h r_h +(n/N^2)\sum_{h=1}^H N_h(N_h-n_h)\tilde{\sigma}^2_h/n_h.
\end{split}
\end{equation} 
The first term on the right hand side of \eqref{eq 6} converges to $0$ and 
$$
(n/N^2)\sum_{h=1}^H N_h(N_h-n_h)\tilde{\sigma}^2_h/n_h=\lambda\sum_{h=1}^H N_h(N_h-n_h)\tilde{\sigma}^2_h/ n_h N+o(1)\text{ as }\nu\rightarrow\infty
$$
 by Assumption $9$. Therefore, 
$$
\Delta_1^2-\Delta_2^2+\Delta_3^2=\lambda\sum_{h=1}^H N_h(N_h-n_h)\tilde{\sigma}^2_h/n_h N+o(1),
$$ 
 and hence 
$$
\sum_{h=1}^H var(\mathcal{T}_h)=\Delta_1^2-\Delta_2^2+\Delta_3^2\rightarrow \lambda\textbf{b}\Gamma_1\textbf{b}^T>0
$$
 as $\nu\rightarrow\infty$ \textit{a.s.} $[\textbf{P}]$ for some p.d. matrix $\Gamma_1$ in view of Assumption $10$. Here, $\Gamma_1$ is as in Assumption $10$. Thus the Lyapunov's condition $\sum_{h=1}^H E|\mathcal{T}_h-E(\mathcal{T}_h)|^{2+\delta}/(\sum_{h=1}^H var(\mathcal{T}_h))^{1+\delta/2}\rightarrow 0$ as $\nu\rightarrow \infty$ for some $\delta>0$, holds under $P(s,\omega)$ \textit{a.s.} $[\textbf{P}]$. Consequently, $\sum_{h=1}^H\mathcal{T}_h\xrightarrow{\mathcal{L}}N(0,\lambda\textbf{b}\Gamma_7 \textbf{b}^T)$ as $\nu\rightarrow\infty$ \textit{a.s.} $[\textbf{P}]$ with $\Gamma_7$=$\Gamma_1$. This completes the proof of \textsl{(ii)}.
\end{proof}
Next, consider $\{Z_i\}_{i=1}^N$  as in ($4$) in the main text  with $F_{y,H}$ replacing $F_y$. Also, consider $B_{z,N}(u,t)$ and $\mathds{B}_n(u,t)$ from \eqref{eq 39}.  Now, we state the following lemma. 
\begin{lemma}\label{lem s9}
\textsl{(i)} Suppose that $H$ is fixed as $\nu\rightarrow\infty$, and Assumptions $1$, $6$ and $8$ in the main article hold. Then, under $P(s,\omega)$, 

$$
E\big[\big(\mathds{B}_n(t_1,t_2)\big)^2\big(\mathds{B}_n(t_2,t_3)\big)^2\big]\leq L_1\big(B_{z,N}(t_1, t_3)\big)^2 \textit{ a.s. } [\textbf{P}]
$$
for any $0\leq t_1<t_2<t_3\leq 1$, $\nu\geq 1$ and some constant $L_1>0$, and
$$
\overline{\lim}_{\nu\rightarrow\infty} E\big(\mathds{B}_n(u,t)\big)^4\leq L_2 (t-u)^2  \textit{ a.s. } [\textbf{P}]
$$
for any $0\leq u<t\leq 1$ and some constant $L_2>0$. 

\noindent{\textsl{(ii)} Further, if $H\rightarrow\infty$ as $\nu\rightarrow\infty$, and Assumptions $1$, $8$ and $9$ in the main article hold, then the same results hold.}
\end{lemma}
\begin{proof} 
Recall $Y_{hjl}^{\prime}$ from the $2^{nd}$ paragraph in Section $4$ of the main text. Let us define $Z_{hjl}^{\prime}$=$F_{y,H}(Y_{hjl}^{\prime})$ for any given $h$=$1,\ldots,H$, $j$=$1,\ldots,M_h$ and $l$=$1,\ldots,N_{hj}$. Consider $F_{z,N}(t)$ and $\mathds{Z}_n(t)$ as in \eqref{eq 37}.  Recall from Section $4$ in the main text that given any $h,j$ and $l$, $Y_{hjl}^{\prime}$=$Y_i$ for some $i\in\{1,\ldots,N\}$. Also, recall from Section $4$ that under $P(s,\omega)$, the inclusion probability of the $i^{th}$ population unit is $\pi_i$=$m_h r_h/M_h N_{hj}$ if it belongs to the $j^{th}$ cluster of the $h^{th}$ stratum. Then, we have $\mathds{Z}_n(t)$=$\sqrt{n}\sum_{h=1}^H \sum_{j\in s_h}\sum_{l\in s_{hj}} M_h N_{hj} \big(\mathds{1}_{[Z_{hjl}^{\prime}\leq t]}-F_{z,N}(t)\big)/m_h r_h N$.
\par
\vspace{.1cm} 

Now, suppose that  for $h$=$1,\ldots, H$, $j$=$1,\ldots,M_h$ and $l$=$1,\ldots,N_{hj}$, 
\begin{align*}
\begin{split}
 \xi_{hjl}=
\begin{cases}
1,\text{ if the }l^{th}\text{ unit of the }j^{th}\text{ cluster in the }h^{th}\text{ stratum is selected in the sample},\text{ and }\\
0,\text{ otherwise.}\\
\end{cases}
\end{split}
\end{align*} 
Then, we have 
\begin{align*}
\begin{split}
\mathds{Z}_n(t)=(\sqrt{n}/N)\sum_{h=1}^H \sum_{j=1}^{M_h}\sum_{l=1}^{N_{hj}} \big((M_h N_{hj}  \xi_{hjl} /m_h r_h)-1\big)\big(\mathds{1}_{[Z_{hjl}^{\prime}\leq t]}-F_{z,N}(t)\big).
\end{split}
\end{align*}
Further, suppose that 
\begin{align*}
\begin{split}
&\tilde{\alpha}_h=\sum_{j=1}^{M_h}\sum_{l=1}^{N_{hj}}\big((M_h N_{hj}  \xi_{hjl} /m_h r_h)-1\big)\bar{A}_{hjl}\text{ and }\\
&\tilde{\beta}_h=\sum_{j=1}^{M_h}\sum_{l=1}^{N_{hj}}\big((M_h N_{hj}  \xi_{hjl} /m_h r_h)-1\big)\bar{B}_{hjl}
\end{split}
\end{align*} 
 for $h$=$1,\ldots, H$ and $0\leq t_1<t_2<t_3\leq 1$, where $\bar{A}_{hjl}$=$\mathds{1}_{[t_1<Z_{hjl}^{\prime}\leq t_2]}-B_{z,N}(t_1,t_2)$ and $\bar{B}_{hjl}$=$\mathds{1}_{[t_2<Z_{hjl}^{\prime}\leq t_3]}-B_{z,N}(t_2,t_3)$. Now, let us define $S_{k,H}$=$\{(h_1,\ldots,h_k):h_1,\ldots,h_k\in\{1,2,\ldots,H\}\text{ and }$ $h_1,\ldots,h_k\text{ are all}$ $ \text{distinct}\}$ for $k$=$2,3,4$. Then, we have
 
\begin{align}\label{eq 7}
\begin{split}
&E\big[\big(\mathds{B}_n(t_1,t_2)\big)^2\big(\mathds{B}_n(t_2,t_3)\big)^2\big]=(n^2/N^4)E\bigg[\sum_{h=1}^H \tilde{\alpha}_h^2\tilde{\beta}_h^2+\sum_{(h_1,h_2)\in S_{2,H}} \tilde{\alpha}_{h_1}^2\tilde{\beta}_{h_2}^2\\
&+\sum_{(h_1,h_2)\in S_{2,H}} \tilde{\alpha}_{h_1}^2\tilde{\beta}_{h_1} \tilde{\beta}_{h_2}+\sum_{(h_1,h_2)\in S_{2,H}} \tilde{\alpha}_{h_1}\tilde{\alpha}_{h_2} \tilde{\beta}^2_{h_2}+\sum_{(h_1,h_2)\in S_{2,H}} \tilde{\alpha}_{h_1}\tilde{\alpha}_{h_2}\tilde{\beta}_{h_1} \tilde{\beta}_{h_2}\\
&+\sum_{(h_1,h_2,h_3)\in S_{3,H}} \tilde{\alpha}_{h_1}^2\tilde{\beta}_{h_2} \tilde{\beta}_{h_3}+ \sum_{(h_1,h_2,h_3)\in S_{3,H}} \tilde{\alpha}_{h_1}\tilde{\alpha}_{h_2} \tilde{\beta}_{h_3}^2+\\
&\sum_{(h_1,h_2,h_3,h_4)\in S_{4,H}} \tilde{\alpha}_{h_1}\tilde{\alpha}_{h_2} \tilde{\beta}_{h_3}\tilde{\beta}_{h_4}\bigg].
\end{split}
\end{align}
\vspace{.2cm}

\textsl{(i)}  Suppose that $\overline{\alpha}_{hjl}$=$\big((M_h N_{hj}  \xi_{hjl} /m_h r_h)-1\big) \bar{A}_{hjl}$, $\overline{\beta}_{hjl}$=$\big((M_h N_{hj}  \xi_{hjl} /m_h r_h)-1\big) \bar{B}_{hjl}$, $\alpha_{hj}^{*}$=$\sum_{l=1}^{N_{hj}}\overline{\alpha}_{hjl}$ and $\beta_{hj}^{*}$=$\sum_{l=1}^{N_{hj}}\overline{\beta}_{hjl}$ for $h$=$1,\ldots, H$, $j$=$1,\ldots,M_{h}$, $l$=$1,\ldots,N_{hj}$ and $0\leq t_1<t_2<t_3\leq 1$. Then, we have $\tilde{\alpha}_h$=$\sum_{j=1}^{M_h}\alpha_{hj}^{*}$ and $\tilde{\beta}_{h}$=$\sum_{j=1}^{M_h}\beta_{hj}^{*}$. Now, let us consider the first term on right hand side of \eqref{eq 7}. Further, suppose that $S_{k,h}$=$\{(j_1,\ldots,j_k):j_1,\ldots,j_k\in \{1,\ldots,M_h\}\text{ and }j_1,\ldots,j_k\text{ are all distinct}\}$, $k$=$2,3,4$, $1\leq h\leq H$. Then, we have
\begin{align}\label{eq s5} 
\begin{split}
&(n^2/N^4)\sum_{h=1}^H  E(\tilde{\alpha}_h^2\tilde{\beta}_h^2)=(n^2/N^4)\sum_{h=1}^H  E\bigg[\sum_{j=1}^{M_h} (\alpha_{hj}^{*}\beta_{hj}^{*})^2+\sum_{(j_1,j_2)\in S_{2,h}} (\alpha_{h j_1}^{*}\beta_{h j_2}^{*})^2\\
&+\sum_{(j_1,j_2)\in S_{2,h}} (\alpha_{h j_1}^{*})^2\beta_{h j_1}^{*} \beta_{h j_2}^{*}+\sum_{(j_1,j_2)\in S_{2,h}} \alpha_{h j_1}^{*}\alpha_{h j_2}^{*} (\beta_{h j_2}^{*})^2+\\
&\sum_{(j_1,j_2)\in S_{2,h}} \alpha_{h j_1}^{*}\alpha_{h j_2}^{*}\beta_{h j_1}^{*} \beta_{h j_2}^{*}+\sum_{(j_1,j_2,j_3)\in S_{3,h}} (\alpha_{h j_1}^{*})^2\beta_{h j_2}^{*} \beta_{h j_3}^{*}+ \\
&\sum_{(j_1,j_2,j_3)\in S_{3,h}} \alpha_{h j_1}^{*}\alpha_{h j_2}^{*} (\beta_{h j_3}^{*})^2+\sum_{(j_1,j_2,j_3, j_4)\in S_{4,h}} \alpha_{h j_1}^{*}\alpha_{h j_2}^{*} \beta_{h j_3}^{*}\beta_{h j_4}^{*}\bigg].
\end{split}
\end{align}
Next, consider the first term on the right hand side of \eqref{eq s5}. Suppose that $S_{k,hj}$=$\{(l_1,\ldots,l_k):l_1,\ldots,l_k\in \{1,\ldots,N_{hj}\}\text{ and }l_1,\ldots,l_k\text{ are all distinct}\}$, $k$=$2,3,4$, $j$=$1,\ldots,M_h$ and $1\leq h\leq H$. Then, we have
\begin{align}\label{eq 12}
\begin{split}
&(n^2/N^4)\sum_{h=1}^H E\bigg[\sum_{j=1}^{M_h} (\alpha_{hj}^{*}\beta_{hj}^{*})^2\bigg]=(n^2/N^4)\sum_{h=1}^H  E\bigg[\sum_{j=1}^{M_h}\bigg(\sum_{l=1}^{N_{hj}} (\overline{\alpha}_{hjl}\overline{\beta}_{hjl})^2\\
&+\sum_{(l_1,l_2)\in S_{2,hj}} (\overline{\alpha}_{hj l_1}\overline{\beta}_{h j l_2})^2+\sum_{(l_1,l_2)\in S_{2,hj}} (\overline{\alpha}_{h j l_1})^2 \overline{\beta}_{h j l_1}  \overline{\beta}_{h j l_2}\\
&+\sum_{(l_1,l_2)\in S_{2,hj}} \overline{\alpha}_{h j l_1} \overline{\alpha}_{h j l_2} (\overline{\beta}_{h j l_2})^2+\sum_{(l_1,l_2)\in S_{2,hj}} \overline{\alpha}_{h j l_1}\overline{\alpha}_{h j l_2}\overline{\beta}_{h j l_1} \overline{\beta}_{h j_2}+\\
&\sum_{(l_1,l_2,l_3)\in S_{3,hj}} (\overline{\alpha}_{h j l_1})^2\overline{\beta}_{h j l_2} \overline{\beta}_{h j l_3}+\sum_{(l_1,l_2,l_3)\in S_{3,hj}} \overline{\alpha}_{h j l_1} \overline{\alpha}_{h j l_2} (\overline{\beta}_{h j l_3})^2\\
&+\sum_{(l_1,l_2,l_3, l_4)\in S_{4,hj}} \overline{\alpha}_{h j l_1}\overline{\alpha}_{h j l_2} \overline{\beta}_{h j l_3}\overline{\beta}_{h j l_4}\bigg)\bigg].
\end{split}
\end{align}
Now, consider the first term on the right hand side of \eqref{eq 12}. Note that $N/n$=$O(1)$ and $\max_{1\leq h\leq H, 1\leq j\leq M_h}(n M_h N_{hj}/r_h m_h N)$=$O(1)$ as $\nu\rightarrow\infty$ by Assumptions $1$ and $6$. Then, we have
\begin{align}\label{eq 74}
\begin{split}
&(n^2/N^4)\sum_{h=1}^H \sum_{j=1}^{M_h}\sum_{l=1}^{N_{hj}}  E(\overline{\alpha}_{hjl}\overline{\beta}_{hjl})^2=(n^2/N^4)\sum_{h=1}^H \sum_{j=1}^{M_h}\sum_{l=1}^{N_{hj}}  E\big((M_h N_{hj}  \xi_{hjl} /m_h r_h)-\\
&1\big)^4 \bar{A}_{hjl}^2 \bar{B}_{hjl}^2\leq(K_1/N^2)\sum_{h=1}^H\sum_{j=1}^{M_h}\sum_{l=1}^{N_{hj}}\bigg(\mathds{1}_{[t_1<Z_{hjl}^{\prime}\leq t_2]} +B_{z,N}(t_1,t_2)\bigg)\times\\
&\bigg(\mathds{1}_{[t_2<Z_{hjl}^{\prime}\leq t_3]}+ B_{z,N}(t_2,t_3)\bigg)\leq K_2 \big(B_{z,N}(t_1,t_3)\big)^2
\end{split}
\end{align}
\textit{a.s.} $[\textbf{P}]$ for all $\nu\geq 1$ and some constants $K_1, K_2>0$. Inequalities similar to \eqref{eq 74} can be shown to hold for the other terms on the right hand side of \eqref{eq 12}. Thus
\begin{equation}\label{eq s4}
(n^2/N^4)\sum_{h=1}^H  \sum_{j=1}^{M_h} E(\alpha_{hj}^{*}\beta_{hj}^{*})^2\leq K_3\big(B_{z,N}(t_1,t_3)\big)^2
\end{equation} 
\textit{a.s.} $[\textbf{P}]$ for any $0\leq t_1<t_2<t_3\leq 1$, $\nu\geq 1$ and some constant $K_3>0$. Inequalities similar to \eqref{eq s4} can also be shown to hold for the other terms on the right hand side of \eqref{eq s5}.  Therefore,
\begin{equation}\label{eq 13}
(n^2/N^4)\sum_{h=1}^H  E(\tilde{\alpha}_h^2\tilde{\beta}_h^2)\leq K_4\big(B_{z,N}(t_1,t_3)\big)^2
\end{equation} 
\textit{a.s.} $[\textbf{P}]$ for any $0\leq t_1<t_2<t_3\leq 1$, $\nu\geq 1$ and some constant $K_4>0$. Furthermore, inequalities similar to \eqref{eq 13} can be shown to hold for the other terms on the right hand side of \eqref{eq 7}. Consequently, $E\big[\big(\mathds{B}_n(t_1,t_2)\big)^2\big(\mathds{B}_n(t_2,t_3)\big)^2\big]\leq K_5\big(B_{z,N}(t_1,t_3)\big)^2$ \textit{a.s.} $[\textbf{P}]$ for any $0\leq t_1<t_2<t_3\leq 1$, $\nu\geq 1$ and some constant $K_5>0$. Moreover, it can be shown in the same way that $\overline{\lim}_{\nu\rightarrow\infty} E\big(\mathds{B}_n(u,t)\big)^4\leq K_6 (t-u)^2$ \textit{a.s.} $[\textbf{P}]$ for any $0\leq u<t\leq 1$ and some constant $K_6>0$ because $B_{z,N}(u,t)\rightarrow (t-u)$ as $\nu\rightarrow\infty$ \textit{a.s.} $[\textbf{P}]$ by Assumption $8$ and SLLN. This completes the proof of \textsl{(i)}.
\vspace{.2cm}

\textsl{(ii)}  It follows from Assumptions $1$ and $9$ that $N/n$=$O(1)$ and $\max_{1\leq h\leq H, 1\leq j\leq M_h}(n M_h N_{hj}$ $/r_h m_h N)$ =$O(1)$ as $\nu\rightarrow\infty$. Then the proof of the result in \textsl{(ii)} follows the same way as the proof of the result in \textsl{(i)}.
\end{proof}
Next, recall $\lambda_h$'s from Assumption $6$, $F_{y,H}$ and $Q_{y,H}$ from the paragraph preceding Assumption $10$ and $\tilde{F}_y$ from Assumption $11$. Let us define $\tilde{Q}_y(p)$=$\inf\{t\in\mathds{R}: \tilde{F}_{y}(t)\geq p\}$ for $0<p<1$. Also, recall $\tilde{F}_{y,H}$ and $\tilde{Q}_{y,H}$ from the paragraph containing \eqref{eq 26}--\eqref{eq 28} in the proof of \textsl{(i)} in Lemma \ref{lem s8}. Then, we state the following lemma.
\begin{lemma}\label{lem s10}
\textsl{(i)} Suppose that $H$ is fixed as $\nu\rightarrow\infty$, and Assumptions $6$ and $8$ in the main text hold. Then, for any $0<\alpha<\beta<1$, 
$$
\sup_{p\in[\alpha,\beta]}|Q_{y,H}(p)-\tilde{Q}_{y,H}(p)|\rightarrow0\text{ as }\nu\rightarrow \infty.
$$  \\
\textsl{(ii)} Further, suppose that $H\rightarrow\infty$ as $\nu\rightarrow\infty$, and Assumptions $8$, $9$ and $11$ in the main text hold. Then, for any $0<\alpha<\beta<1$, 
$$
\sup_{p\in[\alpha,\beta]}|Q_{y,H}(p)-\tilde{Q}_y(p)|\rightarrow0\text{ as }\nu\rightarrow \infty.
$$  
\end{lemma}
\begin{proof}
\textsl{(i)} Note that the inverse of $F_{y,H}|_{\mathcal{C}_y}$, say $F_{y,H}^{-1}:(0,1)\rightarrow \mathcal{C}_y$, exists and is differentiable by Assumption $8$, and $F_{y,H}^{-1}(p)$=$Q_{y,H}(p)$ for any $0<p<1$. Also, note that the inverse of $\tilde{F}_{y,H}|_{\mathcal{C}_y}$, say $\tilde{F}_{y,H}^{-1}:(0,1)\rightarrow \mathcal{C}_y$, exists and is differentiable, and $\tilde{F}_{y,H}^{-1}(p)$=$\tilde{Q}_{y,H}(p)$ for any $0<p<1$. Clearly, $\tilde{Q}_{y,H}$ is uniformly continuous on $[\alpha/2,(1+\beta)/2]$. Then, given any $\epsilon>0$ there exists a $\delta>0$ such that 
$$
|\tilde{Q}_{y,H}(p_1)-\tilde{Q}_{y,H}(p_2)|\leq \epsilon,\text{ whenever }|p_1-p_2|\leq\delta\text{ and }p_1,p_2\in [\alpha/2,(1+\beta)/2].
$$
 Now, it follows that 
$$
\sup_{p\in[\alpha,\beta]}|p-\tilde{F}_{y,H}(Q_{y,H}(p))|=\sup_{p\in[\alpha,\beta]}|F_{y,H}(Q_{y,H}(p))-\tilde{F}_{y,H}(Q_{y,H}(p))|\rightarrow 0
$$
 as $\nu\rightarrow\infty$. This further implies that 
$$
\sup_{p\in[\alpha,\beta]}|p-\tilde{F}_{y,H}(Q_{y,H}(p))|\leq \min\{\alpha/2,(1-\beta)/2,\delta\}
$$
 for all sufficiently large $\nu$. Therefore, 
$$
\alpha/2\leq \tilde{F}_{y,H}(Q_{y,H}(p))\leq (1+\beta)/2\text{ for all }p\in[\alpha,\beta]
$$
 and all sufficiently large $\nu$. Hence, 
$$
\sup_{p\in[\alpha,\beta]}|Q_{y,H}(p)-\tilde{Q}_{y,H}(p)|=\sup_{p\in[\alpha,\beta]}|\tilde{Q}_{y,H}(\tilde{F}_{y,H}(Q_{y,H}(p)))-\tilde{Q}_{y,H}(p)|\leq \epsilon
$$
 for all sufficiently large $\nu$. This completes the proof of \textsl{(i)}. The proof of \textsl{(ii)} follows exactly the same way as the proof of \textsl{(i)}.
\end{proof}
Next, we state the following lemma, which will be required to prove Theorem $5.4$.
\begin{lemma}\label{lem 9}
Fix $0<\alpha<\beta<1$. Suppose that the assumptions of Theorem $4.1$ hold,  $K(p_1,p_2)$ is as in ($8$) in Section $4$ of the main article, and $\hat{K}(p_1,p_2)$ is as in ($16$) in Section $5.1$ of the main text.  Then, the results in \eqref{eq A13} of Lemma \ref{lem 8} hold under stratified multistage cluster sampling design with SRSWOR.
\end{lemma}
\begin{proof}
The proof follows exactly the same way as the proof of \textsl{(i)} in Lemma \ref{lem 8} for the cases, when $H$ is fixed as $\nu\rightarrow\infty$ and $H\rightarrow\infty$ as $\nu\rightarrow\infty$.
\end{proof}
\section{Proofs of several propositions and theorems}\label{sec 5}
 Before we give the proof of Proposition $3.1$, suppose that $P(s,\omega)$ denotes a high entropy sampling design satisfying Assumption $2$, and $Q(s,\omega)$ denotes a rejective sampling design having inclusion probabilities equal to those of $P(s,\omega)$.
\begin{proof}[Proof of Proposition $3.1$] 
  We shall first show that the conclusion of Proposition $3.1$ holds for $Q(s,\omega)$.  Let us recall $F_{z,N}(t)$ and $\mathds{Z}_n(t)$ from the paragraph preceding Lemma \ref{lem 2}. Then, for $d(i,s)$=$(N\pi_i)^{-1}$, we have
\begin{align}\label{eq 38} 
\begin{split}
&\mathds{H}_n:=\big\{\sqrt{n}(\hat{F}_{z}(t)-t):t\in [0,1]\big\}=\mathds{U}_n+\sqrt{n/N}\mathds{W}_N\text{ with }\mathds{U}_n=\bigg\{\mathds{Z}_n(t)\bigg/\\
&\sum_{i\in s}(N\pi_i)^{-1}:t\in[0,1]\bigg\}\text{ and }\mathds{W}_N=\big\{\sqrt{N}(F_{z,N}(t)-t):t\in [0,1]\big\}. 
\end{split}
\end{align}
Next, recall $B_{z,N}(u,t)$ and $\mathds{B}_n(u,t)$ from the paragraph preceding Lemma \ref{lem 2}. Then, by Lemma \ref{lem 2}, we have $E\big[\big(\mathds{B}_n(t_1,t_2)\big)^2 \big(\mathds{B}_n(t_2,t_3)\big)^2\big]\leq K_1\big(B_{z,N}(t_1,t_3)\big)^2$ for all dyadic rational numbers $0\leq t_1<t_2<t_3\leq 1$ \textit{a.s.} $[\textbf{P}]$, where $K_1>0$ is some constant and $\nu\geq 1$. This further implies that
\begin{equation}\label{eq A2}
E\big[\big(\mathds{B}_n(t_1,t_2)\big)^2\big(\mathds{B}_n(t_2,t_3)\big)^2\big]\leq K_1\big(B_{z,N}(t_1,t_3)\big)^2\text{ for any }0\leq t_1<t_2<t_3\leq 1 
\end{equation}
 \textit{a.s.} $[\textbf{P}]$, where $\nu\geq 1$. Suppose that 
\begin{equation}\label{eq 40}
w_n(1/r)=\sup_{|t-u|\leq 1/r} |\mathds{Z}_n(t)-\mathds{Z}_n(u)|\text{ and } B=\{s\in\mathcal{S}:w_n(1/r)\geq\delta\} 
\end{equation} 
for $r$=$1,2,\ldots$. Here, $w_n(1/r)$ is the modulus of continuity of $\{\mathds{Z}_n(t):t\in[0,1]\}$. Then, by using \eqref{eq A2} above and imitating the proof of Lemma $2.3.1$ in \cite{shorack2009empirical} (see p. $49$), we obtain
\begin{align}\label{eq A3}
\begin{split}
&\sum_{s\in B}Q(s,\omega)\leq \delta^{-4}\bigg(\sum_{j=1}^r E\big\{\mathds{B}_n\big((j-1)/r,j/r\big)\big\}^4+\\
&K_2 B_{z,N}(0,1)\max_{1\leq j\leq r}B_{z,N}\big((j-1)/r,j/r\big)\bigg)
\end{split}
\end{align} 
\textit{a.s.} $[\textbf{P}]$ for any $\delta>0$, $r\geq 1,\nu\geq 1$ and some constant $K_2>0$. Next, it follows from \eqref{eq A3} that
\begin{equation}\label{eq 41}
\overline{\lim}_{\nu\rightarrow\infty} E\big\{\mathds{B}_n\big((j-1)/r,j/r\big)\big\}^4\leq K_3 (1/r)^2
\end{equation}
\textit{a.s.} $[\textbf{P}]$ for any $j$=$1,\ldots,r$, $r\geq 1$ and some constant $K_3>0$ by Lemma \ref{lem 2}. Now, note that $\{Z_i\}_{i=1}^N$ are i.i.d. uniform random variables supported on $(0,1)$ since $F_y$ is continuous by Assumption $3$. Then, $B_{z,N}(u,t)\rightarrow t-u$ \textit{a.s.} $[\textbf{P}]$ by SLLN. Therefore, in view of \eqref{eq A3} and \eqref{eq 41}, we have
\begin{equation}\label{eq A1}
\overline{\lim}_{\nu\rightarrow\infty}\sum_{s\in B}Q(s,\omega)\leq \delta^{-4}(K_2/r+K_3/r) \text{ \textit{a.s.} }[\textbf{P}]
\end{equation}
for any $\delta>0$ and $r\geq 1$. Since, $\sum_{s\in B}Q(s,\omega)$ is bounded, by taking expectation of left hand side in \eqref{eq A1} w.r.t. $\textbf{P}$ and applying an extended version of Fatou's lemma, we obtain that 
\begin{equation}
\overline{\lim}_{\nu\rightarrow\infty}\textbf{P}^*\{w_n(1/r)\geq \delta\}\leq \delta^{-4}(K_2/r+K_3/r) 
\end{equation}
for any $\delta>0$ and $r\geq 1$. This further implies that $\overline{\lim}_{\nu\rightarrow\infty}\textbf{P}^*\{w_n(1/r)\geq \delta\}\rightarrow0$ for any $\delta$ as $r\rightarrow \infty$. Then by Theorem $2.3.2$ in \cite{shorack2009empirical} (see p. $46$), $\{\mathds{Z}_n:\nu\geq 1\}$ is weakly/relatively compact in $(\tilde{D}[0,1],\tilde{\mathcal{D}})$ with respect to the sup norm metric under $\textbf{P}^*$. In other words, given any subsequence $\{\nu_k\}$, there exists a further subsequence $\{\nu_{k_l}\}$ such that $E_{\textbf{P}^{*}}(f(\mathds{Z}_n))\rightarrow E(f(\mathds{Z}))$ along the subsequence $\{\nu_{k_l}\}$ for any bounded continuous (with respect to the sup norm metric) and $\tilde{\mathcal{D}}$-measurable function $f$, and for some random function $\mathds{Z}$ in $(\tilde{D}[0,1],\tilde{\mathcal{D}})$ (see p. $44$ in  \cite{shorack2009empirical}).

Now, under $Q(s,\omega)$, $\textbf{c}(\mathds{Z}_n(t_1),\ldots,\mathds{Z}_n(t_k))^T\xrightarrow{\mathcal{L}}N(0,\textbf{c}\Gamma_3 \textbf{c}^T)$ as $\nu\rightarrow\infty$ \textit{a.s.} $[\textbf{P}]$ by Lemma \ref{lem 1}, where $k\geq 1$, $t_1,\ldots,t_k\in(0,1)$, $\textbf{c}\in\mathds{R}^k$, $\textbf{c}\neq 0$ and $\Gamma_3$ is a p.d. matrix. Moreover, $\Gamma_3$=$\lim_{\nu\rightarrow\infty}n N^{-2} \sum_{i=1}^{N}(\textbf{Z}_i-\textbf{T}_{z}\pi_i)^T(\textbf{Z}_i-\textbf{T}_z\pi_i)(\pi_i^{-1}-1)$ \textit{a.s.} $[\textbf{P}]$, where $\textbf{Z}_i$=$(\mathds{1}_{[Z_i\leq t_1]}-F_{z,N}(t_1),\ldots,\mathds{1}_{[Z_i\leq t_k]}-F_{z,N}(t_k))$ and $\textbf{T}_z$=$\sum_{i=1}^N \textbf{Z}_i(1-\pi_i)/\sum_{i=1}^N \pi_i(1-\pi_i)$. Note that $\sum_{i=1}^N ||\textbf{Z}_i||^2/N$ is bounded. Also, note that Assumption $2$-\textsl{(ii)} holds under $Q(s,\omega)$ because $P(s,\omega)$ and $Q(s,\omega)$ have same inclusion probabilities, and Assumption $2$-\textsl{(ii)} holds under $P(s,\omega)$. Then, we have 
\begin{equation}\label{eq 47}
\Gamma_3=\lim_{\nu\rightarrow\infty}E_{\textbf{P}}(n N^{-2}\sum_{i=1}^{N}(\textbf{Z}_i-\textbf{T}_{z}\pi_i)^T(\textbf{Z}_i-\textbf{T}_z\pi_i)(\pi_i^{-1}-1))
\end{equation}
 by DCT. Further, it follows from DCT that under $\textbf{P}^*$, 
\begin{align}
\begin{split}
&\textbf{c}(\mathds{Z}_n(t_1),\ldots,\mathds{Z}_n(t_k))^T\xrightarrow{\mathcal{L}}N(0,\textbf{c}\Gamma_3 \textbf{c}^T)\text{ for any }\textbf{c}\neq 0,\text{ and hence }\\
&(\mathds{Z}_n(t_1),\ldots,\mathds{Z}_n(t_k))\xrightarrow{\mathcal{L}}N(0,\Gamma_3)
\end{split}
\end{align}
 as $\nu\rightarrow\infty$. Relative compactness and weak convergence of finite dimensional distributions of $\{\mathds{Z}_n:\nu\geq 1\}$ imply that $\mathds{Z}_n\xrightarrow{\mathcal{L}}\mathds{Z}$ as $\nu\rightarrow\infty$ in $(\tilde{D}[0,1],\tilde{\mathcal{D}})$ with respect to the sup norm metric, for $Q(s,\omega)$ under $\textbf{P}^*$, where $\mathds{Z}$ has mean $0$ and covariance kernel  
\begin{align} 
\begin{split}
& \lim_{\nu\rightarrow\infty}E_{\textbf{P}}\bigg(n N^{-2}\sum_{i=1}^N (\mathds{1}_{[Z_i\leq t_1]}-F_{z,N}(t_1)-R(t_1)\pi_i)\times\\
&(\mathds{1}_{[Z_i\leq t_2]}-F_{z,N}(t_2)-R(t_2)\pi_i)(\pi_i^{-1}-1)\bigg),
\end{split}
 \end{align} 
with $R(t)$=$\sum_{i=1}^N (\mathds{1}_{[Z_i\leq t]}-F_{z,N}(t))(1-\pi_i)/\sum_{i=1}^N \pi_i(1-\pi_i)$. Moreover, it follows from Theorem $2.3.2$ in \cite{shorack2009empirical} that $\mathds{Z}$ has almost sure continuous paths. Next, note that $\sum_{i=1}^N\pi_i(1-\pi_i)\rightarrow\infty$ as $\nu\rightarrow\infty$ under $Q(s,\omega)$  \textit{a.s.} $[\textbf{P}]$ since $Q(s,\omega)$ satisfies Assumption $2$-\textsl{(ii)}, and Assumption $1$ holds. Then, it can be shown using Theorem $6.1$ in \cite{MR0178555} that under $Q(s,\omega)$, $var(\sum_{i\in s}(N\pi_i)^{-1})\rightarrow 0$ as $\nu\rightarrow\infty$ \textit{a.s.} $[\textbf{P}]$. Consequently, $\sum_{i\in s}(N\pi_i)^{-1}\xrightarrow{p} 1$ as $\nu\rightarrow\infty$ under $\textbf{P}^*$. Then, under $\textbf{P}^*$, $\mathds{U}_n$=$\mathds{Z}_n/\sum_{i\in s}(N\pi_i)^{-1}\xrightarrow{\mathcal{L}}\mathds{U}\overset{\mathcal{L}}{=}\mathds{Z}$ as $\nu\rightarrow\infty$ in $(\tilde{D}[0,1],\tilde{\mathcal{D}})$ with respect to the sup norm metric, for $Q(s,\omega)$. This further implies that under $\textbf{P}^*$, $\mathds{U}_n\xrightarrow{\mathcal{L}}\mathds{Z}$ as $\nu\rightarrow\infty$ in $(\tilde{D}[0,1],\tilde{\mathcal{D}})$ with respect to the Skorohod metric, for $Q(s,\omega)$.
\par

Now, it follows from Donsker theorem that under $\textbf{P}$, $\mathds{W}_N\xrightarrow{\mathcal{L}}\mathds{W}$ as $\nu\rightarrow\infty$ in $(\tilde{D}[0,1],\tilde{\mathcal{D}})$ with respect to the Skorohod metric, where $\mathds{W}$ is the standard Brownian bridge in $\tilde{D}[0,1]$ and has almost sure continuous paths. Hence, under $\textbf{P}^*$, both $\mathds{U}_n$ and $\mathds{W}_N$ are tight in $(\tilde{D}[0,1],\tilde{\mathcal{D}})$ with respect to the Skorohod metric by Theorem $5.2$ in \cite{billingsley2013convergence}. Then, it follows from Lemma B.$2$ in \cite{boistard2017} that under $\textbf{P}^*$, $\mathds{H}_n$=$\mathds{U}_n+\sqrt{n/N}\mathds{W}_N$ is tight in $(\tilde{D}[0,1],\tilde{\mathcal{D}})$ with respect to the Skorohod metric, for $d(i,s)$=$(N\pi_i)^{-1}$ and $Q(s,\omega)$ since Assumption $1$ holds. It also follows from $(iii)$ of Theorem $5.1$ in \cite{rubin2005two} that
\begin{align}\label{eq 45}
\begin{split}
&\textbf{c}\big(\mathds{U}_n(t_1)+\sqrt{n/N}\mathds{W}_N(t_1),\ldots, \mathds{U}_n(t_k)+\sqrt{n/N}\mathds{W}_N(t_k)\big)^T \xrightarrow{\mathcal{L}}\\
&N(0,\textbf{c}(\Gamma_3+\lambda\Gamma_4)\textbf{c}^T)
\end{split}
\end{align} 
 as $\nu\rightarrow\infty$ under $\textbf{P}^*$ for $k\geq 1$ and $\textbf{c}\neq 0$ because $\textbf{c}(\mathds{U}_n(t_1),\ldots, \mathds{U}_n(t_k))^T\xrightarrow{\mathcal{L}}N(0,\textbf{c}\Gamma_3\textbf{c}^T)$ as $\nu\rightarrow\infty$ under $Q(s,\omega)$ \textit{a.s.} $[\textbf{P}]$, and $\sqrt{n/N}\textbf{c}(\mathds{W}_N(t_1),\ldots, \mathds{W}_N(t_k))^T\xrightarrow{\mathcal{L}}N(0,\lambda \textbf{c}\Gamma_4\textbf{c}^T)$ as $\nu\rightarrow\infty$ under $\textbf{P}$. Here, $\Gamma_4$ is a $k\times k$ matrix such that 
\begin{equation}\label{eq 53}
((\Gamma_4))_{ij}=t_i\wedge t_j-t_it_j\text{ for }1\leq i<j\leq k.
\end{equation}  
Therefore, under $\textbf{P}^*$, $\mathds{H}_n\xrightarrow{\mathcal{L}}\mathds{H}$ in $(\tilde{D}[0,1],\tilde{\mathcal{D}})$ with respect to the Skorohod metric, for $d(i,s)$=$(N\pi_i)^{-1}$ and $Q(s,\omega)$, where $\mathds{H}$ is a mean $0$ Gaussian process with covariance kernel 
\begin{align}
\begin{split}
&\lim_{\nu\rightarrow\infty}E_{\textbf{P}}\big(n N^{-2}\sum_{i=1}^N (\mathds{1}_{[Z_i\leq t_1]}-F_{z,N}(t_1)-R(t_1)\pi_i)\times\\
&(\mathds{1}_{[Z_i\leq t_2]}-F_{z,N}(t_2)-R(t_2)\pi_i)(\pi_i^{-1}-1)\big)
+\lambda(t_1\wedge  t_2-t_1 t_2)\text{ for }t_1,t_2\in [0,1].
\end{split}
\end{align}
We can choose independent random functions, $\mathds{H}_1,\mathds{H}_2\in \tilde{D}[0,1]$ defined on some probability space such that $\mathds{H}_1\overset{\mathcal{L}}{=}\mathds{Z}$ and $\mathds{H}_2\overset{\mathcal{L}}{=}\mathds{W}$. $\mathds{H}_1$ and $\mathds{H}_2$ have almost sure continuous paths because $\mathds{Z}$ and $\mathds{W}$ have almost sure continuous paths. Hence, $\mathds{H}_1+\sqrt{\lambda}\mathds{H}_2$ has almost sure continuous paths. Next, note that $\mathds{H}_1$ and $\mathds{H}_2$ are mean $0$ Gaussian processes because $\mathds{Z}$ and $\mathds{W}$ are mean $0$ Gaussian processes. Thus $\mathds{H}_1+\sqrt{\lambda}\mathds{H}_2$ is a mean $0$ Gaussian process. Also, note that the covariance kernel of $\mathds{H}$ is the sum of covariance kernels of $\mathds{Z}$ and $\sqrt{\lambda}\mathds{W}$. Thus the covariance kernel of $\mathds{H}_1+\sqrt{\lambda}\mathds{H}_2$ is the same as that of $\mathds{H}$. Therefore, $\mathds{H}_1+\sqrt{\lambda}\mathds{H}_2\overset{\mathcal{L}}{=}\mathds{H}$. Hence, $\mathds{H}$ has almost sure continuous paths. Then, under $\textbf{P}^*$, $\mathds{H}_n\xrightarrow{\mathcal{L}}\mathds{H}$ in $(\tilde{D}[0,1],\tilde{\mathcal{D}})$ with respect to the sup norm metric, for $d(i,s)$=$(N\pi_i)^{-1}$ and $Q(s,\omega)$ by Skorohod representation theorem.
\par

Finally, we shall show that the conclusion of Proposition $3.1$ holds for the high entropy sampling design $P(s,\omega)$, which satisfies Assumption $2$. Note that for $d(i,s)$=$(N\pi_i)^{-1}$, $E_{\textbf{P}^{*}}(f(\mathds{H}_n))$= $E_{\textbf{P}}(\sum_{s\in\mathcal{S}}f(\mathds{H}_n) Q(s,\omega))\rightarrow \int f dP_{\mathds{H}}$ as $\nu\rightarrow\infty$ given any bounded continuous (with respect to the sup norm metric) $\tilde{\mathcal{D}}$-measurable function $f$, where $P_{\mathds{H}}$ is the probability distribution corresponding to $\mathds{H}$. Then, it follows from Lemmas $2$ and $3$ in \cite{MR1624693} that 
\begin{align}
\begin{split}
&\bigg|\sum_{s\in\mathcal{S}}f(\mathds{H}_n) (P(s,\omega)-Q(s,\omega))\bigg|\leq K_{2}\sum_{s\in \mathcal{S}}|P(s,\omega)-Q(s,\omega)|\leq K_{2}(2D(P||Q))^{1/2}\\
&\leq K_2(2D(P||R))^{1/2},
\end{split}
\end{align}
 for some constant $K_2>0$, where $R(s,\omega)$ is a rejective sampling design such that $D(P||R)\rightarrow 0$ as $\nu\rightarrow\infty$ \textit{a.s.} $[\textbf{P}]$. This implies that $E_{\textbf{P}}(\sum_{s\in\mathcal{S}}f(\mathds{H}_n) P(s,\omega))\rightarrow \int f dP_{\mathds{H}}$ as $\nu\rightarrow\infty$ for $d(i,s)$=$(N\pi_i)^{-1}$ by DCT, and hence, the conclusion of Proposition $3.1$ holds for the high entropy sampling design $P(s,\omega)$. 
\end{proof} 

\begin{proof}[Proof of Proposition $3.2$]
Recall $\tilde{\mathds{Z}}_n(t)$ from the paragraph preceding Lemma \ref{lem 5}. Then, using Lemmas \ref{lem 4} and \ref{lem 5}, it can be shown in the same way as in the first two paragraphs of the proof of Proposition $3.1$ that under $\textbf{P}^*$, $\tilde{\mathds{Z}}_n\xrightarrow{\mathcal{L}}\tilde{\mathds{Z}}$ as $\nu\rightarrow\infty$ in  $(\tilde{D}[0,1],\tilde{\mathcal{D}})$  with respect to the sup norm metric, for RHC sampling design, where $\tilde{\mathds{Z}}$ is a mean $0$ Gaussian process  in $\tilde{D}[0,1]$  with almost sure continuous paths. Moreover, the covariance kernel of $\tilde{\mathds{Z}}$ is 
$$
\lim_{\nu\rightarrow\infty}E_{\textbf{P}}\bigg(n \gamma(\overline{X}/N)\sum_{i=1}^N (\mathds{1}_{[Z_i\leq t_1]}-F_{z,N}(t_1))(\mathds{1}_{[Z_i\leq t_2]}-F_{z,N}(t_2))X_i^{-1}\bigg).
$$
 It can be shown that under RHC sampling design, 
$$
var\bigg(\sum_{i\in s}A_i/NX_i\bigg)=\gamma\sum_{i=1}^N(X_i-\overline{X})^2/NX_i\overline{X}=\gamma\bigg(\overline{X}\sum_{i=1}^N X_i/N-1\bigg) \rightarrow 0 
$$
 as $\nu\rightarrow\infty$ \textit{a.s.} $[\textbf{P}]$ since $n\gamma\rightarrow c>0$ by  Lemma S$1$ in \cite{deychaudhuri2023},  Assumption $4$ holds, and $\{N_r\}_{r=1}^n$ are as in ($6$) of the main text. Consequently, under $\textbf{P}^*$, $\sum_{i\in s}A_i/NX_i\xrightarrow{p} 1$ as $\nu\rightarrow\infty$. Therefore, under $\textbf{P}^*$, 
$$
\tilde{\mathds{U}}_n=\tilde{\mathds{Z}}_n\bigg/\sum_{i\in s}(A_i/NX_i)\xrightarrow{\mathcal{L}}\tilde{\mathds{U}}\overset{\mathcal{L}}{=}\tilde{\mathds{Z}}
$$ 
 as $\nu\rightarrow\infty$ in  $(\tilde{D}[0,1],\tilde{\mathcal{D}})$  with respect to the sup norm metric, for RHC sampling design. Next, note that 
$$
\mathds{H}_n=\big\{\sqrt{n}(\hat{F}_{z}(t)-t):t\in [0,1]\big\}=\tilde{\mathds{U}}_n+\sqrt{n/N}\mathds{W}_N,
$$
for $d(i,s)$=$A_i/NX_i$,  where $\mathds{W}_N$=$\{\sqrt{N}(F_{z,N}(t)-t):t\in [0,1]\}$. Also, note that under $\textbf{P}$, $\mathds{W}_N\xrightarrow{\mathcal{L}}\mathds{W}$ as $\nu\rightarrow\infty$ in  $(\tilde{D}[0,1],\tilde{\mathcal{D}})$  with respect to the Skorohod metric by Donsker theorem, where $\mathds{W}$ is the standard Brownian bridge. Therefore, using the same arguments as in the $3^{rd}$ paragraph of the proof of Proposition $3.1$, we can show that under $\textbf{P}^*$,  $\mathds{H}_n\xrightarrow{\mathcal{L}}  \mathds{H}$  as $\nu\rightarrow\infty$ in  $(\tilde{D}[0,1],\tilde{\mathcal{D}})$  with respect to the sup norm metric, for $d(i,s)$=$A_i/NX_i$ and RHC sampling design, where $\mathds{H}$ is a mean $0$ Gaussian process with covariance kernel 
\begin{align}\label{eq 18}
\begin{split}
&\lim_{\nu\rightarrow\infty}E_{\textbf{P}}\big(n \gamma(\overline{X}/N)\sum_{i=1}^N (\mathds{1}_{[Z_i\leq t_1]}-F_{z,N}(t_1))(\mathds{1}_{[Z_i\leq t_2]}-F_{z,N}(t_2))X_i^{-1}\big)+\\
&\lambda(t_1\wedge  t_2-t_1 t_2),
\end{split}
\end{align}
 for $t_1,t_2\in [0,1]$. Also, $\mathds{H}$ has almost sure continuous paths.  It can be shown using Lemma S$1$ in \cite{deychaudhuri2023}, Assumption $4$ and DCT that the expression in \eqref{eq 18} becomes
\begin{align*}
\begin{split}
&c E_{\textbf{P}}(X_i) E_{\textbf{P}}\bigg(\big(\mathds{1}_{[Z_i\leq t_1]}-\textbf{P}(Z_i\leq t_1)\big)\big(\mathds{1}_{[Z_i\leq t_2]}-\textbf{P}(Z_i\leq t_2)\big)X_i^{-1}\bigg)+\\
& \lambda(t_1\wedge  t_2-t_1 t_2),
\end{split}
\end{align*}
where $c$=$\lim_{\nu\rightarrow\infty}n\gamma$. 
\end{proof}
\begin{proof}[Proof of Theorem $3.2$] 
The proof follows in view of Proposition $3.2$ in the same way as the proof of Theorem $3.1$ follows in view of Proposition $3.1$.
\end{proof}
\begin{proof}[Proof of Proposition $4.1$]
 Let us denote the stratified multistage cluster sampling design by $P(s,\omega)$. 

\textsl{(i)} Recall $F_{y,H}$ from the paragraph preceding Assumption $10$ in the main text, and consider $\{Z_i\}_{i=1}^N$ as in  ($4$)  in the main text with $F_{y,H}$ replacing $F_y$. Also, recall $F_{z,N}(t)$ and $\mathds{Z}_n(t)$ from  \eqref{eq 37}. Note that $F_{z,N}(t)\rightarrow t$ as $\nu\rightarrow\infty$ \textit{a.s.} $[\textbf{P}]$ for any $t\in [0,1]$ by Assumption $8$ and SLLN. Therefore, using Lemmas \ref{lem s8} and \ref{lem s9}, one can show in the same way as in the first two paragraphs of the proof of  Proposition $3.1$ that under $\textbf{P}^*$,  
$$
\mathds{Z}_n\xrightarrow{\mathcal{L}}\mathds{Z}\text{ as }\nu\rightarrow\infty
$$
in $(\tilde{D}[0,1],\tilde{\mathcal{D}})$  with respect to the sup norm metric, for $P(s,\omega)$. Here, $\mathds{Z}$ is a mean $0$ Gaussian process in $\tilde{D}[0,1]$ with covariance kernel 
\begin{align} 
\begin{split}
&K_1(t_1,t_2)=\lambda\sum_{h=1}^H \Lambda_h(\Lambda_h/\lambda\lambda_h-1)E_{\textbf{P}}\big(\mathds{1}_{[Y_{hjl}^{\prime}\leq \tilde{Q}_{y,H}(t_1)]}-\textbf{P}(Y_{hjl}^{\prime}\leq \tilde{Q}_{y,H}(t_1))\big)\times\\
&\big(\mathds{1}_{[Y_{hjl}^{\prime}\leq \tilde{Q}_{y,H}(t_2)]}-\textbf{P}(Y_{hjl}^{\prime}\leq \tilde{Q}_{y,H}(t_2))\big)
\end{split}
\end{align}
 for $t_1,t_2\in [0,1]$. Here, $\tilde{Q}_{y,H}(p)$=$\inf\{t\in\mathds{R}:\tilde{F}_{y,H}(t)\geq p\}$, $\tilde{F}_{y,H}(t)$=$\sum_{h=1}^H \Lambda_h F_{y,h}(t)$, and $\lambda_h$'s and $\Lambda_h$'s are as in Assumption $6$. Moreover, $\mathds{Z}$ has almost sure continuous paths. Next, it can be shown using Assumption $6$ that $var(\sum_{i\in s}(N\pi_i)^{-1})$= $o(1)$, and hence $\sum_{i\in s}(N\pi_i)^{-1}\xrightarrow{p} 1$ as $\nu\rightarrow\infty$ under $P(s,\omega)$ for any given $\omega\in\Omega$. Here, $\pi_i$=$m_h r_h/ M_h N_{hj}$ when the $i^{th}$ population unit belongs to the $j^{th}$ cluster of stratum $h$. Therefore, it follows from DCT that under $\textbf{P}^*$, $\sum_{i\in s}(N\pi_i)^{-1}\xrightarrow{p} 1$, and hence under $\textbf{P}^*$, 
$$
\mathds{U}_n=\mathds{Z}_n\bigg/\sum_{i\in s}(N\pi_i)^{-1}\xrightarrow{\mathcal{L}}\mathds{U}\overset{\mathcal{L}}{=}\mathds{Z}
$$
as $\nu\rightarrow\infty$ in $(\tilde{D}[0,1],\tilde{\mathcal{D}})$  with respect to the sup norm metric, for the sampling design $P(s,\omega)$. 
\par
\vspace{.1cm}

Next, recall $\mathds{W}_N$ from the $1^{st}$ paragraph in the proof of  Proposition $3.1$.  Then, using assumptions Assumptions $6$ and $8$ in the main text, and Lemma \ref{lem s10} in this supplement, it can be shown that
 
\begin{align}\label{eq 9}
\begin{split}
&cov_{\textbf{P}}(\mathds{W}_N(t_1),\mathds{W}_N(t_2))=\sum_{h=1}^H(N_h/N) E_{\textbf{P}}\bigg(\mathds{1}_{[Y_{hjl}^{\prime}\leq Q_{y,H}(t_1)]}-\\
&\textbf{P}(Y_{hjl}^{\prime}\leq Q_{y,H}(t_1))\bigg)\bigg(\mathds{1}_{[Y_{hjl}^{\prime}\leq Q_{y,H}(t_2)]}-\textbf{P}(Y_{hjl}^{\prime}\leq Q_{y,H}(t_2))\bigg)\rightarrow \\
&\sum_{h=1}^H \Lambda_h E_{\textbf{P}}\bigg(\mathds{1}_{[Y_{hjl}^{\prime}\leq \tilde{Q}_{y,H}(t_1)]}-\textbf{P}(Y_{hjl}^{\prime}\leq  \tilde{Q}_{y,H}(t_1))\bigg)\bigg(\mathds{1}_{[Y_{hjl}^{\prime}\leq \tilde{Q}_{y,H}(t_2)]}-\\
&\textbf{P}(Y_{hjl}^{\prime} \leq\tilde{Q}_{y,H}( t_2))\bigg)=K_2(t_1,t_2)\text{ (say) }
\end{split}
\end{align}
 as $\nu\rightarrow\infty$ for any $t_1,t_2\in[0,1]$. Then, under $\textbf{P}$, $\mathds{W}_N\xrightarrow{\mathcal{L}}\mathds{W}$ as $\nu\rightarrow\infty$  in $(\tilde{D}[0,1],\tilde{\mathcal{D}})$  with respect to the Skorohod metric by \eqref{eq 9} above and Theorem $3.3.1$ in \cite{shorack2009empirical} (see p. $109$), where $\mathds{W}$ is a mean $0$ Gaussian process in $\tilde{D}[0,1]$ with covariance kernel $K_2(t_1,t_2)$. Also, $\mathds{W}$ has almost sure continuous paths. Therefore, using similar arguments as in the proof of  Proposition $3.1$,  we can say that under $\textbf{P}^*$, 
$$
\mathds{H}_n=\mathds{U}_n+\sqrt{n/N}\mathds{W}_N=\{\sqrt{n}(\hat{F}_{z}(t)-t):t\in [0,1]\}\xrightarrow{\mathcal{L}}\mathds{H}
$$
 as $\nu\rightarrow\infty$  in $(\tilde{D}[0,1],\tilde{\mathcal{D}})$   with respect to the sup norm metric, for $d(i,s)$=$(N\pi_i)^{-1}$ and $P(s,\omega)$, where $\mathds{H}$ is a mean $0$ Gaussian process in $\tilde{D}[0,1]$ with covariance kernel 
$$
K_1(t_1,t_2)+\lambda K_2(t_1,t_2).
$$
 Moreover, $\mathds{H}$ has almost sure continuous paths. This completes the proof of \textsl{(i)}.
\par
\vspace{.1cm}

\textsl{(ii)} Using Hoeffding's inequality (see \cite{serfling2009approximation}), and Assumptions $1$, $8$ and $9$, it can be shown that $F_{z,N}(t)\rightarrow t$ as $\nu\rightarrow\infty$ \textit{a.s.} $[\textbf{P}]$ for any $t\in [0,1]$. Therefore, using Lemmas \ref{lem s8} and \ref{lem s9}, and the Assumption $9$, one can show in the same way as in \textsl{(i)} that under $\textbf{P}^*$,  
$$
\mathds{U}_n\xrightarrow{\mathcal{L}}\mathds{Z}\text{ as }\nu\rightarrow\infty
$$
 in $(\tilde{D}[0,1],\tilde{\mathcal{D}})$  with respect to the sup norm metric, for $P(s,\omega)$, where $\mathds{Z}$ is a mean $0$ Gaussian process in $\tilde{D}[0,1]$ with covariance kernel 
\begin{align*}
\begin{split}
&K_1(t_1,t_2)=\lim_{\nu\rightarrow\infty}\lambda\sum_{h=1}^H N_h(N_h-n_h)E_{\textbf{P}}\big(\mathds{1}_{[Y_{hjl}^{\prime}\leq Q_{y,H}(t_1)]}-\textbf{P}(Y_{hjl}^{\prime}\leq Q_{y,H}(t_1))\big)\times\\
&\big(\mathds{1}_{[ Y_{hjl}^{\prime}\leq Q_{y,H}(t_2)]}-\textbf{P}( Y_{hjl}^{\prime}\leq Q_{y,H}(t_2))\big)/n_h N, 
\end{split}
\end{align*}
 for $t_1,t_2\in [0,1]$. Moreover, $\mathds{Z}$ has almost sure continuous paths. Next, given any $t_1,t_2\in[0,1]$, 
\begin{align*}
\begin{split}
&cov_{\textbf{P}}(\mathds{W}_N(t_1),\mathds{W}_N(t_2))=\sum_{h=1}^H (N_h/N)E_{\textbf{P}}\big(\mathds{1}_{[Y_{hjl}^{\prime}\leq Q_{y,H}(t_1)]}-\textbf{P}(Y_{hjl}^{\prime}\leq Q_{y,H}(t_1))\big)\times\\
&\big(\mathds{1}_{[Y_{hjl}^{\prime}\leq Q_{y,H}(t_2)]}-\textbf{P}(Y_{hjl}^{\prime}\leq Q_{y,H}(t_2))\big)\rightarrow K_2(t_1, t_2)
\end{split}
\end{align*}
  as $\nu\rightarrow\infty$ for some covariance kernel $K_2(t_1,t_2)$ by Assumption $10$. Then, under $\textbf{P}$, 
$$
\mathds{W}_N\xrightarrow{\mathcal{L}}\mathds{W}\text{ as }\nu\rightarrow\infty
$$
 in $(\tilde{D}[0,1],\tilde{\mathcal{D}})$   with respect to the Skorohod metric by Theorem $3.3.1$ in \cite{shorack2009empirical} (see p. $109$), where $\mathds{W}$ is a mean $0$ Gaussian process in $\tilde{D}[0,1]$ with covariance kernel $K_2(t_1,t_2)$. Also, $\mathds{W}$ has almost sure continuous paths. Therefore, using similar arguments as in the proof of Proposition $3.1$, we can say that under $\textbf{P}^*$, 
 $$
 \mathds{H}_n=\mathds{U}_n+\sqrt{n/N}\mathds{W}_N=\{\sqrt{n}(\hat{F}_{z}(t)-t):t\in [0,1]\}\xrightarrow{\mathcal{L}}\mathds{H}\text{ as }\nu\rightarrow\infty
 $$ 
 in $(\tilde{D}[0,1],\tilde{\mathcal{D}})$   with respect to the sup norm metric, for $d(i,s)$=$(N\pi_i)^{-1}$ and $P(s,\omega)$, where $\mathds{H}$ is a mean $0$ Gaussian process in $\tilde{D}[0,1]$ with almost sure continuous paths and p.d. covariance kernel 
 $$
 K_1(t_1,t_2)+\lambda K_2(t_1,t_2).
 $$
  This completes the proof of \textsl{(ii)}.
\end{proof}
\begin{proof}[Proof of Theorem $4.2$] 
The proof follows in view of Proposition $4.1$ in the same way as the proof of Theorem $3.1$ follows in view of Proposition $3.1$.
\end{proof}
\begin{proof}[Proof of Theorem $5.1$]  By conclusions of Theorems $3.1$ and $3.2$, and continuous mapping theorem, we have 
\begin{equation}
\int_{\alpha}^{\beta}\sqrt{n}(G(p)-Q_{y,N}(p))J(p)dp\xrightarrow{\mathcal{L}}\int_{\alpha}^{\beta}\mathds{Q}(p)J(p)dp\text{ as }\nu\rightarrow\infty
\end{equation}
 for high entropy and RHC sampling deigns under $\textbf{P}^*$. Note that $\mathds{Q}(p)J(p)$ is Riemann integrable on $[\alpha,\beta]$ implying $Z$=$\int_{\alpha}^{\beta}\mathds{Q}(p)J(p)dp$=$\lim_{m\rightarrow\infty}m^{-1}\sum_{i=0}^{m-1}$ $\mathds{Q}\big(\alpha+i(\beta-\alpha)/m\big)J\big(\alpha+i(\beta-\alpha)/m\big)$ under the aforementioned sampling designs. By DCT, we have 
\begin{align}
\begin{split} 
 &E(\exp(itZ))=\lim_{m\rightarrow\infty} \exp\bigg\{-m^{-2}\sum_{i=0}^{m-1}\sum_{j=0}^{m-1}K\big(\alpha+i(\beta-\alpha)/m, \alpha+j(\beta-\alpha)/m\big)\\
 &\times J\big(\alpha+i(\beta-\alpha)/m\big) J\big(\alpha+j(\beta-\alpha)/m\big)(t^2/2)\bigg\}
 \end{split}
 \end{align}
 since $\mathds{Q}$ is a mean $0$ Gaussian process in $D[\alpha,\beta]$ with covariance kernel $K(p_1,p_2)$. Note that $K(p_1,p_2)$ in the case of any high entropy sampling design (see  ($5$) in the main text ) is continuous on $[\alpha,\beta]\times[\alpha,\beta]$ by the assumption of this theorem, whereas $K(p_1,p_2)$ in the case of RHC sampling design (see ($7$) in the main text\color{black}) is continuous on $[\alpha,\beta]\times[\alpha,\beta]$ by Assumption $3$. Then, $E(\exp(itZ))$=$\exp\big(-t^2\int_{\alpha}^{\beta}\int_{\alpha}^{\beta}K(p_1,p_2)J(p_1)J(p_2)dp_1 dp_2/2\big)$ under the above sampling designs since $K(p_1,p_2)$ is continuous on $[\alpha,\beta]\times[\alpha,\beta]$, and hence Riemann integrable on $[\alpha,\beta]\times[\alpha,\beta]$. Therefore, 
\begin{equation}
\int_{\alpha}^{\beta}\mathds{Q}(p)J(p)dp\sim N(0,\sigma^2_1),\text{ where }\sigma^2_1=\int_{\alpha}^{\beta}\int_{\alpha}^{\beta}K(p_1,p_2)J(p_1)J(p_2)dp_1dp_2. 
\end{equation}
 Hence, under $\textbf{P}^*$, $\int_{\alpha}^{\beta}\sqrt{n}(G(p)-Q_{y,N}(p))J(p)dp\xrightarrow{\mathcal{L}}N(0,\sigma^2_1)$ as $\nu\rightarrow\infty$ for high entropy and RHC sampling deigns. 
\par
\vspace{.1cm}

Next, for any $k\geq 1$ and $p_1,\ldots,p_k\in[\alpha,\beta]$, we have 
\begin{equation}
\sqrt{n}(f(G(p_1),\ldots,G(p_k))-f(Q_{y,N}(p_1),\ldots,Q_{y,N}(p_k)))=a_N \sqrt{n}T_n+ \sqrt{n}\mathcal{\epsilon}(T_n)
\end{equation}
 by delta method, where $a_N$=$\nabla f(Q_{y,N}$ $(p_1),\ldots,Q_{y,N}(p_k))$, $T_n$=$G(p_1)-Q_{y,N}(p_1),\ldots,G(p_k)-Q_{y,N}(p_k)$, and $\mathcal{\epsilon}(T_n)\rightarrow 0$ as $T_n\rightarrow 0$. It follows from conclusions of Theorems $3.1$ and $3.2$ that under $\textbf{P}^*$ 
\begin{equation}
\sqrt{n}T_n\xrightarrow{\mathcal{L}}N_k(0,\Delta)\text{ as }\nu\rightarrow\infty
\end{equation}
 for high entropy and RHC sampling deigns, where $\Delta$ is a $k\times k$ matrix such that $((\Delta))_{uv}$= $K(p_u,p_v)$ for $1\leq u,v\leq k$. It can be shown that $Q_{y,N}(p)\rightarrow Q_y(p)$ as $\nu\rightarrow\infty$ \textit{a.s.} $[\textbf{P}]$ for any $p\in (0,1)$, when $\{(Y_i,X_i): 1\leq i \leq N\}$ are i.i.d. Thus $a_N\rightarrow a$ as $\nu\rightarrow\infty$ \textit{a.s.} $[\textbf{P}]$ for some $a$. Consequently, under $\textbf{P}^*$, $\sqrt{n}(f(G(p_1),\ldots,G(p_k))-f(Q_{y,N}(p_1),\ldots,Q_{y,N}(p_k)))\xrightarrow{\mathcal{L}} N(0,\sigma^2_2)$ as $\nu\rightarrow\infty$ for the aforesaid sampling designs, where $\sigma^2_2$=$a\Delta a^T$. This completes the proofs of \textsl{(i)} and \textsl{(ii)}.
\end{proof}
\begin{proof}[Proof of Theorem $5.2$]
It can be shown using Assumptions $1$, $6$ and $8$ in the main text, and Lemma \ref{lem s10} in this supplement that asymptotic covariance kernels of the quantile processes considered in this paper under stratified multistage cluster sampling design with SRSWOR (see  ($8$) in the main text\color{black}) are continuous on $[\alpha,\beta]\times[\alpha,\beta]$, when $H$ is fixed as $\nu\rightarrow\infty$. Moreover, by the assumption of this theorem, asymptotic covariance kernels of the aforementioned quantile processes are continuous on $[\alpha,\beta]\times[\alpha,\beta]$, when $H\rightarrow\infty$ as $\nu\rightarrow\infty$. Then, the asymptotic normality of $\int_{\alpha}^{\beta}\sqrt{n}(G(p)-Q_{y,N}(p))J(p)dp$ for the above sampling design under $\textbf{P}^*$ can be shown using similar arguments as in the $1^{st}$ paragraph of the proof of Theorem $5.1$. 
\par
\vspace{.1cm}

Next, if $H$ is fixed as $\nu\rightarrow\infty$, then it can be shown using Assumption $6$ that $Q_{y,N}(p)\rightarrow \tilde{Q}_{y,H}(p)$ as $\nu\rightarrow\infty$ \textit{a.s.} $[\textbf{P}]$ for any $p\in (0,1)$, where $\tilde{Q}_{y,H}(p)$=$\{t\in\mathds{R}:\tilde{F}_{y,H}(t)\geq p\}$, $\tilde{F}_{y,H}(t)$=$\sum_{h=1}^H\Lambda_h F_{y,h}(t)$, and $\Lambda_h$'s are as in Assumption $6$. Further, if $H\rightarrow\infty$ as $\nu\rightarrow\infty$, then it can be shown using Assumption $11$ that $Q_{y,N}(p)\rightarrow \tilde{Q}_y(p)$ as $\nu\rightarrow\infty$ \textit{a.s.} $[\textbf{P}]$ for any $p\in (0,1)$, where $\tilde{Q}_y(p)$=$\{t\in\mathds{R}:\tilde{F}_y(t)\geq p\}$, and $\tilde{F}_y$ is as in Assumption $11$. Thus $a_N\rightarrow a$ as $\nu\rightarrow\infty$ \textit{a.s.} $[\textbf{P}]$ for some $a$, where $a_N$ is as in the $2^{nd}$ paragraph of the proof of Theorem $5.1$. Then, given any $k\geq 1$ and $p_1,\ldots,p_k\in[\alpha,\beta]$, the the asymptotic normality of $\sqrt{n}(f(G(p_1),\ldots,G(p_k))-f(Q_{y,N}(p_1),\ldots,Q_{y,N}(p_k)))$ for the above sampling design under $\textbf{P}^*$ can be shown using similar arguments as in the $2^{nd}$ paragraph of the proof of Theorem $5.1$. This completes the proofs of \textsl{(i)} and \textsl{(ii)}.
\end{proof}
\begin{proof}[Proof of Theorem $5.3$]
\textsl{(i)} We shall prove this theorem using  \eqref{eq A13} in Lemma \ref{lem 8}.  Fix $\epsilon>0$, and suppose that 
$$
B_{\epsilon}(s,\omega)=\{p_1,p_2\in [\alpha,\beta]:|\hat{K}(p_1,p_2)-K(p_1,p_2)|\leq \epsilon\}\text{ for }s\in\mathcal{S}\text{ and }\omega\in \Omega.
$$
 Then, we have 
\begin{align*}
\begin{split}
&\int_{\alpha}^{\beta}\int_{\alpha}^{\beta}|(\hat{K}(p_1,p_2)-K(p_1,p_2))J(p_1)J(p_2)|dp_1 dp_2 \leq K(\iint_{B_{\epsilon}(s,\omega)}|\hat{K}(p_1,p_2)-\\
&K(p_1,p_2)|dp_1 dp_2+\iint_{(B_{\epsilon}(s,\omega))^c} |\hat{K}(p_1,p_2)-K(p_1,p_2)|dp_1 dp_2)\\
&\leq K(\epsilon(\beta-\alpha)^2+\iint_{(B_{\epsilon}(s,\omega))^c}|\hat{K}(p_1,p_2)-K(p_1,p_2)|dp_1 dp_2)
\end{split}
\end{align*}
 for some constant $K>0$ since $J$ is continuous on $[\alpha,\beta]$. Now, let $W_n$=$\sup_{p_1,p_2\in[\alpha,\beta]}|\hat{K}(p_1,p_2)$ $-K(p_1,p_2)|$. Then, 
\begin{align*}
\begin{split}
&\iint_{(B_{\epsilon}(s,\omega))^c}|\hat{K}(p_1,p_2) -K(p_1,p_2)|dp_1 dp_2\leq\\
& W_n\int_{\alpha}^{\beta}\int_{\alpha}^{\beta}\mathds{1}_{[(B_{\epsilon}(s,\omega))^c]}(p_1,p_2) dp_1 dp_2. 
\end{split}
\end{align*}
 Further, under a high entropy sampling design, 
\begin{align*}
\begin{split}
&E_{\textbf{P}^{*}}\bigg(\int_{\alpha}^{\beta}\int_{\alpha}^{\beta}\mathds{1}_{[(B_{\epsilon}(s,\omega))^c]}( p_1,p_2) dp_1 dp_2\bigg)=\int_{\alpha}^{\beta}\int_{\alpha}^{\beta} \textbf{P}^*\big(|\hat{K}(p_1,p_2)-K(p_1,p_2)|>\epsilon\big)dp_1 dp_2\\
&\rightarrow 0\text{ as }\nu\rightarrow\infty\text{ by DCT since }\hat{K}(p_1,p_2)\xrightarrow{p}K(p_1,p_2)\text{ as }\nu\rightarrow\infty 
\end{split}
\end{align*}
for any $p_1,p_2\in [\alpha,\beta]$ under $\textbf{P}^*$  by \eqref{eq A13} in Lemma \ref{lem 8}.  Therefore, under $\textbf{P}^*$, 
\begin{align*}
\begin{split}
&\int_{\alpha}^{\beta}\int_{\alpha}^{\beta}\mathds{1}_{[(B_{\epsilon}(s,\omega))^c]}(p_1,p_2)dp_1 dp_2\xrightarrow{p} 0,\text{ and }\\
&\iint_{(B_{\epsilon}(s,\omega))^c}|\hat{K}(p_1,p_2)-K(p_1,p_2)|dp_1 dp_2\xrightarrow{p} 0\text{ as }\nu\rightarrow\infty
\end{split}
\end{align*}
 for a high entropy sampling design because $W_n$=$O_p(1)$ as $\nu\rightarrow\infty$ by  by \eqref{eq A13} in Lemma \ref{lem 8}.  Hence, $\int_{\alpha}^{\beta}\int_{\alpha}^{\beta}|(\hat{K}(p_1,p_2)-K(p_1,p_2))$ $J(p_1)$ $J(p_2)|dp_1 dp_2\xrightarrow{p} 0$ as $\nu\rightarrow\infty$ under $\textbf{P}^*$. This completes the proof of the first part of \textsl{(i)}. The proof of the other part of \textsl{(i)} follows in a straight forward way. Also, the proof of \textsl{(ii)} follows exactly the same way as the proof of \textsl{(i)}.
\end{proof}
\begin{proof}[Proof of Theorem $5.4$]
The proof follows exactly the same way as the proof of Theorem $5.3$ in view of  Lemma \ref{lem 9} .
\end{proof} 
\begin{proof}[Proof of Theorem $6.1$]
 \textsl{(i)} Suppose that $\delta^2_1$, $\delta^2_2$, $\delta^2_3$ and $\delta^2_4$ are the asymptotic variances of the estimators of $\int_{\alpha}^{\beta}Q_{y,N}(p)J(p)dp$ based on $\hat{Q}_{y}(p)$, $\hat{Q}_{y,RA}(p)$, $\hat{Q}_{y,DI}(p)$ and $\hat{Q}_{y,REG}(p)$, respectively, under $P(s,\omega)$. Here, $P(s,\omega)$ is one of SRSWOR, RHC and any HE$\pi$PS sampling designs. It follows from Lemma \ref{lem s1} that Assumption $2$ holds under SRSWOR and any HE$\pi$PS sampling designs by the assumptions of Theorem $6.1$. Then, in view of Theorem $5.1$, we have
\begin{equation}
\delta^2_u=\int_{\alpha}^{\beta} \int_{\alpha}^{\beta}K_u(p_1,p_2)J(p_1)J(p_2)dp_1 dp_2 \text{ for }1\leq u\leq 4
\end{equation}
where $K_u(p_1,p_2)$'s are as in the paragraph preceding Theorem $6.1$ in the main article. Therefore, the conclusion of \textsl{(i)} in Theorem $6.1$ holds in a straightforward way. 
\par

\textsl{(ii)} The proof follows exactly the same way as the proof of \textsl{(i)}.
\end{proof}
\begin{proof}[Proof of Theorem $6.2$]
 \textsl{(i)} Suppose that $\eta^2_1$, $\eta^2_2$ and $\eta^2_3$ are the asymptotic variances of the estimators of $\int_{\alpha}^{\beta}Q_{y,N}(p)J(p)dp$ based on $G(p)$ under SRSWOR, RHC and any HE$\pi$PS sampling designs, respectively. Here, $G(p)$ is one of $\hat{Q}_{y}(p)$, $\hat{Q}_{y,RA}(p)$, $\hat{Q}_{y,DI}(p)$ and $\hat{Q}_{y,REG}(p)$. Then, in view of Theorem $5.1$, we have
\begin{equation}
\eta^2_u=\int_{\alpha}^{\beta} \int_{\alpha}^{\beta}K_u^*(p_1,p_2)J(p_1)J(p_2)dp_1 dp_2 \text{ for }1\leq u\leq 3
\end{equation}
where $K_u^*(p_1,p_2)$'s are as in the paragraph preceding Theorem $6.2$ in the main text. Therefore, the conclusion of \textsl{(i)} in Theorem $6.2$ holds in a straightforward way. 
\par

\textsl{(ii)} The proof follows exactly the same way as the proof of \textsl{(i)}.
\end{proof}
\begin{proof}[Proof of Theorem $6.3$]
 It follows from ($33$) in the proof of Theorem $3.1$ in the main article that under $\textbf{P}^*$
\begin{equation}\label{ceq 24}
\begin{split}
&\{\sqrt{n}(\hat{Q}_y(p)-Q_{y}(p)):p\in[\alpha,\beta]\}\xrightarrow{\mathcal{L}} -\tilde{\mathds{V}}_1/f_y\circ Q_y
\end{split}
\end{equation} 
as $\nu\rightarrow\infty$ in $(D[\alpha,\beta], \mathcal{D})$ with respect to the sup norm metric, for $d(i,s)$=$(N\pi_i)^{-1}$ and SRSWOR. Here, $Q_y$ and $f_y$ are superpopulation quantile and density functions of $y$, respectively, and $\tilde{\mathds{V}}_1$ is a mean $0$ Gaussian process in $D[\alpha,\beta]$ with covariance kernel
\begin{align}\label{eq 25}
\begin{split}
& K(p_1,p_2)=\lim_{\nu\rightarrow\infty}(1-n/N)E_{\textbf{P}}\bigg(\sum_{i=1}^N (\mathds{1}_{[Y_i\leq Q_y(p_1)]}-F_{y,N}(Q_y(p_1))\times\\
&(\mathds{1}_{[Y_i\leq Q_y(p_2)]}-F_{y,N}(Q_y(p_2))/N\bigg)+\lambda (p_1\wedge p_2-p_1 p_2)\\
&=p_1\wedge p_2-p_1 p_2\text{ for }p_1,p_2\in [\alpha,\beta].
\end{split}
 \end{align}
 The result in \eqref{ceq 24} implies that under $\textbf{P}^*$
\begin{equation}\label{eq 68}
\sqrt{n}(\hat{Q}_y(0.5)-Q_{y}(0.5))\xrightarrow{\mathcal{L}}N(0,\sigma^2_1)\text{ as }\nu\rightarrow\infty
\end{equation}
for $d(i,s)$=$(N\pi_i)^{-1}$ and SRSWOR, where $\sigma^2_1$=$1/4f_y^2(Q_y(0.5))$. Next, it can be shown using Theorems $1$ and $3$ in \cite{scott1981asymptotic} that under SRSWOR,
\begin{equation}
\sqrt{n}( \overline{y}-\overline{Y})\xrightarrow{\mathcal{L}}N(0,\sigma^2_2)\text{ and }\sqrt{n}(\hat{\overline{Y}}_{GREG}-\overline{Y})\xrightarrow{\mathcal{L}}N(0,\sigma^2_3)
\end{equation}
as $\nu\rightarrow\infty$ \textit{ a.s. }[$\textbf{P}$], where $\sigma^2_2$=$(1-\lambda)\sigma^2_y$, $\sigma^2_3$=$(1-\lambda)\sigma^2_y(1-\rho^2_{xy})$, $\sigma^2_{y}$ is the superpopulation variance of $y$, and $\rho_{xy}$ is the superpopulation correlation coefficient between $x$ and $y$. Further, it can be shown in the same way as the proof of the result in \eqref{eq 45} that under $\textbf{P}^*$,
\begin{equation}\label{eq 69}
\sqrt{n}( \overline{y}-E_{\textbf{P}}(Y_i))\xrightarrow{\mathcal{L}}N(0,\sigma^2_2+\lambda \sigma^2_y)\text{ and }\sqrt{n}(\hat{\overline{Y}}_{GREG}-E_{\textbf{P}}(Y_i))\xrightarrow{\mathcal{L}}N(0,\sigma^2_3+\lambda \sigma^2_y)
\end{equation}
as $\nu\rightarrow\infty$. Therefore, the conclusion of Theorem $6.3$ holds in a straightforward way in view of \eqref{eq 68} and \eqref{eq 69}. 
\end{proof} 
\section{Additional data analysis}\label{sec 6}
In this section, we compute and plot ratios of asymptotic and true MSEs for different parameters, estimators, sampling designs and sample sizes considered in Section $7$ of the main text. We first compute asymptotic mean squared errors (MSEs) of the estimators following the procedure described below. Recall from  ($5$) and ($7$) in Section $3$, and ($8$) in Section $4$ of the main text  the expressions of the asymptotic covariance kernels $K(p_1,p_2)$ of several quantile processes considered in this paper. Note that $K(p_1,p_2)$=$\lim_{\nu\rightarrow\infty}$ $E_{\textbf{P}}(\sigma_1(p_1,p_2))$ for $d(i,s)$=$(N\pi_i)^{-1}$ under high entropy sampling designs, $K(p_1,p_2)$=$\lim_{\nu\rightarrow\infty}$ $E_{\textbf{P}}(\sigma_2(p_1,p_2))$ for $d(i,s)$=$A_i/NX_i$ under RHC sampling design, and $K(p_1,p_2)$=$\lim_{\nu\rightarrow\infty}(n/N^2)\sum_{h=1}^H N_h (N_h-n_h)\sigma_h(p_1,p_2)/n_h$ for $d(i,s)$=$(N\pi_i)^{-1}$ under stratified multistage cluster sampling design with SRSWOR, where 
\begin{align}
\begin{split}
&\sigma_1(p_1,p_2)=(n/N^2)\sum_{i=1}^N(\zeta_i(p_1)-\overline{\zeta}(p_1)-S(p_1)\pi_i)(\zeta_i(p_2)-\overline{\zeta}(p_2)-S(p_2)\pi_i)\times \\
&(\pi_i^{-1}-1),\\
&\sigma_2(p_1,p_2)=(n\gamma)(\overline{X}/N)\sum_{i=1}^N (\zeta_i(p_1)-\overline{\zeta}(p_1))(\zeta_i(p_2)-\overline{\zeta}(p_2))X_i^{-1},\text{ and }\\
&\sigma_h(p_1,p_2)=E_{\textbf{P}}\big(\zeta_{hjl}^{\prime}(p_1)- E_{\textbf{P}}(\zeta_{hjl}^{\prime}(p_1))\big)\big(\zeta_{hjl}^{\prime}(p_2)-E_{\textbf{P}}(\zeta_{hjl}^{\prime}(p_2))\big)
\end{split}
\end{align} 
 for $h$=$1,\ldots,H$. Here, $\zeta_i(p)$'s, $\zeta_{hjl}^{\prime}(p)$'s, $\overline{\zeta}(p)$, $S(p)$ and $\gamma$ are as in  Sections $3$ and $4$ of the main article,  and $N_h$ and $n_h$ are as in the paragraph preceding Assumption $6$ in Section $4$ of the main paper.  Note that $\zeta_i(p)$'s in $\sigma_1(p_1,p_2)$ and $\sigma_2(p_1,p_2)$ involve various superpopulation parameters (see  Table $1$  in the main text). We approximate these superpopulation parameters  by their finite population versions as given in Table \ref{table 1} below. 
\begin{table}[h] 
\caption{Several superpopulation parameters and their finite population versions}
\label{table 1}
\renewcommand{\arraystretch}{1.4}
\begin{threeparttable}[b]
\centering
\begin{tabular}{cc} 
\hline
Superpopulation parameters& Finite population version\\
\hline
$E_{\textbf{P}}(X_i)$& $\overline{X}$\\
\hline
$E_{\textbf{P}}(Y_i)$ & $\overline{Y}$\\
\hline
$E_{\textbf{P}}(X_iY_i)$ & $\sum_{i=1}^N X_i Y_i/N$\\
\hline
$E_{\textbf{P}}(X_i^2)$& $\sum_{i=1}^N X_i^2/N$\\
\hline
$1/f_y(Q_y(p))$&\tnote{3} \ $\sqrt{N}(Q_{y,N}(p+1/\sqrt{N})-Q_{y,N}(p-1/\sqrt{N}))/2$\\
\hline
$1/f_x(Q_x(p))$&\tnote{3} \ $\sqrt{N}(Q_{x,N}(p+1/\sqrt{N})-Q_{x,N}(p-1/\sqrt{N}))/2$\\
\hline
\end{tabular}
\begin{tablenotes}
\item[3]  We approximate $1/f_y(Q_y(p))$ and $1/f_x(Q_x(p))$ following the ideas in \cite{MR1292550}.  
\end{tablenotes}
\end{threeparttable}
\end{table} 
Next, we approximate the superpopulation covariance $\sigma_h(p_1,p_2)$ between $\zeta_{hjl}^{\prime}(p_1)$ and $\zeta_{hjl}^{\prime}(p_2)$ by $\sum_{j=1}^{M_h}\sum_{l=1}^{N_{hj}} (\zeta_{hjl}^{\prime}(p_1)-\overline{\zeta}_h^{\prime}(p_1))(\zeta_{hjl}^{\prime}(p_2)-\overline{\zeta}_h^{\prime}(p_2))/N_h$, where $\overline{\zeta}_h^{\prime}(p)$=$\sum_{j=1}^{M_h}\sum_{l=1}^{N_{hj}} \zeta_{hjl}^{\prime}(p)/N_h$. Further, we approximate the superpopulation parameters involved in the expressions of $\zeta_{hjl}^{\prime}(p)$'s (see  Table ($3$)  in the main text) by their finite population versions as given in Table \ref{table 2} below. 
\begin{table}[h] 
\caption{Several superpopulation parameters and their finite population versions}
\label{table 2} 
\renewcommand{\arraystretch}{1.4}
\begin{threeparttable}[b]
\centering
\begin{tabular}{cc} 
\hline
Superpopulation parameters& Finite population version\\
\hline
$\sum_{h=1}^H (N_h/N)E_{\textbf{P}}(X_{hjl}^{\prime})$ (as well as $\Theta_1$)& $\overline{X}$\\
\hline
$\sum_{h=1}^H (N_h/N)E_{\textbf{P}}(Y_{hjl}^{\prime})$ (as well as $\Theta_2$) & $\overline{Y}$\\
\hline
$\sum_{h=1}^H (N_h/N)E_{\textbf{P}} (X_{hjl}^{\prime}Y_{hjl}^{\prime})$ (as well as $\Theta_3$) & $\sum_{i=1}^N X_i Y_i/N$\\
\hline
$\sum_{h=1}^H (N_h/N) E_{\textbf{P}}$ $(X_{hjl}^{\prime})^2$ (as well as $\Theta_4$)& $\sum_{i=1}^N X_i^2/N$\\
\hline
$1/f_{y,H}(Q_{y,H}(p))$ & $\sqrt{N}(Q_{y,N}(p+1/\sqrt{N})-Q_{y,N}(p-1/\sqrt{N}))/2$\\
\hline
$1/f_{x,H}(Q_{x,H}(p))$& $\sqrt{N}(Q_{x,N}(p+1/\sqrt{N})-Q_{x,N}(p-1/\sqrt{N}))/2$\\
\hline
\end{tabular}
\end{threeparttable}
\end{table} 
Let $\tilde{\sigma}_1(p_1,p_2)$, $\tilde{\sigma}_2(p_1,p_2)$ and $\tilde{\sigma}_h(p_1,p_2)$ denote the approximated $\sigma_1(p_1,p_2)$, $\sigma_2(p_1,p_2)$ and $\sigma_h(p_1,p_2)$, respectively. Then, asymptotic MSEs of several estimators of the parameters considered in Section $7$ of the main text are computed by replacing $K(p_1,p_2)$ in the expressions of $\sigma^2_1$ and $\sigma^2_2$ (see  Theorem $5.1$)  by $\tilde{\sigma}_1(p_1,p_2)/n$, $\tilde{\sigma}_2(p_1,p_2)/n$ and $(1/N^2)\sum_{h=1}^H N_h(N_h-n_h)\tilde{\sigma}_h(p_1,p_2)/n_h$. We approximate the double integral in the expression of $\sigma_1^2$ by sum after dividing $[\alpha,1-\alpha]$ into $100$ sub intervals of equal width. We also approximate $a$ in the expression of $\sigma_2^2$ by $\tilde{a}$=$\nabla f(Q_{y,N} (p_1),\ldots,Q_{y,N}(p_k))$.
\par

The true MSE of an estimator $\hat{\theta}$ of $\theta$ under a sampling design $P(s)$ is estimated as 
$$
MSE(\hat{\theta},P)=\sum_{k=1}^I(\hat{\theta}_{k}-\theta_0)^2/I,
$$
 where $\theta_0$ is the true value of $\theta$, and $\hat{\theta}_k$ is the estimate of $\theta$ computed based on the $k^{th}$ sample using the sampling design $P(s)$ for $k$=$1,\ldots,I$. The plots of ratios of asymptotic and true MSEs for different parameters, estimators, sampling designs and sample sizes are presented in Figures \ref{Fig 17}--\ref{Fig 20} and \ref{Fig 1}--\ref{Fig 4} in this supplement. Also, boxplots of ratios of asymptotic and true MSEs for different parameters and estimators in the cases of single stage sampling designs and stratified multistage cluster sampling design with SRSWOR are given in Figure \ref{Fig 34} in this supplement. 
\par
 
 Next, we empirically check whether the conditions ($18$) and ($19$) in the main text hold for different parameters and sampling designs in the case of the real dataset considered in this article. For this, we fix a sampling design among SRSWOR, RHC and Rao-Sampford (RS) sampling designs and a parameter among the median, the $\alpha$-trimmed means with $\alpha$=$0.1$ and $0.3$, the interquartile range and the Bowley's measure of skewness. Suppose that for the given parameter,  $\tilde{\delta}^2_1$, $\tilde{\delta}^2_2$, $\tilde{\delta}^2_3$ and $\tilde{\delta}^2_4$ are the asymptotic MSEs of the estimators constructed based on $\hat{Q}_{y}(p)$, $\hat{Q}_{y,RA}(p)$, $\hat{Q}_{y,DI}(p)$ and $\hat{Q}_{y,REG}(p)$, respectively, under the given sampling design. Then, we check whether the condition 
 \begin{equation}\label{eq 70}
 \max_{2\leq u\leq 4}\big(\tilde{\delta}^2_1-\tilde{\delta}^2_u\big)<0
 \end{equation}
 holds. Here, asymptotic MSEs are computed based on all the population values as described in the first paragraph of this section. Note that RS sampling design is a HE$\pi$PS sampling design. We consider RS sampling design as it is easier to implement than other HE$\pi$PS sampling designs. Next, we empirically check whether the conditions ($21$) and ($22$) in the main text hold for the above parameters and their estimators in the case of the real dataset. For this, we fix one of the aforementioned parameters and one of its estimators among the estimators constructed based on $\hat{Q}_{y}(p)$, $\hat{Q}_{y,RA}(p)$, $\hat{Q}_{y,DI}(p)$ and $\hat{Q}_{y,REG}(p)$. Suppose that for the given parameter,  $\tilde{\eta}^2_1$, $\tilde{\eta}^2_2$ and $\tilde{\eta}^2_3$ are the asymptotic MSEs of the given estimator under SRSWOR, RHC and RS sampling designs, respectively. Then, we check whether the condition 
 \begin{equation}\label{eq 71}
 \max_{2\leq u\leq 3}\big(\tilde{\eta}^2_1-\tilde{\eta}^2_u\big)<0
 \end{equation}
 holds. Finally, we empirically check whether the conditions ($23$) and ($24$) in the main text hold in the case of the real dataset. For this, we check whether the conditions
 \begin{equation}\label{eq 72}
S^2_{y}> N(Q_{y,N}(0.5+1/\sqrt{N})-Q_{y,N}(0.5-1/\sqrt{N}))^2\text{ and }
\end{equation}
\begin{equation}\label{eq 73}
r_{xy}^2> (1-n/N)^{-1}\big(1-N(Q_{y,N}(0.5+1/\sqrt{N})-Q_{y,N}(0.5-1/\sqrt{N}))^2/S^2_y \big)\
\end{equation}
hold. Here, $S^2_y$ is the finite population variance of $y$, and $r_{xy}$ is the finite population correlation coefficient between $x$ and $y$. We check the conditions in \eqref{eq 70}--\eqref{eq 73} for $n$=$200$.
\par

Furthermore, for each of the parameters and sampling designs considered in the preceding paragraph, we compare various estimators mentioned in the same paragraph in terms of true MSEs. For each of these parameters, we also compare the performance of each of its estimators under SRSWOR, RHC and RS sampling designs based on true MSEs. Moreover, we compare the performance of the sample mean, the sample median and the GREG estimator under SRSWOR in terms of true MSEs. As in the preceding paragraph, here also we carry out the aforementioned comparison for $n$=$200$.
\par

Next, for each of the parameters and sampling designs considered in the third paragraph of this section, we compare average lengths of nominal $90\%$ as well as $95\%$ confidence intervals constructed based on $\hat{Q}_{y}(p)$, $\hat{Q}_{y,RA}(p)$, $\hat{Q}_{y,DI}(p)$ and $\hat{Q}_{y,REG}(p)$. Also, for each of the above-mentioned parameters, we compare average  lengths of nominal $90\%$ as well as $95\%$ confidence intervals constructed based on each of its estimators under SRSWOR, RHC and RS sampling designs. We construct each of these confidence intervals $I$=$1000$ times based on $I$=$1000$ samples each of size $n$=$200$ in the same way as in Section $5.1$ of the main text, and compute its average length by taking the average of the lengths of these intervals. Finally, we assume that the real dataset considered in this article is generated from a superpopulation distribution with the same mean and median, and compare average lengths of nominal $90\%$ as well as $95\%$ confidence intervals of the superpopulation mean/median constructed based on the sample mean, the sample median and the GREG estimator under SRSWOR. While constructing these latter confidence intervals, we estimate the superpopulation parameters involved in the asymptotic variances (see the proof of Theorem $6.3$ in Section \ref{sec 5} of this supplement) of the above estimators  in the same way as in Table \ref{table 7} of this supplement. As in the preceding cases, here also we compute average lengths of the confidence intervals based on $I$=$1000$ samples each of size $n$=$200$. 
\par  

 The results obtained from the above data analysis are summarised as follows. 
\begin{itemize}
\item[(i)] It follows from Figures \ref{Fig 17}--\ref{Fig 20}, \ref{Fig 1}--\ref{Fig 4} and \ref{Fig 34} that  ratios of asymptotic and true MSEs for different parameters, estimators and sampling designs become closer to $1$  as the sample size increases from $n$=$200$ to $n$=$500$. Thus  the accuracy of the asymptotic approximations of the distributions of different estimators considered in this section increases as the sample size increases. 
\item[(ii)] The condition \eqref{eq 70} does not hold for any of the parameters and sampling designs considered in the third paragraph of this section. Thus for each of these parameters and sampling designs, the estimator based on one of $\hat{Q}_{y,RA}(p)$, $\hat{Q}_{y,DI}(p)$ and $\hat{Q}_{y,REG}(p)$ has the least asymptotic MSE (see Table \ref{table 4} below). We observe exactly the same scenario, when we compare the estimators in terms of true MSEs, and the confidence intervals based on different estimators in terms of average lengths (see Table \ref{table 4} below). This probably happens because of the approximate linear relationship between the study and the auxiliary variables through the origin (see the scatter plot and the regression line in Figure \ref{Fig 37} drawn based on all the population values).
\item[(iii)] The condition \eqref{eq 71} holds for all the parameters and their estimators considered in the third paragraph of this section except for those estimators of the $\alpha$-trimmed mean with $\alpha$=$0.1$, which are constructed based on $\hat{Q}_{y,DI}(p)$ and $\hat{Q}_{y,REG}(p)$. Thus  for all these parameters except the $\alpha$-trimmed mean with $\alpha$=$0.1$, the estimators based on $\hat{Q}_{y}(p)$, $\hat{Q}_{y,RA}(p)$, $\hat{Q}_{y,DI}(p)$ and $\hat{Q}_{y,REG}(p)$ have the least asymptotic MSEs under SRSWOR among SRSWOR, RS and RHC sampling designs (see Table \ref{table 5} below). For the $\alpha$-trimmed mean with $\alpha$=$0.1$, the estimators based on $\hat{Q}_{y}(p)$ and $\hat{Q}_{y,RA}(p)$ have the least asymptotic MSEs under SRSWOR, and the estimators based on $\hat{Q}_{y,DI}(p)$ and $\hat{Q}_{y,REG}(p)$ have the least asymptotic MSEs under RS sampling design (see Table \ref{table 5} below). We find exactly the same situation, when we compare the estimators in terms of true MSEs, and the confidence intervals based on different estimators in terms of average lengths (see Table \ref{table 5} below).
\item[(iv)] The conditions \eqref{eq 72} and \eqref{eq 73} holds in the case of the real dataset considered in this article. Furthermore, in terms of true MSEs, the GREG estimator performs better than the sample median under SRSWOR, and the sample median performs better than the sample mean under SRSWOR. Also, the confidence interval based on the GREG estimator has smaller average length than the confidence interval based on the sample median under SRSWOR, and the confidence interval based on the sample median has smaller average length than the confidence interval based on the sample mean under SRSWOR.
\item[(v)] The estimators having least asymptotic MSEs coincide with the estimators  having least true MSEs (see Tables \ref{table 4} and \ref{table 5}  below). Also, the estimators having least asymptotic and true MSEs coincide with the estimators based on which nominal $90\%$ as well as $95\%$ confidence intervals have the least average lengths. Thus theoretical and empirical results corroborate each other.
\end{itemize}  
\section{Results of data analysis} 
\subsection{Results of data analysis carried out in Section $7$ of the main article} The results obtained from the analysis described in Section  $7$  of the main article are summarized in this section. Here, we provide some plots that were mentioned in Section  $7$  of the main article. Figures \ref{Fig 21}--\ref{Fig 24} and \ref{Fig 5}--\ref{Fig 8} present plots of relative biases. Figures \ref{Fig 25}--\ref{Fig 32} and \ref{Fig 9}--\ref{Fig 16} present plots of observed coverage probabilities of nominal $90\%$ and $95\%$ confidence intervals. Figure \ref{Fig 33} presents boxplots of relative biases for different parameters and estimators in the cases of single stage sampling designs and stratified multistage cluster sampling design with SRSWOR. Figures \ref{Fig 35} and \ref{Fig 36} present boxplots of observed coverage probabilities of nominal $90\%$ and $95\%$ confidence intervals for different parameters and estimators in the cases of single stage sampling designs and stratified multistage cluster sampling design with SRSWOR.

\begin{figure}[h!]
\begin{center}
\includegraphics[height=9cm,width=14cm]{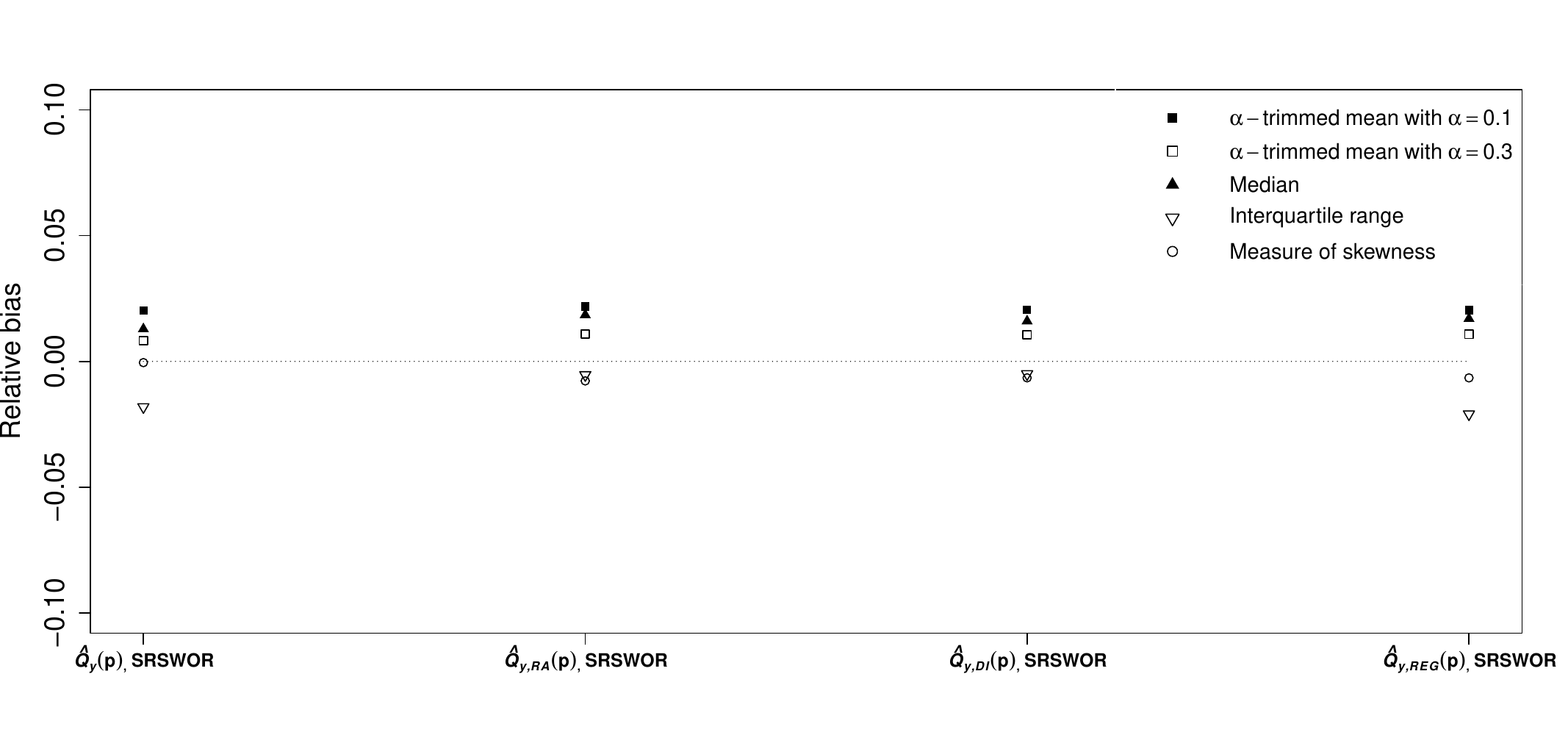}
\caption{Relative biases of different estimators for $n$=$500$ in the case of SRSWOR.}
\label{Fig 21}
\end{center}
\end{figure}

\begin{figure}[h!]
\begin{center}
\includegraphics[height=9cm,width=14cm]{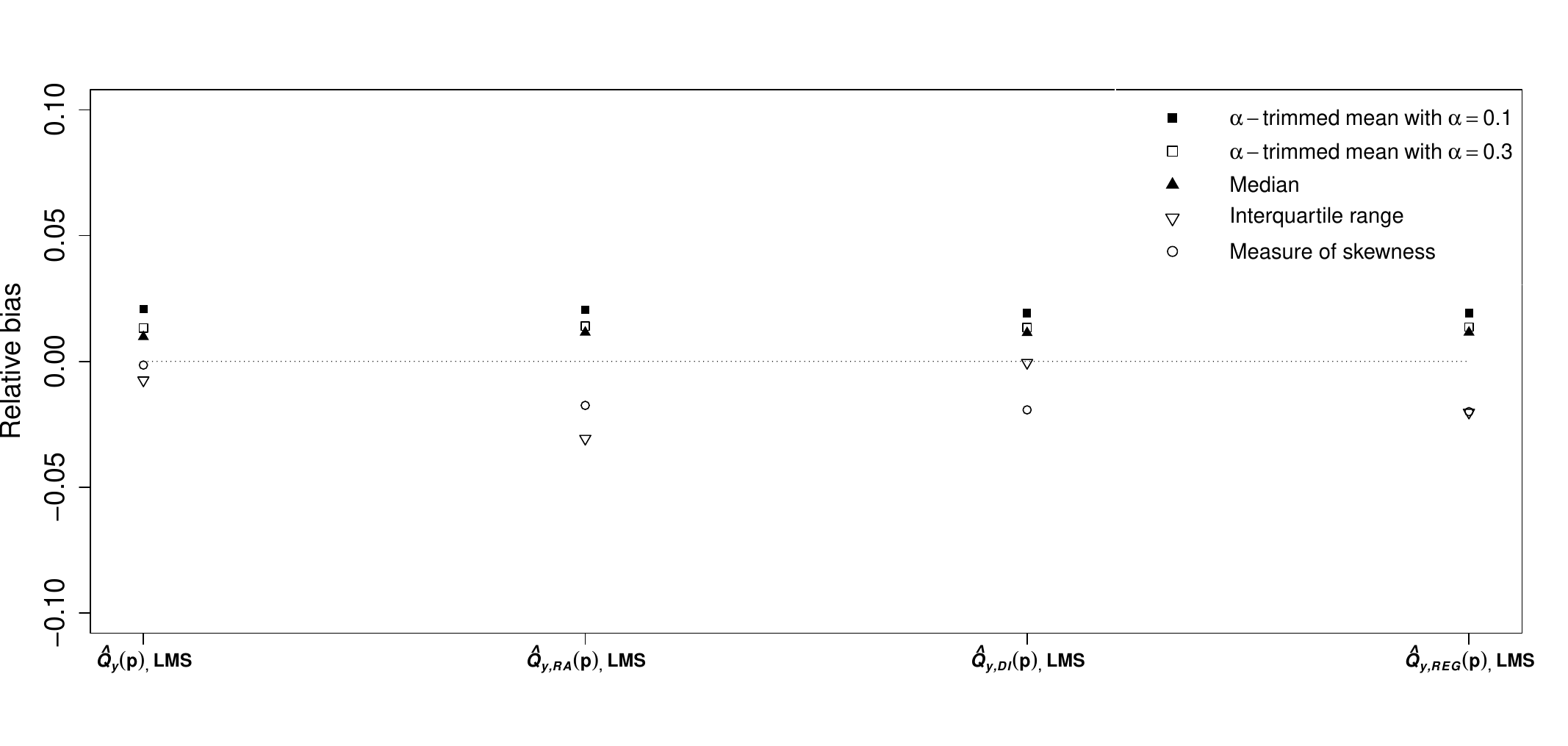}
\caption{Relative biases of different estimators for $n$=$500$ in the case of LMS sampling design.}
\label{Fig 22}
\end{center}
\end{figure}

\begin{figure}[h!]
\begin{center}
\includegraphics[height=9cm,width=14cm]{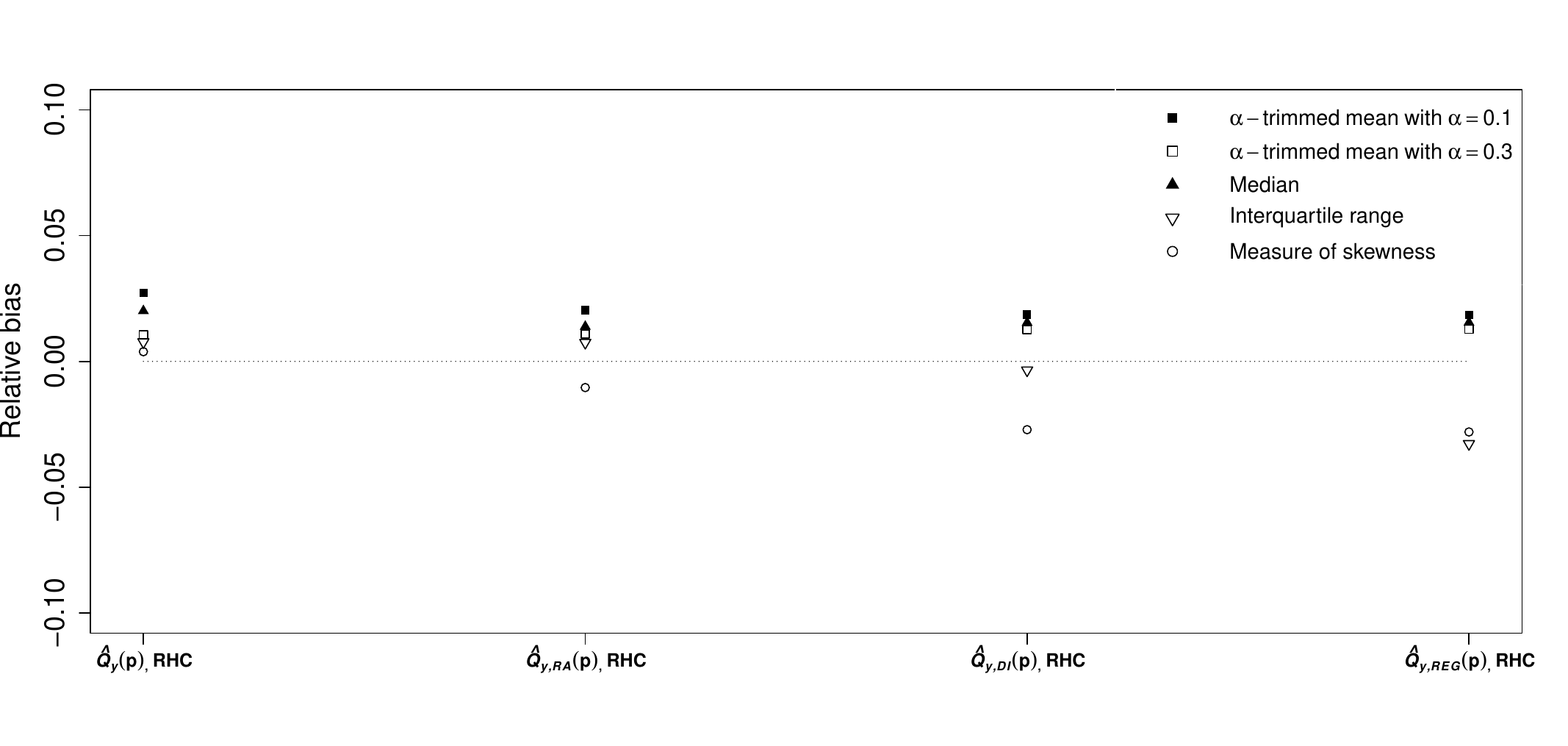}
\caption{Relative biases of different estimators for $n$=$500$ in the case of RHC sampling design.}
\label{Fig 23}
\end{center}
\end{figure}

\begin{figure}[h!]
\begin{center}
\includegraphics[height=9cm,width=14cm]{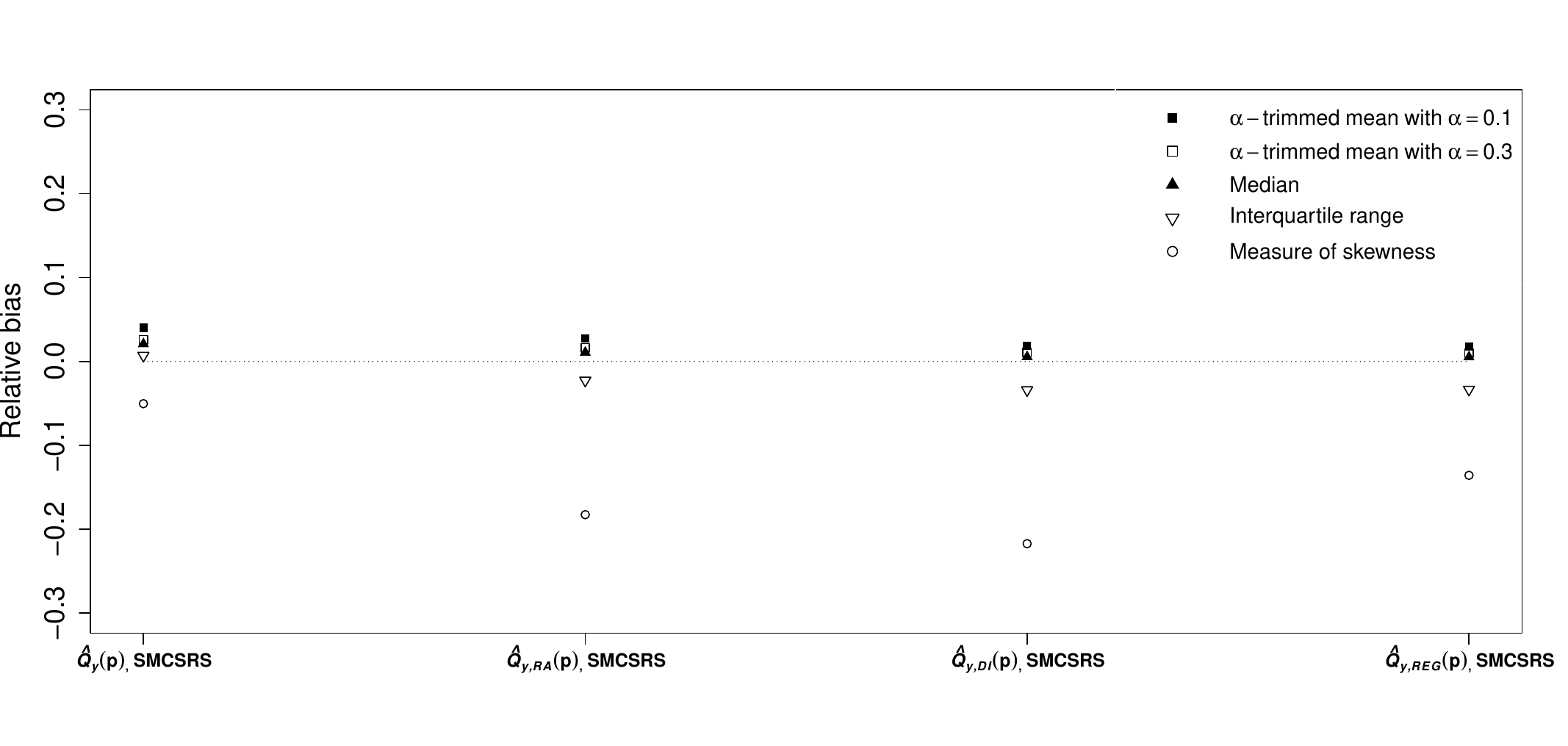}
\caption{Relative biases of different estimators for $n$=$216$ in the case of SMCSRS. In this figure, SMCSRS stands for stratified multistage cluster sampling design with SRSWOR.}
\label{Fig 24}
\end{center}
\end{figure}

\begin{figure}[h!]
\begin{center}
\includegraphics[height=8cm,width=14cm]{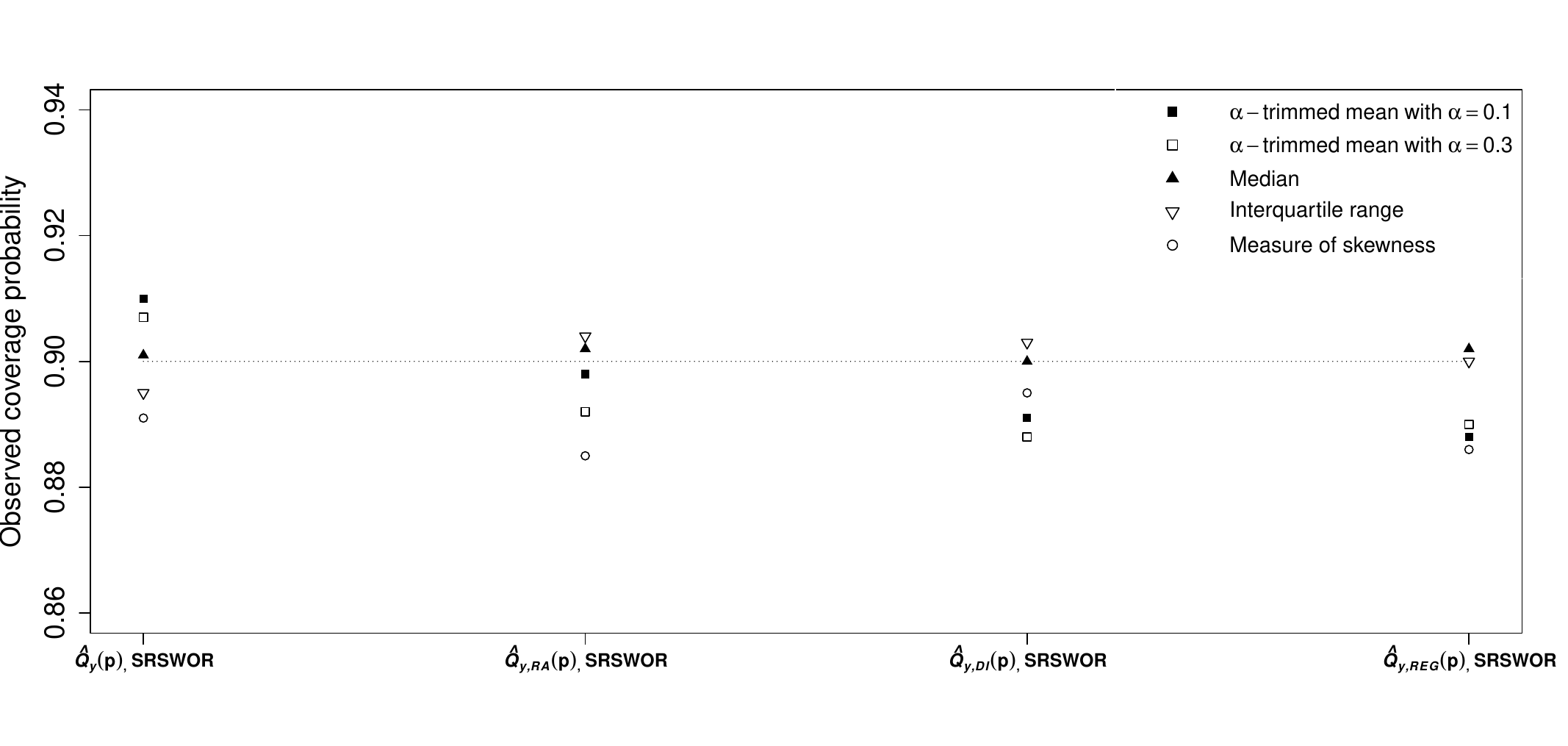}
\caption{Observed coverage probabilities of nominal $90 \%$ confidence intervals for $n$=$500$ in the case of SRSWOR  (the magnitude of the Monte Carlo standard error for observed coverage probabilities is $0.009$). }
\label{Fig 25}
\end{center}
\end{figure}

\begin{figure}[h!]
\begin{center}
\includegraphics[height=8cm,width=14cm]{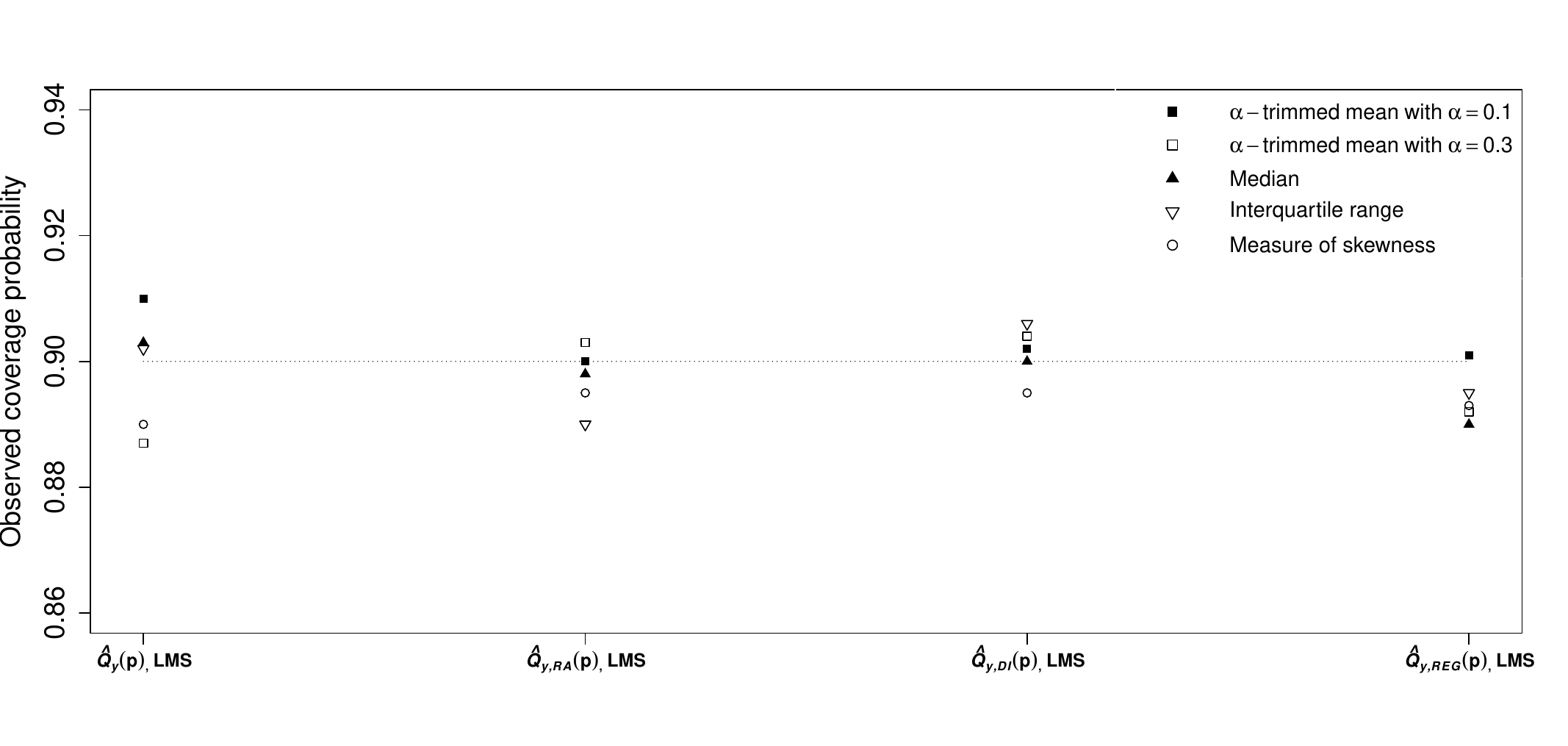}
\caption{Observed coverage probabilities of nominal $90 \%$ confidence intervals for $n$=$500$ in the case of LMS sampling design  (the magnitude of the Monte Carlo standard error for observed coverage probabilities is $0.009$). }
\label{Fig 26}
\end{center}
\end{figure}

\begin{figure}[h!]
\begin{center}
\includegraphics[height=8cm,width=14cm]{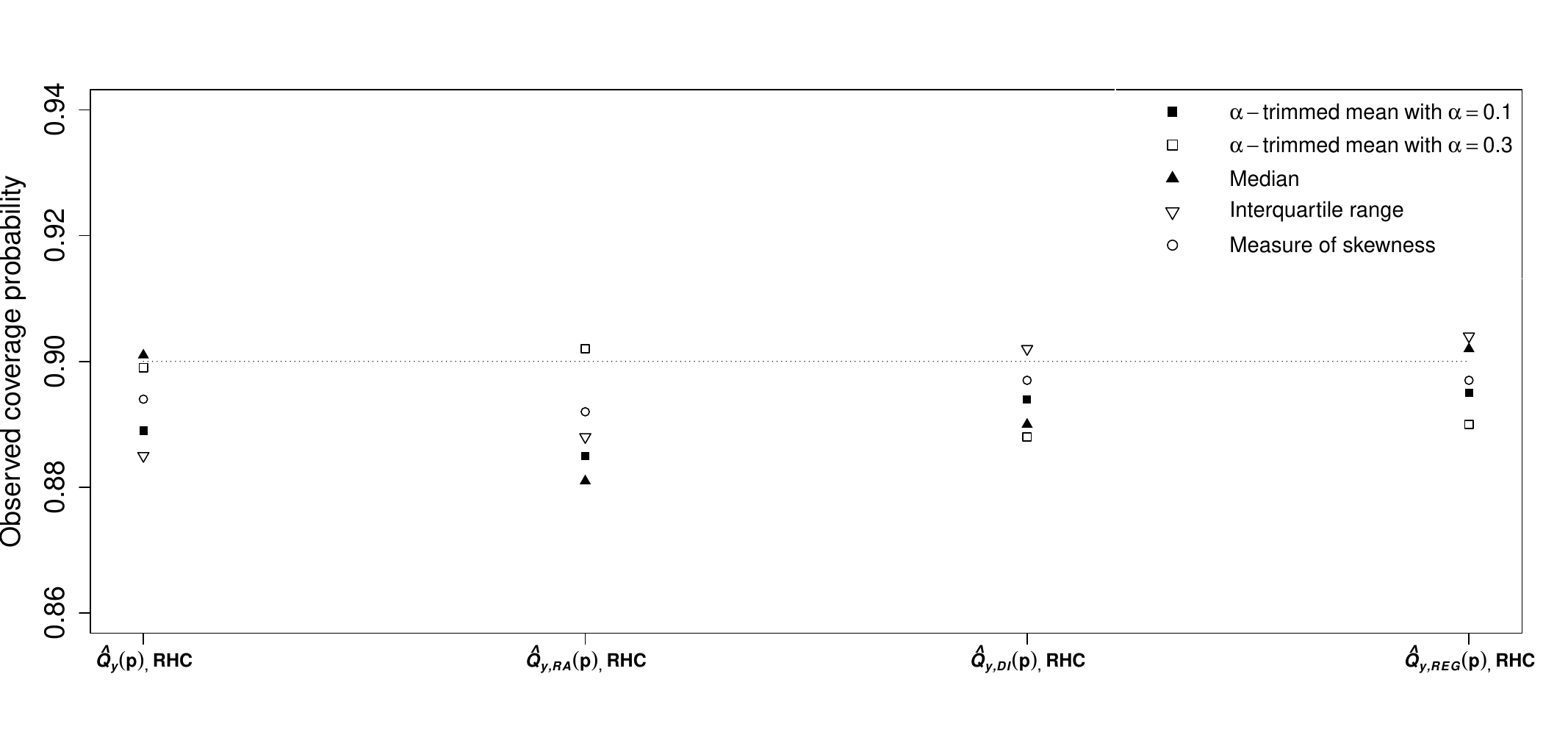}
\caption{Observed coverage probabilities of nominal $90 \%$ confidence intervals for $n$=$500$ in the case of RHC sampling design  (the magnitude of the Monte Carlo standard error for observed coverage probabilities is $0.009$). }
\label{Fig 27}
\end{center}
\end{figure}

\begin{figure}[h!]
\begin{center}
\includegraphics[height=8cm,width=14cm]{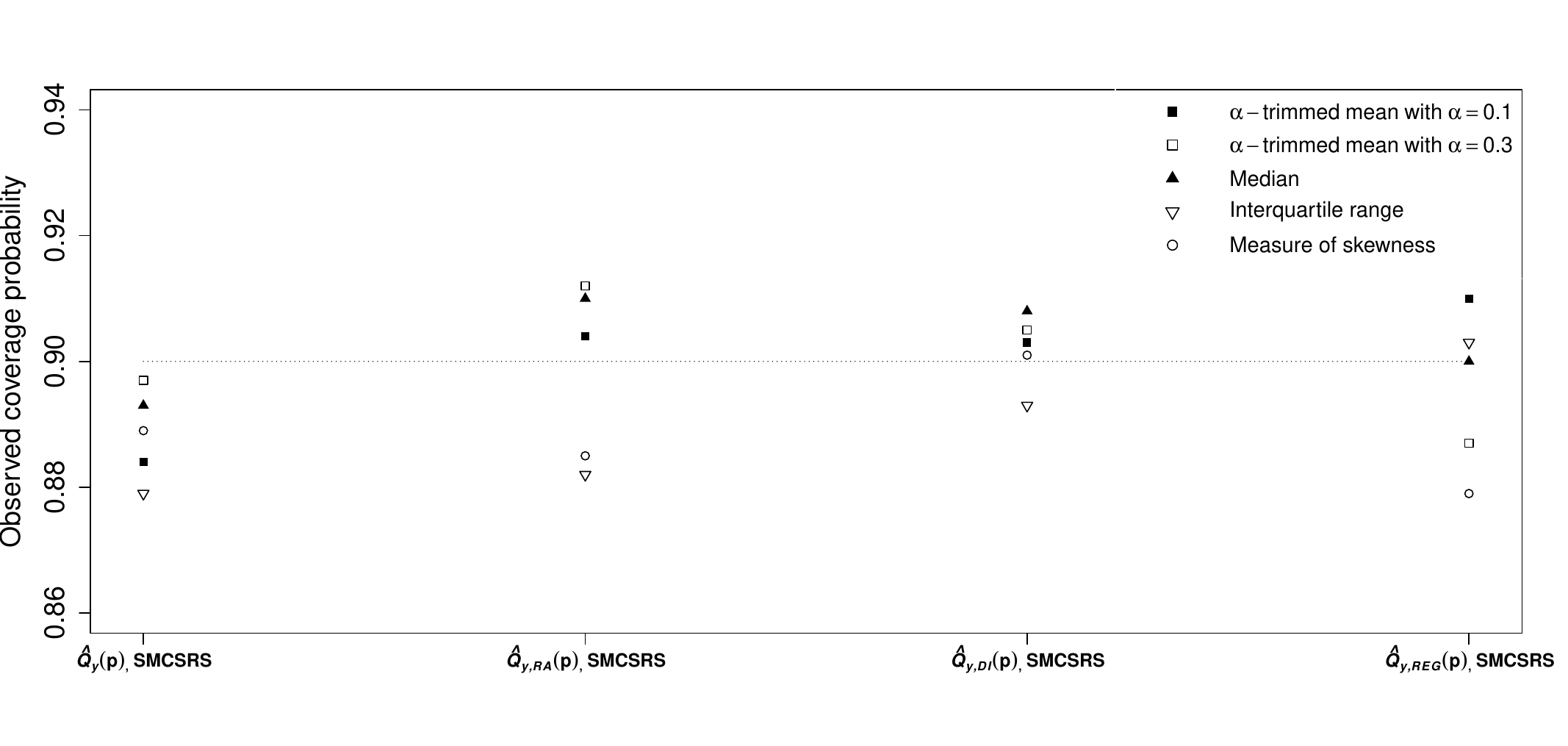}
\caption{Observed coverage probabilities of nominal $90 \%$ confidence intervals for $n$=$216$ in the case of SMCSRS  (the magnitude of the Monte Carlo standard error for observed coverage probabilities is $0.009$).  In this figure, SMCSRS stands for stratified multistage cluster sampling design with SRSWOR.}
\label{Fig 28}
\end{center}
\end{figure}

\begin{figure}[h!]
\begin{center}
\includegraphics[height=8cm,width=14cm]{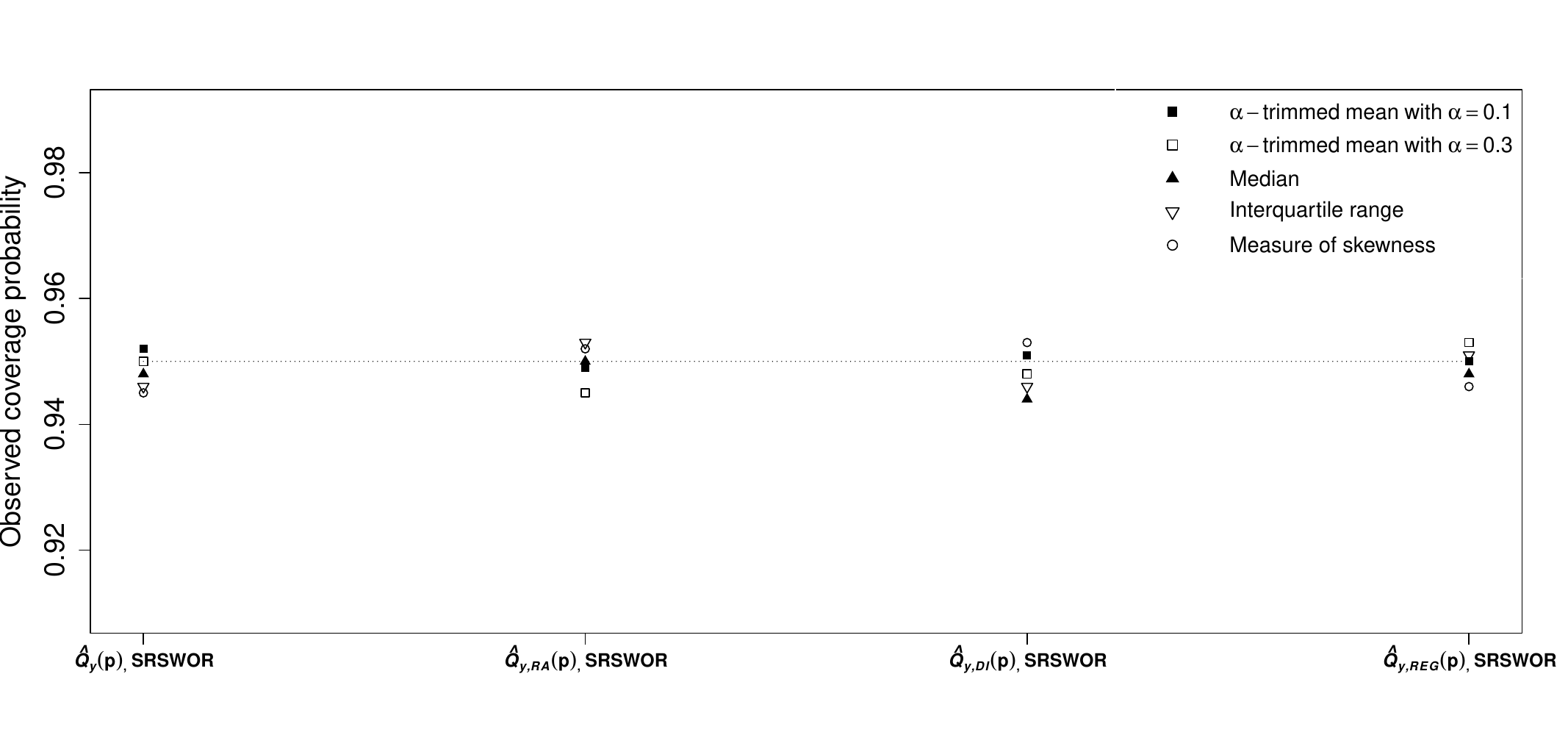}
\caption{Observed coverage probabilities of nominal $95 \%$ confidence intervals for $n$=$500$ in the case of SRSWOR  (the magnitude of the Monte Carlo standard error for observed coverage probabilities is $0.007$). }
\label{Fig 29}
\end{center}
\end{figure}

\begin{figure}[h!]
\begin{center}
\includegraphics[height=8cm,width=14cm]{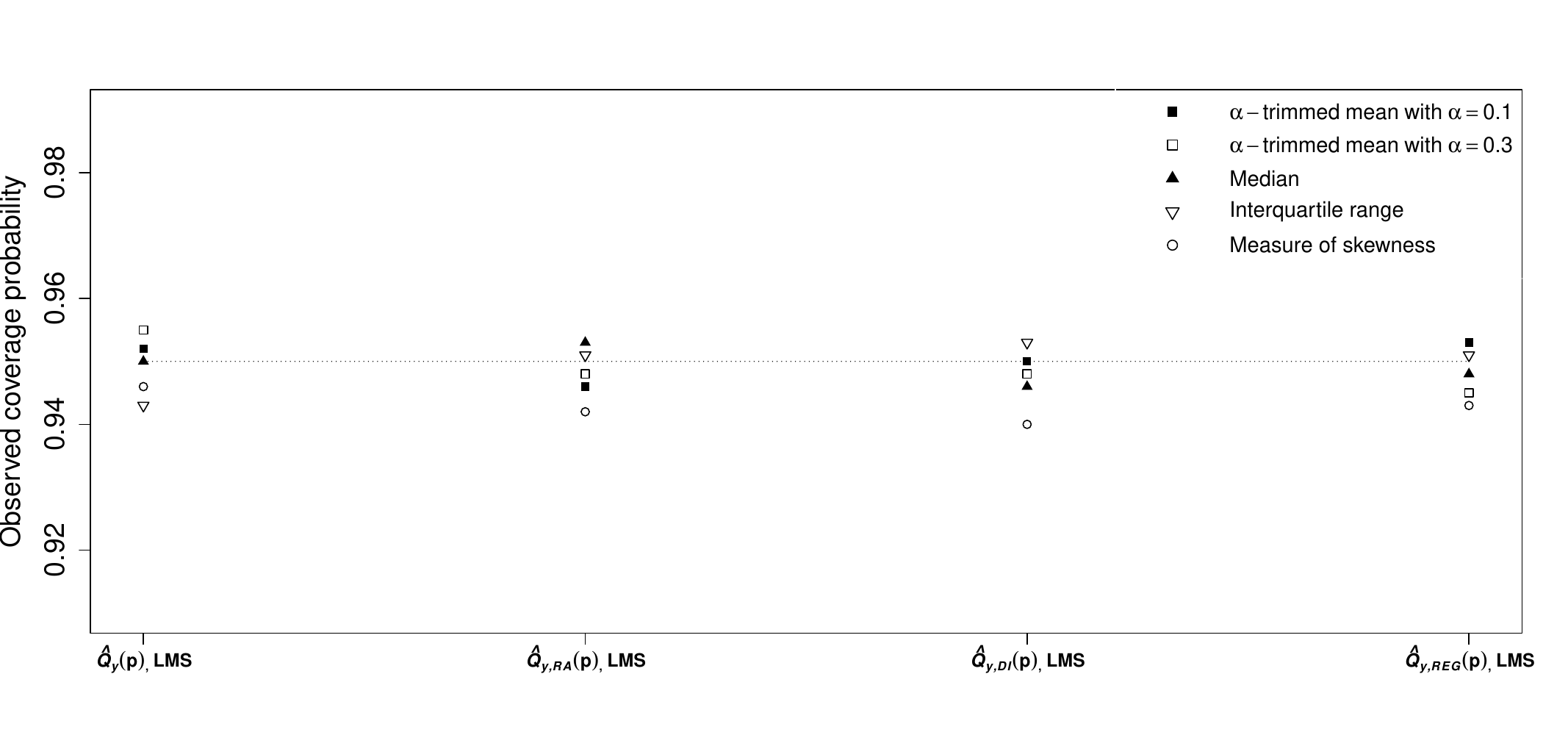}
\caption{Observed coverage probabilities of nominal $95 \%$ confidence intervals for $n$=$500$ in the case of LMS sampling design  (the magnitude of the Monte Carlo standard error for observed coverage probabilities is $0.007$). }
\label{Fig 30}
\end{center}
\end{figure}

\begin{figure}[h!]
\begin{center}
\includegraphics[height=8cm,width=14cm]{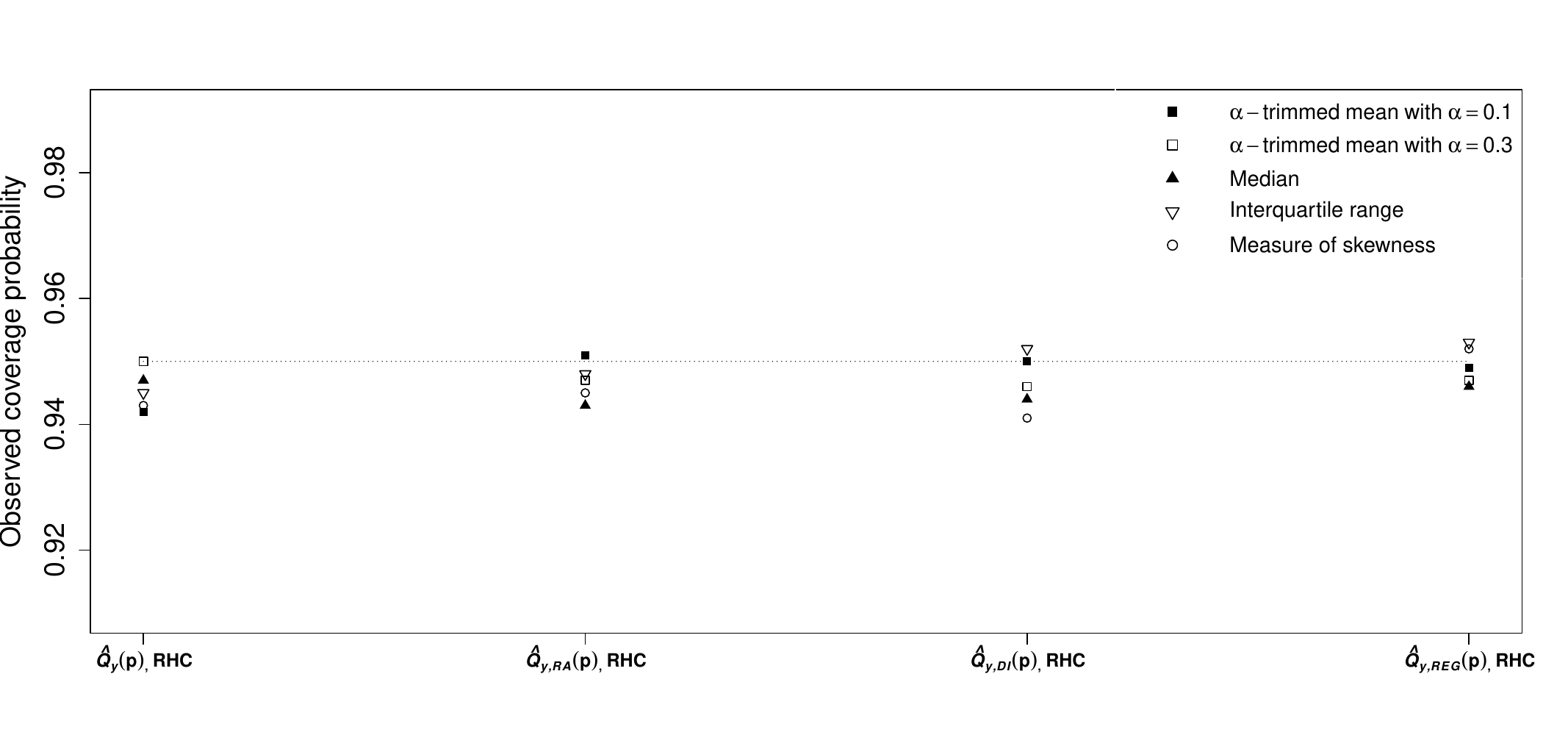}
\caption{Observed coverage probabilities of nominal $95 \%$ confidence intervals for $n$=$500$ in the case of RHC sampling design  (the magnitude of the Monte Carlo standard error for observed coverage probabilities is $0.007$). }
\label{Fig 31}
\end{center}
\end{figure}

\begin{figure}[h!]
\begin{center}
\includegraphics[height=8cm,width=14cm]{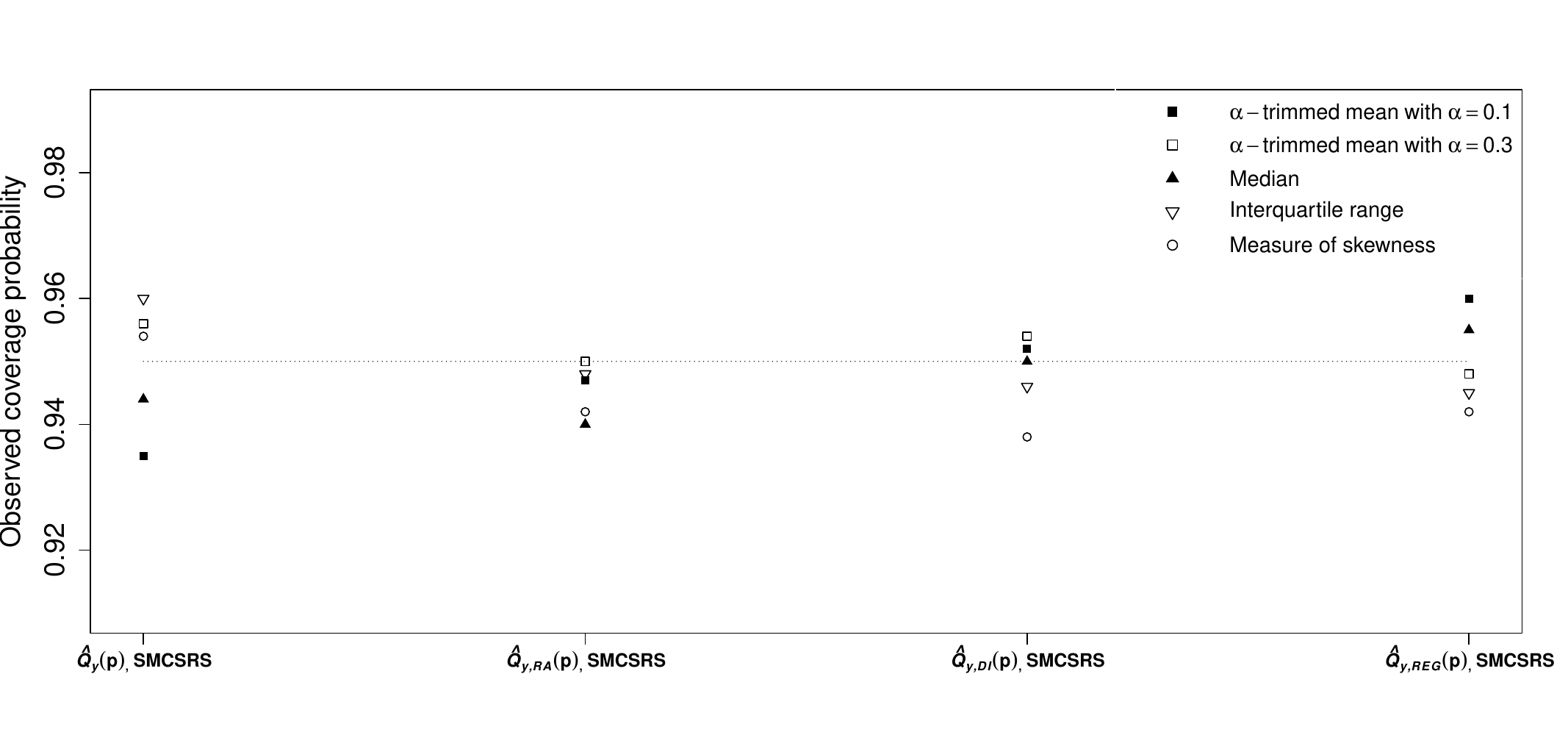}
\caption{Observed coverage probabilities of nominal $95 \%$ confidence intervals for $n$=$216$ in the case of SMCSRS  (the magnitude of the Monte Carlo standard error for observed coverage probabilities is $0.007$).  In this figure, SMCSRS stands for stratified multistage cluster sampling design with SRSWOR.}
\label{Fig 32}
\end{center}
\end{figure}

\begin{figure}[h!]
\begin{center}
\includegraphics[height=9cm,width=14cm]{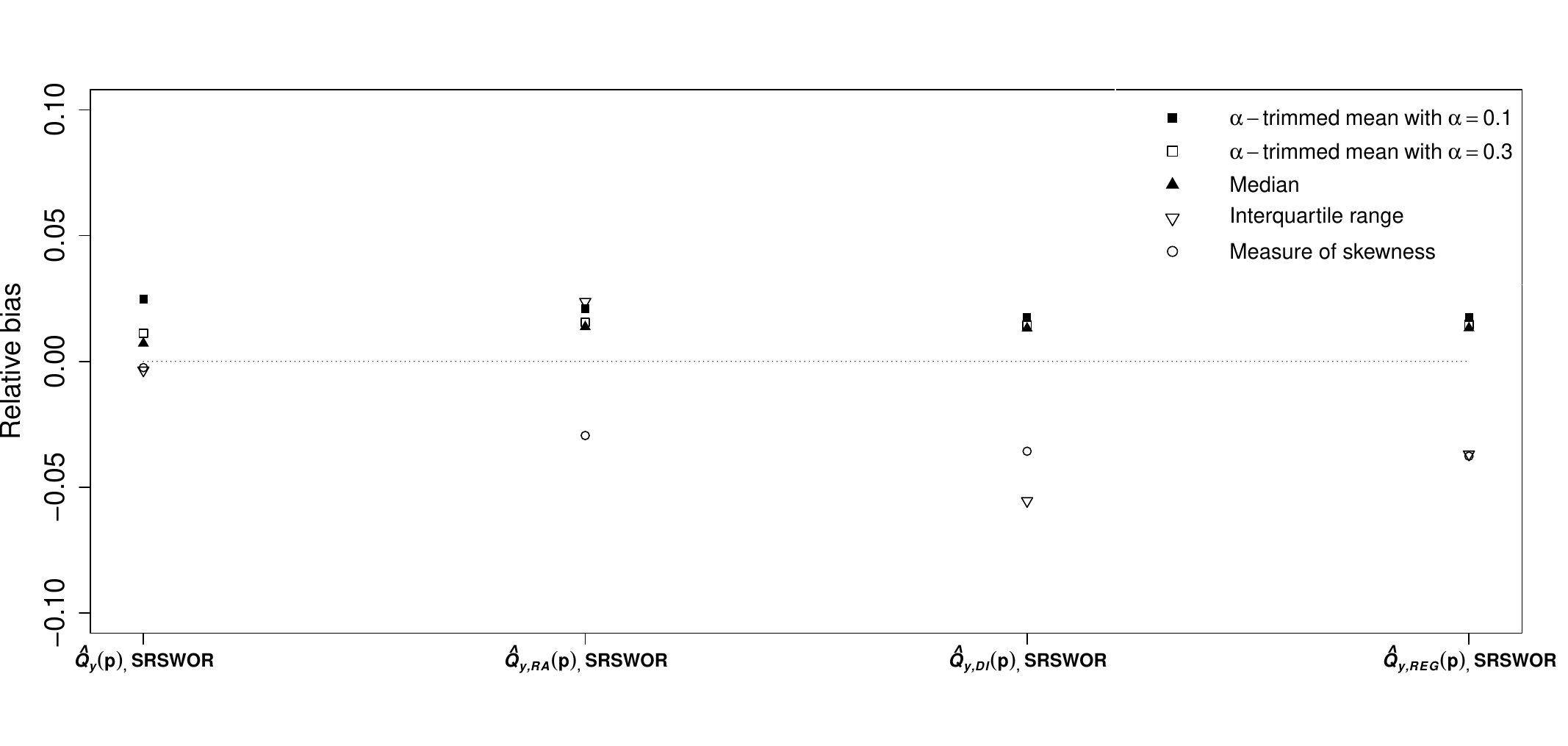}
\caption{Relative biases of different estimators for $n$=$200$ in the case of SRSWOR.}
\label{Fig 5}
\end{center}
\end{figure}

\begin{figure}[h!]
\begin{center}
\includegraphics[height=9cm,width=14cm]{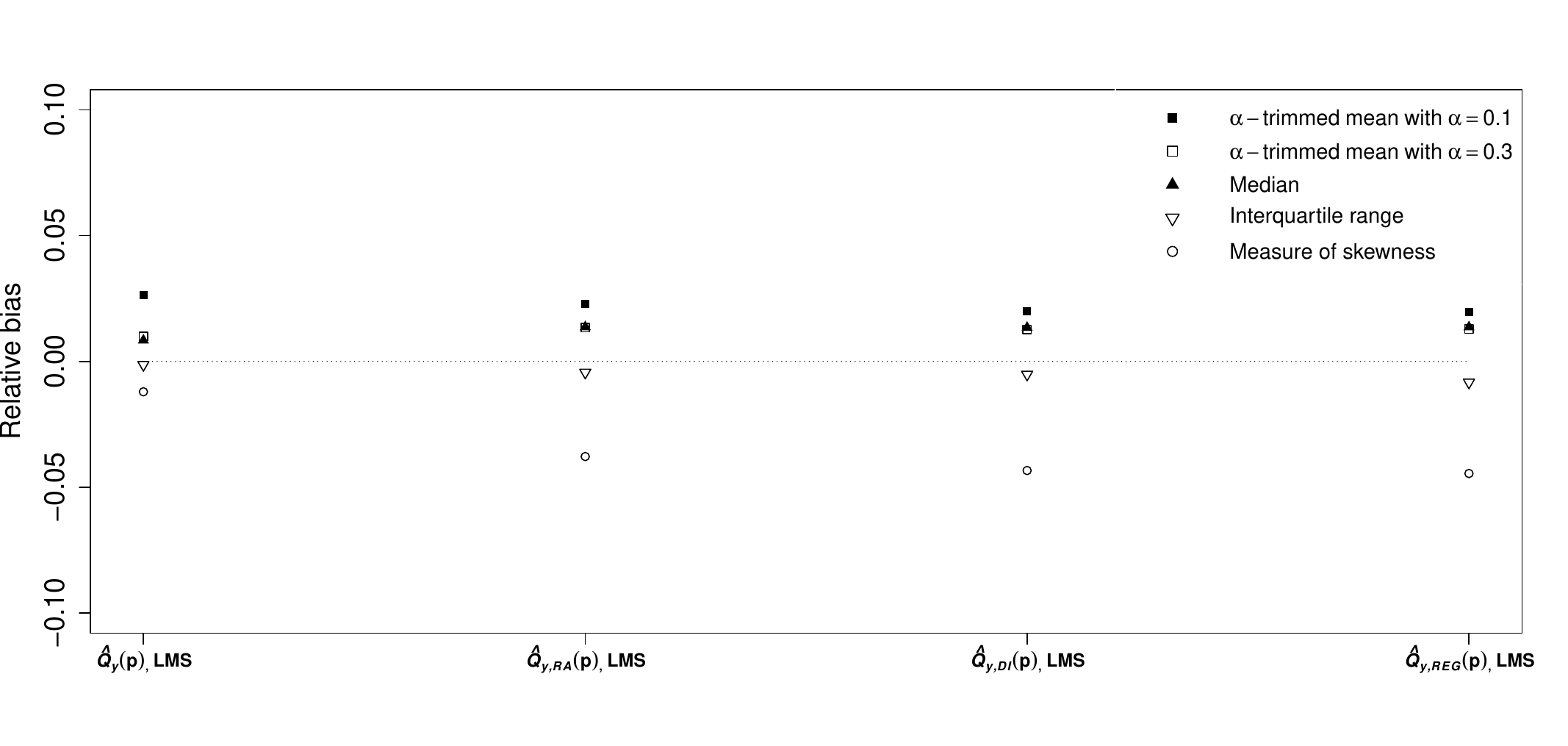}
\caption{Relative biases of different estimators for $n$=$200$ in the case of LMS sampling design.}
\label{Fig 6}
\end{center}
\end{figure}

\begin{figure}[h!]
\begin{center}
\includegraphics[height=9cm,width=14cm]{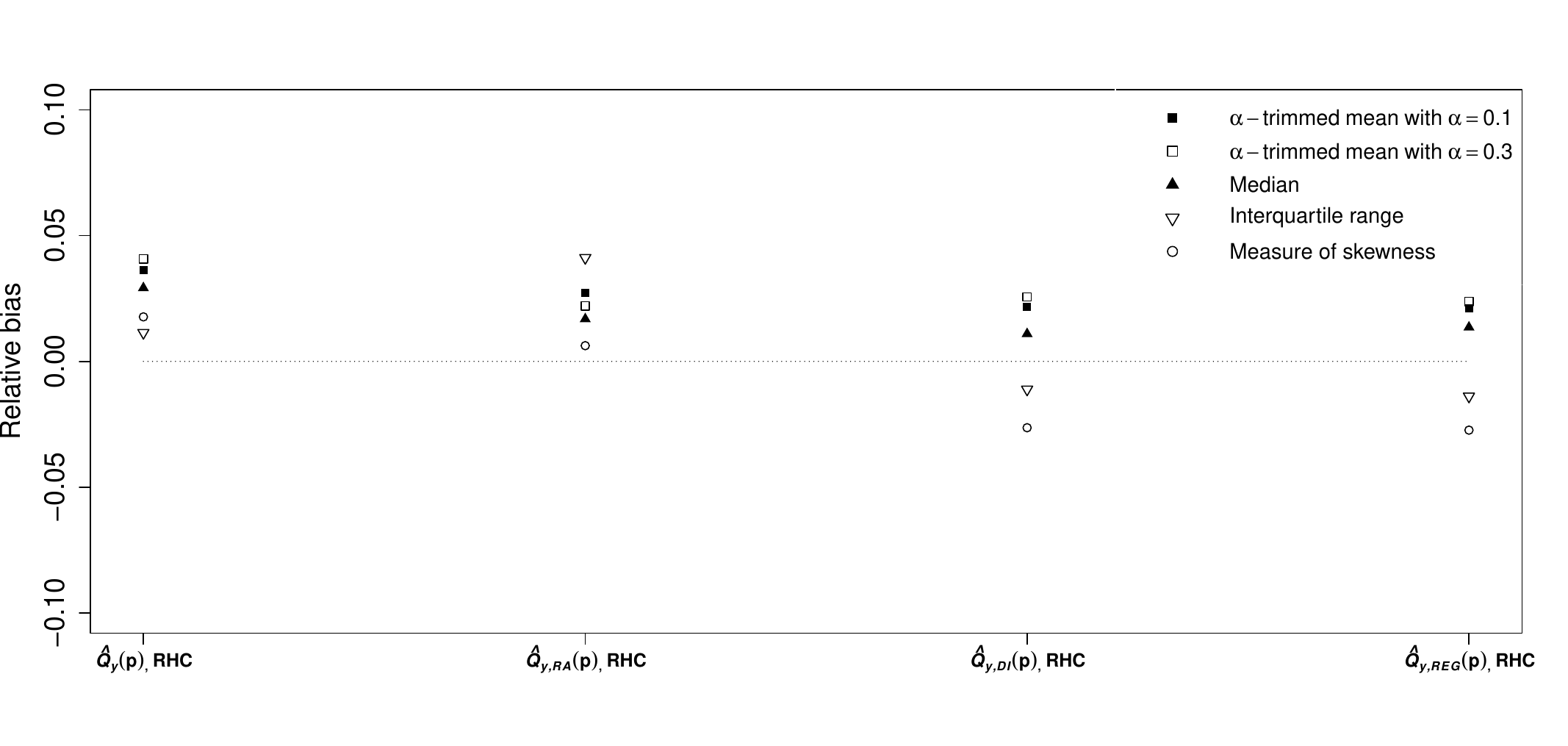}
\caption{Relative biases of different estimators for $n$=$200$ in the case of RHC sampling design.}
\label{Fig 7}
\end{center}
\end{figure}

\begin{figure}[h!]
\begin{center}
\includegraphics[height=9cm,width=14cm]{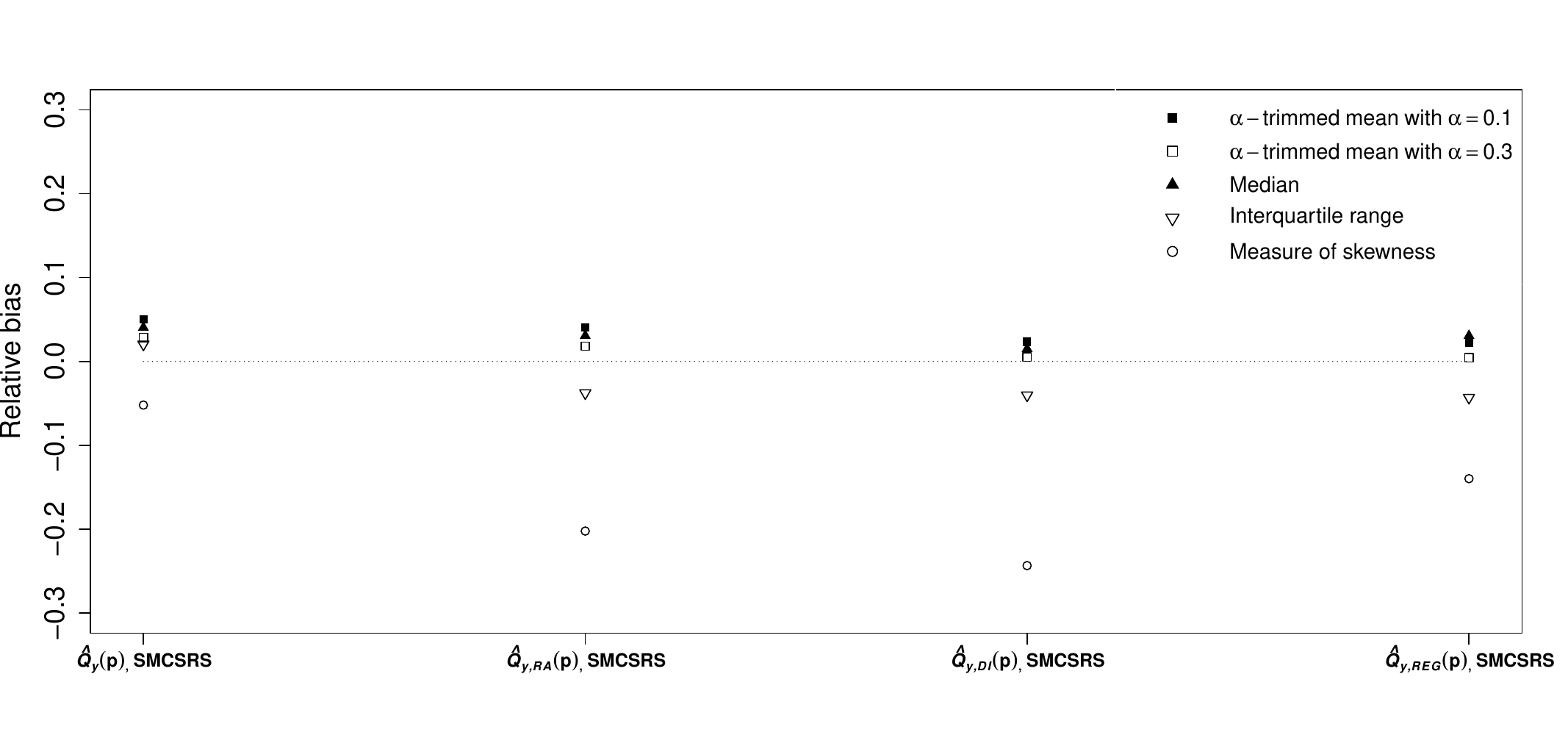}
\caption{Relative biases of different estimators for $n$=$108$ in the case of SMCSRS. In this figure, SMCSRS stands for stratified multistage cluster sampling design with SRSWOR.}
\label{Fig 8}
\end{center}
\end{figure}

\begin{figure}[h!]
\begin{center}
\includegraphics[height=8cm,width=14cm]{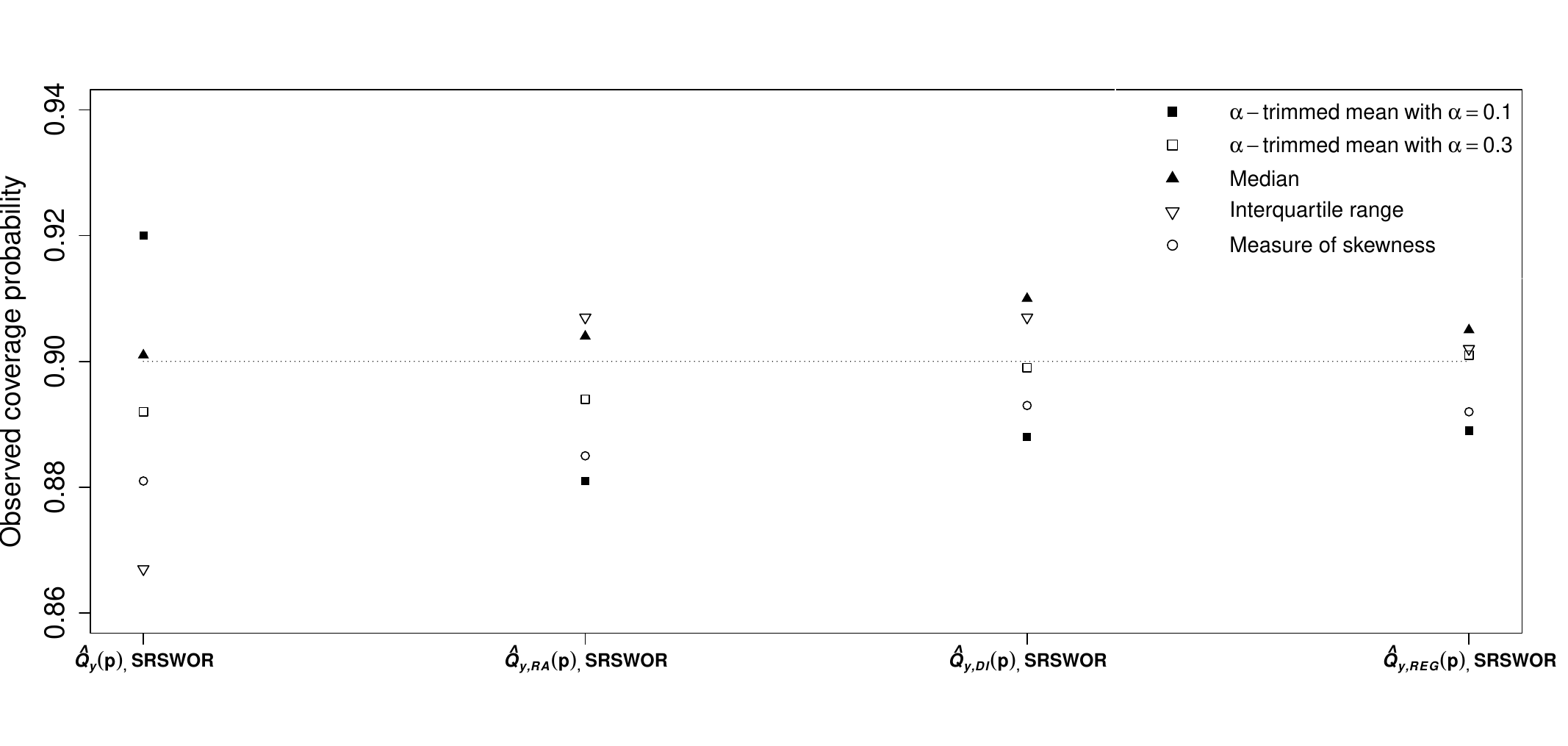}
\caption{Observed coverage probabilities of nominal $90 \%$ confidence intervals for $n$=$200$ in the case of SRSWOR  (the magnitude of the Monte Carlo standard error for observed coverage probabilities is $0.009$). }
\label{Fig 9}
\end{center}
\end{figure}

\begin{figure}[h!]
\begin{center}
\includegraphics[height=8cm,width=14cm]{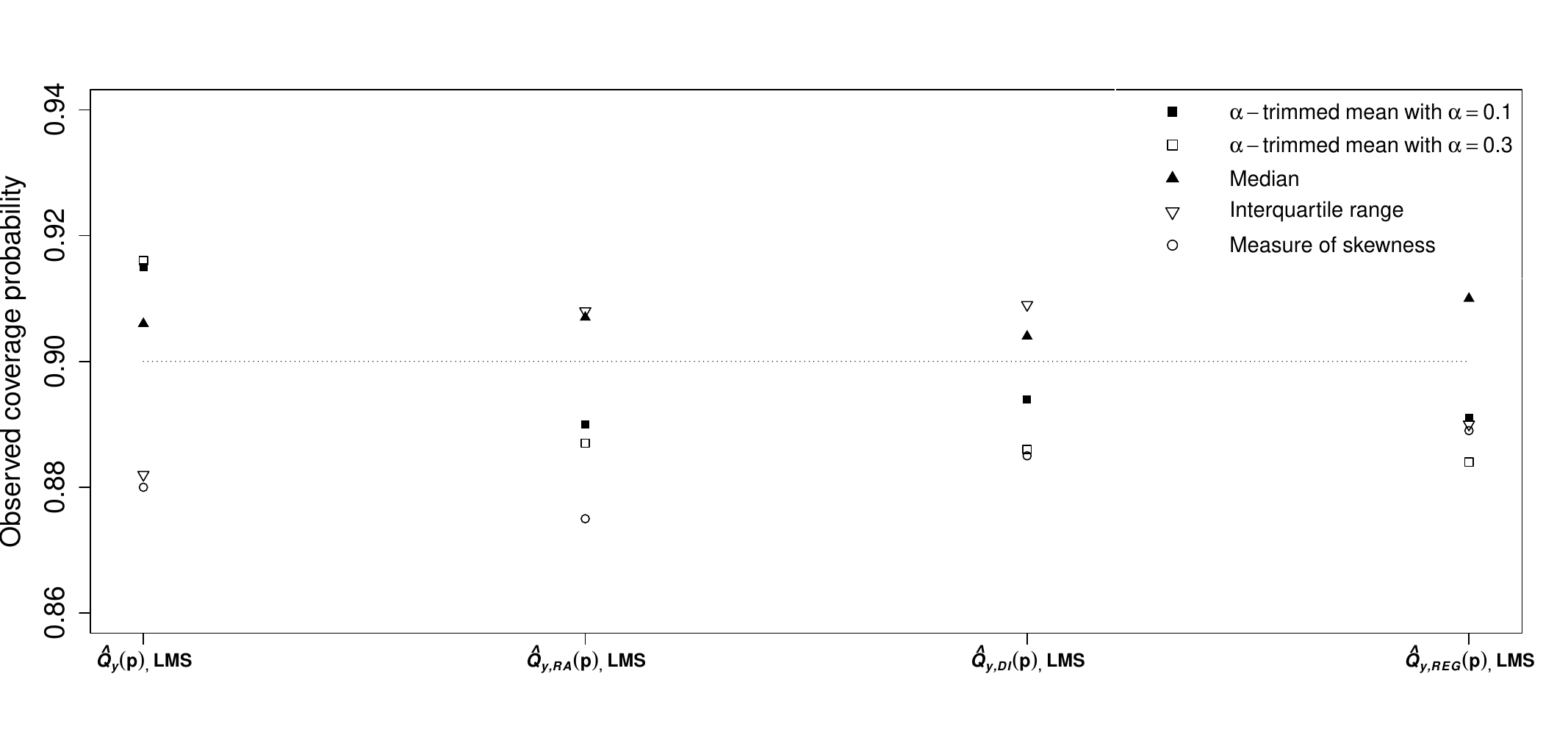}
\caption{Observed coverage probabilities of nominal $90 \%$ confidence intervals for $n$=$200$ in the case of LMS sampling design  (the magnitude of the Monte Carlo standard error for observed coverage probabilities is $0.009$). }
\label{Fig 10}
\end{center}
\end{figure}

\begin{figure}[h!]
\begin{center}
\includegraphics[height=8cm,width=14cm]{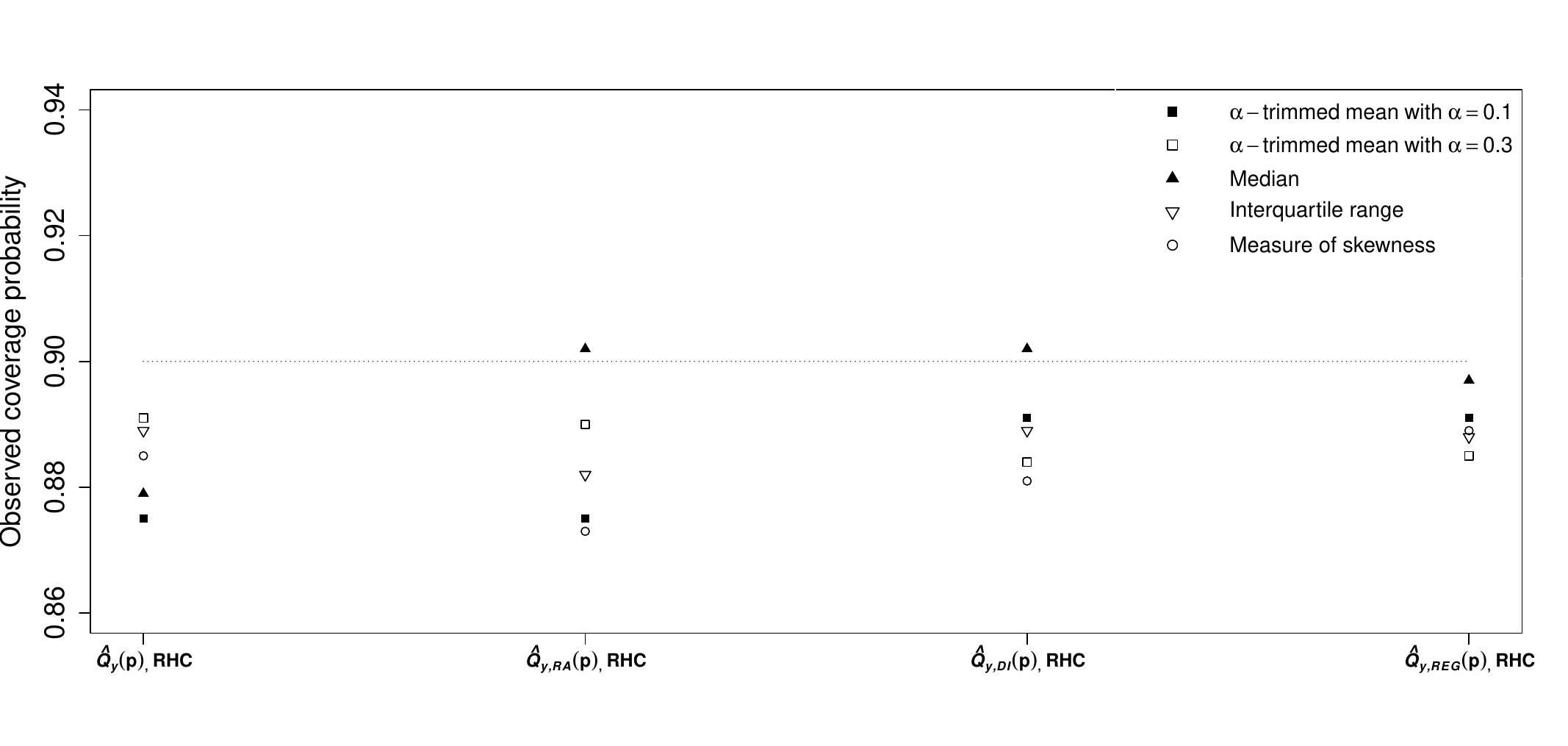}
\caption{Observed coverage probabilities of nominal $90 \%$ confidence intervals for $n$=$200$ in the case of RHC sampling design  (the magnitude of the Monte Carlo standard error for observed coverage probabilities is $0.009$). }
\label{Fig 11}
\end{center}
\end{figure}

\begin{figure}[h!]
\begin{center}
\includegraphics[height=8cm,width=14cm]{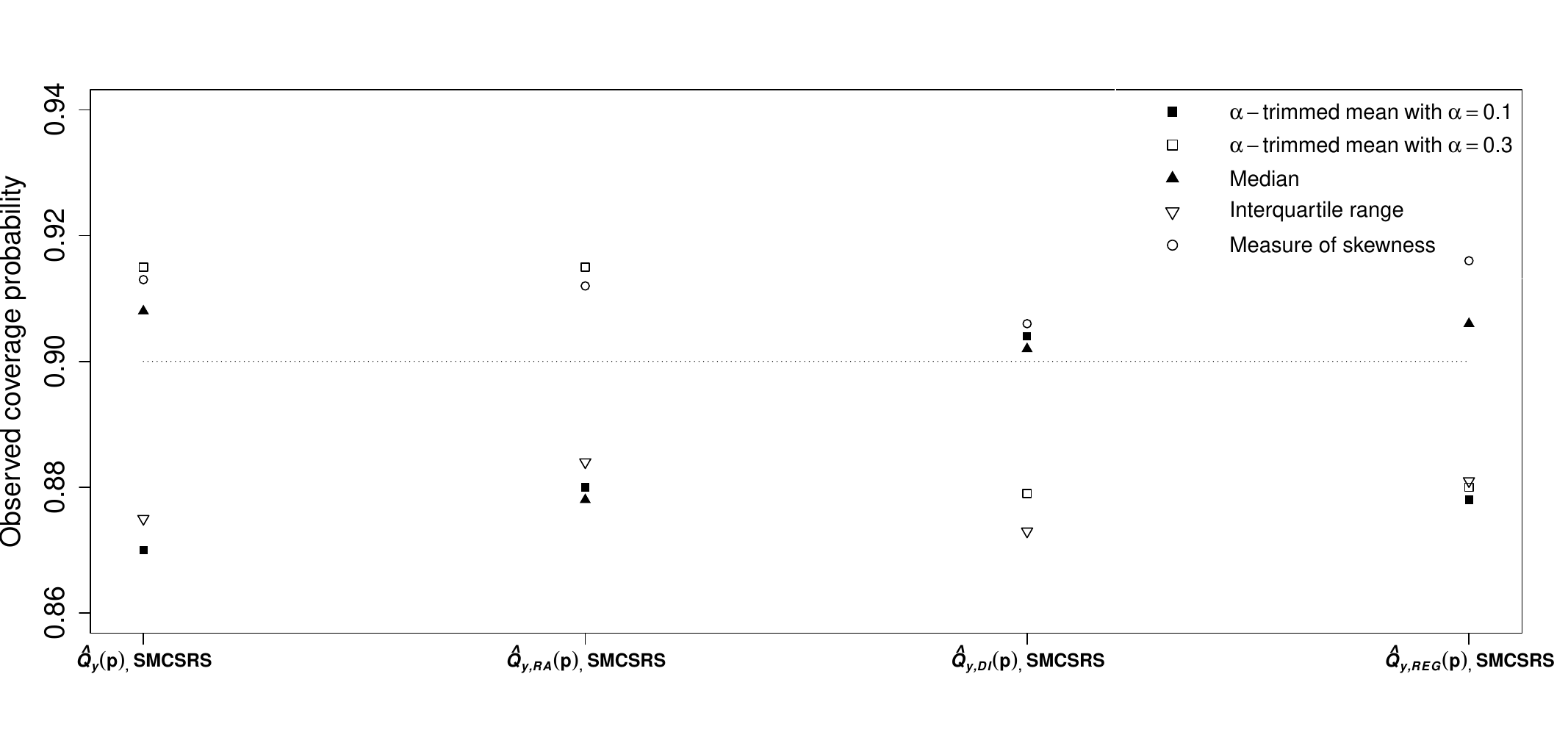}
\caption{Observed coverage probabilities of nominal $90 \%$ confidence intervals for $n$=$108$ in the case of SMCSRS  (the magnitude of the Monte Carlo standard error for observed coverage probabilities is $0.009$).  In this figure, SMCSRS stands for stratified multistage cluster sampling design with SRSWOR.}
\label{Fig 12}
\end{center}
\end{figure}

\begin{figure}[h!]
\begin{center}
\includegraphics[height=8cm,width=14cm]{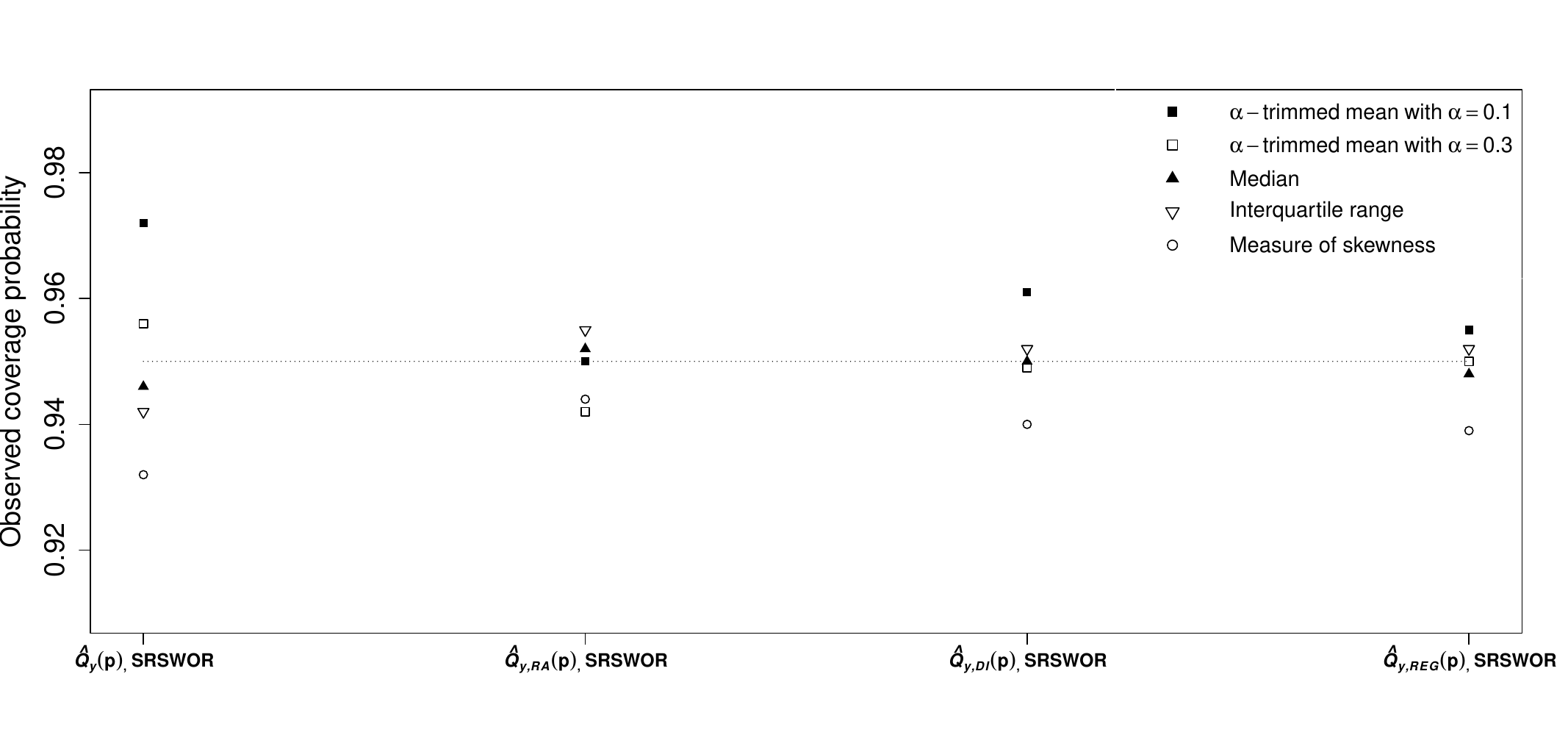}
\caption{Observed coverage probabilities of nominal $95 \%$ confidence intervals for $n$=$200$ in the case of SRSWOR  (the magnitude of the Monte Carlo standard error for observed coverage probabilities is $0.007$). }
\label{Fig 13}
\end{center}
\end{figure}

\begin{figure}[h!]
\begin{center}
\includegraphics[height=8cm,width=14cm]{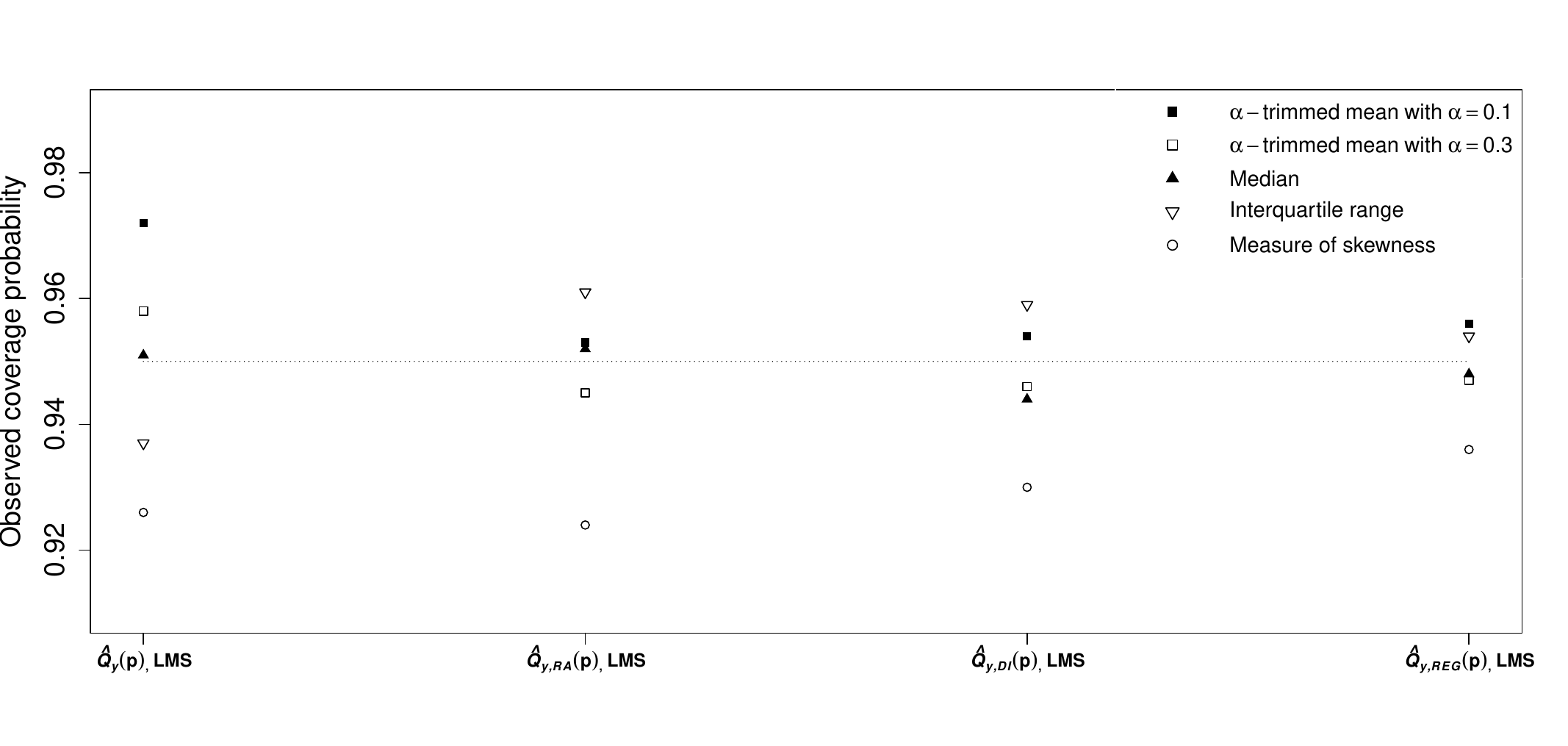}
\caption{Observed coverage probabilities of nominal $95 \%$ confidence intervals for $n$=$200$ in the case of LMS sampling design  (the magnitude of the Monte Carlo standard error for observed coverage probabilities is $0.007$). }
\label{Fig 14}
\end{center}
\end{figure}

\begin{figure}[h!]
\begin{center}
\includegraphics[height=8cm,width=14cm]{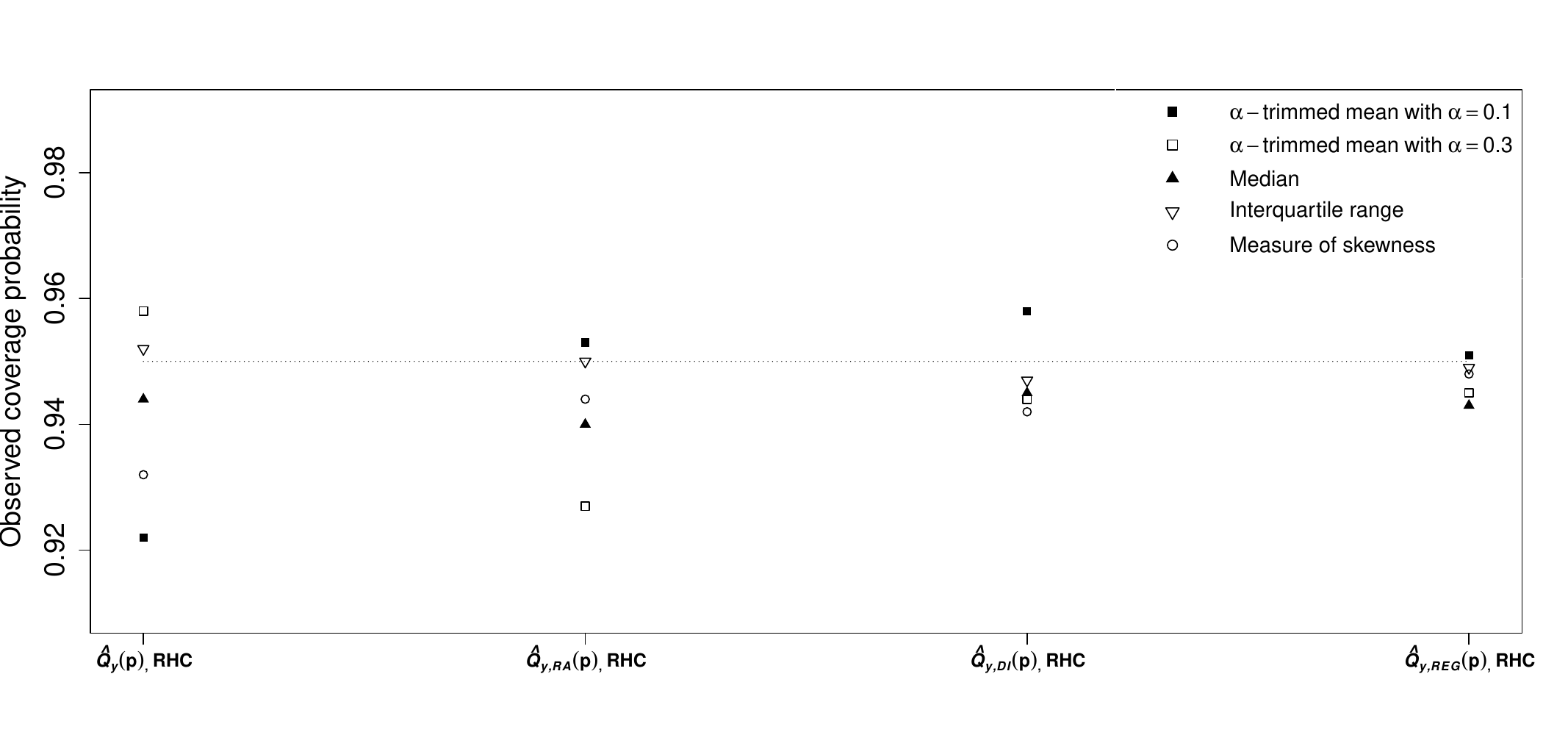}
\caption{Observed coverage probabilities of nominal $95 \%$ confidence intervals for $n$=$200$ in the case of RHC sampling design  (the magnitude of the Monte Carlo standard error for observed coverage probabilities is $0.007$). }
\label{Fig 15}
\end{center}
\end{figure}

\begin{figure}[h!]
\begin{center}
\includegraphics[height=8cm,width=14cm]{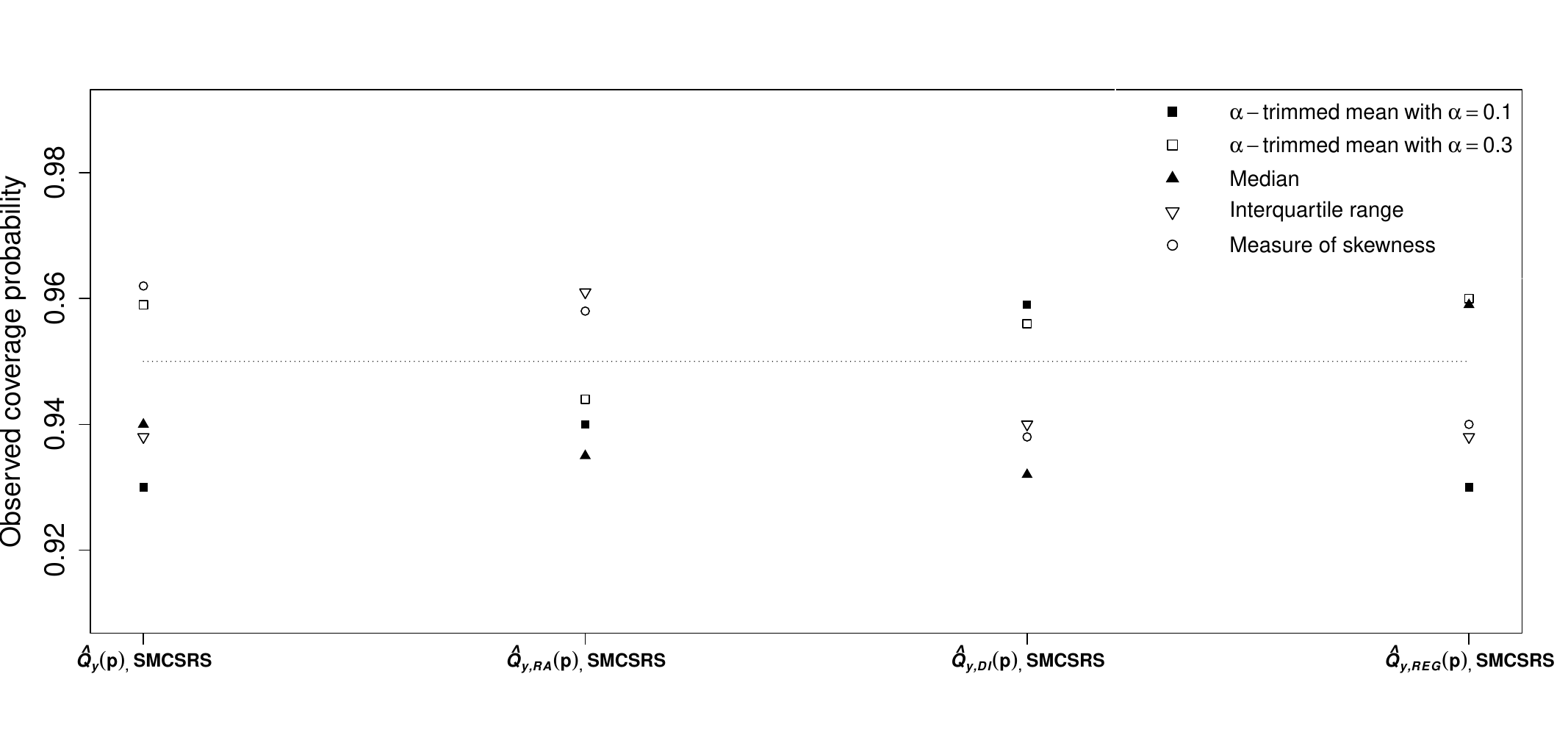}
\caption{Observed coverage probabilities of nominal $95 \%$ confidence intervals for $n$=$108$ in the case of SMCSRS  (the magnitude of the Monte Carlo standard error for observed coverage probabilities is $0.007$).  In this figure, SMCSRS stands for stratified multistage cluster sampling design with SRSWOR.}
\label{Fig 16}
\end{center}
\end{figure}

\begin{figure}[h]
\begin{center}
\includegraphics[height=8cm,width=14cm]{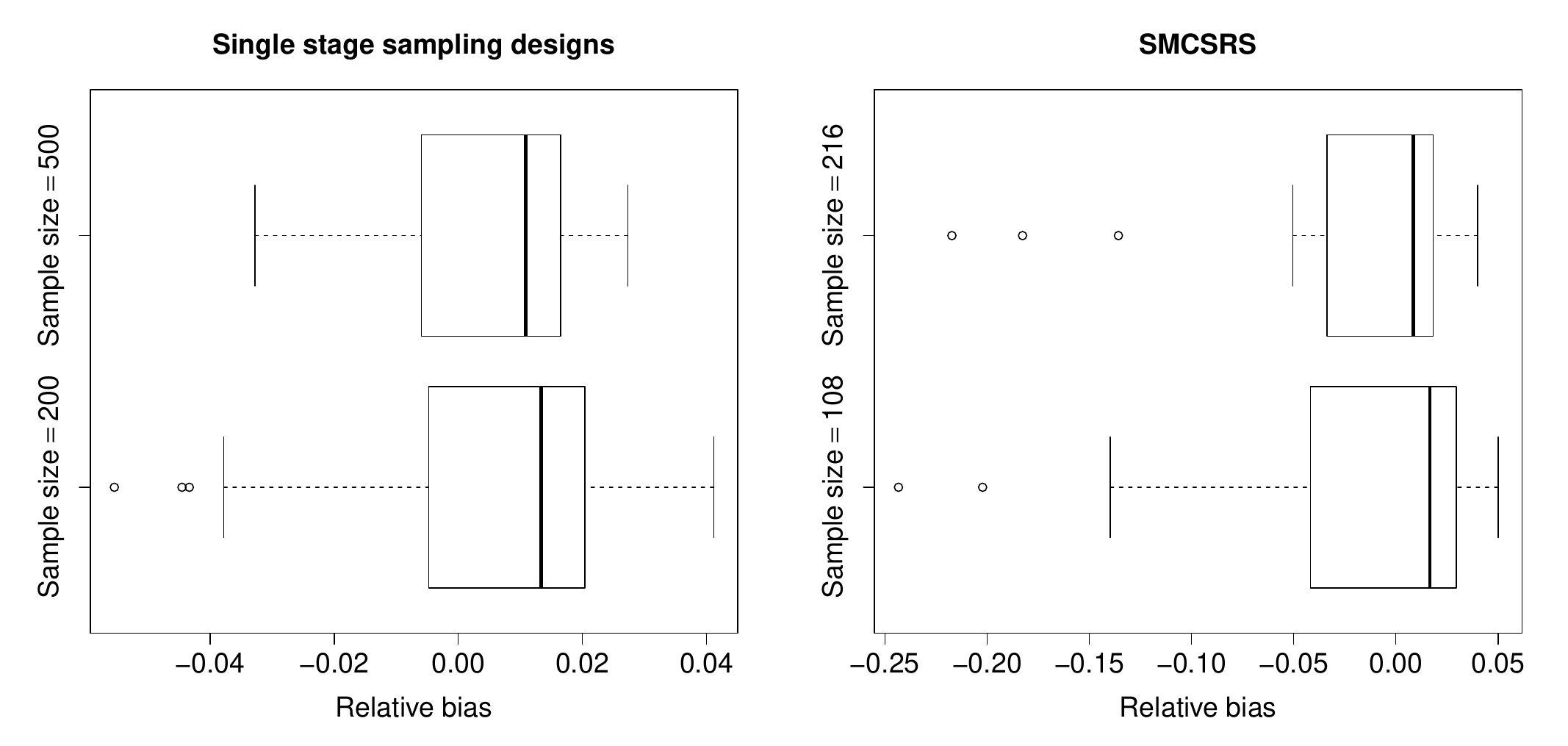}
\caption{Boxplots of relative biases for different parameters and estimators in the cases of single stage sampling designs and SMCSRS. In this figure, SMCSRS stands for stratified multistage cluster sampling design with SRSWOR.}
\label{Fig 33}
\end{center}
\end{figure}

\begin{figure}[h]
\begin{center}
\includegraphics[height=8cm,width=14cm]{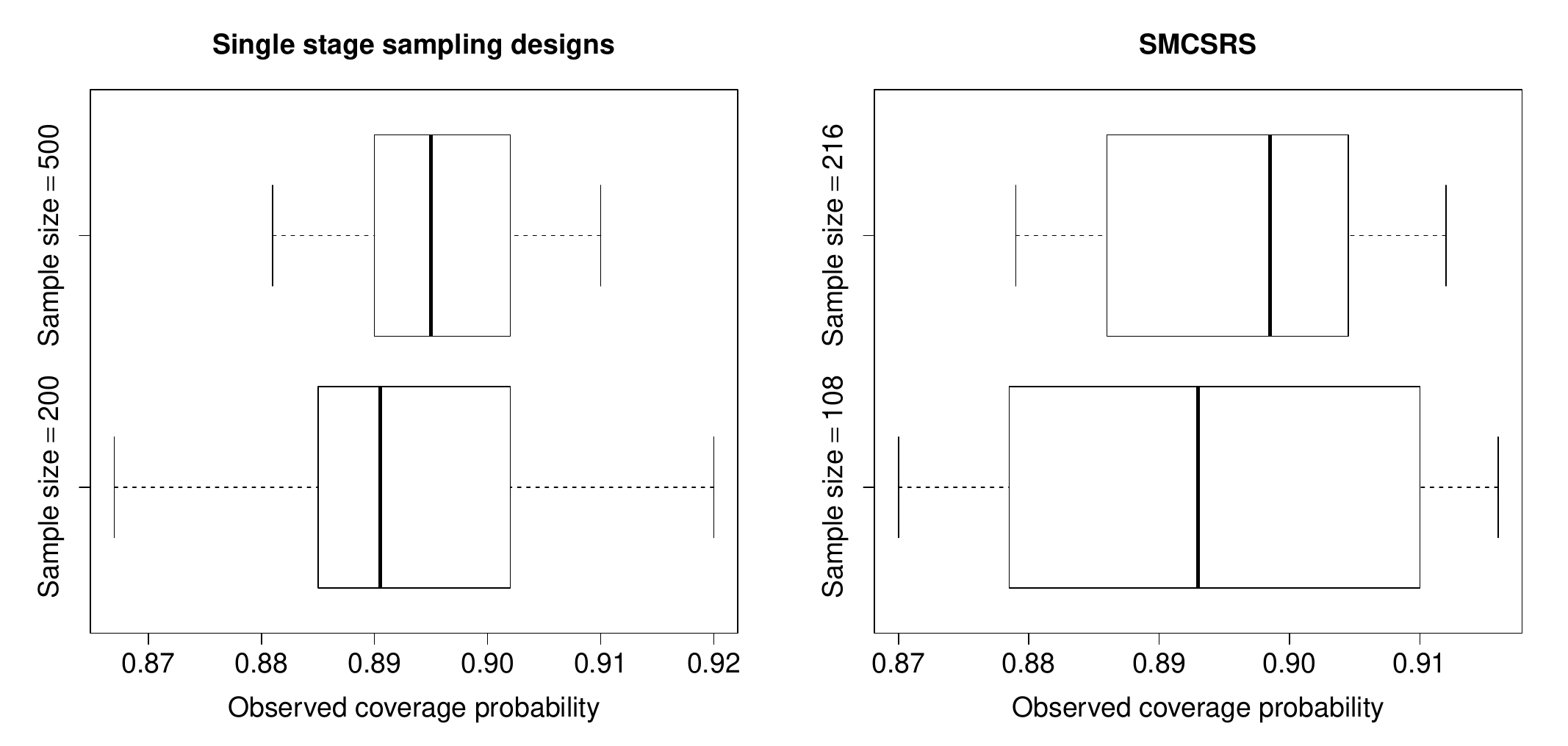}
\caption{Boxplots of observed coverage probabilities of nominal $90\%$ confidence intervals for different estimators and parameters in the cases of single stage sampling designs and SMCSRS. In this figure, SMCSRS stands for stratified multistage cluster sampling design with SRSWOR.}
\label{Fig 35}
\end{center}
\end{figure}
\clearpage

\begin{figure}[h]
\begin{center}
\includegraphics[height=8cm,width=14cm]{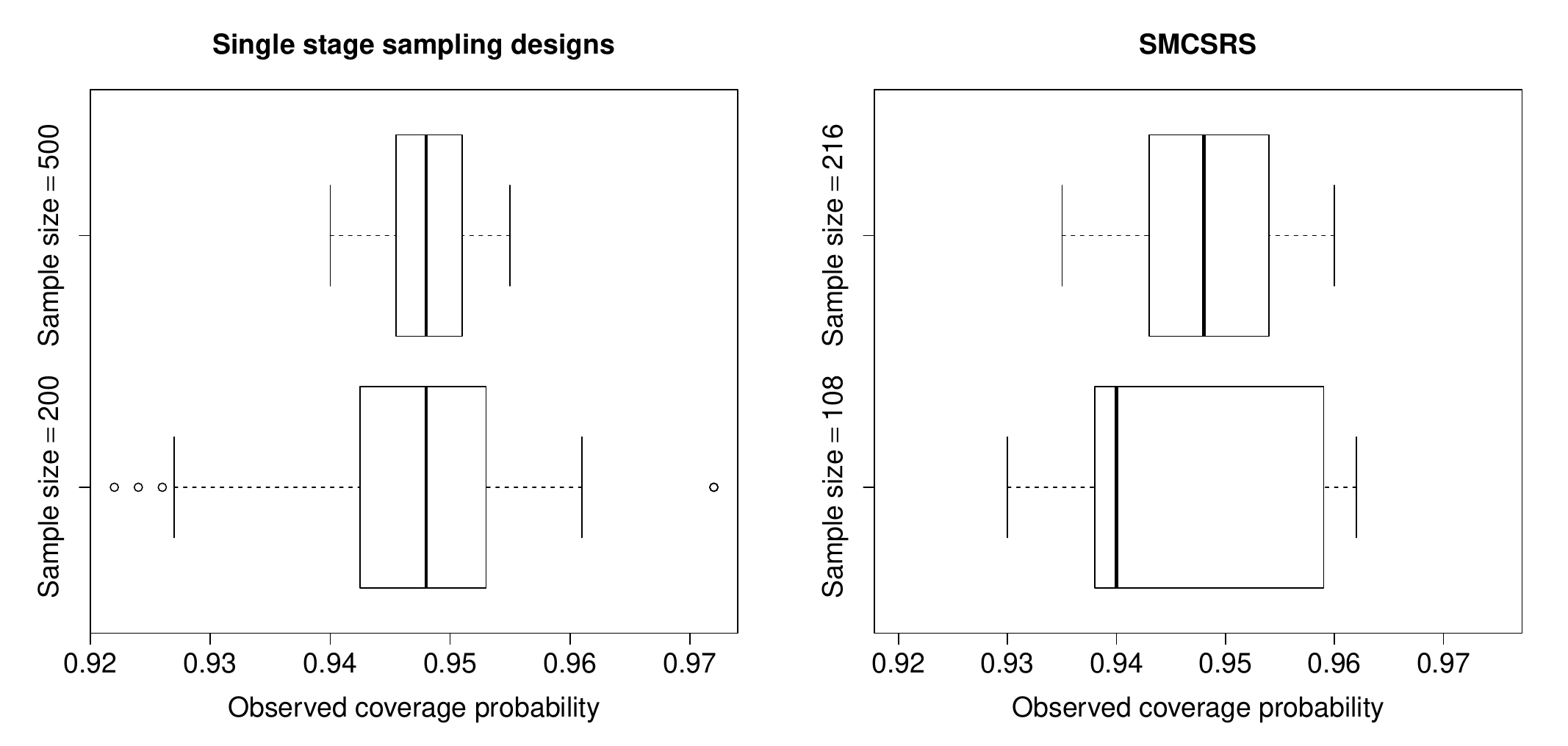}
\caption{Boxplots of observed coverage probabilities of nominal $95\%$ confidence intervals for different estimators and parameters in the cases of single stage sampling designs and SMCSRS. In this figure, SMCSRS stands for stratified multistage cluster sampling design with SRSWOR.}
\label{Fig 36}
\end{center}
\end{figure}
\subsection{Results of data analysis carried out in Section $6$ of this supplement}  The results obtained from the analysis described in Section \ref{sec 6} of this supplement  are summarized in this section. Here, we provide some plots that were mentioned in Section \ref{sec 6} of this supplement. Figures \ref{Fig 17}--\ref{Fig 4} present plots of ratios of asymptotic and true MSEs.  Figure \ref{Fig 34} presents boxplots of ratios of asymptotic and true MSEs for different parameters and estimators in the cases of single stage sampling designs and stratified multistage cluster sampling design with SRSWOR.  Table \ref{table 4} presents the estimators having least asymptotic and true MSEs for different parameters and sampling designs. Table \ref{table 5} presents the sampling designs under which various estimators of various parameters have the least asymptotic and true MSEs. Figure \ref{Fig 37} presents the scatter plot and the regression line between the study and the auxiliary variables.
\clearpage

\begin{figure}[h!]
\begin{center}
\includegraphics[height=8.5cm,width=14cm]{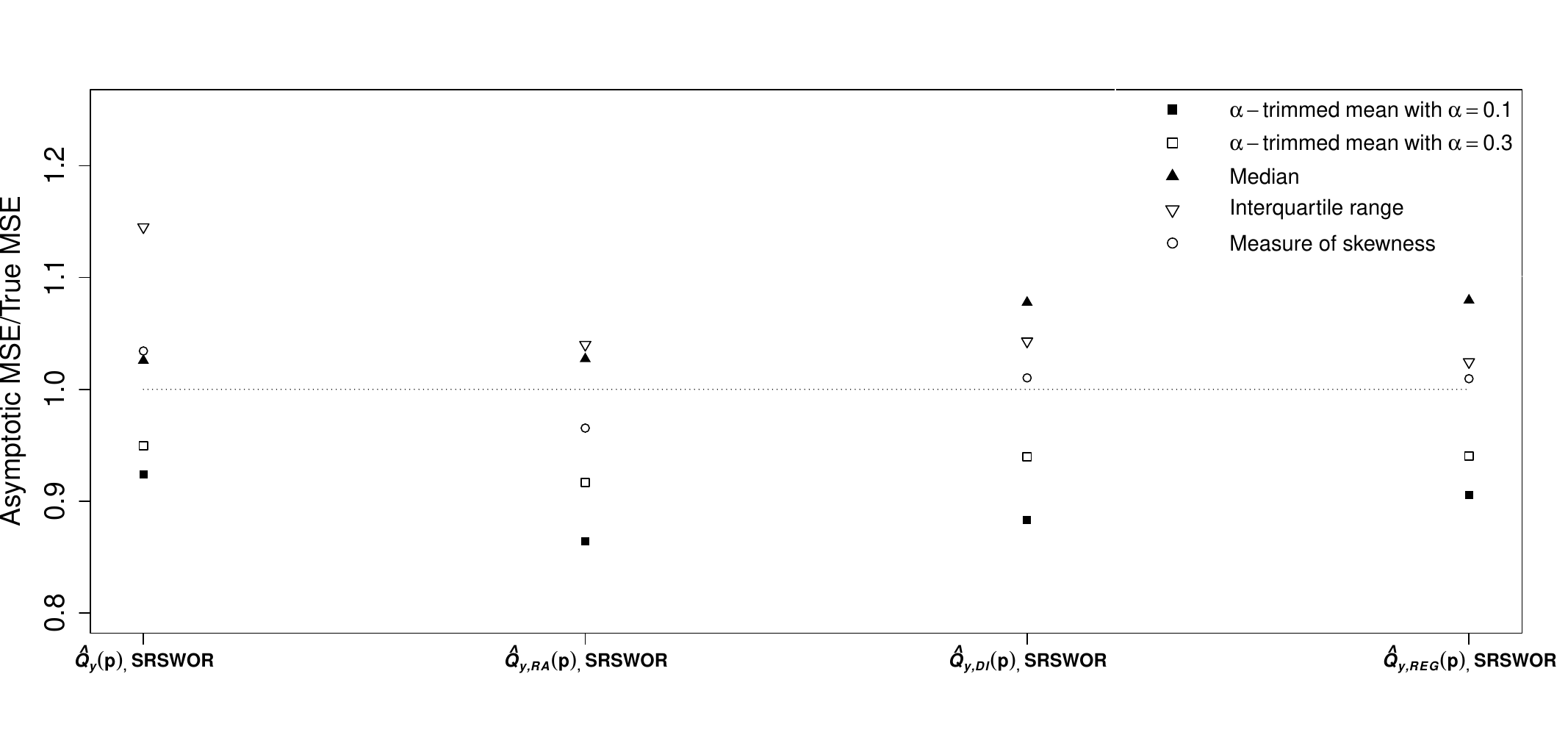}
\caption{Ratios of asymptotic and true MSEs of different estimators for $n$=$500$ in the case of SRSWOR.}
\label{Fig 17}
\end{center}
\end{figure}

\begin{figure}[h!]
\begin{center}
\includegraphics[height=8.5cm,width=14cm]{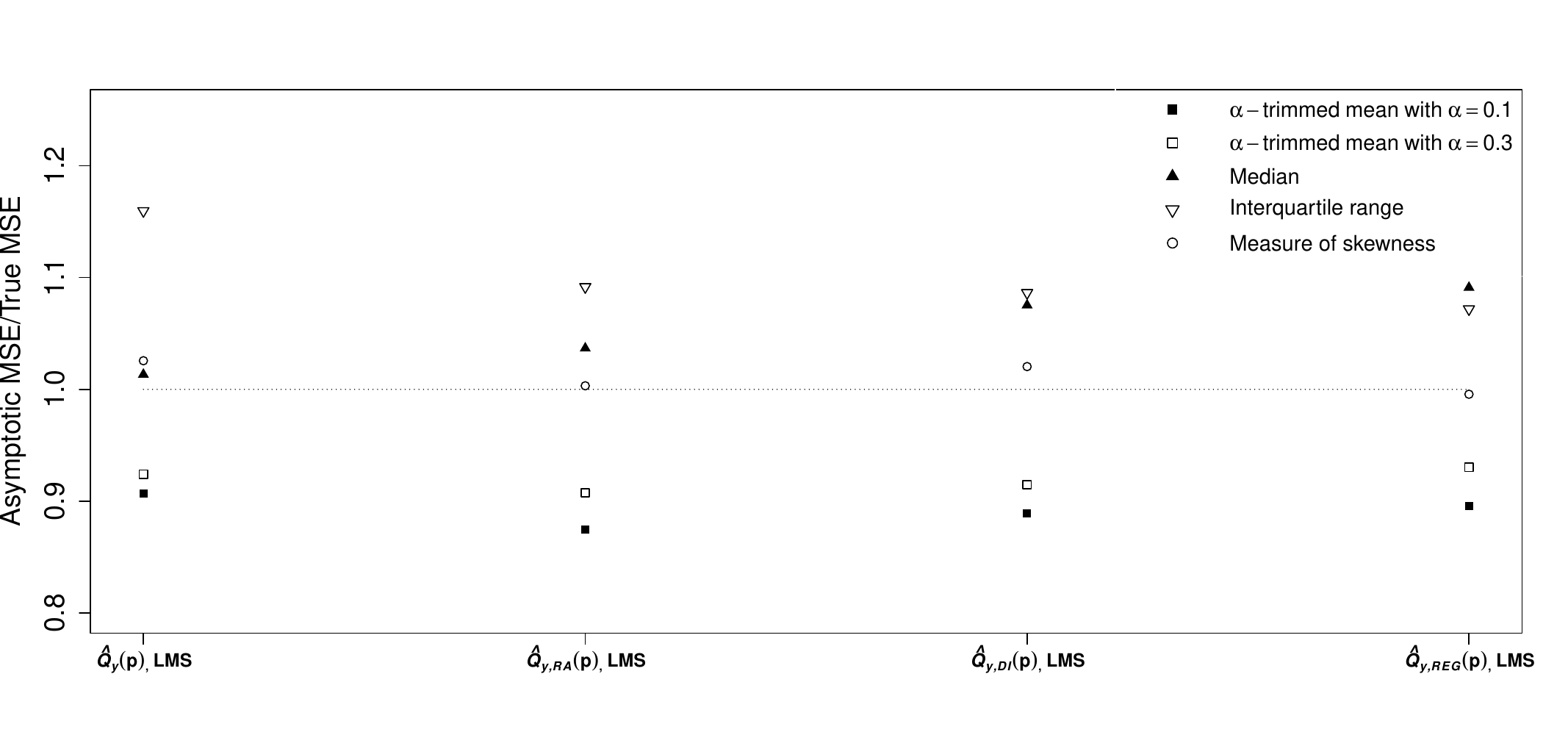}
\caption{Ratios of asymptotic and true MSEs of different estimators for $n$=$500$ in the case of LMS sampling design.}
\label{Fig 18}
\end{center}
\end{figure}

\begin{figure}[h!]
\begin{center}
\includegraphics[height=8.5cm,width=14cm]{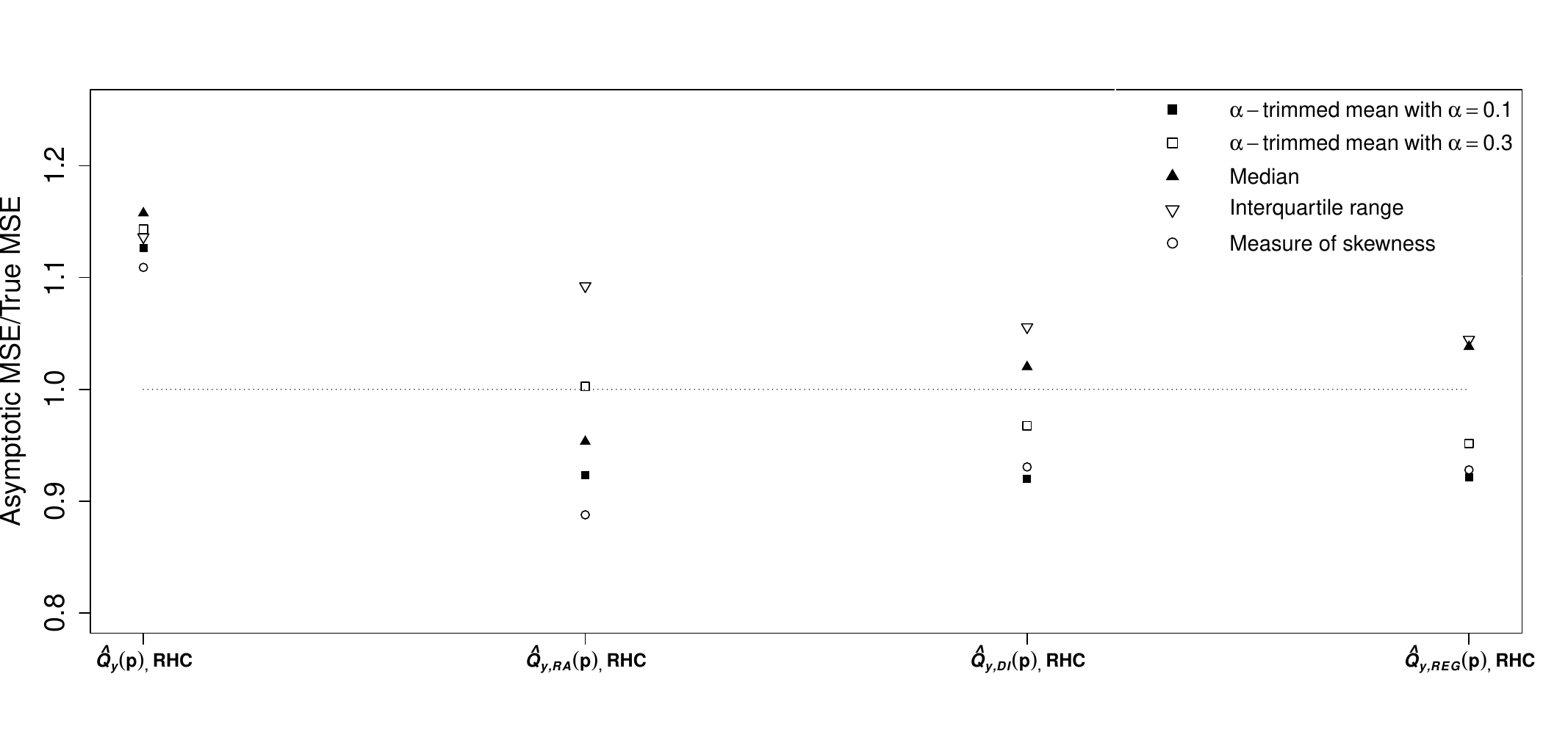}
\caption{Ratios of asymptotic and true MSEs of different estimators for $n$=$500$ in the case of RHC sampling design.}
\label{Fig 19}
\end{center}
\end{figure}

\begin{figure}[h!]
\begin{center}
\includegraphics[height=8.5cm,width=14cm]{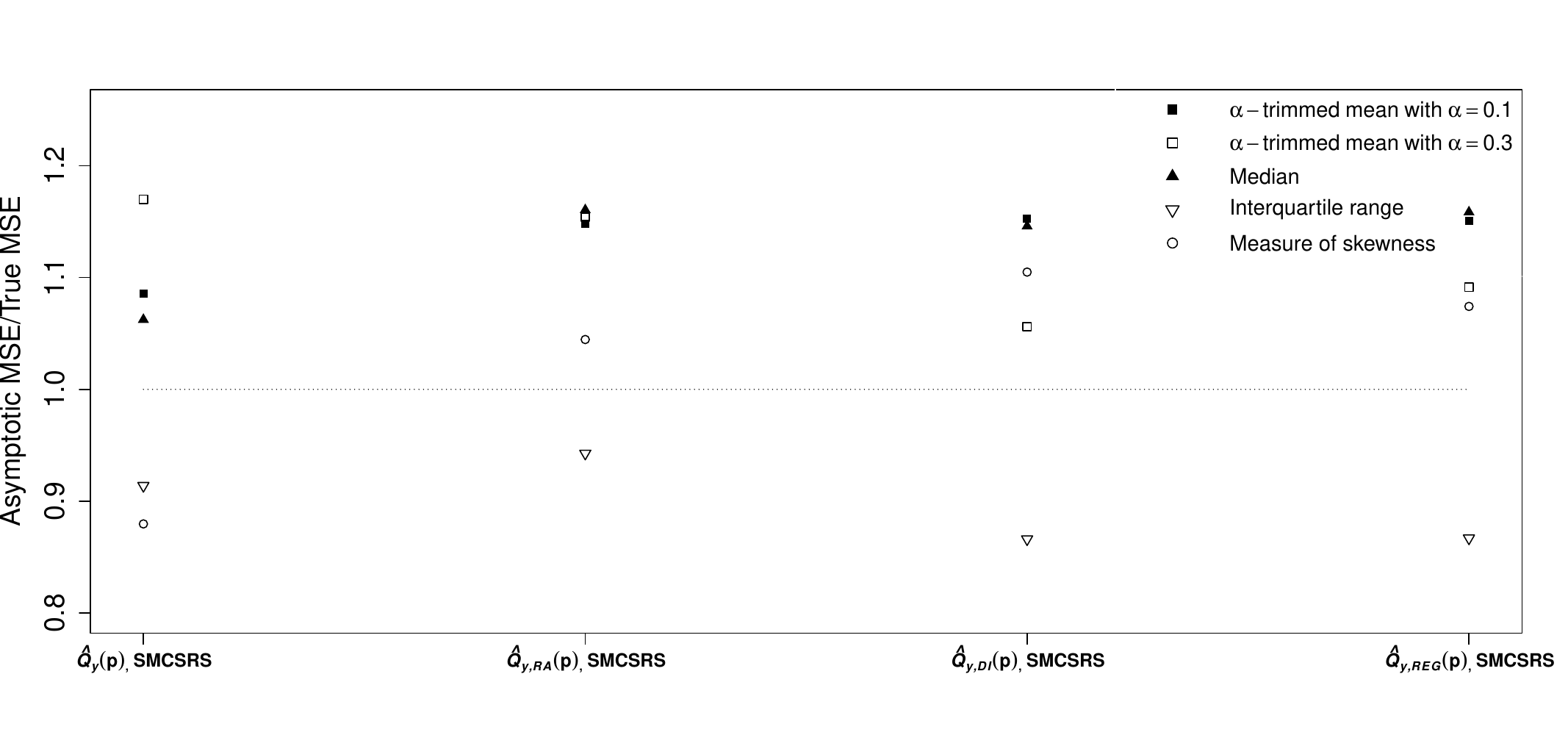}
\caption{Ratios of asymptotic and true MSEs of different estimators for $n$=$216$ in the case of SMCSRS. In this figure, SMCSRS stands for stratified multistage cluster sampling design with SRSWOR.}
\label{Fig 20}
\end{center}
\end{figure}

\begin{figure}[h]
\begin{center}
\includegraphics[height=8.5cm,width=14cm]{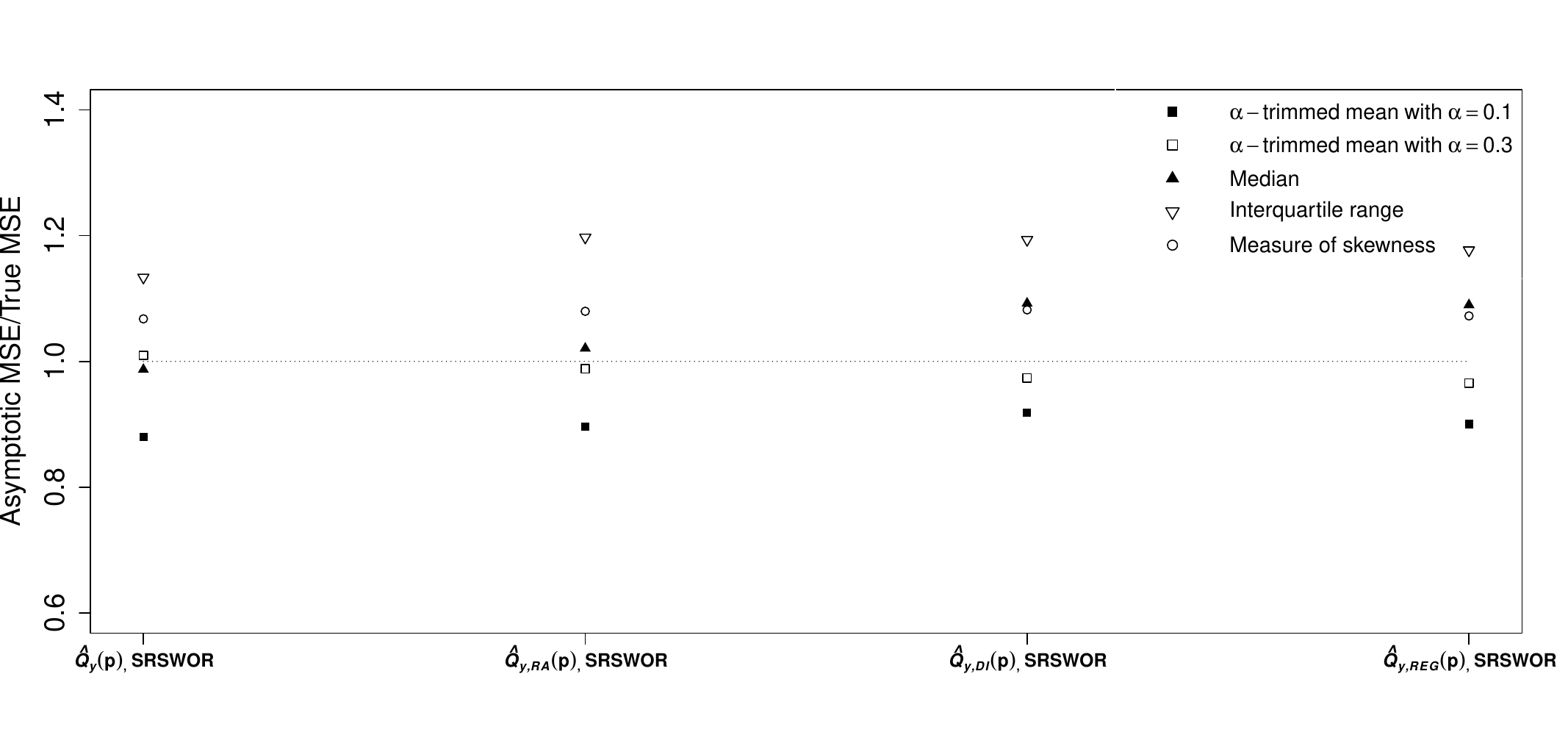}
\caption{Ratios of asymptotic and true MSEs of different estimators for $n$=$200$ in the case of SRSWOR.}
\label{Fig 1}
\end{center}
\end{figure}

\begin{figure}[h!]
\begin{center}
\includegraphics[height=8.5cm,width=14cm]{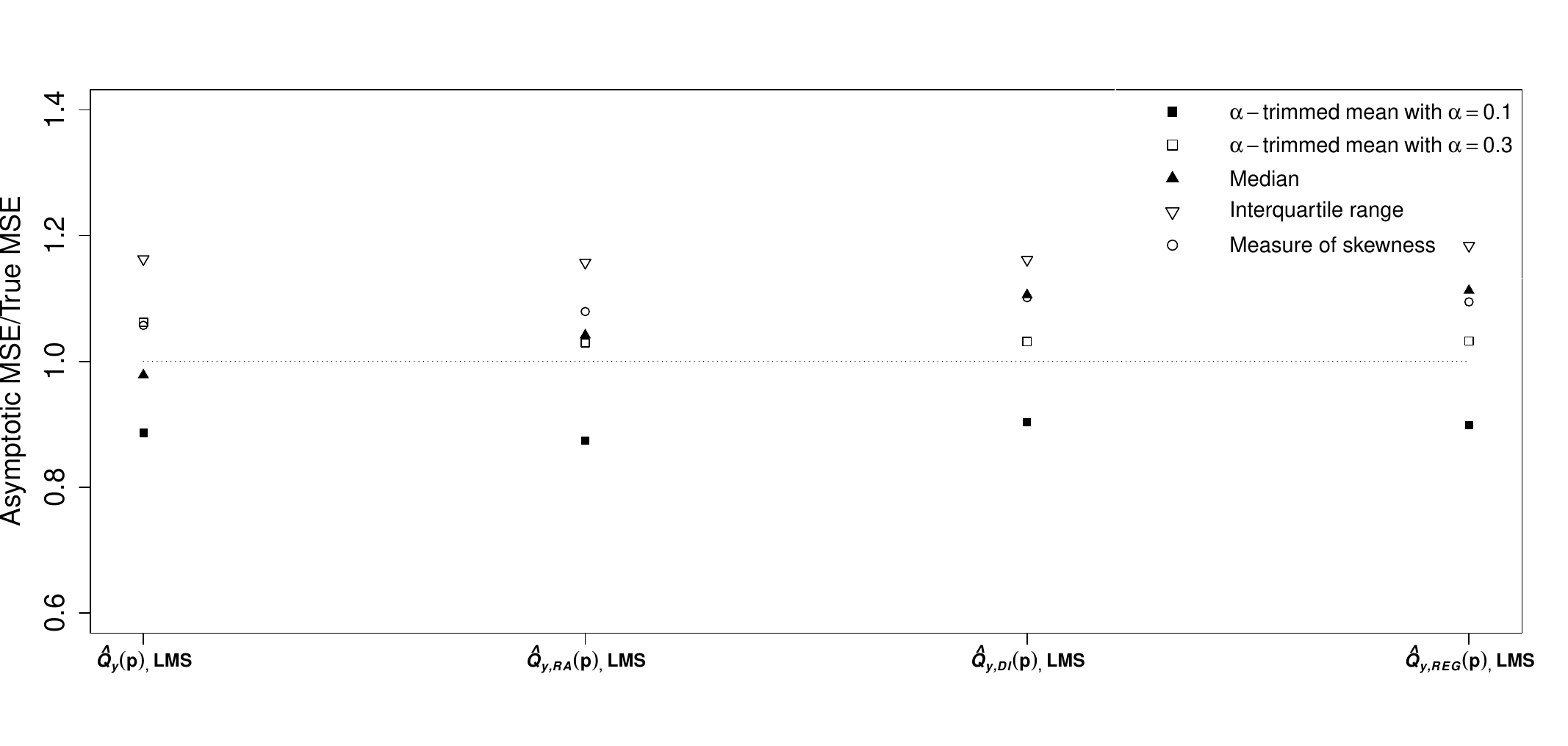}
\caption{Ratios of asymptotic and true MSEs of different estimators for $n$=$200$ in the case of LMS sampling design.}
\label{Fig 2}
\end{center}
\end{figure}

\begin{figure}[h!]
\begin{center}
\includegraphics[height=8.5cm,width=14cm]{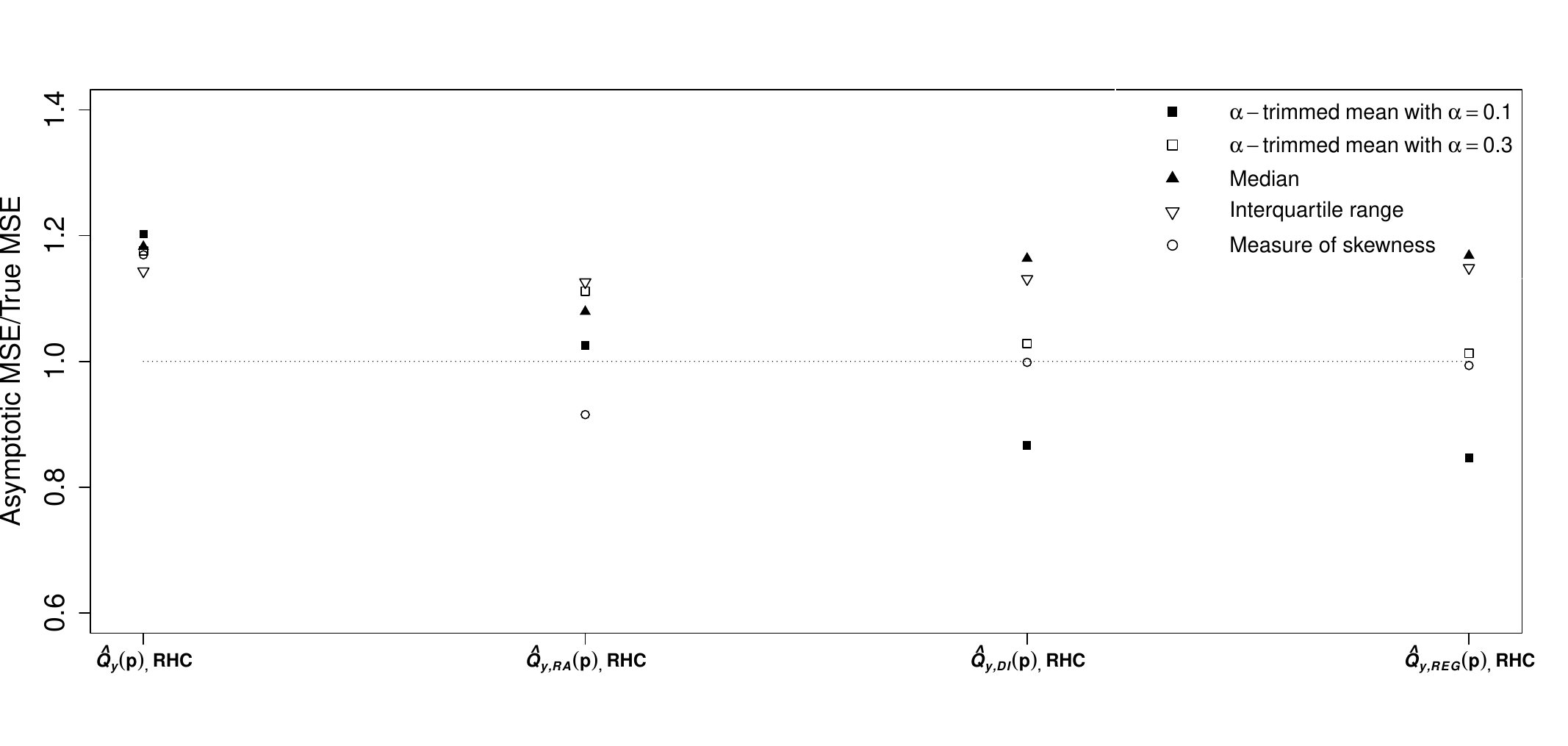}
\caption{Ratios of asymptotic and true MSEs of different estimators for $n$=$200$ in the case of RHC sampling design.}
\label{Fig 3}
\end{center}
\end{figure}

\begin{figure}[h!]
\begin{center}
\includegraphics[height=8.5cm,width=14cm]{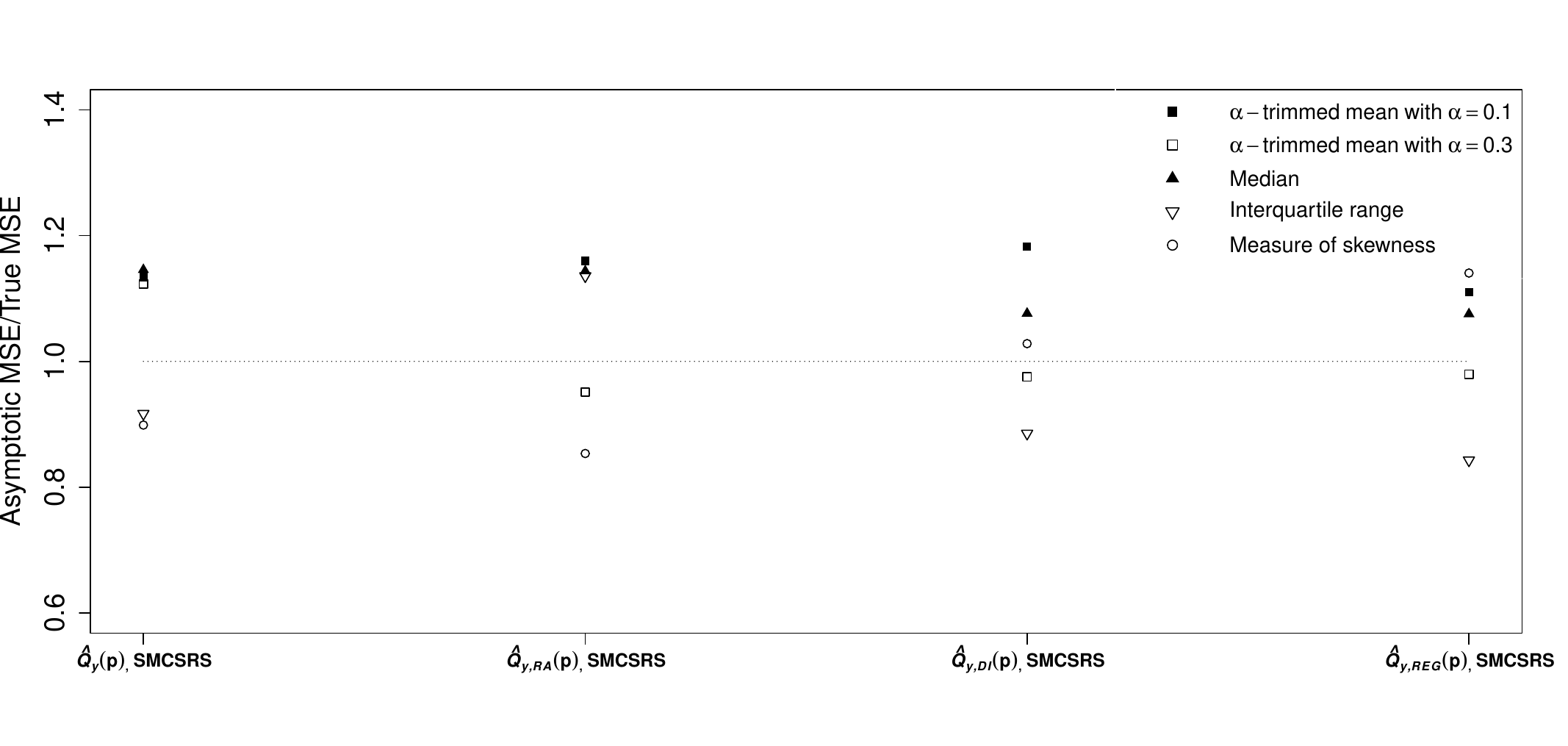}
\caption{Ratios of asymptotic and true MSEs of different estimators for $n$=$108$ in the case of SMCSRS. In this figure, SMCSRS stands for stratified multistage cluster sampling design with SRSWOR.}
\label{Fig 4}
\end{center}
\end{figure}
\clearpage

\begin{figure}[h!]
\begin{center}
\includegraphics[height=8cm,width=14cm]{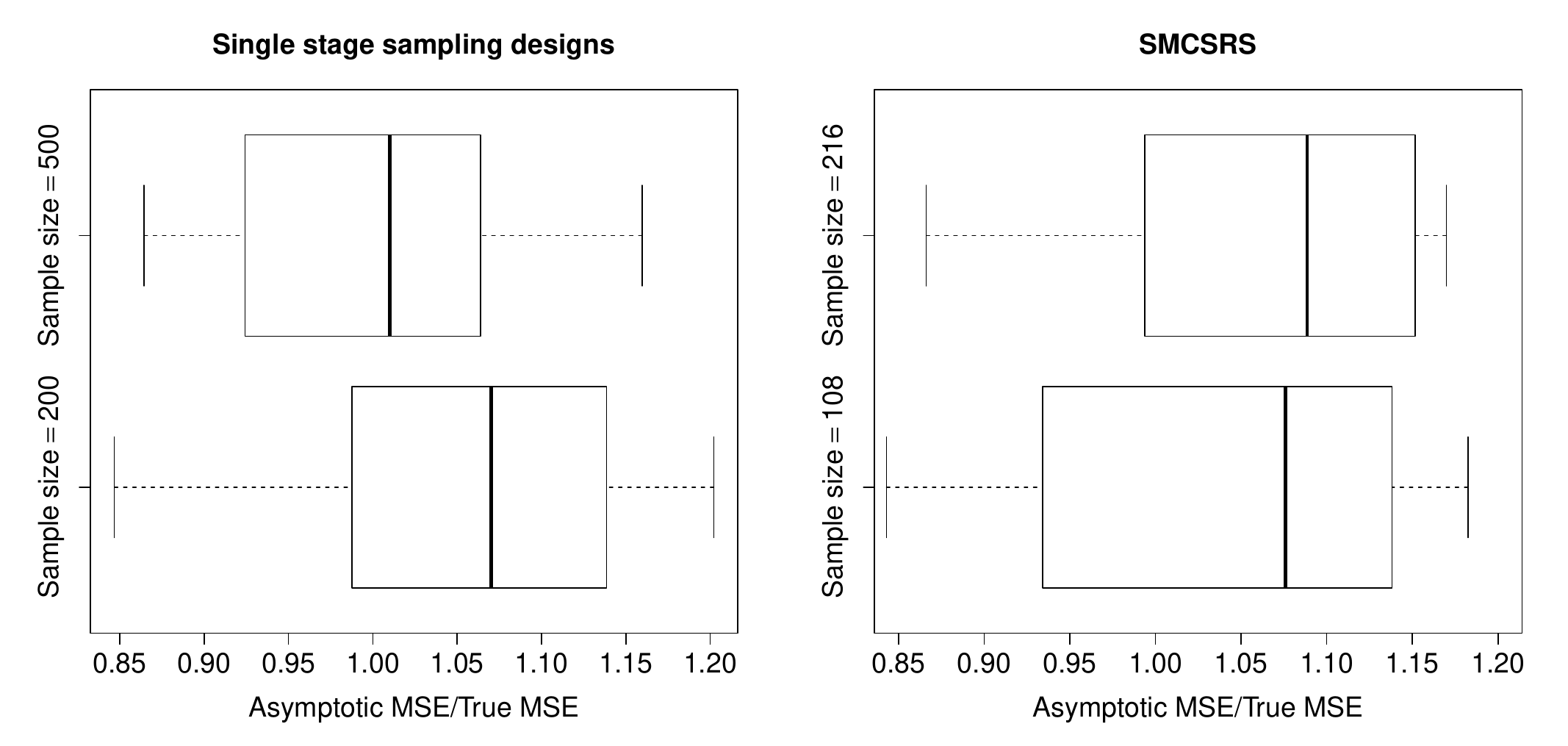}
\caption{Boxplots of ratios of asymptotic and true MSEs for different estimators and parameters in the cases of single stage sampling designs and SMCSRS. In this figure, SMCSRS stands for stratified multistage cluster sampling design with SRSWOR.}
\label{Fig 34}
\end{center}
\end{figure}

\begin{table}[h!] 
\centering  
\caption{Estimators with least asymptotic and true MSEs (note that nominal $90\%$ as well as $95\%$ confidence intervals based on these estimators have the least average lengths).}

\label{table 4}
\renewcommand{\arraystretch}{0.8}
\begin{tabular}{cccc}
\hline
&\multicolumn{3}{c}{ Sampling design}\\
\hline
Parameter&  SRSWOR& RHC& RS\\
 \hline
\multirow{2}{*}{ Median }& Estimator based on &Estimator based on&Estimator based on\\   
& $\hat{Q}_{y,RA}(p)$&$\hat{Q}_{y,DI}(p)$& $\hat{Q}_{y,DI}(p)$\\
\hline
$\alpha$-trimmed mean &Estimator based on &Estimator based on&Estimator based on\\  
with $\alpha$=$0.1$& $\hat{Q}_{y,DI}(p)$&$\hat{Q}_{y,DI}(p)$& $\hat{Q}_{y,DI}(p)$\\
\hline
$\alpha$-trimmed mean &Estimator based on &Estimator based on&Estimator based on\\  
with $\alpha$=$0.3$&  $\hat{Q}_{y,RA}(p)$&$\hat{Q}_{y,REG}(p)$& $\hat{Q}_{y,REG}(p)$\\
\hline
Interquartile &Estimator based on &Estimator based on&Estimator based on\\  
range& $\hat{Q}_{y,RA}(p)$&$\hat{Q}_{y,RA}(p)$& $\hat{Q}_{y,RA}(p)$\\
\hline
Bowley's measure &Estimator based on &Estimator based on&Estimator based on\\  
of skewness& $\hat{Q}_{y,RA}(p)$&$\hat{Q}_{y,DI}(p)$& $\hat{Q}_{y,DI}(p)$\\
\hline
\end{tabular}
\end{table}
 
\begin{table}[h!] 
\centering  
\caption{Sampling designs under which estimators have the least asymptotic and true MSEs (note that nominal $90\%$ as well as $95\%$ confidence intervals under these sampling designs have the least average lengths).}

\label{table 5}
\begin{tabular}{ccccc}
\hline
&\multicolumn{4}{c}{ Estimator based on}\\
\hline
Parameter& $\hat{Q}_y(p)$& $\hat{Q}_{y,RA}(p)$& $\hat{Q}_{y,DI}(p)$& $\hat{Q}_{y,REG}(p)$\\
 \hline
Median& SRSWOR&SRSWOR&SRSWOR& SRSWOR\\  
\hline
$\alpha$-trimmed mean with $\alpha$=$0.1$& SRSWOR&SRSWOR&RS& RS\\  
\hline
$\alpha$-trimmed mean with $\alpha$=$0.3$& SRSWOR&SRSWOR&SRSWOR& SRSWOR\\  
\hline
Interquartile range&SRSWOR&SRSWOR&SRSWOR& SRSWOR\\   
\hline
Bowley's measure of skewness& SRSWOR&SRSWOR&SRSWOR& SRSWOR\\    
\hline
\end{tabular}
\end{table} 

\begin{figure}[h!]
\begin{center}
\includegraphics[height=8.5cm,width=14cm]{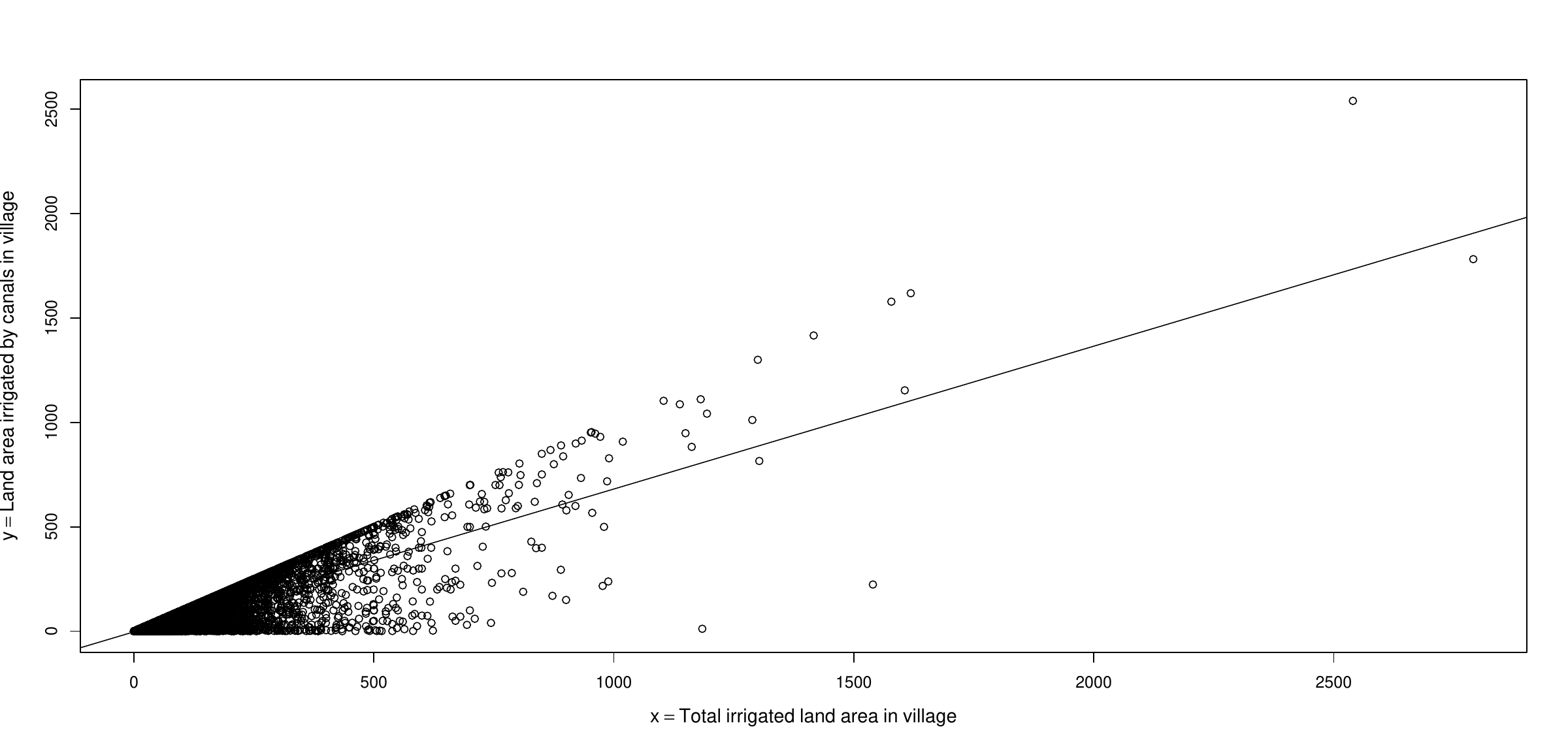}
\caption{Scatter plot and regression line between the study and the auxiliary variables.}
\label{Fig 37}
\end{center}
\end{figure}

\vskip .65cm
\noindent
Anurag Dey\\
{\it Indian Statistical Institute}
\vskip 2pt
\noindent
E-mail: deyanuragsaltlake64@gmail.com
\vskip 2pt

\noindent
Probal Chaudhuri\\
{\it Indian Statistical Institute}
\vskip 2pt
\noindent
E-mail:probalchaudhuri@gmail.com